\documentclass[numbers=enddot,12pt,final,onecolumn,notitlepage]{scrartcl}%
\usepackage[headsepline,footsepline,manualmark]{scrlayer-scrpage}
\usepackage[all,cmtip]{xy}
\usepackage{amssymb}
\usepackage{amsmath}
\usepackage{amsthm}
\usepackage{framed}
\usepackage{comment}
\usepackage{color}
\usepackage{hyperref}[breaklinks=True]
\usepackage[sc]{mathpazo}
\usepackage[T1]{fontenc}
\usepackage{tikz}
\usepackage{needspace}
\usepackage{tabls}
\usepackage{wasysym}
\providecommand{\U}[1]{\protect\rule{.1in}{.1in}}
\usetikzlibrary{arrows.meta}
\usetikzlibrary{chains}
\newcounter{exer}
\numberwithin{exer}{section}
\theoremstyle{definition}
\newtheorem{theo}{Theorem}[section]
\newenvironment{theorem}[1][]
{\begin{theo}[#1]\begin{leftbar}}
{\end{leftbar}\end{theo}}
\newtheorem{lem}[theo]{Lemma}
\newenvironment{lemma}[1][]
{\begin{lem}[#1]\begin{leftbar}}
{\end{leftbar}\end{lem}}
\newtheorem{prop}[theo]{Proposition}
\newenvironment{proposition}[1][]
{\begin{prop}[#1]\begin{leftbar}}
{\end{leftbar}\end{prop}}
\newtheorem{defi}[theo]{Definition}
\newenvironment{definition}[1][]
{\begin{defi}[#1]\begin{leftbar}}
{\end{leftbar}\end{defi}}
\newtheorem{remk}[theo]{Remark}
\newenvironment{remark}[1][]
{\begin{remk}[#1]\begin{leftbar}}
{\end{leftbar}\end{remk}}
\newtheorem{coro}[theo]{Corollary}
\newenvironment{corollary}[1][]
{\begin{coro}[#1]\begin{leftbar}}
{\end{leftbar}\end{coro}}
\newtheorem{conv}[theo]{Convention}
\newenvironment{convention}[1][]
{\begin{conv}[#1]\begin{leftbar}}
{\end{leftbar}\end{conv}}
\newtheorem{quest}[theo]{Question}
\newenvironment{question}[1][]
{\begin{quest}[#1]\begin{leftbar}}
{\end{leftbar}\end{quest}}
\newtheorem{warn}[theo]{Warning}

\newtheorem{conj}[theo]{Conjecture}

\newtheorem{exam}[theo]{Example}
\newenvironment{example}[1][]
{\begin{exam}[#1]\begin{leftbar}}
{\end{leftbar}\end{exam}}
\newtheorem{exmp}[exer]{Exercise}

\newenvironment{statement}{\begin{quote}}{\end{quote}}

\let\sumnonlimits\sum
\let\prodnonlimits\prod
\let\cupnonlimits\bigcup
\let\capnonlimits\bigcap
\let\sqcupnonlimits\bigsqcup
\let\oplusnonlimits\bigoplus
\renewcommand{\sum}{\sumnonlimits\limits}
\renewcommand{\prod}{\prodnonlimits\limits}
\renewcommand{\bigcup}{\cupnonlimits\limits}
\renewcommand{\bigcap}{\capnonlimits\limits}
\renewcommand{\bigsqcup}{\sqcupnonlimits\limits}
\renewcommand{\bigoplus}{\oplusnonlimits\limits}
\newenvironment{verlong}{}{}
\newenvironment{vershort}{}{}
\newenvironment{noncompile}{}{}

\excludecomment{verlong}
\includecomment{vershort}
\excludecomment{noncompile}
\excludecomment{teachingnote}

\newcommand{\powset}[2][]{\ifthenelse{\equal{#2}{}}{\mathcal{P}\left(#1\right)}{\mathcal{P}_{#1}\left(#2\right)}}

\ihead{The Redei--Berge symmetric functions}
\ohead{page \thepage}
\begin{document}

\title{The Redei--Berge symmetric function of a directed graph}
\author{Darij Grinberg and Richard P. Stanley}
\date{draft (Sections 1--2 finished, 3--8 outlined),
July 9, 2023}
\maketitle

\begin{abstract}
\textbf{Abstract.} Let $D=\left(  V,A\right)  $ be a digraph with $n$
vertices, where each arc $a\in A$ is a pair $\left(  u,v\right)  $ of two
vertices. We study the \emph{Redei--Berge symmetric function} $U_{D}$, defined
as the quasisymmetric function%
\[
\sum L_{\operatorname*{Des}\left(  w,D\right)  ,\ n}\in\operatorname*{QSym}.
\]
Here, the sum ranges over all lists $w=\left(  w_{1},w_{2},\ldots
,w_{n}\right)  $ that contain each vertex of $D$ exactly once, and the
corresponding addend is%
\[
L_{\operatorname*{Des}\left(  w,D\right)  ,\ n}:=\sum_{\substack{i_{1}\leq
i_{2}\leq\cdots\leq i_{n};\\i_{p}<i_{p+1}\text{ for each }p\text{ satisfying
}\left(  w_{p},w_{p+1}\right)  \in A}}x_{i_{1}}x_{i_{2}}\cdots x_{i_{n}}%
\]
(an instance of Gessel's fundamental quasisymmetric functions).

While $U_{D}$ is a specialization of Chow's path-cycle symmetric function,
which has been studied before, we prove some new formulas that express $U_{D}$
in terms of the power-sum symmetric functions. We show that $U_{D}$ is always
$p$-integral, and furthermore is $p$-positive whenever $D$ has no $2$-cycles.
When $D$ is a tournament, $U_{D}$ can be written as a polynomial in
$p_{1},2p_{3},2p_{5},2p_{7},\ldots$ with nonnegative integer coefficients. By
specializing these results, we obtain the famous theorems of Redei and Berge
on the number of Hamiltonian paths in digraphs and tournaments, as well as a
modulo-$4$ refinement of Redei's theorem. \medskip

\textbf{Keywords:} directed graph, symmetric function, tournament, Hamiltonian
path, power sum symmetric function.

\textbf{Mathematics Subject Classification 2020:} 05A05, 05A19, 05E05, 05C30.

\end{abstract}
\tableofcontents

\section{Definitions and the main theorems}

We begin with introducing the notations, some of which come from \cite[Problem
120]{EC2supp}. We use standard notations as defined, e.g., in \cite[Chapter
7]{EC2} and \cite[Chapters 2 and 5]{GriRei}.

\subsection{Digraphs, $V$-listings and $D$-descents}

We let $\mathbb{N}:=\left\{  0,1,2,\ldots\right\}  $ and $\mathbb{P}:=\left\{
1,2,3,\ldots\right\}  $. We set $\left[  n\right]  :=\left\{  1,2,\ldots
,n\right\}  $ for each $n\in\mathbb{Z}$. (This set $\left[  n\right]  $ is
empty if $n\leq0$.)

The words \textquotedblleft list\textquotedblright\ and \textquotedblleft
tuple\textquotedblright\ will be used interchangeably, and will always mean
finite ordered tuples.

We shall next introduce some basic notations regarding digraphs (i.e.,
directed graphs):

\begin{definition}
A \emph{digraph} means a pair $\left(  V,A\right)  $, where $V$ is a finite
set and where $A$ is a subset of $V\times V$. The elements of $V$ are called
the \emph{vertices} of this digraph, and the elements of $A$ are called the
\emph{arcs} of this digraph. For any further notations, we refer to standard
literature (the definitions in \cite[\S 1.1-\S 1.2]{17s-lec7} should suffice)
and common sense. (Our digraphs are allowed to have loops, but this has no
effect on what follows.)
\end{definition}

\begin{definition}
Let $D=\left(  V,A\right)  $ be a digraph. Then, the digraph $\left(
V,\ \left(  V\times V\right)  \setminus A\right)  $ will be denoted by
$\overline{D}$ and called the \emph{complement} of the digraph $D$. Its arcs
will be called the \emph{non-arcs} of $D$ (since they are precisely the pairs
$\left(  u,v\right)  \in V\times V$ that are not arcs of $D$).
\end{definition}

\Needspace{18pc}

\begin{example}
\label{exa.complement.3verts}If $D$ is the digraph%
\[
\left(  \left\{  1,2,3\right\}  ,\ \left\{  \left(  1,2\right)  ,\ \left(
2,2\right)  ,\ \left(  3,3\right)  \right\}  \right)  ,
\]
then its complement $\overline{D}$ is the digraph%
\[
\left(  \left\{  1,2,3\right\}  ,\ \left\{  \left(  1,1\right)  ,\ \left(
1,3\right)  ,\ \left(  2,1\right)  ,\ \left(  2,3\right)  ,\ \left(
3,1\right)  ,\ \left(  3,2\right)  \right\}  \right)  .
\]
Here are the two digraphs, drawn side by side:%
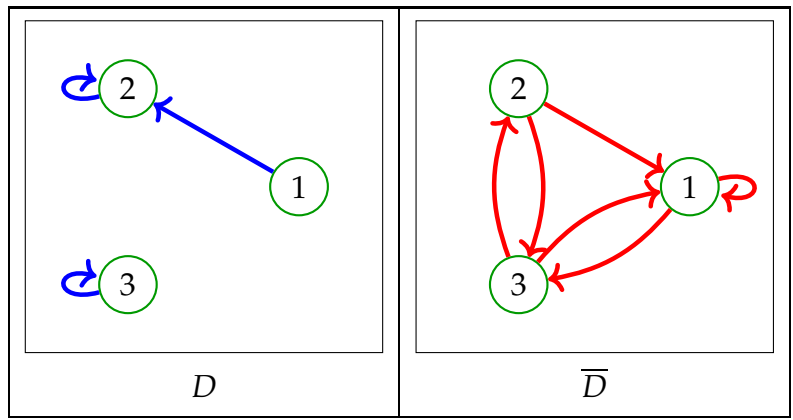
\[%
\begin{tabular}
[c]{|c|c|}\hline
$%
\begin{tikzpicture}
\draw[draw=black] (-2.1, -2.2) rectangle (2.6, 2.2);
\begin{scope}[every node/.style={circle,thick,draw=green!60!black}]
\node(1) at (0:1.5) {$1$};
\node(2) at (120:1.5) {$2$};
\node(3) at (240:1.5) {$3$};
\end{scope}
\begin{scope}[every edge/.style={draw=blue,line width=1.7pt}]
\path[->] (1) edge (2) (2) edge[loop left] (2) (3) edge[loop left] (3);
\end{scope}
\end{tikzpicture}%
$ & $%
\begin{tikzpicture}
\draw[draw=black] (-2.1, -2.2) rectangle (2.6, 2.2);
\begin{scope}[every node/.style={circle,thick,draw=green!60!black}]
\node(1) at (0:1.5) {$1$};
\node(2) at (120:1.5) {$2$};
\node(3) at (240:1.5) {$3$};
\end{scope}
\begin{scope}[every edge/.style={draw=red,line width=1.7pt}]
\path[->] (1) edge[loop right] (1) edge[bend left=20] (3);
\path[->] (2) edge (1) edge[bend left=20] (3);
\path[->] (3) edge[bend left=20] (1) edge[bend left=20] (2);
\end{scope}
\end{tikzpicture}%
$\\
$D$ & $\overline{D}$\\\hline
\end{tabular}
\ \ \ \ \ \ \ \ \ \ .
\]

\end{example}

\begin{definition}
Let $V$ be a finite set. A $V$\emph{-listing} will mean a list of elements of
$V$ that contains each element of $V$ exactly once.
\end{definition}

For example, $\left(  2,1,3\right)  $ is a $\left\{  1,2,3\right\}  $-listing.

Of course, if $V$ is a finite set, then there are exactly $\left\vert
V\right\vert !$ many $V$-listings. They are in a canonical bijection with the
bijective maps from $\left[  \left\vert V\right\vert \right]  $ to $V$, and in
a non-canonical bijection with the permutations of $V$.

\begin{convention}
\label{conv.wi}If $w$ is any list (i.e., tuple), and if $i$ is a positive
integer, then $w_{i}$ shall denote the $i$-th entry of $w$. (Thus, $w=\left(
w_{1},w_{2},\ldots,w_{k}\right)  $, where $k$ is the length of $w$.)
\end{convention}

\begin{definition}
Let $D=\left(  V,A\right)  $ be a digraph. Let $w=\left(  w_{1},w_{2}%
,\ldots,w_{n}\right)  $ be a $V$-listing. Then:

\begin{enumerate}
\item[\textbf{(a)}] A $D$\emph{-descent} of $w$ means an $i\in\left[
n-1\right]  $ satisfying $\left(  w_{i},w_{i+1}\right)  \in A$.

\item[\textbf{(b)}] We let $\operatorname*{Des}\left(  w,D\right)  $ denote
the set of all $D$-descents of $w$.
\end{enumerate}
\end{definition}

\begin{example}
Let $D$ be the digraph $D$ from Example \ref{exa.complement.3verts}, and let
$w$ be the $V$-listing $\left(  3,1,2\right)  $. Then, $2$ is a $D$-descent of
$w$ (since $\left(  w_{2},w_{3}\right)  =\left(  1,2\right)  \in A$), but $1$
is not a $D$-descent of $w$ (since $\left(  w_{1},w_{2}\right)  =\left(
3,1\right)  \notin A$). Hence, $\operatorname*{Des}\left(  w,D\right)
=\left\{  2\right\}  $.
\end{example}

\begin{example}
Let $n\in\mathbb{N}$, and let $V=\left[  n\right]  $. Let $D$ be the digraph
whose vertices are the elements of $V$ and whose arcs are all the pairs
$\left(  i,j\right)  \in\left[  n\right]  ^{2}$ satisfying $i>j$. Let $w$ be a
$V$-listing. Then, the $D$-descents of $w$ are exactly the descents of $w$ in
the usual sense (i.e., the numbers $i\in\left[  n-1\right]  $ satisfying
$w_{i}>w_{i+1}$).
\end{example}

We note that $D$-descents for general digraphs $D$ have already implicitly
appeared in the work of Foata and Zeilberger \cite{FoaZei96}, which considers
the number $\operatorname*{maj}\nolimits_{D}^{\prime}w:=\sum_{i\in
\operatorname*{Des}\left(  w,D\right)  }i$ for each $V$-listing $w$. We would
not be surprised if what follows can shed some new light on the results of
\cite{FoaZei96}, but so far we have not found any deeper connections.

\subsection{Quasisymmetric functions}

Next, we introduce some notations from the theory of quasisymmetric functions
(see, e.g., \cite[\S 7.19]{EC2} or \cite[Chapter 5]{GriRei}):

\begin{definition}
\ \ 

\begin{enumerate}
\item[\textbf{(a)}] A \emph{composition} means a finite list of positive
integers. If $\alpha=\left(  \alpha_{1},\alpha_{2},\ldots,\alpha_{k}\right)  $
is a composition, then the number $k$ is called the \emph{length} of $\alpha$,
whereas the number $\alpha_{1}+\alpha_{2}+\cdots+\alpha_{k}$ is called the
\emph{size} of $\alpha$. If $n\in\mathbb{N}$, then a \emph{composition of }$n$
shall mean a composition having size $n$.

\item[\textbf{(b)}] A \emph{partition} (or \emph{integer partition}) means a
composition that is weakly decreasing.
\end{enumerate}
\end{definition}

For example, $\left(  2,5,3\right)  $ is a composition of $10$ that has length
$3$ and is not a partition (since $2<5$).

\begin{definition}
Let $n\in\mathbb{N}$. For any subset $I$ of $\left[  n-1\right]  $, we let
$\operatorname*{comp}\left(  I,n\right)  $ be the list%
\[
\left(  i_{1}-i_{0},\ i_{2}-i_{1},\ i_{3}-i_{2},\ \ldots,\ i_{k}%
-i_{k-1}\right)  ,
\]
where $i_{0},i_{1},\ldots,i_{k}$ are the elements of $\left\{  0\right\}  \cup
I\cup\left\{  n\right\}  $ listed in strictly increasing order. This list
$\operatorname*{comp}\left(  I,n\right)  $ is a composition of $n$.
\end{definition}

\begin{example}
If $n=6$ and $I=\left\{  2,3,5\right\}  $, then $\operatorname*{comp}\left(
I,n\right)  =\left(  2,1,2,1\right)  $.
\end{example}

Note that $\operatorname*{comp}\left(  I,n\right)  $ is denoted by
$\operatorname*{co}\left(  I\right)  $ in \cite[\S 7.19]{EC2}, but we prefer
to make the dependence on $n$ explicit here. In the notation of
\cite[Definition 5.1.10]{GriRei}, the composition $\operatorname*{comp}\left(
I,n\right)  $ is the preimage of $I$ under the bijection
$D:\operatorname*{Comp}\nolimits_{n}\rightarrow2^{\left[  n-1\right]  }$.

For any $n\in\mathbb{N}$, the map
\begin{align*}
\left\{  \text{subsets of }\left[  n-1\right]  \right\}   &  \rightarrow
\left\{  \text{compositions of }n\right\}  ,\\
I  &  \mapsto\operatorname*{comp}\left(  I,n\right)
\end{align*}
is a bijection.

\begin{definition}
Consider the ring $\mathbb{Z}\left[  \left[  x_{1},x_{2},x_{3},\ldots\right]
\right]  $ of formal power series in countably many indeterminates
$x_{1},x_{2},x_{3},\ldots$. Two subrings of this ring $\mathbb{Z}\left[
\left[  x_{1},x_{2},x_{3},\ldots\right]  \right]  $ are:

\begin{itemize}
\item the ring $\Lambda$ of \emph{symmetric functions} (defined, e.g., in
\cite[\S 7.1]{EC2} or in \cite[\S 2.1]{GriRei});

\item the ring $\operatorname*{QSym}$ of \emph{quasisymmetric functions}
(defined, e.g., in \cite[\S 7.19]{EC2} or in \cite[\S 5.1]{GriRei}).
\end{itemize}
\end{definition}

We will not actually use any properties of these rings $\Lambda$ and
$\operatorname*{QSym}$ anywhere except in Sections \ref{sec.antipode},
\ref{sec.redei} and \ref{sec.mod4} (and even there, only $\Lambda$ will be
used); thus, a reader unfamiliar with symmetric functions can read
$\mathbb{Z}\left[  \left[  x_{1},x_{2},x_{3},\ldots\right]  \right]  $ instead
of $\Lambda$ or $\operatorname*{QSym}$ everywhere else.

\begin{definition}
Let $\alpha$ be a composition. Then, $L_{\alpha}$ will denote the
\emph{fundamental quasisymmetric function} corresponding to $\alpha$. This is
a formal power series in $\operatorname*{QSym}$, and is defined as follows:
Let $I$ be the unique subset of $\left[  n-1\right]  $ satisfying
$\alpha=\operatorname*{comp}\left(  I,n\right)  $. Then, we set%
\[
L_{\alpha}=\sum_{\substack{i_{1}\leq i_{2}\leq\cdots\leq i_{n};\\i_{p}%
<i_{p+1}\text{ for each }p\in I}}x_{i_{1}}x_{i_{2}}\cdots x_{i_{n}}%
\in\operatorname*{QSym}%
\]
(where the summation indices $i_{1},i_{2},\ldots,i_{n}$ range over
$\mathbb{P}$).
\end{definition}

See \cite[\S 7.19]{EC2} or \cite[\S 5]{GriRei} for more about these
fundamental quasisymmetric functions $L_{\alpha}$ (originally introduced by
Ira Gessel)\footnote{Note that the definition of $L_{\alpha}$ given in
\cite[Definition 5.2.4]{GriRei} differs from ours. However, it is equivalent
to ours, since \cite[Proposition 5.2.9]{GriRei} shows that the $L_{\alpha}$
defined in \cite[Definition 5.2.4]{GriRei} satisfy the same formula that we
used to define our $L_{\alpha}$.}. We will actually find it easier to index
them not by the compositions $\alpha$ but rather by the corresponding subsets
$I$ of $\left[  n-1\right]  $. Thus, we define:

\begin{definition}
Let $n\in\mathbb{N}$, and let $I$ be a subset of $\left[  n-1\right]  $. Then,
we will use the notation $L_{I,n}$ for $L_{\operatorname*{comp}\left(
I,n\right)  }$. Explicitly, we have%
\begin{equation}
L_{I,n}=\sum_{\substack{i_{1}\leq i_{2}\leq\cdots\leq i_{n};\\i_{p}%
<i_{p+1}\text{ for each }p\in I}}x_{i_{1}}x_{i_{2}}\cdots x_{i_{n}}%
\in\operatorname*{QSym} \label{eq.def.Lal.LIn}%
\end{equation}
(where the summation indices $i_{1},i_{2},\ldots,i_{n}$ range over
$\mathbb{P}$).
\end{definition}

\begin{example}
If $n=3$ and $I=\left\{  2\right\}  $, then%
\begin{align*}
L_{I,n}  &  =L_{\left\{  2\right\}  ,3}=\sum_{\substack{i_{1}\leq i_{2}\leq
i_{3};\\i_{p}<i_{p+1}\text{ for each }p\in\left\{  2\right\}  }}x_{i_{1}%
}x_{i_{2}}x_{i_{3}}=\sum_{i_{1}\leq i_{2}<i_{3}}x_{i_{1}}x_{i_{2}}x_{i_{3}}\\
&  =x_{1}x_{1}x_{2}+x_{1}x_{1}x_{3}+\cdots+x_{1}x_{2}x_{3}+x_{1}x_{2}%
x_{4}+\cdots+\cdots+x_{2}x_{2}x_{3}+\cdots.
\end{align*}

\end{example}

\subsection{The Redei--Berge symmetric function}

We are now ready to define the main protagonist of this paper:

\begin{definition}
\label{def.UD}Let $n\in\mathbb{N}$. Let $D=\left(  V,A\right)  $ be a digraph
with $n$ vertices. We define the \emph{Redei--Berge symmetric function}
$U_{D}$ to be the quasisymmetric function%
\[
\sum_{w\text{ is a }V\text{-listing}}L_{\operatorname*{Des}\left(  w,D\right)
,\ n}\in\operatorname*{QSym}.
\]

\end{definition}

\begin{example}
\label{exa.UD.1}Let $D$ be the digraph $D$ from Example
\ref{exa.complement.3verts}. Then,%
\begin{align*}
U_{D}  &  =\sum_{w\text{ is a }V\text{-listing}}L_{\operatorname*{Des}\left(
w,D\right)  ,\ 3}\\
&  =L_{\operatorname*{Des}\left(  \left(  1,2,3\right)  ,D\right)
,\ 3}+L_{\operatorname*{Des}\left(  \left(  1,3,2\right)  ,D\right)
,\ 3}+L_{\operatorname*{Des}\left(  \left(  2,1,3\right)  ,D\right)  ,\ 3}\\
&  \ \ \ \ \ \ \ \ \ \ +L_{\operatorname*{Des}\left(  \left(  2,3,1\right)
,D\right)  ,\ 3}+L_{\operatorname*{Des}\left(  \left(  3,1,2\right)
,D\right)  ,\ 3}+L_{\operatorname*{Des}\left(  \left(  3,2,1\right)
,D\right)  ,\ 3}\\
&  =L_{\left\{  1\right\}  ,\ 3}+L_{\varnothing,\ 3}+L_{\varnothing
,\ 3}+L_{\varnothing,\ 3}+L_{\left\{  2\right\}  ,\ 3}+L_{\varnothing,\ 3}\\
&  =4\cdot\underbrace{L_{\varnothing,\ 3}}_{=\sum_{i_{1}\leq i_{2}\leq i_{3}%
}x_{i_{1}}x_{i_{2}}x_{i_{3}}}+\underbrace{L_{\left\{  1\right\}  ,\ 3}}%
_{=\sum_{i_{1}<i_{2}\leq i_{3}}x_{i_{1}}x_{i_{2}}x_{i_{3}}}%
+\underbrace{L_{\left\{  2\right\}  ,\ 3}}_{=\sum_{i_{1}\leq i_{2}<i_{3}%
}x_{i_{1}}x_{i_{2}}x_{i_{3}}}\\
&  =4\cdot\sum_{i_{1}\leq i_{2}\leq i_{3}}x_{i_{1}}x_{i_{2}}x_{i_{3}}%
+\sum_{i_{1}<i_{2}\leq i_{3}}x_{i_{1}}x_{i_{2}}x_{i_{3}}+\sum_{i_{1}\leq
i_{2}<i_{3}}x_{i_{1}}x_{i_{2}}x_{i_{3}}.
\end{align*}
From this expression, we can easily see that $U_{D}$ is actually a symmetric
function, and can be written (e.g.) as $p_{1}^{3}+2p_{1}p_{2}+p_{3}$, where
$p_{k}:=x_{1}^{k}+x_{2}^{k}+x_{3}^{k}+\cdots$ is the $k$-th power-sum
symmetric function for each $k\geq1$.
\end{example}

The name \textquotedblleft Redei--Berge symmetric function\textquotedblright%
\ for the power series $U_{D}$ was chosen because (as we will soon see) it is
actually a symmetric function and is related to two classical results of Redei
and Berge on the number of Hamiltonian paths in digraphs. In \cite[Problem
120]{EC2supp}, it is called $U_{X}$, where $X$ is what we call $A$ (that is,
the set of arcs of $D$); but we shall here put the entire digraph $D$ into the subscript.

The Redei--Berge symmetric function $U_{D}$ equals the quasisymmetric function
$\Xi_{\overline{D}}\left(  x,0\right)  $ from Chow's \cite{Chow96}%
.\footnote{Indeed, this equality follows immediately from \cite[Proposition
7]{Chow96}, since the quasisymmetric function we call $L_{I,n}$ appears under
the name of $Q_{I,n}$ in \cite{Chow96}, and since our $\operatorname*{Des}%
\left(  w,\overline{D}\right)  $ is what is called $S\left(  w\right)  $ in
\cite{Chow96}.} It also is denoted by $\Pi_{\overline{D}}$ in \cite{Wisema07}%
.\footnote{Indeed, comparing the definition of $\Pi_{D}$ in \cite[Definition
2.2]{Wisema07} with the definition of $\Xi_{D}$ in \cite[\S 2]{Chow96} shows
that $\Pi_{D}=\Xi_{D}\left(  x,0\right)  $. Thus, $\Pi_{\overline{D}}%
=\Xi_{\overline{D}}\left(  x,0\right)  =U_{D}$ (as we already know).} Several
properties of $U_{D}$ have been shown in \cite{Chow96} and in \cite{Wisema07},
and some of them will be reproved here for the sake of self-containedness and
variety. However, our main results -- Theorems \ref{thm.UX.1}, \ref{thm.UX.2}
and \ref{thm.UX.3} -- appear to be new.

\begin{question}
Can these results be extended to the more general functions $\Xi_{D}\left(
x,y\right)  $ from \cite{Chow96}?
\end{question}

\subsection{Arcs and cyclic arcs}

The main results of this paper are explicit (albeit not, in general,
subtraction-free) expansions of $U_{D}$ in terms of the power-sum symmetric
functions. To state these, we need some more notations. We shall soon define
cycles of digraphs and cycles of permutations, and we will then connect the
two notions. First, some auxiliary notations:

\begin{definition}
\label{def.Arcs-and-Carcs}Let $V$ be a set. Let $v=\left(  v_{1},v_{2}%
,\ldots,v_{k}\right)  \in V^{k}$ be a nonempty tuple of elements of $V$.

\begin{enumerate}
\item[\textbf{(a)}] We define a subset $\operatorname*{Arcs}v$ of $V\times V$
by
\begin{align}
\operatorname*{Arcs}v:=\  &  \left\{  \left(  v_{i},v_{i+1}\right)
\ \mid\ i\in\left[  k-1\right]  \right\} \nonumber\\
=\  &  \left\{  \left(  v_{1},v_{2}\right)  ,\ \left(  v_{2},v_{3}\right)
,\ \ldots,\ \left(  v_{k-1},v_{k}\right)  \right\}
\label{eq.def.Arcs-and-Carcs.a.2}\\
\subseteq\  &  V\times V.\nonumber
\end{align}
This subset $\operatorname*{Arcs}v$ is called the \emph{arc set} of the tuple
$v$. Its elements $\left(  v_{i},v_{i+1}\right)  $ are called the \emph{arcs}
of $v$.

\item[\textbf{(b)}] We define a subset $\operatorname*{CArcs}v$ of $V\times V$
by%
\begin{align}
\operatorname*{CArcs}v:=\  &  \left\{  \left(  v_{i},v_{i+1}\right)
\ \mid\ i\in\left[  k\right]  \right\} \nonumber\\
=\  &  \left\{  \left(  v_{1},v_{2}\right)  ,\ \left(  v_{2},v_{3}\right)
,\ \ldots,\ \left(  v_{k-1},v_{k}\right)  ,\ \left(  v_{k},v_{1}\right)
\right\} \label{eq.def.Arcs-and-Carcs.b.2}\\
\subseteq\  &  V\times V,\nonumber
\end{align}
where we set $v_{k+1}:=v_{1}$. This subset $\operatorname*{CArcs}v$ is called
the \emph{cyclic arc set} of the tuple $v$. Its elements $\left(
v_{i},v_{i+1}\right)  $ are called the \emph{cyclic arcs} of $v$.

\item[\textbf{(c)}] The \emph{reversal} of $v$ is defined to be the tuple
$\left(  v_{k},v_{k-1},\ldots,v_{1}\right)  \in V^{k}$.
\end{enumerate}
\end{definition}

\begin{example}
Let $V=\mathbb{N}$ and $v=\left(  1,4,2,6\right)  \in V^{4}$. Then,
\begin{align*}
\operatorname*{Arcs}v  &  =\left\{  \left(  1,4\right)  ,\ \left(  4,2\right)
,\ \left(  2,6\right)  \right\}  \ \ \ \ \ \ \ \ \ \ \text{and}\\
\operatorname*{CArcs}v  &  =\left\{  \left(  1,4\right)  ,\ \left(
4,2\right)  ,\ \left(  2,6\right)  ,\ \left(  6,1\right)  \right\}  .
\end{align*}

\end{example}

Note that if we cyclically rotate a nonempty tuple $v\in V^{k}$, then the set
$\operatorname*{CArcs}v$ remains unchanged: i.e., for any $\left(  v_{1}%
,v_{2},\ldots,v_{k}\right)  \in V^{k}$, we have%
\[
\operatorname*{CArcs}\left(  v_{1},v_{2},\ldots,v_{k}\right)
=\operatorname*{CArcs}\left(  v_{2},v_{3},\ldots,v_{k},v_{1}\right)  .
\]

\subsection{Permutations and their cycles}

Now, let us discuss permutations and their cycles. We start with some basic notations:

\begin{definition}
Let $V$ be a finite set. Then, $\mathfrak{S}_{V}$ shall denote the symmetric
group of $V$ (that is, the group of all permutations of $V$).
\end{definition}

Note that the order of this group is $\left\vert \mathfrak{S}_{V}\right\vert
=\left\vert V\right\vert !$.

\begin{definition}
Let $V$ be a set.

\begin{enumerate}
\item[\textbf{(a)}] Two tuples $v\in V^{k}$ and $w\in V^{\ell}$ of elements of
$V$ are said to be \emph{rotation-equivalent} if $w$ can be obtained from $v$
by cyclic rotation, i.e., if $\ell=k$ and $w=\left(  v_{i},v_{i+1}%
,\ldots,v_{k},v_{1},v_{2},\ldots,v_{i-1}\right)  $ for some $i\in\left[
k\right]  $.

\item[\textbf{(b)}] The relation \textquotedblleft
rotation-equivalent\textquotedblright\ is an equivalence relation on the set
of all nonempty tuples of elements of $V$. Its equivalence classes are called
the \emph{rotation-equivalence classes}. In other words, the
rotation-equivalence classes are the orbits of the operation%
\[
\left(  a_{1},a_{2},\ldots,a_{k}\right)  \mapsto\left(  a_{2},a_{3}%
,\ldots,a_{k},a_{1}\right)
\]
on the set of all nonempty tuples of elements of $V$.

\item[\textbf{(c)}] The rotation-equivalence class that contains a given
nonempty tuple $v\in V^{k}$ will be denoted by $v_{\sim}$.
\end{enumerate}
\end{definition}

For instance, the tuple $\left(  1,2,3,4\right)  $ is rotation-equivalent to
$\left(  3,4,1,2\right)  $, but not to $\left(  4,3,2,1\right)  $. Thus,%
\[
\left(  1,2,3,4\right)  _{\sim}=\left(  3,4,1,2\right)  _{\sim}\neq\left(
4,3,2,1\right)  _{\sim}.
\]
Also,%
\[
\left(  1,3,6\right)  _{\sim}=\left\{  \left(  1,3,6\right)  ,\ \left(
3,6,1\right)  ,\ \left(  6,1,3\right)  \right\}  .
\]

\begin{definition}
\label{def.reqc.features}Let $V$ be a set. Let $\gamma$ be a
rotation-equivalence class (of nonempty tuples of elements of $V$). Then:

\begin{enumerate}
\item[\textbf{(a)}] All tuples $v\in\gamma$ have the same length (i.e., number
of entries). This length is denoted by $\ell\left(  \gamma\right)  $, and is
called the \emph{length} of $\gamma$. Thus, if $\gamma=v_{\sim}$ for some
tuple $v\in V^{k}$, then $\ell\left(  \gamma\right)  =k$.

\item[\textbf{(b)}] All tuples $v\in\gamma$ have the same cyclic arc set
$\operatorname*{CArcs}v$ (since $\operatorname*{CArcs}v$ remains unchanged if
we cyclically rotate $v$). This cyclic arc set is denoted by
$\operatorname*{CArcs}\gamma$, and is called the \emph{cyclic arc set} of
$\gamma$. Thus, the cyclic arc set of a rotation-equivalence class
$\gamma=\left(  v_{1},v_{2},\ldots,v_{k}\right)  _{\sim}$ is%
\[
\operatorname*{CArcs}\gamma=\left\{  \left(  v_{1},v_{2}\right)  ,\ \left(
v_{2},v_{3}\right)  ,\ \ldots,\ \left(  v_{k-1},v_{k}\right)  ,\ \left(
v_{k},v_{1}\right)  \right\}  .
\]

\item[\textbf{(c)}] All tuples $v\in\gamma$ have the same entries (up to
order). These entries are called the \emph{entries} of $v$. Thus, the entries
of a rotation-equivalence class $\gamma=\left(  v_{1},v_{2},\ldots
,v_{k}\right)  _{\sim}$ are $v_{1},v_{2},\ldots,v_{k}$.

\item[\textbf{(d)}] The reversals of all tuples $v\in\gamma$ are the elements
of a single rotation-equivalence class $\operatorname*{rev}\gamma$. This
latter class will be called the \emph{reversal} of $\gamma$. Thus, the
reversal of a rotation-equivalence class $\gamma=\left(  v_{1},v_{2}%
,\ldots,v_{k}\right)  _{\sim}$ is the rotation-equivalence class $\left(
v_{k},v_{k-1},\ldots,v_{1}\right)  _{\sim}$.

\item[\textbf{(e)}] We say that $\gamma$ is \emph{nontrivial} if $\ell\left(
\gamma\right)  >1$.
\end{enumerate}
\end{definition}

For instance, the rotation-equivalence class $\left(  3,1,4\right)  _{\sim}$
has length $3$, cyclic arc set $\left\{  \left(  3,1\right)  ,\ \left(
1,4\right)  ,\ \left(  4,3\right)  \right\}  $, and entries $3,1,4$. Its
reversal is $\left(  4,1,3\right)  _{\sim}$, and it is nontrivial (since
$\ell\left(  \left(  3,1,4\right)  _{\sim}\right)  =3>1$).

\begin{definition}
\label{def.perm.cycs}Let $V$ be a finite set. Let $\sigma\in\mathfrak{S}_{V}$
be any permutation.

\begin{enumerate}
\item[\textbf{(a)}] The \emph{cycles} of $\sigma$ are the rotation-equivalence
classes of the tuples of the form%
\[
\left(  \sigma^{0}\left(  i\right)  ,\ \sigma^{1}\left(  i\right)
,\ \ldots,\ \sigma^{k-1}\left(  i\right)  \right)  ,
\]
where $i$ is some element of $V$, and where $k$ is the smallest positive
integer satisfying $\sigma^{k}\left(  i\right)  =i$.

For example, the permutation $w_{0}\in\mathfrak{S}_{\left[  7\right]  }$ that
sends each $i\in\left[  7\right]  $ to $8-i$ has cycles $\left(  1,7\right)
_{\sim}$, $\left(  2,6\right)  _{\sim}$, $\left(  3,5\right)  _{\sim}$ and
$\left(  4\right)  _{\sim}$. (Note that we do allow a cycle to have length $1$.)

\item[\textbf{(b)}] The \emph{cycle type} of $\sigma$ means the partition
whose entries are the lengths of the cycles of $\sigma$. We denote this cycle
type by $\operatorname*{type}\sigma$. It is a partition of the number
$\left\vert V\right\vert $.

\item[\textbf{(c)}] We let $\operatorname*{Cycs}\sigma$ denote the set of all
cycles of $\sigma$.
\end{enumerate}
\end{definition}

\begin{example}
Let $w_{0}\in\mathfrak{S}_{\left[  7\right]  }$ be the permutation that sends
each $i\in\left[  7\right]  $ to $8-i$. We have already seen that $w_{0}$ has
cycles $\left(  1,7\right)  _{\sim}$, $\left(  2,6\right)  _{\sim}$, $\left(
3,5\right)  _{\sim}$ and $\left(  4\right)  _{\sim}$. Their respective lengths
are $2,2,2,1$. Thus, the cycle type of $w_{0}$ is $\operatorname*{type}%
w_{0}=\left(  2,2,2,1\right)  $. We have $\operatorname*{Cycs}\sigma=\left\{
\left(  1,7\right)  _{\sim},\ \left(  2,6\right)  _{\sim},\ \left(
3,5\right)  _{\sim},\ \left(  4\right)  _{\sim}\right\}  $. The first three of
the four cycles $\left(  1,7\right)  _{\sim}$, $\left(  2,6\right)  _{\sim}$,
$\left(  3,5\right)  _{\sim}$ and $\left(  4\right)  _{\sim}$ are nontrivial.
\end{example}

\subsection{$D$-paths and $D$-cycles}

Next, we define paths and cycles in a digraph:

\begin{definition}
Let $D=\left(  V,A\right)  $ be a digraph.

\begin{enumerate}
\item[\textbf{(a)}] A $D$\emph{-path} (or \emph{path of }$D$) shall mean a
nonempty tuple $v$ of distinct elements of $V$ such that $\operatorname*{Arcs}%
v\subseteq A$.

\item[\textbf{(b)}] A $D$\emph{-cycle} (or \emph{cycle of }$D$) shall mean a
rotation-equivalence class $\gamma$ of nonempty tuples of distinct elements of
$V$ such that $\operatorname*{CArcs}\gamma\subseteq A$.
\end{enumerate}
\end{definition}

We note that our notion of \textquotedblleft cycle of $D$\textquotedblright%
\ differs slightly from the common one used in graph theory\footnote{Namely,
cycles in graph theory have their first vertex repeated at the end, whereas
our cycles don't. However, this difference is purely notational: A cycle
$\left(  v_{1},v_{2},\ldots,v_{k}\right)  _{\sim}$ in our sense corresponds to
the cycle $\left(  v_{1},v_{2},\ldots,v_{k},v_{1}\right)  $ in the
graph-theorists' terminology.}.

\begin{example}
Let $D$ be the digraph $D$ from Example \ref{exa.complement.3verts}. Then:

\begin{enumerate}
\item[\textbf{(a)}] The pair $\left(  1,2\right)  $ as well as the three
$1$-tuples $\left(  1\right)  $, $\left(  2\right)  $ and $\left(  3\right)  $
are $D$-paths. The triple $\left(  1,2,2\right)  $ is not a $D$-path (even
though it satisfies the \textquotedblleft$\operatorname*{Arcs}v\subseteq
A$\textquotedblright\ condition), since its entries $1,2,2$ are not distinct.
The triple $\left(  1,2,3\right)  $ is not a $D$-path, since $\left(
2,3\right)  $ is not an arc of $D$.

The triple $\left(  2,3,1\right)  $ is a $\overline{D}$-path (and there are
several others).

\item[\textbf{(b)}] The only $D$-cycles are the rotation-equivalence classes
$\left(  2\right)  _{\sim}$ and $\left(  3\right)  _{\sim}$. The $\overline
{D}$-cycles are $\left(  1\right)  _{\sim}$, $\left(  1,3\right)  _{\sim}$,
$\left(  2,3\right)  _{\sim}$ and $\left(  2,1,3\right)  _{\sim}$.
\end{enumerate}
\end{example}

\subsection{The sets $\mathfrak{S}_{V}\left(  D\right)  $ and $\mathfrak{S}%
_{V}\left(  D,\overline{D}\right)  $}

Now, we can connect digraphs with permutations by comparing their cycles:

\begin{definition}
\label{def.SVD}Let $D=\left(  V,A\right)  $ be a digraph. Then, we
define\footnotemark%
\[
\mathfrak{S}_{V}\left(  D\right)  :=\left\{  \sigma\in\mathfrak{S}_{V}%
\ \mid\ \text{each nontrivial cycle of }\sigma\text{ is a }D\text{-cycle}%
\right\}
\]
and%
\[
\mathfrak{S}_{V}\left(  D,\overline{D}\right)  :=\left\{  \sigma
\in\mathfrak{S}_{V}\ \mid\ \text{each cycle of }\sigma\text{ is a
}D\text{-cycle or a }\overline{D}\text{-cycle}\right\}  .
\]

\end{definition}

\footnotetext{As we warned in Definition \ref{def.perm.cycs} \textbf{(a)}, we
are being cavalier about the distinction between rotation-equivalence classes
and their representatives. Thus, when we say that a certain cycle $\gamma$ of
$\sigma$ is a $D$-cycle, we really mean that some tuple in the
rotation-equivalence class $\gamma$ (and therefore every tuple in $\gamma$) is
a $D$-cycle.}Note that we could just as well replace \textquotedblleft each
cycle\textquotedblright\ by \textquotedblleft each nontrivial
cycle\textquotedblright\ in the definition of $\mathfrak{S}_{V}\left(
D,\overline{D}\right)  $, since a cycle of length $1$ is always a $D$-cycle or
a $\overline{D}$-cycle (depending on whether its only cyclic arc belongs to
$A$ or not). However, we could not replace \textquotedblleft nontrivial
cycle\textquotedblright\ by \textquotedblleft cycle\textquotedblright\ in the
definition of $\mathfrak{S}_{V}\left(  D\right)  $.

\begin{example}
Let $D$ be the digraph $D$ from Example \ref{exa.complement.3verts}. Let
$V=\left\{  1,2,3\right\}  $ be its set of vertices. Then:

\begin{enumerate}
\item[\textbf{(a)}] We have $\mathfrak{S}_{V}\left(  D\right)  =\left\{
\operatorname*{id}\nolimits_{V}\right\}  $, since the only $D$-cycles have
length $1$.

\item[\textbf{(b)}] We have%
\[
\mathfrak{S}_{V}\left(  D,\overline{D}\right)  =\left\{  \operatorname*{id}%
\nolimits_{V},\ \operatorname*{cyc}\nolimits_{1,3},\ \operatorname*{cyc}%
\nolimits_{2,3},\ \operatorname*{cyc}\nolimits_{1,3,2}\right\}  ,
\]
where $\operatorname*{cyc}\nolimits_{i_{1},i_{2},\ldots,i_{k}}$ denotes the
permutation that cyclically permutes the elements $i_{1},i_{2},\ldots,i_{k}$
while leaving all other elements of $V$ unchanged.
\end{enumerate}
\end{example}

\subsection{Formulas for $U_{D}$}

\subsubsection{The power-sum symmetric functions}

We now introduce some of the best-known (and easiest to define) symmetric functions:

\begin{definition}
\label{def.powersum}\ \ 

\begin{enumerate}
\item[\textbf{(a)}] For each positive integer $n$, we define the
\emph{power-sum symmetric function}
\[
p_{n}:=x_{1}^{n}+x_{2}^{n}+x_{3}^{n}+\cdots\in\Lambda.
\]

\item[\textbf{(b)}] If $\lambda=\left(  \lambda_{1},\lambda_{2},\ldots
,\lambda_{k}\right)  $ is a partition with $k$ positive entries, then we set%
\[
p_{\lambda}:=p_{\lambda_{1}}p_{\lambda_{2}}\cdots p_{\lambda_{k}}\in\Lambda.
\]

\end{enumerate}
\end{definition}

For instance, $p_{\left(  2,2,1\right)  }=p_{2}p_{2}p_{1}=\left(  x_{1}%
^{2}+x_{2}^{2}+x_{3}^{2}+\cdots\right)  ^{2}\left(  x_{1}+x_{2}+x_{3}%
+\cdots\right)  $.

\subsubsection{The first main theorem: general digraphs}

We now state our first main theorem (which will be proved in Section
\ref{sec.pf.thm.UX.1}):

\begin{theorem}
\label{thm.UX.1}Let $D=\left(  V,A\right)  $ be a digraph. Set%
\[
\varphi\left(  \sigma\right)  :=\sum_{\substack{\gamma\in\operatorname*{Cycs}%
\sigma;\\\gamma\text{ is a }D\text{-cycle}}}\left(  \ell\left(  \gamma\right)
-1\right)  \ \ \ \ \ \ \ \ \ \ \text{for each }\sigma\in\mathfrak{S}_{V}.
\]
Then,%
\[
U_{D}=\sum_{\sigma\in\mathfrak{S}_{V}\left(  D,\overline{D}\right)  }\left(
-1\right)  ^{\varphi\left(  \sigma\right)  }p_{\operatorname*{type}\sigma}.
\]

\end{theorem}

\begin{example}
Let $V=\left\{  1,2,3,4,5,6\right\}  $ and $D=\left(  V,\ V\times V\right)  $.
Let $\sigma\in\mathfrak{S}_{V}$ be the permutation whose cycles are $\left(
1,3\right)  _{\sim}$, $\left(  2,4,5\right)  _{\sim}$ and $\left(  6\right)
_{\sim}$. Then, every cycle of $\sigma$ is a $D$-cycle, and the number
$\varphi\left(  \sigma\right)  $ (as defined in Theorem \ref{thm.UX.1}) is
\begin{align*}
&  \left(  \ell\left(  \left(  1,3\right)  _{\sim}\right)  -1\right)  +\left(
\ell\left(  \left(  2,4,5\right)  _{\sim}\right)  -1\right)  +\left(
\ell\left(  \left(  6\right)  _{\sim}\right)  -1\right) \\
&  =\left(  2-1\right)  +\left(  3-1\right)  +\left(  1-1\right)  =3.
\end{align*}

\end{example}

\begin{example}
Let $D$ be the digraph $D$ from Example \ref{exa.complement.3verts}. Recall
that $\mathfrak{S}_{V}\left(  D,\overline{D}\right)  =\left\{
\operatorname*{id}\nolimits_{V},\ \operatorname*{cyc}\nolimits_{1,3}%
,\ \operatorname*{cyc}\nolimits_{2,3},\ \operatorname*{cyc}\nolimits_{1,3,2}%
\right\}  $. Thus, Theorem \ref{thm.UX.1} yields%
\begin{align*}
U_{D}  &  =\underbrace{\left(  -1\right)  ^{\varphi\left(  \operatorname*{id}%
\nolimits_{V}\right)  }}_{=\left(  -1\right)  ^{0}=1}%
\underbrace{p_{\operatorname*{type}\left(  \operatorname*{id}\nolimits_{V}%
\right)  }}_{=p_{\left(  1,1,1\right)  }=p_{1}^{3}}+\underbrace{\left(
-1\right)  ^{\varphi\left(  \operatorname*{cyc}\nolimits_{1,3}\right)  }%
}_{=\left(  -1\right)  ^{0}=1}\underbrace{p_{\operatorname*{type}\left(
\operatorname*{cyc}\nolimits_{1,3}\right)  }}_{=p_{\left(  2,1\right)  }%
=p_{2}p_{1}}\\
&  \ \ \ \ \ \ \ \ \ \ +\underbrace{\left(  -1\right)  ^{\varphi\left(
\operatorname*{cyc}\nolimits_{2,3}\right)  }}_{=\left(  -1\right)  ^{0}%
=1}\underbrace{p_{\operatorname*{type}\left(  \operatorname*{cyc}%
\nolimits_{2,3}\right)  }}_{=p_{\left(  2,1\right)  }=p_{2}p_{1}%
}+\underbrace{\left(  -1\right)  ^{\varphi\left(  \operatorname*{cyc}%
\nolimits_{1,3,2}\right)  }}_{=\left(  -1\right)  ^{0}=1}%
\underbrace{p_{\operatorname*{type}\left(  \operatorname*{cyc}%
\nolimits_{1,3,2}\right)  }}_{=p_{\left(  3\right)  }=p_{3}}\\
&  =p_{1}^{3}+p_{2}p_{1}+p_{2}p_{1}+p_{3}=p_{1}^{3}+2p_{1}p_{2}+p_{3}.
\end{align*}
This agrees with the result found in Example \ref{exa.UD.1}.
\end{example}

\begin{example}
Let $D$ be the digraph $\left(  V,A\right)  $, where $V=\left\{
1,2,3\right\}  $ and
\[
A=\left\{  \left(  1,3\right)  ,\ \left(  2,1\right)  ,\ \left(  3,1\right)
,\ \left(  3,2\right)  \right\}  .
\]
Then, a straightforward computation using Theorem \ref{thm.UX.1} shows that
$U_{D}=p_{1}^{3}-p_{1}p_{2}+p_{3}$. (This example is due to Ira Gessel.)
\end{example}

The following two corollaries can be easily obtained from Theorem
\ref{thm.UX.1} (see Section \ref{sec.pf.cors} for their proofs):

\begin{corollary}
\label{cor.UX.p-int}Let $D=\left(  V,A\right)  $ be a digraph. Then, $U_{D}$
is a $p$-integral symmetric function (i.e., a symmetric function that can be
written as a polynomial in $p_{1},p_{2},p_{3},\ldots$). That is, we have
$U_{D}\in\mathbb{Z}\left[  p_{1},p_{2},p_{3},\ldots\right]  $.
\end{corollary}

\begin{corollary}
\label{cor.UX.if-each-D-cyc-odd}Let $D=\left(  V,A\right)  $ be a digraph.
Assume that every $D$-cycle has odd length. Then,
\[
U_{D}=\sum_{\sigma\in\mathfrak{S}_{V}\left(  D,\overline{D}\right)
}p_{\operatorname*{type}\sigma}\in\mathbb{N}\left[  p_{1},p_{2},p_{3}%
,\ldots\right]  .
\]

\end{corollary}

\subsubsection{The second main theorem: tournaments}

After we will have proved Theorem \ref{thm.UX.1}, we will use it to derive a
simpler formula, which however is specific to tournaments. First, we recall
the definition of a tournament:

\begin{definition}
A \emph{tournament} means a digraph $D=\left(  V,A\right)  $ that satisfies
the following two axioms:

\begin{itemize}
\item \emph{Looplessness}: We have $\left(  u,u\right)  \notin A$ for any
$u\in V$.

\item \emph{Tournament axiom}: For any two distinct vertices $u$ and $v$ of
$D$, \textbf{exactly} one of the two pairs $\left(  u,v\right)  $ and $\left(
v,u\right)  $ is an arc of $D$.
\end{itemize}
\end{definition}

\begin{example}
Neither the digraph $D$ from Example \ref{exa.complement.3verts}, nor its
complement $\overline{D}$, is a tournament. Here is a tournament:%
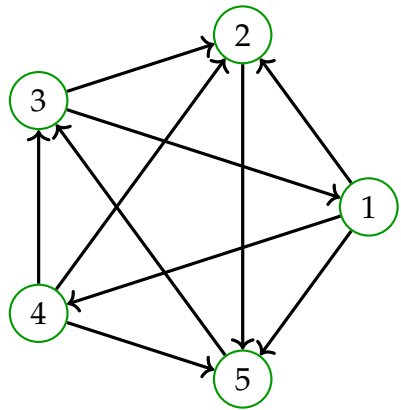
\[%
\begin{tikzpicture}[scale=1.2]
\begin{scope}[every node/.style={circle,thick,draw=green!60!black}]
\node(1) at (0*360/5 : 2) {$1$};
\node(2) at (1*360/5 : 2) {$2$};
\node(3) at (2*360/5 : 2) {$3$};
\node(4) at (3*360/5 : 2) {$4$};
\node(5) at (4*360/5 : 2) {$5$};
\end{scope}
\begin{scope}[every edge/.style={draw=black,very thick}, every loop/.style={}]
\path[->] (1) edge (2) edge (4) edge (5);
\path[->] (2) edge (5);
\path[->] (3) edge (1) edge (2);
\path[->] (4) edge (2) edge (3) edge (5);
\path[->] (5) edge (3);
\end{scope}
\end{tikzpicture}%
\ \ .
\]

\end{example}

We can now state our second main theorem (which we will prove in Section
\ref{sec.pf.thm.UX.2}):

\begin{theorem}
\label{thm.UX.2}Let $D=\left(  V,A\right)  $ be a tournament. For each
$\sigma\in\mathfrak{S}_{V}$, let $\psi\left(  \sigma\right)  $ denote the
number of nontrivial cycles of $\sigma$. Then,%
\[
U_{D}=\sum_{\substack{\sigma\in\mathfrak{S}_{V}\left(  D\right)  ;\\\text{all
cycles of }\sigma\text{ have odd length}}}2^{\psi\left(  \sigma\right)
}p_{\operatorname*{type}\sigma}.
\]

\end{theorem}

Once this is proved, the following corollary will be easy to derive (see
Section \ref{sec.pf.cors} for the details):

\begin{corollary}
\label{cor.UX.tournament-N}Let $D=\left(  V,A\right)  $ be a tournament.
Then,
\[
U_{D}\in\mathbb{N}\left[  p_{1},2p_{3},2p_{5},2p_{7},\ldots\right]
=\mathbb{N}\left[  p_{1},\ 2p_{i}\ \mid\ i>1\text{ is odd}\right]  .
\]
(Here, $\mathbb{N}\left[  p_{1},2p_{3},2p_{5},2p_{7},\ldots\right]  $ means
the set of all values of the form $f\left(  p_{1},2p_{3},2p_{5},2p_{7}%
,\ldots\right)  $, where $f$ is a polynomial in countably many indeterminates
with coefficients in $\mathbb{N}$.)
\end{corollary}

\subsubsection{The third main theorem: digraphs with no $2$-cycles}

A more general version of Theorem \ref{thm.UX.2} is the following:

\begin{theorem}
\label{thm.UX.3}Let $D=\left(  V,A\right)  $ be a digraph. Assume that there
exist no two distinct vertices $u$ and $v$ of $D$ such that both pairs
$\left(  u,v\right)  $ and $\left(  v,u\right)  $ belong to $A$.

\begin{enumerate}
\item[\textbf{(a)}] Then, $U_{D}$ is a $p$-positive symmetric function (i.e.,
a symmetric function that can be written as a polynomial in $p_{1},p_{2}%
,p_{3},\ldots$ with coefficients in $\mathbb{N}$). That is, we have $U_{D}%
\in\mathbb{N}\left[  p_{1},p_{2},p_{3},\ldots\right]  $.

\item[\textbf{(b)}] Let us say that a rotation-equivalence class $\gamma$ of
nonempty tuples of elements of $V$ is \emph{risky} if its length is even and
it has the property that either $\gamma$ or the reversal of $\gamma$ is a
$D$-cycle. Then,%
\[
U_{D}=\sum_{\substack{\sigma\in\mathfrak{S}_{V}\left(  D,\overline{D}\right)
;\\\text{no cycle of }\sigma\text{ is risky}}}p_{\operatorname*{type}\sigma}.
\]

\end{enumerate}
\end{theorem}

We will prove this in Section \ref{sec.pf.thm.UX.3}. Note that Theorem
\ref{thm.UX.3} \textbf{(a)} generalizes \cite[Theorem 7]{Chow96}.\footnote{To
see how, one needs to observe that
\par
\begin{enumerate}
\item any acyclic digraph $D$ satisfies the assumption of Theorem
\ref{thm.UX.3};
\par
\item the $\omega_{x}\Xi_{D}$ from \cite{Chow96} equals our $U_{D}$ in the
case when $D$ is acyclic.
\end{enumerate}
\par
The first of these two observations is obvious. The second follows from the
equality (\ref{eq.thm.antipode.U.omega}) further below, combined with the fact
that $\Xi_{D}=\Xi_{D}\left(  x,0\right)  $ when $D$ acyclic (since the
$y$-variables do not actually appear in $\Xi_{D}$ for lack of cycles), and the
fact that $U_{\overline{D}}=\Xi_{D}\left(  x,0\right)  $ (stated above in the
equivalent form $U_{D}=\Xi_{\overline{D}}\left(  x,0\right)  $).}

\begin{remark}
The converse of Theorem \ref{thm.UX.3} \textbf{(a)} does not hold. Indeed,
consider the digraph $D=\left(  V,A\right)  $ with $V=\left\{
1,2,3,4\right\}  $ and%
\[
A=\left\{  \left(  1,2\right)  ,\ \left(  2,1\right)  ,\ \left(  2,3\right)
,\ \left(  2,4\right)  ,\ \left(  3,4\right)  \right\}  .
\]
Then, $D$ does not satisfy the assumption of Theorem \ref{thm.UX.3} (since the
two distinct vertices $1$ and $2$ satisfy both $\left(  1,2\right)  \in A$ and
$\left(  2,1\right)  \in A$), but the corresponding symmetric function $U_{D}$
is $p$-positive (indeed, $U_{D}=p_{1}^{4}+p_{2}p_{1}^{2}+p_{3}p_{1}$). It
would be interesting to know some more precise criteria for the $p$-positivity
of $U_{D}$.
\end{remark}

The next sections are devoted to the proofs of the above results. Afterwards,
we will proceed with further properties of the Redei-Berge symmetric functions
$U_{D}$ (Section \ref{sec.antipode}), applications to reproving Redei's and
Berge's theorems (Section \ref{sec.redei}) and a (not very substantial)
generalization (Section \ref{sec.matrixgen}).

\subsection*{Remark on alternative versions}

\begin{vershort}
This paper also has a detailed version \cite{verlong}, which includes more
details (and less handwaving) in some of the proofs (and some straightforward
proofs that have been omitted from the present version).
\end{vershort}

\begin{verlong}
You are reading the detailed version of this paper. For the standard version
(which is shorter by virtue of omitting some straightforward proofs and some
details), see \cite{vershort}.
\end{verlong}

\section{\label{sec.pf.thm.UX.1}Proof of Theorem \ref{thm.UX.1}}

In the following, we will outline the proof of Theorem \ref{thm.UX.1}. We hope
that the proof can still be simplified further.

\subsection{Basic conventions}

The following two conventions are popular in enumerative combinatorics, and we
too will use them on occasion:

\begin{convention}
\label{conv.number}The symbol \# shall mean \textquotedblleft
number\textquotedblright. For instance, $\left(  \text{\# of subsets of
}\left\{  1,2,3\right\}  \right)  =8$.
\end{convention}

\begin{convention}
\label{conv.iverson}We shall use the \emph{Iverson bracket notation}: For any
logical statement $\mathcal{A}$, we let $\left[  \mathcal{A}\right]  $ denote
the truth value of $\mathcal{A}$. This is the number $%
\begin{cases}
1, & \text{if }\mathcal{A}\text{ is true};\\
0, & \text{if }\mathcal{A}\text{ is false}.
\end{cases}
$
\end{convention}

Our proof of Theorem \ref{thm.UX.1} will rely on many lemmas. The first is a
well-known cancellation lemma (see, e.g., \cite[Proposition 7.8.10]{mps}):

\begin{lemma}
\label{lem.cancel}Let $B$ be a finite set. Then, $\sum_{F\subseteq B}\left(
-1\right)  ^{\left\vert F\right\vert }=\left[  B=\varnothing\right]  $.
\end{lemma}

\subsection{Path covers and linear sets}

We begin with some more notations:

\begin{definition}
Let $V$ be a finite set.

\begin{enumerate}
\item[\textbf{(a)}] A \emph{path} of $V$ means a nonempty tuple of distinct
elements of $V$.

\item[\textbf{(b)}] An element $v$ is said to \emph{belong} to a given tuple
$t$ if $v$ is an entry of $t$.

\item[\textbf{(c)}] A \emph{path cover} of $V$ means a set of paths of $V$
such that each $v\in V$ belongs to exactly one of these paths.
\end{enumerate}
\end{definition}

For example, $\left\{  \left(  1,4,3\right)  ,\left(  2,8\right)  ,\left(
5\right)  ,\left(  7,6\right)  \right\}  $ is a path cover of $\left[
8\right]  $. We stress once again the words \textquotedblleft exactly
one\textquotedblright\ in the definition of a path cover. Thus, the paths
constituting a path cover are disjoint (i.e., have no entries in common). For
instance, $\left\{  \left(  1,2\right)  ,\left(  2,3\right)  \right\}  $ is
\textbf{not} a path cover of $\left[  3\right]  $.

In Definition \ref{def.Arcs-and-Carcs} \textbf{(a)}, we have introduced the
arc set of a path of $V$ (and, more generally, of any nonempty tuple of
elements of $V$). We now extend this to path covers in the obvious way:

\begin{definition}
Let $V$ be a finite set.

\begin{enumerate}
\item[\textbf{(a)}] If $C$ is a path cover of $V$, then the \emph{arc set} of
$C$ is defined to be the subset%
\[
\bigcup_{v\in C}\operatorname*{Arcs}v\ \ \ \ \ \ \ \ \ \ \text{of }V\times V.
\]
This arc set will be denoted by $\operatorname*{Arcs}C$.

\item[\textbf{(b)}] A subset $F$ of $V\times V$ will be called \emph{linear}
if it is the arc set of some path cover of $V$.
\end{enumerate}
\end{definition}

For example, the path cover $\left\{  \left(  1,4,3\right)  ,\left(
2,8\right)  ,\left(  5\right)  ,\left(  7,6\right)  \right\}  $ of $\left[
8\right]  $ has arc set
\begin{align*}
&  \operatorname*{Arcs}\left\{  \left(  1,4,3\right)  ,\left(  2,8\right)
,\left(  5\right)  ,\left(  7,6\right)  \right\} \\
&  =\operatorname*{Arcs}\left(  1,4,3\right)  \cup\operatorname*{Arcs}\left(
2,8\right)  \cup\operatorname*{Arcs}\left(  5\right)  \cup\operatorname*{Arcs}%
\left(  7,6\right) \\
&  =\left\{  \left(  1,4\right)  ,\ \left(  4,3\right)  \right\}  \cup\left\{
\left(  2,8\right)  \right\}  \cup\varnothing\cup\left\{  \left(  7,6\right)
\right\} \\
&  =\left\{  \left(  1,4\right)  ,\ \left(  4,3\right)  ,\ \left(  2,8\right)
,\ \left(  7,6\right)  \right\}  .
\end{align*}
Thus, the latter set is linear (as a subset of $\left[  8\right]
\times\left[  8\right]  $).

Note that the notion of \textquotedblleft path of $V$\textquotedblright%
\ depends on $V$ alone, not on any digraph structure on $V$. Thus, if $V$ is
the vertex set of a digraph $D=\left(  V,A\right)  $, then a path of $V$ is
not the same as a $D$-path; in fact, the $D$-paths are precisely the paths $v$
of $V$ that satisfy $\operatorname*{Arcs}v\subseteq A$.

We shall now see a few properties and characterizations of linear subsets of
$V\times V$. Here is a first one, which will not be used in what follows but
might help in visualizing the concept:

\begin{proposition}
Let $V$ be a finite set. Let $F$ be a subset of $V\times V$. Then, $F$ is
linear if and only if the digraph $\left(  V,F\right)  $ has no cycles and no
vertices with outdegree $>1$ and no vertices with indegree $>1$.
\end{proposition}

We omit the proof of this proposition, since we shall have no use for it.

The following is also easy to see:

\begin{proposition}
\label{prop.linear.subset}Let $V$ be a finite set. Let $F$ be a linear subset
of $V\times V$. Then, any subset of $F$ is linear as well.
\end{proposition}

\begin{vershort}

\begin{proof}
It suffices to show that removing a single element $e$ from a linear subset
$F$ of $V\times V$ yields a linear subset. But this follows from the fact that
if we remove an arc $f$ from a path, then the path breaks up into two smaller
paths (the \textquotedblleft part before $f$\textquotedblright\ and the
\textquotedblleft part after $f$\textquotedblright).
\end{proof}
\end{vershort}

\begin{verlong}

\begin{proof}
It suffices to show that removing a single element $e$ from a linear subset
$F$ of $V\times V$ yields a linear subset. But this is easy:

Let $F$ be a linear subset of $V\times V$, and let $e$ be an element of $F$.
We must prove that $F\setminus\left\{  e\right\}  $ is again a linear subset.

The set $F$ is linear, i.e., is the arc set of some path cover of $V$ (by the
definition of \textquotedblleft linear\textquotedblright). In other words, we
have $F=\operatorname*{Arcs}C$ for some path cover $C$ of $V$. Consider this
path cover $C$.

We have $e\in F=\operatorname*{Arcs}C$. Thus, $e$ is an arc of some path $p\in
C$. Consider this $p$.

If we remove an arc $f$ from a path, then the path breaks up into two smaller
paths (the \textquotedblleft part before $f$\textquotedblright\ and the
\textquotedblleft part after $f$\textquotedblright)\ \ \ \ \footnote{To be
specific: If the path is $\left(  v_{1},v_{2},\ldots,v_{k}\right)  $, and if
the arc $f$ is $\left(  v_{i},v_{i+1}\right)  $, then the resulting two
smaller paths are $\left(  v_{1},v_{2},\ldots,v_{i}\right)  $ and $\left(
v_{i+1},v_{i+2},\ldots,v_{k}\right)  $.}. Thus, if we remove the arc $e$ from
the path $p$, then this path $p$ breaks up into the \textquotedblleft part
before $e$\textquotedblright\ and the \textquotedblleft part after
$e$\textquotedblright. Let us denote these two parts by $p^{\prime}$ and
$p^{\prime\prime}$. Let $C^{\prime}$ be the path cover of $V$ obtained from
$C$ by breaking up the path $p$ into its two parts $p^{\prime}$ and
$p^{\prime\prime}$ (that is, let $C^{\prime}:=\left(  C\setminus\left\{
p\right\}  \right)  \cup\left\{  p^{\prime},p^{\prime\prime}\right\}  $).
Then, $\operatorname*{Arcs}\left(  C^{\prime}\right)  =\underbrace{\left(
\operatorname*{Arcs}C\right)  }_{=F}\setminus\left\{  e\right\}
=F\setminus\left\{  e\right\}  $. This shows that $F\setminus\left\{
e\right\}  $ is the arc set of a path cover of $V$ (namely, of $C^{\prime}$).
In other words, $F\setminus\left\{  e\right\}  $ is linear. As explained
above, this completes our proof of Proposition \ref{prop.linear.subset}.
\end{proof}
\end{verlong}

This quickly leads to the following alternative characterization of linear subsets:

\begin{proposition}
\label{prop.linear.listing}Let $V$ be a finite set. Let $F$ be a subset of
$V\times V$. Then:

\begin{enumerate}
\item[\textbf{(a)}] If the subset $F$ is not linear, then there exists no
$V$-listing $v$ satisfying $F\subseteq\operatorname*{Arcs}v$.

\item[\textbf{(b)}] If $F=\operatorname*{Arcs}C$ for some path cover $C$ of
$V$, then there are exactly $\left\vert C\right\vert !$ many $V$-listings $v$
satisfying $F\subseteq\operatorname*{Arcs}v$. (Note that $\left\vert
C\right\vert $ is the number of paths in $C$.)

\item[\textbf{(c)}] The subset $F$ is linear if and only if it is a subset of
$\operatorname*{Arcs}v$ for some $V$-listing $v$.
\end{enumerate}
\end{proposition}

\begin{vershort}

\begin{proof}
\textbf{(a)} It clearly suffices to prove the contrapositive: i.e., that if
$F\subseteq\operatorname*{Arcs}v$ for some $V$-listing $v$, then $F$ is linear.

Let us prove this. Assume that $F\subseteq\operatorname*{Arcs}v$ for some
$V$-listing $v$. Consider this $F$. Then, $\operatorname*{Arcs}v$ is linear
(since $\operatorname*{Arcs}v=\operatorname*{Arcs}\left\{  v\right\}  $ for
the path cover $\left\{  v\right\}  $), and thus Proposition
\ref{prop.linear.subset} shows that $F$ is also linear (since $F$ is a subset
of $\operatorname*{Arcs}v$). This completes the proof of part \textbf{(a)}.
\medskip

\textbf{(b)} Assume that $F=\operatorname*{Arcs}C$ for some path cover $C$ of
$V$. Consider this $C$.

Then, each $V$-listing $v$ satisfying $F\subseteq\operatorname*{Arcs}v$ can be
obtained by concatenating the paths in $C$ in some order (and conversely, each
such concatenation is a $V$-listing $v$ satisfying $F\subseteq
\operatorname*{Arcs}v$). There are clearly $\left\vert C\right\vert !$ many
such concatenations (since there are $\left\vert C\right\vert !$ many orders),
and they all lead to different $V$-listings $v$ (since the paths in $C$ are
disjoint and nonempty). Hence, there are exactly $\left\vert C\right\vert !$
many $V$-listings $v$ satisfying $F\subseteq\operatorname*{Arcs}v$. This
proves Proposition \ref{prop.linear.listing} \textbf{(b)}. \medskip

\textbf{(c)} $\Longrightarrow:$ This follows from part \textbf{(b)} (since
$\left\vert C\right\vert !>0$).

$\Longleftarrow:$ This is just the contrapositive of part \textbf{(a)}.
\end{proof}
\end{vershort}

\begin{verlong}

\begin{proof}
\textbf{(a)} We shall prove the contrapositive: i.e., that if there exists a
$V$-listing $v$ satisfying $F\subseteq\operatorname*{Arcs}v$, then $F$ is linear.

Indeed, assume that there exists a $V$-listing $v$ satisfying $F\subseteq
\operatorname*{Arcs}v$. Consider this $v$. Then, $v$ is a path of $V$ that
contains all elements of $V$. Hence, the $1$-element set $\left\{  v\right\}
$ is a path cover of $V$. Its arc set $\operatorname*{Arcs}\left\{  v\right\}
$ is therefore linear (by the definition of \textquotedblleft
linear\textquotedblright). In other words, the set $\operatorname*{Arcs}v$ is
linear (since $\operatorname*{Arcs}\left\{  v\right\}  =\operatorname*{Arcs}%
v$). Hence, Proposition \ref{prop.linear.subset} (applied to
$\operatorname*{Arcs}v$ instead of $F$) shows that any subset of
$\operatorname*{Arcs}v$ is linear as well. Thus, $F$ is linear (since $F$ is a
subset of $\operatorname*{Arcs}v$). This completes the proof of Proposition
\ref{prop.linear.listing} \textbf{(a)}. \medskip

\textbf{(b)} Assume that $F=\operatorname*{Arcs}C$ for some path cover $C$ of
$V$. Consider this $C$.

Then, each $V$-listing $v$ satisfying $F\subseteq\operatorname*{Arcs}v$ can be
obtained by concatenating the paths in $C$ in some
order\footnote{\textit{Proof.} This might appear intuitively clear, but let us
give a proof nevertheless.
\par
Let $v$ be a $V$-listing satisfying $F\subseteq\operatorname*{Arcs}v$. We must
prove that $v$ can be obtained by concatenating the paths in $C$ in some
order.
\par
Let $p_{1},p_{2},\ldots,p_{k}$ be the paths in $C$ (listed with no
repetitions). Thus, we must prove that $v$ is a concatenation of $p_{1}%
,p_{2},\ldots,p_{k}$ in some order.
\par
Note that the paths $p_{1},p_{2},\ldots,p_{k}$ are the distinct paths of the
path cover $C$. Thus, they are disjoint, i.e., have no entries in common.
Moreover, every element of $V$ belongs to one of the paths $p_{1},p_{2}%
,\ldots,p_{k}$ (since $C$ is a path cover of $V$).
\par
Note that $v$ is a $V$-listing. Hence, each element of $V$ appears exactly
once in $v$. In particular, no entry appears more than once in $v$.
\par
Fix $i\in\left[  k\right]  $. Write the path $p_{i}$ as $p_{i}=\left(
w_{1},w_{2},\ldots,w_{\ell}\right)  $.
\par
Let $j\in\left[  \ell-1\right]  $. Then, in particular, $w_{j}$ appears
exactly once in $v$ (since each element of $V$ appears exactly once in $v$).
Moreover, we have
\begin{align*}
\left(  w_{j},w_{j+1}\right)   &  \in\operatorname*{Arcs}\left(  p_{i}\right)
\subseteq\operatorname*{Arcs}C\ \ \ \ \ \ \ \ \ \ \left(  \text{since }%
p_{i}\in C\right) \\
&  =F\subseteq\operatorname*{Arcs}v.
\end{align*}
Thus, the tuple $v$ must have the form $\left(  \ldots,w_{j},w_{j+1}%
,\ldots\right)  $ (where each \textquotedblleft$\ldots$\textquotedblright%
\ stands for an arbitrary number of entries). In other words, $w_{j+1}$ must
be the next entry after $w_{j}$ in the $V$-listing $v$ (since $w_{j}$ appears
exactly once in $v$).
\par
Forget that we fixed $j$. We thus have shown that for each $j\in\left[
\ell-1\right]  $, the element $w_{j+1}$ must be the next entry after $w_{j}$
in the $V$-listing $v$. In other words, the entries $w_{1},w_{2}%
,\ldots,w_{\ell}$ must appear in $v$ in this order and as a contiguous block.
In other words, the tuple $\left(  w_{1},w_{2},\ldots,w_{\ell}\right)  $ is a
factor of the tuple $v$ (where a \textquotedblleft\emph{factor}%
\textquotedblright\ of a tuple $\left(  a_{1},a_{2},\ldots,a_{g}\right)  $
means a contiguous block $\left(  a_{j},a_{j+1},\ldots,a_{j+s}\right)  $ of
this tuple). In other words, the path $p_{i}$ is a factor of the tuple $v$
(since $p_{i}=\left(  w_{1},w_{2},\ldots,w_{\ell}\right)  $).
\par
Forget that we fixed $i$. We thus have shown that for each $i\in\left[
k\right]  $, the path $p_{i}$ is a factor of $v$. In other words, all $k$
paths $p_{1},p_{2},\ldots,p_{k}$ are factors of $v$. These factors are
nonempty (since paths are nonempty by definition) and do not overlap (since
the paths $p_{1},p_{2},\ldots,p_{k}$ have no entries in common), and therefore
appear in $v$ in a well-defined order. In other words, there exists a
permutation $\sigma$ of $\left[  k\right]  $ such that the paths
$p_{\sigma\left(  1\right)  },p_{\sigma\left(  2\right)  },\ldots
,p_{\sigma\left(  k\right)  }$ appear as factors of $v$ in this order (i.e.,
the factor $p_{\sigma\left(  1\right)  }$ appears before $p_{\sigma\left(
2\right)  }$, which in turn appears before $p_{\sigma\left(  3\right)  }$, and
so on). Consider this $\sigma$. Note that every element of $V$ belongs to one
of the paths $p_{\sigma\left(  1\right)  },p_{\sigma\left(  2\right)  }%
,\ldots,p_{\sigma\left(  k\right)  }$ (since every element of $V$ belongs to
one of the paths $p_{1},p_{2},\ldots,p_{k}$).
\par
We claim that these $k$ factors $p_{\sigma\left(  1\right)  },p_{\sigma\left(
2\right)  },\ldots,p_{\sigma\left(  k\right)  }$ cover the entire tuple $v$
(that is, every entry of $v$ belongs to one of these factors). In fact, if
this was not the case, then some entry of $v$ would lie outside all of these
$k$ factors $p_{\sigma\left(  1\right)  },p_{\sigma\left(  2\right)  }%
,\ldots,p_{\sigma\left(  k\right)  }$; but then this same entry would also
appear again inside one of these $k$ factors (since every element of $V$
belongs to one of the paths $p_{\sigma\left(  1\right)  },p_{\sigma\left(
2\right)  },\ldots,p_{\sigma\left(  k\right)  }$), and therefore would appear
in $v$ twice (once outside the $k$ factors, and once again inside one of
them), which would contradict the fact that no entry appears more than once in
$v$. Hence, the $k$ factors $p_{\sigma\left(  1\right)  },p_{\sigma\left(
2\right)  },\ldots,p_{\sigma\left(  k\right)  }$ cover the entire tuple $v$.
Since these $k$ factors do not overlap (because they are just a permutation of
the $k$ factors $p_{1},p_{2},\ldots,p_{k}$, which do not overlap), and since
they appear in $v$ in this order, we thus conclude that $v$ is the
concatenation of $p_{\sigma\left(  1\right)  },p_{\sigma\left(  2\right)
},\ldots,p_{\sigma\left(  k\right)  }$ in this order. Therefore, $v$ is a
concatenation of $p_{1},p_{2},\ldots,p_{k}$ in some order. This completes our
proof.} (and conversely, each such concatenation is a $V$-listing $v$
satisfying $F\subseteq\operatorname*{Arcs}v$). There are clearly $\left\vert
C\right\vert !$ many orders in which the paths in $C$ can be concatenated, and
they all lead to different concatenations (since the paths in $C$ are disjoint
and nonempty\footnote{Here is this argument in more detail: If you concatenate
the paths in $C$ in some order, then the order in which you concatenate them
can be reconstructed from the resulting concatenation, since it is precisely
the order in which the first entries of the paths appear in the concatenation.
This relies on the fact that the paths are nonempty (so that these first
entries exist) and disjoint (so that each of these first entries appears only
once in the concatenation). Thus, different orders in which we can concatenate
the paths lead to different resulting concatenations.}), i.e., to different
$V$-listings $v$. Hence, there are exactly $\left\vert C\right\vert !$ many
$V$-listings $v$ satisfying $F\subseteq\operatorname*{Arcs}v$. This proves
Proposition \ref{prop.linear.listing} \textbf{(b)}. \medskip

\textbf{(c)} $\Longrightarrow:$ Assume that $F$ is linear. Thus, $F$ is the
arc set of a path cover of $V$. In other words, $F=\operatorname*{Arcs}C$ for
some path cover $C$ of $V$. Consider this $C$. Proposition
\ref{prop.linear.listing} \textbf{(b)} yields that there are exactly
$\left\vert C\right\vert !$ many $V$-listings $v$ satisfying $F\subseteq
\operatorname*{Arcs}v$. Hence, there is at least one such $V$-listing $v$
(since $\left\vert C\right\vert !\geq1$). Thus, $F$ is a subset of
$\operatorname*{Arcs}v$ for some $V$-listing $v$ (namely, the $V$-listing we
just mentioned). This proves the \textquotedblleft$\Longrightarrow
$\textquotedblright\ direction of Proposition \ref{prop.linear.listing}
\textbf{(c)}.

$\Longleftarrow:$ Assume that $F$ is a subset of $\operatorname*{Arcs}v$ for
some $V$-listing $v$. In other words, there exists a $V$-listing $v$
satisfying $F\subseteq\operatorname*{Arcs}v$. However, if the set $F$ was not
linear, then Proposition \ref{prop.linear.listing} \textbf{(a)} would yield
that there exists no such $V$-listing; this would contradict the preceding
sentence. Hence, the set $F$ must be linear. This proves the \textquotedblleft%
$\Longleftarrow$\textquotedblright\ direction of Proposition
\ref{prop.linear.listing} \textbf{(c)}.
\end{proof}
\end{verlong}

Next, let us address a technical issue. We defined the notion of a
\textquotedblleft linear subset of $V\times V$\textquotedblright\ using path
covers of $V$. When we say that a certain set is \textquotedblleft
linear\textquotedblright, we are thus tacitly assuming that it is clear what
the relevant set $V$ is. This may cause an ambiguity: Sometimes, two different
sets $V_{1}$ and $V_{2}$ can reasonably qualify as $V$, and we may have a
subset $F$ of $V_{1}\times V_{1}$ that is also a subset of $V_{2}\times V_{2}%
$. In that case, when we say that $F$ is \textquotedblleft
linear\textquotedblright, do we mean that $F$ is linear as a subset of
$V_{1}\times V_{1}$ or as a subset of $V_{2}\times V_{2}$ ? Fortunately, this
does not matter (at least when $V_{1}$ is a subset of $V_{2}$), as the
following proposition shows:

\begin{proposition}
\label{prop.linear.VW}Let $V$ be a finite set. Let $W$ be a subset of $V$. Let
$F$ be a subset of $W\times W$. Then, $F$ is linear as a subset of $W\times W$
if and only if $F$ is linear as a subset of $V\times V$.
\end{proposition}

\begin{vershort}

\begin{proof}
$\Longrightarrow:$ Assume that $F$ is linear as a subset of $W\times W$. Thus,
$F$ is the arc set of some path cover $C$ of $W$. If we add a trivial path
$\left(  v\right)  $ for each $v \in V \setminus W$ to this path cover $C$,
then it becomes a path cover of $V$, but its arc set does not change (and thus
remains $F$). Hence, $F$ is the arc set of the resulting path cover of $V$. In
other words, $F$ is linear as a subset of $V \times V$. \medskip

$\Longleftarrow:$ Assume that $F$ is linear as a subset of $V\times V$. Thus,
$F$ is the arc set of some path cover $C$ of $V$. Consider this $C$. For each
$v \in V \setminus W$, there must be a path in $C$ that contains $v$, and this
path must be the trivial path $\left(  v\right)  $ (since otherwise, this path
would have at least one arc containing $v$, and this arc would then belong to
$\operatorname{Arcs} C = F$; but this would contradict the fact that $F
\subseteq W \times W$). Hence, the path cover $C$ contains the trivial path
$\left(  v\right)  $ for each $v \in V \setminus W$. Removing all these
trivial paths will turn $C$ into a path cover of $W$, while leaving its arc
set unchanged (so it remains $F$). Hence, $F$ is the arc set of the resulting
path cover of $W$. In other words, $F$ is linear as a subset of $W \times W$.
\end{proof}
\end{vershort}

\begin{verlong}

\begin{proof}
A path $\left(  v_{1},v_{2},\ldots,v_{k}\right)  $ of $V$ will be called
\emph{trivial} if $k=1$, and \emph{nontrivial} otherwise. Clearly, if $v$ is a
trivial path, then $\operatorname*{Arcs}v=\varnothing$. On the other hand, if
$v$ is a nontrivial path, then $\operatorname*{Arcs}v\neq\varnothing$ (since a
path cannot be empty).

Now, we shall prove the \textquotedblleft$\Longrightarrow$\textquotedblright%
\ and \textquotedblleft$\Longleftarrow$\textquotedblright\ directions of
Proposition \ref{prop.linear.VW} separately:

$\Longrightarrow:$ Assume that $F$ is linear as a subset of $W\times W$. Thus,
$F$ is the arc set of some path cover of $W$. Let $C=\left\{  c_{1}%
,c_{2},\ldots,c_{k}\right\}  $ be this path cover; thus,
$F=\operatorname*{Arcs}C$. Now, let $v_{1},v_{2},\ldots,v_{\ell}$ be the
elements of $V\setminus W$ (each listed exactly once), and let us define a set%
\[
D:=\left\{  c_{1},c_{2},\ldots,c_{k},\left(  v_{1}\right)  ,\left(
v_{2}\right)  ,\ldots,\left(  v_{\ell}\right)  \right\}  =C\cup\left\{
\left(  v_{1}\right)  ,\left(  v_{2}\right)  ,\ldots,\left(  v_{\ell}\right)
\right\}  .
\]
Then, $D$ is a path cover of $V$ (in fact, $D$ is just the result of extending
$C$ to a path cover of $V$ by inserting a trivial path $\left(  v\right)  $
for each $v\in V\setminus W$). Furthermore,%
\begin{align*}
\operatorname*{Arcs}D  &  =\operatorname*{Arcs}\left(  C\cup\left\{  \left(
v_{1}\right)  ,\left(  v_{2}\right)  ,\ldots,\left(  v_{\ell}\right)
\right\}  \right)  \ \ \ \ \ \ \ \ \ \ \left(  \text{since }D=C\cup\left\{
\left(  v_{1}\right)  ,\left(  v_{2}\right)  ,\ldots,\left(  v_{\ell}\right)
\right\}  \right) \\
&  =\left(  \operatorname*{Arcs}C\right)  \cup\underbrace{\left(
\operatorname*{Arcs}\left(  \left(  v_{1}\right)  \right)  \right)
\cup\left(  \operatorname*{Arcs}\left(  \left(  v_{2}\right)  \right)
\right)  \cup\cdots\cup\left(  \operatorname*{Arcs}\left(  \left(  v_{\ell
}\right)  \right)  \right)  }_{\substack{=\varnothing\\\text{(since each
trivial path }\left(  v_{i}\right)  \text{ satisfies }\operatorname*{Arcs}%
\left(  \left(  v_{i}\right)  \right)  =\varnothing\text{)}}}\\
&  =\operatorname*{Arcs}C.
\end{align*}
Hence, $F=\operatorname*{Arcs}C=\operatorname*{Arcs}D$. This shows that $F$ is
the arc set of some path cover of $V$ (since $D$ is a path cover of $V$). In
other words, $F$ is linear as a subset of $V\times V$. The \textquotedblleft%
$\Longrightarrow$\textquotedblright\ direction of Proposition
\ref{prop.linear.VW} is thus proved.

$\Longleftarrow:$ Assume that $F$ is linear as a subset of $V\times V$. Thus,
$F$ is the arc set of some path cover of $V$. Let $C=\left\{  c_{1}%
,c_{2},\ldots,c_{k}\right\}  $ be this path cover; thus,
$F=\operatorname*{Arcs}C$. Now, let $v_{1},v_{2},\ldots,v_{\ell}$ be the
elements of $V\setminus W$ (each listed exactly once). Thus, $V\setminus
W=\left\{  v_{1},v_{2},\ldots,v_{\ell}\right\}  $.

Let $i\in\left[  \ell\right]  $. We shall prove that the trivial path $\left(
v_{i}\right)  $ belongs to $C$.

Indeed, from $v_{i}\in V\setminus W$, we obtain $v_{i}\in V$ and $v_{i}\notin
W$. The element $v_{i}$ must clearly belong to some path in $C$ (since $C$ is
a path cover of $V$). Let $p=\left(  p_{1},p_{2},\ldots,p_{r}\right)  $ be
this path. Thus, $v_{i}=p_{s}$ for some $s\in\left[  r\right]  $. Consider
this $s$.

Since the path $p$ belongs to $C$, we have $\operatorname*{Arcs}%
p\subseteq\operatorname*{Arcs}C=F\subseteq W\times W$.

However, if we had $s<r$, then we would have $\left(  p_{s},p_{s+1}\right)
\in\operatorname*{Arcs}p\subseteq W\times W$, which would entail $p_{s}\in W$,
which contradicts $p_{s}=v_{i}\notin W$. Thus, we cannot have $s<r$. Hence,
$s=r$ (since $s\in\left[  r\right]  $).

Furthermore, if we had $s>1$, then we would have $\left(  p_{s-1}%
,p_{s}\right)  \in\operatorname*{Arcs}p\subseteq W\times W$, which would
entail $p_{s}\in W$, which contradicts $p_{s}=v_{i}\notin W$. Thus, we cannot
have $s>1$. Hence, $s=1$ (since $s\in\left[  r\right]  $). Therefore,
$p_{s}=p_{1}$, so that $p_{1}=p_{s}=v_{i}$.

Comparing $s=1$ with $s=r$, we obtain $r=1$, so that $\left(  p_{1}%
,p_{2},\ldots,p_{r}\right)  =\left(  p_{1}\right)  =\left(  v_{i}\right)  $
(since $p_{1}=v_{i}$). Thus, $p=\left(  p_{1},p_{2},\ldots,p_{r}\right)
=\left(  v_{i}\right)  $, so that $\left(  v_{i}\right)  =p\in C$. In other
words, the trivial path $\left(  v_{i}\right)  $ belongs to $C$.

Forget that we fixed $i$. We thus have shown that for each $i\in\left[
\ell\right]  $, the trivial path $\left(  v_{i}\right)  $ belongs to $C$. In
other words, all $\ell$ trivial paths $\left(  v_{1}\right)  ,\left(
v_{2}\right)  ,\ldots,\left(  v_{\ell}\right)  $ belong to $C$. Let us set
$D:=C\setminus\left\{  \left(  v_{1}\right)  ,\left(  v_{2}\right)
,\ldots,\left(  v_{\ell}\right)  \right\}  $. Then,%
\[
C=D\cup\left\{  \left(  v_{1}\right)  ,\left(  v_{2}\right)  ,\ldots,\left(
v_{\ell}\right)  \right\}
\]
(since $\left(  v_{1}\right)  ,\left(  v_{2}\right)  ,\ldots,\left(  v_{\ell
}\right)  $ belong to $C$), so that%
\begin{align*}
\operatorname*{Arcs}C  &  =\operatorname*{Arcs}\left(  D\cup\left\{  \left(
v_{1}\right)  ,\left(  v_{2}\right)  ,\ldots,\left(  v_{\ell}\right)
\right\}  \right) \\
&  =\left(  \operatorname*{Arcs}D\right)  \cup\underbrace{\left(
\operatorname*{Arcs}\left(  \left(  v_{1}\right)  \right)  \right)
\cup\left(  \operatorname*{Arcs}\left(  \left(  v_{2}\right)  \right)
\right)  \cup\cdots\cup\left(  \operatorname*{Arcs}\left(  \left(  v_{\ell
}\right)  \right)  \right)  }_{\substack{=\varnothing\\\text{(since each
trivial path }\left(  v_{i}\right)  \text{ satisfies }\operatorname*{Arcs}%
\left(  \left(  v_{i}\right)  \right)  =\varnothing\text{)}}}\\
&  =\operatorname*{Arcs}D.
\end{align*}
In other words, $F=\operatorname*{Arcs}D$ (since $F=\operatorname*{Arcs}C$).

Recall that $C$ is a path cover of $V$. Hence, each $v\in V$ belongs to
exactly one path in $C$. Thus, it is easy to see that each path in $D$ is a
path of $W$\ \ \ \ \footnote{\textit{Proof.} Assume the contrary. Thus, some
path in $D$ is not a path of $W$. In other words, some path in $D$ contains an
element of $V\setminus W$ (since each path in $D$ is a path of $V$). Let $p$
be this path, and let $v$ be this element. Thus, $p\in D$ and $v\in V\setminus
W$, and $v$ belongs to $p$. Note that $p\in D\subseteq C$ (by the definition
of $D$).
\par
However, $v\in V\setminus W=\left\{  v_{1},v_{2},\ldots,v_{\ell}\right\}  $.
Thus, $v=v_{i}$ for some $i\in\left[  \ell\right]  $. Consider this $i$. Thus,
$v$ belongs to the trivial path $\left(  v_{i}\right)  $ (since $v=v_{i}$). We
have $\left(  v_{i}\right)  \in C$ (since all $\ell$ trivial paths $\left(
v_{1}\right)  ,\left(  v_{2}\right)  ,\ldots,\left(  v_{\ell}\right)  $ belong
to $C$) and $\left(  v_{i}\right)  \notin D$ (since $D$ was defined as
$C\setminus\left\{  \left(  v_{1}\right)  ,\left(  v_{2}\right)
,\ldots,\left(  v_{\ell}\right)  \right\}  $).
\par
Recall that $v$ belongs to exactly one path in $C$ (since $C$ is a path cover
of $V$). Since $v$ belongs to both paths $\left(  v_{i}\right)  $ and $p$
(both of which are paths in $C$, since $\left(  v_{i}\right)  \in C$ and $p\in
C$), this entails that these two paths $\left(  v_{i}\right)  $ and $p$ must
be identical. Hence, $\left(  v_{i}\right)  =p\in D$. But this contradicts
$\left(  v_{i}\right)  \notin D$. This contradiction shows that our assumption
was false, qed.}. Therefore, $D$ is a set of paths of $W$. Furthermore, each
$v\in W$ belongs to exactly one of these paths\footnote{\textit{Proof.} Let
$v\in W$. We must prove that $v$ belongs to exactly one of the paths in $D$.
\par
First of all, $v\in W\subseteq V$. Hence, $v$ belongs to exactly one path in
$C$ (since $C$ is a path cover of $V$).
\par
On the other hand, we have $v\notin\left\{  v_{1},v_{2},\ldots,v_{\ell
}\right\}  $ (because otherwise, we would have $v\in\left\{  v_{1}%
,v_{2},\ldots,v_{\ell}\right\}  =V\setminus W$, which would entail $v\notin
W$, but this would contradict $v\in W$). In other words, $v$ is not one of the
$\ell$ elements $v_{1},v_{2},\ldots,v_{\ell}$. In other words, $v$ belongs to
none of the trivial paths $\left(  v_{1}\right)  ,\left(  v_{2}\right)
,\ldots,\left(  v_{\ell}\right)  $. In other words, $v$ belongs to no paths in
$\left\{  \left(  v_{1}\right)  ,\left(  v_{2}\right)  ,\ldots,\left(
v_{\ell}\right)  \right\}  $.
\par
Now we know that:
\par
\begin{itemize}
\item the element $v$ belongs to exactly one path in $C$, but
\par
\item the element $v$ belongs to no paths in $\left\{  \left(  v_{1}\right)
,\left(  v_{2}\right)  ,\ldots,\left(  v_{\ell}\right)  \right\}  $.
\end{itemize}
\par
Combining these two facts, we see that $v$ must belong to exactly one path in
$C\setminus\left\{  \left(  v_{1}\right)  ,\left(  v_{2}\right)
,\ldots,\left(  v_{\ell}\right)  \right\}  $. In other words, $v$ belongs to
exactly one path in $D$ (since $D=C\setminus\left\{  \left(  v_{1}\right)
,\left(  v_{2}\right)  ,\ldots,\left(  v_{\ell}\right)  \right\}  $). This
completes our proof.}. Hence, $D$ is a path cover of $W$. Since
$F=\operatorname*{Arcs}D$, we thus conclude that $F$ is the arc set of some
path cover of $W$. In other words, $F$ is linear as a subset of $W\times W$.
The \textquotedblleft$\Longleftarrow$\textquotedblright\ direction of
Proposition \ref{prop.linear.VW} is thus proved.
\end{proof}
\end{verlong}

We will also use the following fact:

\begin{proposition}
\label{prop.linear.Vi}Let $V$ be a finite set. Let $V_{1},V_{2},\ldots,V_{k}$
be several disjoint subsets of $V$ such that $V=V_{1}\cup V_{2}\cup\cdots\cup
V_{k}$. For each $i\in\left[  k\right]  $, let $F_{i}$ be a subset of
$V_{i}\times V_{i}$. Let $F=F_{1}\cup F_{2}\cup\cdots\cup F_{k}$. Then, the
set $F$ is linear (as a subset of $V\times V$) if and only if all the subsets
$F_{i}$ for $i\in\left[  k\right]  $ are linear.
\end{proposition}

\begin{vershort}
\begin{proof}
This is straightforward.
\end{proof}
\end{vershort}

\begin{verlong}
\begin{proof}
$\Longrightarrow:$ Assume that $F$ is linear. Then, for each $i\in\left[
k\right]  $, the set $F_{i}$ is a subset of $F$ (since $F=F_{1}\cup F_{2}%
\cup\cdots\cup F_{k}\supseteq F_{i}$) and therefore is also
linear\footnote{Here, we are tacitly using Proposition \ref{prop.linear.VW},
which allows us to equivocate between \textquotedblleft linear as a subset of
$V\times V$\textquotedblright\ and \textquotedblleft linear as a subset of
$V_{i}\times V_{i}$\textquotedblright.} (by Proposition
\ref{prop.linear.subset}). Thus, the \textquotedblleft$\Longrightarrow
$\textquotedblright\ direction of Proposition \ref{prop.linear.Vi} is proved.

$\Longleftarrow:$ Assume that all the subsets $F_{i}$ for $i\in\left[
k\right]  $ are linear. We must prove that $F$ is linear.

Let $i\in\left[  k\right]  $. Then, $F_{i}$ is a linear subset of $V_{i}\times
V_{i}$. In other words, $F_{i}$ is the arc set of some path cover of $V_{i}$.
Let $C_{i}$ be this path cover; thus, $F_{i}=\operatorname*{Arcs}\left(
C_{i}\right)  $.

Forget that we fixed $i$. Thus, for each $i\in\left[  k\right]  $, we have
constructed a path cover $C_{i}$ of $V_{i}$ satisfying $F_{i}%
=\operatorname*{Arcs}\left(  C_{i}\right)  $. It is easy to see that the union
$C_{1}\cup C_{2}\cup\cdots\cup C_{k}$ of these path covers $C_{1},C_{2}%
,\ldots,C_{k}$ is a path cover of $V_{1}\cup V_{2}\cup\cdots\cup V_{k}$ (since
$V_{1},V_{2},\ldots,V_{k}$ are disjoint sets). In other words, $C_{1}\cup
C_{2}\cup\cdots\cup C_{k}$ is a path cover of $V$ (since $V=V_{1}\cup
V_{2}\cup\cdots\cup V_{k}$). Moreover, the arc set of this path cover is%
\begin{align*}
\operatorname*{Arcs}\left(  C_{1}\cup C_{2}\cup\cdots\cup C_{k}\right)   &
=\left(  \operatorname*{Arcs}\left(  C_{1}\right)  \right)  \cup\left(
\operatorname*{Arcs}\left(  C_{2}\right)  \right)  \cup\cdots\cup\left(
\operatorname*{Arcs}\left(  C_{k}\right)  \right) \\
&  =F_{1}\cup F_{2}\cup\cdots\cup F_{k}\ \ \ \ \ \ \ \ \ \ \left(  \text{since
}\operatorname*{Arcs}\left(  C_{i}\right)  =F_{i}\text{ for each }i\in\left[
k\right]  \right) \\
&  =F.
\end{align*}
Hence, $F$ is the arc set of a path cover of $V$ (namely, of the path cover
$C_{1}\cup C_{2}\cup\cdots\cup C_{k}$). In other words, $F$ is linear. This
proves the \textquotedblleft$\Longleftarrow$\textquotedblright\ direction of
Proposition \ref{prop.linear.Vi}.
\end{proof}
\end{verlong}

\subsection{The arrow set of a permutation}

We will now see another way to obtain subsets of $V\times V$:

\begin{definition}
\label{def.Asigma}Let $V$ be a finite set. Let $\sigma$ be a permutation of
$V$. Then, $\mathbf{A}_{\sigma}$ shall denote the subset%
\[
\left\{  \left(  v,\sigma\left(  v\right)  \right)  \ \mid\ v\in V\right\}
=\bigcup_{c\in\operatorname*{Cycs}\sigma}\operatorname*{CArcs}c
\]
of $V\times V$.
\end{definition}

\begin{example}
Let $V=\left\{  1,2,3,4,5,6\right\}  $, and let $\sigma$ be the permutation of
$V$ that sends $1,2,3,4,5,6$ to $2,3,1,5,4,6$ (respectively). Then,
\[
\operatorname*{Cycs}\sigma=\left\{  \left(  1,2,3\right)  ,\ \left(
4,5\right)  ,\ \left(  6\right)  \right\}
\]
and%
\begin{align*}
\mathbf{A}_{\sigma}  &  =\left\{  \left(  1,2\right)  ,\ \left(  2,3\right)
,\ \left(  3,1\right)  ,\ \left(  4,5\right)  ,\ \left(  5,4\right)
,\ \left(  6,6\right)  \right\} \\
&  =\underbrace{\operatorname*{CArcs}\left(  1,2,3\right)  }_{=\left\{
\left(  1,2\right)  ,\ \left(  2,3\right)  ,\ \left(  3,1\right)  \right\}
}\cup\underbrace{\operatorname*{CArcs}\left(  4,5\right)  }_{=\left\{  \left(
4,5\right)  ,\ \left(  5,4\right)  \right\}  }\cup
\underbrace{\operatorname*{CArcs}\left(  6\right)  }_{=\left\{  \left(
6,6\right)  \right\}  }.
\end{align*}

\end{example}

The following is a counterpart to Proposition \ref{prop.linear.listing}
\textbf{(b)}:

\begin{proposition}
\label{prop.linear.Aisg-num}Let $V$ be a finite set. Let $F$ be a subset of
$V\times V$. If $F=\operatorname*{Arcs}C$ for some path cover $C$ of $V$, then
there are exactly $\left\vert C\right\vert !$ many permutations $\sigma
\in\mathfrak{S}_{V}$ satisfying $F\subseteq\mathbf{A}_{\sigma}$. (Note that
$\left\vert C\right\vert $ is the number of paths in $C$.)
\end{proposition}

\begin{vershort}
\begin{proof}
Assume that $F=\operatorname*{Arcs}C$ for some path cover $C$ of $V$. Consider
this $C$. We shall refer to the paths in $C$ as \textquotedblleft%
$C$-paths\textquotedblright.

Let $k=\left\vert C\right\vert $. Let $s_{1},s_{2},\ldots,s_{k}$ be the
starting points (i.e., first entries) of the $C$-paths, and let $t_{1}%
,t_{2},\ldots,t_{k}$ be their respective ending points (i.e., last entries).
We note that a permutation $\sigma\in\mathfrak{S}_{V}$ satisfies
$F\subseteq\mathbf{A}_{\sigma}$ if and only if it has the property that
$\sigma\left(  v\right)  =w$ whenever $v$ and $w$ are two consecutive entries
of a $C$-path. Thus, the condition $F\subseteq\mathbf{A}_{\sigma}$ uniquely
determines the value $\sigma\left(  v\right)  $ for each $v\in V\setminus
\left\{  t_{1},t_{2},\ldots,t_{k}\right\}  $ (namely, $\sigma\left(  v\right)
$ has to be the next entry after $v$ on the $C$-path that contains $v$), and
uniquely determines the value $\sigma^{-1}\left(  w\right)  $ for each $w\in
V\setminus\left\{  s_{1},s_{2},\ldots,s_{k}\right\}  $ (namely, $\sigma
^{-1}\left(  w\right)  $ has to be the entry just before $v$ on the $C$-path
that contains $v$).

Hence, in order to construct a permutation $\sigma\in\mathfrak{S}_{V}$
satisfying $F\subseteq\mathbf{A}_{\sigma}$, we only need to specify the $k$
values $\sigma\left(  t_{1}\right)  ,\sigma\left(  t_{2}\right)
,\ldots,\sigma\left(  t_{k}\right)  $ (since all other values $\sigma\left(
v\right)  $ are already decided by the requirement $F\subseteq\mathbf{A}%
_{\sigma}$), and we must choose these $k$ values from the set $\left\{
s_{1},s_{2},\ldots,s_{k}\right\}  $ (since all other elements of $V$ have
already been assigned preimages under $\sigma$ by the requirement
$F\subseteq\mathbf{A}_{\sigma}$). Thus, we must choose a bijection from the
$k$-element set $\left\{  t_{1},t_{2},\ldots,t_{k}\right\}  $ to the
$k$-element set $\left\{  s_{1},s_{2},\ldots,s_{k}\right\}  $. This can be
done in $k!$ many ways, i.e., in $\left\vert C\right\vert !$ many ways (since
$k=\left\vert C\right\vert $). Thus, there are exactly $\left\vert
C\right\vert !$ many permutations $\sigma\in\mathfrak{S}_{V}$ satisfying
$F\subseteq\mathbf{A}_{\sigma}$.
\end{proof}
\end{vershort}

\begin{verlong}
Proposition \ref{prop.linear.Aisg-num} is easily proved, but the proof is
tricky to formalize due to its reliance on some enumerative ideas that are
intuitively clear yet notationally intricate. To prepare for this proof, we
begin with some basic enumerative results. First, a notation:

\begin{definition}
An \emph{injection} shall mean an injective map.
\end{definition}

(Of course, this is analogous to the concept of a \emph{bijection}, which
means a bijective map.)

Now, we can state our first enumerative result:\footnote{Recall Convention
\ref{conv.number}.}

\begin{proposition}
\label{prop.count.inj-maps.extend}Let $X$, $Y$ and $Z$ be three finite sets
such that $Y\subseteq X$. Let $f:Y\rightarrow Z$ be any injection. Then,%
\[
\left(  \text{\# of injections }g:X\rightarrow Z\text{ such that }\left.
g\mid_{Y}\right.  =f\right)  =\prod_{k=0}^{\left\vert X\right\vert -\left\vert
Y\right\vert -1}\left(  \left\vert Z\right\vert -\left\vert Y\right\vert
-k\right)  .
\]

\end{proposition}

\begin{proof}
We proceed by induction on $\left\vert X\setminus Y\right\vert $:

\textit{Base case:} If $\left\vert X\setminus Y\right\vert =0$, then the claim
of Proposition \ref{prop.count.inj-maps.extend} is easy to
verify\footnote{\textit{Proof.} Assume that $\left\vert X\setminus
Y\right\vert =0$. Then, $X\setminus Y=\varnothing$, so that $X\subseteq Y$.
Combining this with $Y\subseteq X$, we obtain $Y=X$. Hence, for any injection
$g:X\rightarrow Z$, we have $\left.  g\mid_{Y}\right.  =\left.  g\mid
_{X}\right.  =g$. Thus,%
\begin{align*}
&  \left(  \text{\# of injections }g:\underbrace{X}_{=Y}\rightarrow Z\text{
such that }\underbrace{\left.  g\mid_{Y}\right.  }_{=g}=f\right) \\
&  =\left(  \text{\# of injections }g:Y\rightarrow Z\text{ such that
}g=f\right) \\
&  =1
\end{align*}
(since $f$ itself is an injection $g:Y\rightarrow Z$ such that $g=f$, and
clearly there are no other such injections). Comparing this with%
\begin{align*}
&  \prod_{k=0}^{\left\vert X\right\vert -\left\vert Y\right\vert -1}\left(
\left\vert Z\right\vert -\left\vert Y\right\vert -k\right) \\
&  =\prod_{k=0}^{-1}\left(  \left\vert Z\right\vert -\left\vert Y\right\vert
-k\right)  \ \ \ \ \ \ \ \ \ \ \left(  \text{since }\left\vert X\right\vert
-\left\vert \underbrace{Y}_{=X}\right\vert -1=\left\vert X\right\vert
-\left\vert X\right\vert -1=-1\right) \\
&  =\left(  \text{empty product}\right)  =1,
\end{align*}
we obtain%
\[
\left(  \text{\# of injections }g:X\rightarrow Z\text{ such that }\left.
g\mid_{Y}\right.  =f\right)  =\prod_{k=0}^{\left\vert X\right\vert -\left\vert
Y\right\vert -1}\left(  \left\vert Z\right\vert -\left\vert Y\right\vert
-k\right)  .
\]
Thus, Proposition \ref{prop.count.inj-maps.extend} is proved under the
assumption that $\left\vert X\setminus Y\right\vert =0$.}. This completes the
base case.

\textit{Induction step:} Let $m\in\mathbb{N}$. Assume (as the induction
hypothesis) that Proposition \ref{prop.count.inj-maps.extend} holds for
$\left\vert X\setminus Y\right\vert =m$. We must now prove that Proposition
\ref{prop.count.inj-maps.extend} holds for $\left\vert X\setminus Y\right\vert
=m+1$ as well.

So let $X$, $Y$ and $Z$ be three finite sets such that $Y\subseteq X$. Let
$f:Y\rightarrow Z$ be any injection. Assume that $\left\vert X\setminus
Y\right\vert =m+1$. We must then prove that%
\[
\left(  \text{\# of injections }g:X\rightarrow Z\text{ such that }\left.
g\mid_{Y}\right.  =f\right)  =\prod_{k=0}^{\left\vert X\right\vert -\left\vert
Y\right\vert -1}\left(  \left\vert Z\right\vert -\left\vert Y\right\vert
-k\right)  .
\]

It is well-known that if $A$ and $B$ are two finite sets, and if
$\varphi:A\rightarrow B$ is an injection, then $\left\vert \varphi\left(
A\right)  \right\vert =\left\vert A\right\vert $ (since an injection sends
distinct elements to distinct elements, and thus all the $\left\vert
A\right\vert $ distinct elements of $A$ give rise to $\left\vert A\right\vert
$ distinct elements of $\varphi\left(  A\right)  $). We can apply this to
$A=Y$ and $B=Z$ and $\varphi=f$ (since $f:Y\rightarrow Z$ is an injection);
thus we obtain $\left\vert f\left(  Y\right)  \right\vert =\left\vert
Y\right\vert $.

From $Y\subseteq X$, we obtain $\left\vert X\setminus Y\right\vert =\left\vert
X\right\vert -\left\vert Y\right\vert $, so that $\left\vert X\right\vert
-\left\vert Y\right\vert =\left\vert X\setminus Y\right\vert =m+1\geq1$ (since
$m\geq0$). Hence, $\left\vert X\right\vert -\left\vert Y\right\vert -1\geq0$.

We have $\left\vert X\setminus Y\right\vert =m+1>m\geq0$, so that the set
$X\setminus Y$ is nonempty. In other words, the set $X\setminus Y$ has at
least one element $p$. Consider this $p$. (We can choose $p$ arbitrarily, but
we then keep it fixed for the rest of this proof.)

Thus, $p\in X\setminus Y$. In other words, $p\in X$ and $p\notin Y$. Let
$Y^{\prime}$ be the set $Y\cup\left\{  p\right\}  $. Thus,%
\[
Y^{\prime}=Y\cup\left\{  p\right\}  \subseteq X\ \ \ \ \ \ \ \ \ \ \left(
\text{since }Y\subseteq X\text{ and }p\in X\right)  .
\]
Furthermore, from $Y^{\prime}=Y\cup\left\{  p\right\}  $, we obtain%
\[
\left\vert Y^{\prime}\right\vert =\left\vert Y\cup\left\{  p\right\}
\right\vert =\left\vert Y\right\vert +1\ \ \ \ \ \ \ \ \ \ \left(  \text{since
}p\notin Y\right)  .
\]
Since $Y^{\prime}\subseteq X$, we have%
\begin{align*}
\left\vert X\setminus Y^{\prime}\right\vert  &  =\left\vert X\right\vert
-\underbrace{\left\vert Y^{\prime}\right\vert }_{=\left\vert Y\right\vert
+1}=\left\vert X\right\vert -\left(  \left\vert Y\right\vert +1\right)
=\underbrace{\left\vert X\right\vert -\left\vert Y\right\vert }%
_{\substack{=\left\vert X\setminus Y\right\vert \\\text{(since }Y\subseteq
X\text{)}}}-1\\
&  =\underbrace{\left\vert X\setminus Y\right\vert }_{=m+1}-1=\left(
m+1\right)  -1=m.
\end{align*}

From $Y^{\prime}=Y\cup\left\{  p\right\}  $, we also obtain
\[
Y^{\prime}\setminus\left\{  p\right\}  =\left(  Y\cup\left\{  p\right\}
\right)  \setminus\left\{  p\right\}  =Y\ \ \ \ \ \ \ \ \ \ \left(
\text{since }p\notin Y\right)  .
\]
Also,%
\[
Y\subseteq Y\cup\left\{  p\right\}  =Y^{\prime}.
\]

If $y\in Y^{\prime}$ is an element that satisfies $y\neq p$, then $y\in Y$
(since $y\in Y^{\prime}$ and $y\neq p$ lead to $y\in Y^{\prime}\setminus
\left\{  p\right\}  =Y$), and therefore $f\left(  y\right)  $ is well-defined
(since $f:Y\rightarrow Z$ is a map).

For any $z\in Z\setminus f\left(  Y\right)  $, we define a map
$f_{p\rightarrow z}:Y^{\prime}\rightarrow Z$ by setting%
\[
f_{p\rightarrow z}\left(  y\right)  =%
\begin{cases}
z, & \text{if }y=p;\\
f\left(  y\right)  , & \text{if }y\neq p
\end{cases}
\ \ \ \ \ \ \ \ \ \ \text{for each }y\in Y^{\prime}.
\]
This is well-defined, because if $y\in Y^{\prime}$ is an element that
satisfies $y\neq p$, then $f\left(  y\right)  $ is well-defined (as we saw in
the previous paragraph).

Now we claim the following:

\begin{statement}
\textit{Claim 1:} Let $z\in Z\setminus f\left(  Y\right)  $. Then, the map
$f_{p\rightarrow z}:Y^{\prime}\rightarrow Z$ is an injection.
\end{statement}

\begin{proof}
[Proof of Claim 1.]Let $u$ and $v$ be two elements of $Y^{\prime}$ satisfying
$f_{p\rightarrow z}\left(  u\right)  =f_{p\rightarrow z}\left(  v\right)  $.
We shall show that $u=v$.

Indeed, we are in one of the following four cases:

\textit{Case 1:} We have $u=p$ and $v=p$.

\textit{Case 2:} We have $u=p$ but not $v=p$.

\textit{Case 3:} We have $v=p$ but not $u=p$.

\textit{Case 4:} We have neither $u=p$ nor $v=p$.

Let us first consider Case 1. In this case, we have $u=p$ and $v=p$. Hence,
$u=p=v$. Thus, $u=v$ has been proved in Case 1.

Let us now consider Case 2. In this case, we have $u=p$ but not $v=p$. Hence,
$v\neq p$ (since we do not have $v=p$). Combining $v\in Y^{\prime}$ with
$v\neq p$, we obtain $v\in Y^{\prime}\setminus\left\{  p\right\}  =Y$.

The definition of $f_{p\rightarrow z}$ yields%
\[
f_{p\rightarrow z}\left(  u\right)  =%
\begin{cases}
z, & \text{if }u=p;\\
f\left(  u\right)  , & \text{if }u\neq p
\end{cases}
\ \ =z\ \ \ \ \ \ \ \ \ \ \left(  \text{since }u=p\right)  ,
\]
so that
\begin{align*}
z  &  =f_{p\rightarrow z}\left(  u\right)  =f_{p\rightarrow z}\left(  v\right)
\\
&  =%
\begin{cases}
z, & \text{if }v=p;\\
f\left(  v\right)  , & \text{if }v\neq p
\end{cases}
\ \ \ \ \ \ \ \ \ \ \left(  \text{by the definition of }f_{p\rightarrow
z}\right) \\
&  =f\left(  v\right)  \ \ \ \ \ \ \ \ \ \ \left(  \text{since }v\neq p\right)
\\
&  \in f\left(  Y\right)  \ \ \ \ \ \ \ \ \ \ \left(  \text{since }v\in
Y\right)  .
\end{align*}
However, from $z\in Z\setminus f\left(  Y\right)  $, we obtain $z\notin
f\left(  Y\right)  $. This contradicts $z\in f\left(  Y\right)  $. Thus we
have obtained a contradiction in Case 2. Hence, Case 2 cannot occur.

A similar argument (with the roles of $u$ and $v$ swapped) shows that Case 3
cannot occur.

Let us finally consider Case 4. In this case, we have neither $u=p$ nor $v=p$.
In other words, we have $u\neq p$ and $v\neq p$. Combining $v\in Y^{\prime}$
with $v\neq p$, we obtain $v\in Y^{\prime}\setminus\left\{  p\right\}  =Y$.
Similarly, $u\in Y$.

The definition of $f_{p\rightarrow z}$ yields%
\[
f_{p\rightarrow z}\left(  u\right)  =%
\begin{cases}
z, & \text{if }u=p;\\
f\left(  u\right)  , & \text{if }u\neq p
\end{cases}
\ \ =f\left(  u\right)  \ \ \ \ \ \ \ \ \ \ \left(  \text{since }u\neq
p\right)  .
\]
Hence,%
\begin{align*}
f\left(  u\right)   &  =f_{p\rightarrow z}\left(  u\right)  =f_{p\rightarrow
z}\left(  v\right) \\
&  =%
\begin{cases}
z, & \text{if }v=p;\\
f\left(  v\right)  , & \text{if }v\neq p
\end{cases}
\ \ \ \ \ \ \ \ \ \ \left(  \text{by the definition of }f_{p\rightarrow
z}\right) \\
&  =f\left(  v\right)  \ \ \ \ \ \ \ \ \ \ \left(  \text{since }v\neq
p\right)  .
\end{align*}

However, the map $f$ is an injection, i.e., is injective. Thus, if $a$ and $b$
are two elements of $Y$ satisfying $f\left(  a\right)  =f\left(  b\right)  $,
then $a=b$. Applying this to $a=u$ and $b=v$, we obtain $u=v$ (since $f\left(
u\right)  =f\left(  v\right)  $). Thus, we have proved $u=v$ in Case 4.

Let us summarize: We have proved that Cases 2 and 3 cannot occur; thus, we
must actually be in one of the two Cases 1 and 4. But we have also proved that
$u=v$ in each of the latter two Cases 1 and 4. Hence, $u=v$ always holds.

Forget that we fixed $u$ and $v$. We thus have shown that if $u$ and $v$ are
two elements of $Y^{\prime}$ satisfying $f_{p\rightarrow z}\left(  u\right)
=f_{p\rightarrow z}\left(  v\right)  $, then $u=v$. In other words, the map
$f_{p\rightarrow z}:Y^{\prime}\rightarrow Z$ is injective, i.e., an injection.
This proves Claim 1.
\end{proof}

\begin{statement}
\textit{Claim 2:} Let $g:X\rightarrow Z$ be any injection such that $\left.
g\mid_{Y}\right.  =f$. Then, $g\left(  p\right)  \in Z\setminus f\left(
Y\right)  $.
\end{statement}

\begin{proof}
[Proof of Claim 2.]Clearly, $g\left(  p\right)  \in Z$. We shall now show that
$g\left(  p\right)  \notin f\left(  Y\right)  $.

Indeed, assume the contrary. Thus, $g\left(  p\right)  \in f\left(  Y\right)
$. In other words, $g\left(  p\right)  =f\left(  y\right)  $ for some $y\in
Y$. Consider this $y$. From $y\in Y$, we obtain $\left(  g\mid_{Y}\right)
\left(  y\right)  =g\left(  y\right)  $. Hence, $g\left(  y\right)
=\underbrace{\left(  g\mid_{Y}\right)  }_{=f}\left(  y\right)  =f\left(
y\right)  =g\left(  p\right)  $ (since $g\left(  p\right)  =f\left(  y\right)
$).

However, the map $g$ is an injection, i.e., is injective. Hence, if $a$ and
$b$ are two elements of $X$ satisfying $g\left(  a\right)  =g\left(  b\right)
$, then $a=b$. Applying this to $a=y$ and $b=p$, we obtain $y=p$ (since
$g\left(  y\right)  =g\left(  p\right)  $). Hence, $p=y\in Y$. But this
contradicts $p\notin Y$. This contradiction shows that our assumption was false.

Hence, $g\left(  p\right)  \notin f\left(  Y\right)  $ is proved. Now,
combining $g\left(  p\right)  \in Z$ with $g\left(  p\right)  \notin f\left(
Y\right)  $, we obtain $g\left(  p\right)  \in Z\setminus f\left(  Y\right)
$. This proves Claim 2.
\end{proof}

\begin{statement}
\textit{Claim 3:} Let $g:X\rightarrow Z$ be any injection. Let $z\in
Z\setminus f\left(  Y\right)  $. Then, the statement \textquotedblleft$\left.
g\mid_{Y}\right.  =f$ and $g\left(  p\right)  =z$\textquotedblright\ is
equivalent to the statement \textquotedblleft$\left.  g\mid_{Y^{\prime}%
}\right.  =f_{p\rightarrow z}$\textquotedblright.
\end{statement}

\begin{proof}
[Proof of Claim 3.]We must prove that these two statements are equivalent,
i.e., that each of them implies the other.

Let us first show that the statement \textquotedblleft$\left.  g\mid
_{Y}\right.  =f$ and $g\left(  p\right)  =z$\textquotedblright\ implies the
statement \textquotedblleft$\left.  g\mid_{Y^{\prime}}\right.
=f_{p\rightarrow z}$\textquotedblright.

\textit{Proof that \textquotedblleft}$\left.  g\mid_{Y}\right.  =f$\textit{
and }$g\left(  p\right)  =z$\textit{\textquotedblright\ implies
\textquotedblleft}$\left.  g\mid_{Y^{\prime}}\right.  =f_{p\rightarrow z}%
$\textit{\textquotedblright:} Assume that the statement \textquotedblleft%
$\left.  g\mid_{Y}\right.  =f$ and $g\left(  p\right)  =z$\textquotedblright%
\ holds. We must prove that \textquotedblleft$\left.  g\mid_{Y^{\prime}%
}\right.  =f_{p\rightarrow z}$\textquotedblright\ holds as well.

Indeed, let $y\in Y^{\prime}$. We shall prove the equality $g\left(  y\right)
=f_{p\rightarrow z}\left(  y\right)  $. This equality is easily proved in the
case when $y=p$\ \ \ \ \footnote{\textit{Proof.} Assume that $y=p$. We must
prove that $g\left(  y\right)  =f_{p\rightarrow z}\left(  y\right)  $.
\par
Indeed, from $y=p$, we obtain $g\left(  y\right)  =g\left(  p\right)  =z$
(since we assumed \textquotedblleft$\left.  g\mid_{Y}\right.  =f$ and
$g\left(  p\right)  =z$\textquotedblright). On the other hand, the definition
of $f_{p\rightarrow z}$ yields%
\[
f_{p\rightarrow z}\left(  y\right)  =%
\begin{cases}
z, & \text{if }y=p;\\
f\left(  y\right)  , & \text{if }y\neq p
\end{cases}
\ \ =z\ \ \ \ \ \ \ \ \ \ \left(  \text{since }y=p\right)  .
\]
Comparing this with $g\left(  y\right)  =z$, we obtain $g\left(  y\right)
=f_{p\rightarrow z}\left(  y\right)  $. Qed.}. Thus, for the rest of this
proof, we WLOG assume that $y\neq p$. Combining $y\in Y^{\prime}$ with $y\neq
p$, we obtain $y\in Y^{\prime}\setminus\left\{  p\right\}  =Y$. Hence,
$\left(  g\mid_{Y}\right)  \left(  y\right)  =g\left(  y\right)  $. However,
$\left.  g\mid_{Y}\right.  =f$ (since we assumed \textquotedblleft$\left.
g\mid_{Y}\right.  =f$ and $g\left(  p\right)  =z$\textquotedblright). Thus,
$\underbrace{\left(  g\mid_{Y}\right)  }_{=f}\left(  y\right)  =f\left(
y\right)  $. Furthermore, the definition of $f_{p\rightarrow z}$ yields%
\begin{align*}
f_{p\rightarrow z}\left(  y\right)   &  =%
\begin{cases}
z, & \text{if }y=p;\\
f\left(  y\right)  , & \text{if }y\neq p
\end{cases}
\ \ =f\left(  y\right)  \ \ \ \ \ \ \ \ \ \ \left(  \text{since }y\neq
p\right) \\
&  =\left(  g\mid_{Y}\right)  \left(  y\right)  \ \ \ \ \ \ \ \ \ \ \left(
\text{since }\left(  g\mid_{Y}\right)  \left(  y\right)  =f\left(  y\right)
\right) \\
&  =g\left(  y\right)  .
\end{align*}
In other words, $g\left(  y\right)  =f_{p\rightarrow z}\left(  y\right)  $.
Thus, we have proved the equality $g\left(  y\right)  =f_{p\rightarrow
z}\left(  y\right)  $.

Hence, $\left(  g\mid_{Y^{\prime}}\right)  \left(  y\right)  =g\left(
y\right)  =f_{p\rightarrow z}\left(  y\right)  $.

Forget that we fixed $y$. We thus have shown that $\left(  g\mid_{Y^{\prime}%
}\right)  \left(  y\right)  =f_{p\rightarrow z}\left(  y\right)  $ for each
$y\in Y^{\prime}$. In other words, $\left.  g\mid_{Y^{\prime}}\right.
=f_{p\rightarrow z}$. We conclude that the statement \textquotedblleft$\left.
g\mid_{Y^{\prime}}\right.  =f_{p\rightarrow z}$\textquotedblright\ holds.

Thus, we have proved that the statement \textquotedblleft$\left.  g\mid
_{Y}\right.  =f$ and $g\left(  p\right)  =z$\textquotedblright\ implies the
statement \textquotedblleft$\left.  g\mid_{Y^{\prime}}\right.
=f_{p\rightarrow z}$\textquotedblright.

Let us now prove the reverse implication:

\textit{Proof that \textquotedblleft}$\left.  g\mid_{Y^{\prime}}\right.
=f_{p\rightarrow z}$\textit{\textquotedblright\ implies \textquotedblleft%
}$\left.  g\mid_{Y}\right.  =f$\textit{ and }$g\left(  p\right)
=z$\textquotedblright\textit{:} Assume that the statement \textquotedblleft%
$\left.  g\mid_{Y^{\prime}}\right.  =f_{p\rightarrow z}$\textquotedblright%
\ \ holds. We must prove that \textquotedblleft$\left.  g\mid_{Y}\right.  =f$
and $g\left(  p\right)  =z$\textquotedblright\ holds as well.

Indeed, we have $\left.  g\mid_{Y^{\prime}}\right.  =f_{p\rightarrow z}$
(since we assumed that \textquotedblleft$\left.  g\mid_{Y^{\prime}}\right.
=f_{p\rightarrow z}$\textquotedblright\ \ holds). Now, for each $y\in
Y^{\prime}$, we have%
\begin{align}
g\left(  y\right)   &  =\underbrace{\left(  g\mid_{Y^{\prime}}\right)
}_{=f_{p\rightarrow z}}\left(  y\right)  \ \ \ \ \ \ \ \ \ \ \left(
\text{since }y\in Y^{\prime}\right) \nonumber\\
&  =f_{p\rightarrow z}\left(  y\right) \nonumber\\
&  =%
\begin{cases}
z, & \text{if }y=p;\\
f\left(  y\right)  , & \text{if }y\neq p
\end{cases}
\label{pf.prop.count.inj-maps.extend.c4.pf.4}%
\end{align}
(by the definition of $f_{p\rightarrow z}$).

Hence, each $y\in Y$ satisfies%
\begin{align*}
\left(  g\mid_{Y}\right)  \left(  y\right)   &  =g\left(  y\right)
\ \ \ \ \ \ \ \ \ \ \left(  \text{since }y\in Y\right) \\
&  =%
\begin{cases}
z, & \text{if }y=p;\\
f\left(  y\right)  , & \text{if }y\neq p
\end{cases}
\ \ \ \ \ \ \ \ \ \ \left(  \text{by
(\ref{pf.prop.count.inj-maps.extend.c4.pf.4}) (since }y\in Y\subseteq
Y^{\prime}\text{)}\right) \\
&  =f\left(  y\right)  \ \ \ \ \ \ \ \ \ \ \left(  \text{since }y\neq p\text{
(because }y\in Y\text{, but }p\notin Y\text{)}\right)  .
\end{align*}
In other words, $\left.  g\mid_{Y}\right.  =f$. Furthermore, we have
$p\in\left\{  p\right\}  \subseteq Y\cup\left\{  p\right\}  =Y^{\prime}$.
Hence, (\ref{pf.prop.count.inj-maps.extend.c4.pf.4}) (applied to $y=p$) yields%
\[
g\left(  p\right)  =%
\begin{cases}
z, & \text{if }p=p;\\
f\left(  p\right)  , & \text{if }p\neq p
\end{cases}
\ \ =z\ \ \ \ \ \ \ \ \ \ \left(  \text{since }p=p\right)  .
\]

Thus, we have shown that $\left.  g\mid_{Y}\right.  =f$ and $g\left(
p\right)  =z$. In other words, \textquotedblleft$\left.  g\mid_{Y}\right.  =f$
and $g\left(  p\right)  =z$\textquotedblright\ holds. This completes the proof
that the statement \textquotedblleft$\left.  g\mid_{Y^{\prime}}\right.
=f_{p\rightarrow z}$\textquotedblright\ implies \textquotedblleft$\left.
g\mid_{Y}\right.  =f$ and $g\left(  p\right)  =z$\textquotedblright.

Altogether, we have now shown that the statement \textquotedblleft$\left.
g\mid_{Y}\right.  =f$ and $g\left(  p\right)  =z$\textquotedblright\ implies
the statement \textquotedblleft$\left.  g\mid_{Y^{\prime}}\right.
=f_{p\rightarrow z}$\textquotedblright, and vice versa. In other words, these
two statements are equivalent. This proves Claim 3.
\end{proof}

\begin{statement}
\textit{Claim 4:} Let $z\in Z\setminus f\left(  Y\right)  $. Then,%
\begin{align*}
&  \left(  \text{\# of injections }g:X\rightarrow Z\text{ such that }\left.
g\mid_{Y}\right.  =f\text{ and }g\left(  p\right)  =z\right) \\
&  =\prod_{k=1}^{\left\vert X\right\vert -\left\vert Y\right\vert -1}\left(
\left\vert Z\right\vert -\left\vert Y\right\vert -k\right)  .
\end{align*}

\end{statement}

\begin{proof}
[Proof of Claim 4.]Recall that $Y^{\prime}\subseteq X$ and $\left\vert
X\setminus Y^{\prime}\right\vert =m$. Also, the map $f_{p\rightarrow
z}:Y^{\prime}\rightarrow Z$ is an injection (by Claim 1). However, our
induction hypothesis says that Proposition \ref{prop.count.inj-maps.extend}
holds for $\left\vert X\setminus Y\right\vert =m$. Thus, we can apply
Proposition \ref{prop.count.inj-maps.extend} to $Y^{\prime}$ and
$f_{p\rightarrow z}$ instead of $Y$ and $f$. As a result, we obtain%
\begin{align*}
&  \left(  \text{\# of injections }g:X\rightarrow Z\text{ such that }\left.
g\mid_{Y^{\prime}}\right.  =f_{p\rightarrow z}\right) \\
&  =\prod_{k=0}^{\left\vert X\right\vert -\left\vert Y^{\prime}\right\vert
-1}\left(  \left\vert Z\right\vert -\left\vert Y^{\prime}\right\vert -k\right)
\\
&  =\underbrace{\prod_{k=0}^{\left\vert X\right\vert -\left(  \left\vert
Y\right\vert +1\right)  -1}}_{\substack{=\prod_{k=0}^{\left\vert X\right\vert
-\left\vert Y\right\vert -1-1}\\\text{(since }\left\vert X\right\vert -\left(
\left\vert Y\right\vert +1\right)  -1=\left\vert X\right\vert -\left\vert
Y\right\vert -1-1\text{)}}}\underbrace{\left(  \left\vert Z\right\vert
-\left(  \left\vert Y\right\vert +1\right)  -k\right)  }_{=\left\vert
Z\right\vert -\left\vert Y\right\vert -\left(  k+1\right)  }%
\ \ \ \ \ \ \ \ \ \ \left(  \text{since }\left\vert Y^{\prime}\right\vert
=\left\vert Y\right\vert +1\right) \\
&  =\prod_{k=0}^{\left\vert X\right\vert -\left\vert Y\right\vert -1-1}\left(
\left\vert Z\right\vert -\left\vert Y\right\vert -\left(  k+1\right)  \right)
\\
&  =\prod_{k=1}^{\left\vert X\right\vert -\left\vert Y\right\vert -1}\left(
\left\vert Z\right\vert -\left\vert Y\right\vert -k\right)
\end{align*}
(here, we have substituted $k$ for $k+1$ in the product).

However, for any injection $g:X\rightarrow Z$, the statement \textquotedblleft%
$\left.  g\mid_{Y}\right.  =f$ and $g\left(  p\right)  =z$\textquotedblright%
\ is equivalent to the statement \textquotedblleft$\left.  g\mid_{Y^{\prime}%
}\right.  =f_{p\rightarrow z}$\textquotedblright\ (by Claim 3). Hence, we
have
\begin{align*}
&  \left(  \text{\# of injections }g:X\rightarrow Z\text{ such that }\left.
g\mid_{Y}\right.  =f\text{ and }g\left(  p\right)  =z\right) \\
&  =\left(  \text{\# of injections }g:X\rightarrow Z\text{ such that }\left.
g\mid_{Y^{\prime}}\right.  =f_{p\rightarrow z}\right) \\
&  =\prod_{k=1}^{\left\vert X\right\vert -\left\vert Y\right\vert -1}\left(
\left\vert Z\right\vert -\left\vert Y\right\vert -k\right)  .
\end{align*}
This proves Claim 4.
\end{proof}

We are now almost done. Since $f\left(  Y\right)  $ is a subset of $Z$, we
have $\left\vert Z\setminus f\left(  Y\right)  \right\vert =\left\vert
Z\right\vert -\underbrace{\left\vert f\left(  Y\right)  \right\vert
}_{=\left\vert Y\right\vert }=\left\vert Z\right\vert -\left\vert Y\right\vert
$.

However, if $g:X\rightarrow Z$ is any injection such that $\left.  g\mid
_{Y}\right.  =f$, then $g\left(  p\right)  \in Z\setminus f\left(  Y\right)  $
(by Claim 2). Hence, in order to count the injections $g:X\rightarrow Z$ such
that $\left.  g\mid_{Y}\right.  =f$, we can split them up according to the
value of $g\left(  p\right)  $ as follows:%
\begin{align*}
&  \left(  \text{\# of injections }g:X\rightarrow Z\text{ such that }\left.
g\mid_{Y}\right.  =f\right) \\
&  =\sum_{z\in Z\setminus f\left(  Y\right)  }\underbrace{\left(  \text{\# of
injections }g:X\rightarrow Z\text{ such that }\left.  g\mid_{Y}\right.
=f\text{ and }g\left(  p\right)  =z\right)  }_{\substack{=\prod_{k=1}%
^{\left\vert X\right\vert -\left\vert Y\right\vert -1}\left(  \left\vert
Z\right\vert -\left\vert Y\right\vert -k\right)  \\\text{(by Claim 4)}}}\\
&  =\sum_{z\in Z\setminus f\left(  Y\right)  }\ \ \prod_{k=1}^{\left\vert
X\right\vert -\left\vert Y\right\vert -1}\left(  \left\vert Z\right\vert
-\left\vert Y\right\vert -k\right)  =\underbrace{\left\vert Z\setminus
f\left(  Y\right)  \right\vert }_{\substack{=\left\vert Z\right\vert
-\left\vert Y\right\vert \\=\left\vert Z\right\vert -\left\vert Y\right\vert
-0}}\cdot\prod_{k=1}^{\left\vert X\right\vert -\left\vert Y\right\vert
-1}\left(  \left\vert Z\right\vert -\left\vert Y\right\vert -k\right) \\
&  =\left(  \left\vert Z\right\vert -\left\vert Y\right\vert -0\right)
\cdot\prod_{k=1}^{\left\vert X\right\vert -\left\vert Y\right\vert -1}\left(
\left\vert Z\right\vert -\left\vert Y\right\vert -k\right)  .
\end{align*}
Comparing this with
\begin{align*}
\prod_{k=0}^{\left\vert X\right\vert -\left\vert Y\right\vert -1}\left(
\left\vert Z\right\vert -\left\vert Y\right\vert -k\right)   &  =\left(
\left\vert Z\right\vert -\left\vert Y\right\vert -0\right)  \cdot\prod
_{k=1}^{\left\vert X\right\vert -\left\vert Y\right\vert -1}\left(  \left\vert
Z\right\vert -\left\vert Y\right\vert -k\right) \\
&  \ \ \ \ \ \ \ \ \ \ \ \ \ \ \ \ \ \ \ \ \left(
\begin{array}
[c]{c}%
\text{here, we have split off the factor for }k=0\\
\text{from the product, since }\left\vert X\right\vert -\left\vert
Y\right\vert -1\geq0
\end{array}
\right)  ,
\end{align*}
we obtain%
\[
\left(  \text{\# of injections }g:X\rightarrow Z\text{ such that }\left.
g\mid_{Y}\right.  =f\right)  =\prod_{k=0}^{\left\vert X\right\vert -\left\vert
Y\right\vert -1}\left(  \left\vert Z\right\vert -\left\vert Y\right\vert
-k\right)  .
\]
In other words, the claim of Proposition \ref{prop.count.inj-maps.extend}
holds for our $X$, $Y$, $Z$ and $f$.

This completes the induction step. Thus, we have proved Proposition
\ref{prop.count.inj-maps.extend} by induction.
\end{proof}

From Proposition \ref{prop.count.inj-maps.extend}, we can easily derive the following:

\begin{corollary}
\label{cor.count.bij-maps.extend}Let $X$ be a finite set, and let $Y$ be a
subset of $X$. Let $f:Y\rightarrow X$ be any injection. Then,%
\[
\left(  \text{\# of permutations }\sigma\in\mathfrak{S}_{X}\text{ such that
}\left.  \sigma\mid_{Y}\right.  =f\right)  =\left(  \left\vert X\right\vert
-\left\vert Y\right\vert \right)  !.
\]

\end{corollary}

\begin{proof}
We have $Y\subseteq X$ (since $Y$ is a subset of $X$) and thus $\left\vert
Y\right\vert \leq\left\vert X\right\vert $. Hence, $\left\vert X\right\vert
-\left\vert Y\right\vert \in\mathbb{N}$.

We recall the following basic fact about finite sets (one of the Pigeonhole
Principles): If $U$ and $V$ are two finite sets having the same size (i.e.,
satisfying $\left\vert U\right\vert =\left\vert V\right\vert $), then any
injective map from $U$ to $V$ is bijective.

Applying this to $U=X$ and $V=X$, we conclude that any injective map from $X$
to $X$ is bijective (since $X$ and $X$ are two finite sets having the same
size). In other words, any injection from $X$ to $X$ is bijective (since an
injection is the same as an injective map). Hence, any injection from $X$ to
$X$ is a bijection from $X$ to $X$. The converse of this claim is true as well
(since any bijection is obviously an injection). Combining the preceding two
sentences, we conclude that the injections from $X$ to $X$ are precisely the
bijections from $X$ to $X$. Therefore,
\begin{align*}
&  \left\{  \text{injections from }X\text{ to }X\right\} \\
&  =\left\{  \text{bijections from }X\text{ to }X\right\} \\
&  =\left\{  \text{permutations of }X\right\}  \ \ \ \ \ \ \ \ \ \ \left(
\begin{array}
[c]{c}%
\text{since a permutation of }X\text{ is defined}\\
\text{as a bijection from }X\text{ to }X
\end{array}
\right) \\
&  =\left\{  \text{permutations }\sigma\in\mathfrak{S}_{X}\right\}
\ \ \ \ \ \ \ \ \ \ \left(  \text{by the definition of }\mathfrak{S}%
_{X}\right)  .
\end{align*}
Thus,%
\[
\left\{  \text{permutations }\sigma\in\mathfrak{S}_{X}\right\}  =\left\{
\text{injections from }X\text{ to }X\right\}  .
\]
In other words, the permutations $\sigma\in\mathfrak{S}_{X}$ are precisely the
injections from $X$ to $X$. Hence,%
\begin{align*}
&  \left(  \text{\# of permutations }\sigma\in\mathfrak{S}_{X}\text{ such that
}\left.  \sigma\mid_{Y}\right.  =f\right) \\
&  =\left(  \text{\# of injections }\sigma\text{ from }X\text{ to }X\text{
such that }\left.  \sigma\mid_{Y}\right.  =f\right) \\
&  =\left(  \text{\# of injections }\sigma:X\rightarrow X\text{ such that
}\left.  \sigma\mid_{Y}\right.  =f\right) \\
&  =\left(  \text{\# of injections }g:X\rightarrow X\text{ such that }\left.
g\mid_{Y}\right.  =f\right) \\
&  \ \ \ \ \ \ \ \ \ \ \ \ \ \ \ \ \ \ \ \ \left(  \text{here, we have renamed
the index }\sigma\text{ as }g\right) \\
&  =\prod_{k=0}^{\left\vert X\right\vert -\left\vert Y\right\vert -1}\left(
\left\vert X\right\vert -\left\vert Y\right\vert -k\right)
\ \ \ \ \ \ \ \ \ \ \left(  \text{by Proposition
\ref{prop.count.inj-maps.extend}, applied to }Z=X\right) \\
&  =\prod_{i=1}^{\left\vert X\right\vert -\left\vert Y\right\vert
}i\ \ \ \ \ \ \ \ \ \ \left(  \text{here, we have substituted }i\text{ for
}\left\vert X\right\vert -\left\vert Y\right\vert -k\text{ in the
product}\right) \\
&  =1\cdot2\cdot\cdots\cdot\left(  \left\vert X\right\vert -\left\vert
Y\right\vert \right) \\
&  =\left(  \left\vert X\right\vert -\left\vert Y\right\vert \right)
!\ \ \ \ \ \ \ \ \ \ \left(  \text{since }\left\vert X\right\vert -\left\vert
Y\right\vert \in\mathbb{N}\right)  .
\end{align*}
This proves Corollary \ref{cor.count.bij-maps.extend}.
\end{proof}

Next, we introduce some notations for paths:

\begin{definition}
\label{def.linear.Aisg-num.next}Let $V$ be a finite set. Let $p$ be a path of
$V$. Then:

\begin{enumerate}
\item[\textbf{(a)}] We let $p_{\operatorname*{last}}$ denote the last entry of
$p$. (This is well-defined, since $p$ is a path, thus a nonempty tuple, and
therefore has a last entry.)

\item[\textbf{(b)}] Let $v$ be any entry of $p$ distinct from
$p_{\operatorname*{last}}$. Then, the tuple $p$ contains $v$ exactly
once\footnotemark. Furthermore, $v$ is an entry of $p$, but is not the last
entry of $p$ (since $v$ is distinct from $p_{\operatorname*{last}}$, which is
the last entry of $p$). Hence, the tuple $p$ has at least one entry coming
after $v$. We let $\operatorname*{next}\left(  p,v\right)  $ denote the next
entry after $v$ in the tuple $p$. (This is well-defined, since the tuple $p$
contains $v$ exactly once and has at least one entry coming after $v$.)
\end{enumerate}
\end{definition}

\textit{\footnotetext{\textit{Proof.} Clearly, the tuple $p$ contains $v$
(since $v$ is an entry of $p$). Furthermore, $p$ is a path of $V$, that is, a
nonempty tuple of distinct elements of $V$ (by the definition of a path of
$V$). Hence, in particular, the entries of $p$ are distinct. In other words,
$p$ does not contain any entry more than once. Hence, $p$ contains each entry
of $p$ exactly once. Thus, in particular, $p$ contains $v$ exactly once (since
$v$ is an entry of $p$).}}

\begin{example}
Assume that $V=\left[  10\right]  $ and $p=\left(  3,4,1,6,7\right)  $. Then,
$p_{\operatorname*{last}}=7$ and $\operatorname*{next}\left(  p,3\right)  =4$
and $\operatorname*{next}\left(  p,4\right)  =1$ and $\operatorname*{next}%
\left(  p,1\right)  =6$ and $\operatorname*{next}\left(  p,6\right)  =7$.
\end{example}

\begin{definition}
Let $V$ be a finite set. Let $C$ be a path cover of $V$. Let $w\in V$.

Recall that $C$ is a path cover of $V$. In other words, $C$ is a set of paths
of $V$ such that each $v\in V$ belongs to exactly one of these paths (by the
definition of a path cover). In particular, each $v\in V$ belongs to exactly
one of the paths in $C$. Applying this to $v=w$, we conclude that $w$ belongs
to exactly one of the paths in $C$. In other words, there is exactly one path
$p\in C$ such that $w$ belongs to $p$. In other words, there is exactly one
path $p\in C$ that contains $w$. We shall denote the latter path $p$ by
$\operatorname*{path}\left(  C,w\right)  $.
\end{definition}

\begin{example}
Assume that $V=\left[  6\right]  $ and $C=\left\{  \left(  1,6,4\right)
,\ \left(  5\right)  ,\ \left(  2,3\right)  \right\}  $. Then,%
\begin{align*}
\operatorname*{path}\left(  C,1\right)   &  =\operatorname*{path}\left(
C,6\right)  =\operatorname*{path}\left(  C,4\right)  =\left(  1,6,4\right)
;\\
\operatorname*{path}\left(  C,5\right)   &  =\left(  5\right)  ;\\
\operatorname*{path}\left(  C,2\right)   &  =\operatorname*{path}\left(
C,3\right)  =\left(  2,3\right)  .
\end{align*}

\end{example}

The following lemma will help us reduce Proposition \ref{prop.linear.Aisg-num}
to Corollary \ref{cor.count.bij-maps.extend}:

\begin{lemma}
\label{lem.linear.Aisg-num.1}Let $V$ be a finite set. Let $C$ be a path cover
of $V$.

Let $L=\left\{  p_{\operatorname*{last}}\ \mid\ p\in C\right\}  $. This is a
subset of $V$.

Define a map $f:V\setminus L\rightarrow V$ as follows: Let $w\in V\setminus
L$. Thus, $w\in V$ and $w\notin L$. Let $q$ be the path $\operatorname*{path}%
\left(  C,w\right)  $. Thus, $q$ is the unique path $p\in C$ that contains $w$
(by the definition of $\operatorname*{path}\left(  C,w\right)  $). Hence,
$q\in C$ is a path that contains $w$. We have $w\neq q_{\operatorname*{last}}$
(because otherwise, we would have%
\begin{align*}
w  &  =q_{\operatorname*{last}}\in\left\{  p_{\operatorname*{last}}%
\ \mid\ p\in C\right\}  \ \ \ \ \ \ \ \ \ \ \left(  \text{since }q\in C\right)
\\
&  =L,
\end{align*}
which would contradict $w\notin L$). Hence, $w$ is an entry of $q$ (since $q$
contains $w$) that is distinct from $q_{\operatorname*{last}}$ (since $w\neq
q_{\operatorname*{last}}$). Thus, $\operatorname*{next}\left(  q,w\right)  $
is well-defined (by Definition \ref{def.linear.Aisg-num.next} \textbf{(b)}).
We set $f\left(  w\right)  :=\operatorname*{next}\left(  q,w\right)  $.

Thus, we have defined a map $f:V\setminus L\rightarrow V$.

Now, we claim the following:

\begin{enumerate}
\item[\textbf{(a)}] This map $f$ is an injection.

\item[\textbf{(b)}] Let $\sigma\in\mathfrak{S}_{V}$. Let
$F=\operatorname*{Arcs}C$. Then, we have $F\subseteq\mathbf{A}_{\sigma}$ if
and only if $\left.  \sigma\mid_{V\setminus L}\right.  =f$.
\end{enumerate}
\end{lemma}

\begin{proof}
We begin by showing a general property of the map $f$:

\begin{statement}
\textit{Claim 1:} For any $w\in V\setminus L$, we have $f\left(  w\right)
=\operatorname*{next}\left(  \operatorname*{path}\left(  C,w\right)
,\ w\right)  $ and $\operatorname*{path}\left(  C,f\left(  w\right)  \right)
=\operatorname*{path}\left(  C,w\right)  $.
\end{statement}

\begin{proof}
[Proof of Claim 1.]Let $w\in V\setminus L$. Thus, $w\in V$ and $w\notin L$.
Let $q$ be the path $\operatorname*{path}\left(  C,w\right)  $. Thus, $q$ is
the unique path $p\in C$ that contains $w$ (by the definition of
$\operatorname*{path}\left(  C,w\right)  $). Hence, $q\in C$ is a path that
contains $w$. The definition of $f$ yields $f\left(  w\right)
=\operatorname*{next}\left(  q,w\right)  $. Since $q=\operatorname*{path}%
\left(  C,w\right)  $, we can rewrite this as $f\left(  w\right)
=\operatorname*{next}\left(  \operatorname*{path}\left(  C,w\right)
,\ w\right)  $.

However, $\operatorname*{next}\left(  q,w\right)  $ is defined as the next
entry after $w$ in the tuple $q$. Hence, in particular, $\operatorname*{next}%
\left(  q,w\right)  $ is an entry of $q$. In other words, $f\left(  w\right)
$ is an entry of $q$ (since $f\left(  w\right)  =\operatorname*{next}\left(
q,w\right)  $). In other words, the path $q$ contains $f\left(  w\right)  $.

However, $\operatorname*{path}\left(  C,f\left(  w\right)  \right)  $ is
defined as the unique path $p\in C$ that contains $f\left(  w\right)  $.
Hence, if $p\in C$ is a path that contains $f\left(  w\right)  $, then
$p=\operatorname*{path}\left(  C,f\left(  w\right)  \right)  $. Applying this
to $p=q$, we obtain $q=\operatorname*{path}\left(  C,f\left(  w\right)
\right)  $ (since $q\in C$ is a path that contains $f\left(  w\right)  $).
Comparing this with $q=\operatorname*{path}\left(  C,w\right)  $, we find
$\operatorname*{path}\left(  C,f\left(  w\right)  \right)
=\operatorname*{path}\left(  C,w\right)  $.

Thus, we have now shown that $f\left(  w\right)  =\operatorname*{next}\left(
\operatorname*{path}\left(  C,w\right)  ,\ w\right)  $ and
$\operatorname*{path}\left(  C,f\left(  w\right)  \right)
=\operatorname*{path}\left(  C,w\right)  $. This proves Claim 1.
\end{proof}

We shall now prove the two parts of Lemma \ref{lem.linear.Aisg-num.1}:
\medskip

\textbf{(a)} Let $u$ and $v$ be two elements of $V\setminus L$ satisfying
$f\left(  u\right)  =f\left(  v\right)  $. We shall prove that $u=v$.

Claim 1 (applied to $w=u$) yields $f\left(  u\right)  =\operatorname*{next}%
\left(  \operatorname*{path}\left(  C,u\right)  ,\ u\right)  $ and
$\operatorname*{path}\left(  C,f\left(  u\right)  \right)
=\operatorname*{path}\left(  C,u\right)  $.

Claim 1 (applied to $w=v$) yields $f\left(  v\right)  =\operatorname*{next}%
\left(  \operatorname*{path}\left(  C,v\right)  ,\ v\right)  $ and
$\operatorname*{path}\left(  C,f\left(  v\right)  \right)
=\operatorname*{path}\left(  C,v\right)  $.

From $\operatorname*{path}\left(  C,f\left(  u\right)  \right)
=\operatorname*{path}\left(  C,u\right)  $, we obtain%
\[
\operatorname*{path}\left(  C,u\right)  =\operatorname*{path}\left(
C,\underbrace{f\left(  u\right)  }_{=f\left(  v\right)  }\right)
=\operatorname*{path}\left(  C,f\left(  v\right)  \right)
=\operatorname*{path}\left(  C,v\right)  .
\]
Let us set $q:=\operatorname*{path}\left(  C,u\right)  $. Thus,
$q=\operatorname*{path}\left(  C,u\right)  =\operatorname*{path}\left(
C,v\right)  $.

We have $q=\operatorname*{path}\left(  C,u\right)  $. In other words, $q$ is
the unique path $p\in C$ that contains $u$ (since $\operatorname*{path}\left(
C,u\right)  $ is defined to be the unique path $p\in C$ that contains $u$).
Hence, $q\in C$ is a path that contains $u$. The same argument (applied to $v$
instead of $u$) shows that $q\in C$ is a path that contains $v$ (since
$q=\operatorname*{path}\left(  C,v\right)  $).

In particular, $q$ is a path of $V$. In other words, $q$ is a nonempty tuple
of distinct elements of $V$ (by the definition of a path of $V$). Hence, in
particular, the entries of $q$ are distinct.

Write the path $q$ as $q=\left(  q_{1},q_{2},\ldots,q_{k}\right)  $. Then,
$u=q_{i}$ for some $i\in\left[  k\right]  $ (since $q$ contains $u$).
Similarly, $v=q_{j}$ for some $j\in\left[  k\right]  $ (since $q$ contains
$v$). Consider this $i$ and this $j$.

But $\operatorname*{next}\left(  q,u\right)  $ is the next entry after $u$ in
the tuple $q$ (by the definition of $\operatorname*{next}\left(  q,u\right)
$). In other words,%
\[
\operatorname*{next}\left(  q,u\right)  =\left(  \text{the next entry after
}u\text{ in the tuple }q\right)  .
\]

Now,
\begin{align*}
f\left(  u\right)   &  =\operatorname*{next}\left(
\underbrace{\operatorname*{path}\left(  C,u\right)  }_{=q},\ u\right)
=\operatorname*{next}\left(  q,u\right) \\
&  =\left(  \text{the next entry after }u\text{ in the tuple }q\right) \\
&  =\left(  \text{the next entry after }q_{i}\text{ in the tuple }q\right)
\ \ \ \ \ \ \ \ \ \ \left(  \text{since }u=q_{i}\right) \\
&  =q_{i+1}\ \ \ \ \ \ \ \ \ \ \left(
\begin{array}
[c]{c}%
\text{since }q=\left(  q_{1},q_{2},\ldots,q_{k}\right)  \text{, so that the
entry }q_{i}\\
\text{is followed by }q_{i+1}\text{ in the tuple }q
\end{array}
\right)  .
\end{align*}
The same argument (applied to $v$ and $j$ instead of $u$ and $i$) shows that
$f\left(  v\right)  =q_{j+1}$ (since $\operatorname*{path}\left(  C,v\right)
=q$ and $v=q_{j}$). Note that the equalities $f\left(  u\right)  =q_{i+1}$ and
$f\left(  v\right)  =q_{j+1}$ show (in particular) that $q_{i+1}$ and
$q_{j+1}$ are well-defined, i.e., that the elements $i+1$ and $j+1$ belong to
$\left[  k\right]  $.

Now, from $f\left(  u\right)  =q_{i+1}$, we obtain%
\[
q_{i+1}=f\left(  u\right)  =f\left(  v\right)  =q_{j+1}.
\]

But we know that the entries of $q$ are distinct. In other words, $q_{1}%
,q_{2},\ldots,q_{k}$ are distinct (since $q_{1},q_{2},\ldots,q_{k}$ are the
entries of $q$ (because $q=\left(  q_{1},q_{2},\ldots,q_{k}\right)  $)). In
other words, if $a$ and $b$ are two elements of $\left[  k\right]  $
satisfying $q_{a}=q_{b}$, then $a=b$. Applying this to $a=i+1$ and $b=j+1$, we
conclude that $i+1=j+1$ (since $q_{i+1}=q_{j+1}$). Therefore, $i=j$. Hence,
$q_{i}=q_{j}=v$ (since $v=q_{j}$), so that $u=q_{i}=v$.

Now, forget that we fixed $u$ and $v$. We thus have shown that if $u$ and $v$
are two elements of $V\setminus L$ satisfying $f\left(  u\right)  =f\left(
v\right)  $, then $u=v$. In other words, the map $f$ is injective. In other
words, $f$ is an injection. This proves Lemma \ref{lem.linear.Aisg-num.1}
\textbf{(a)}. \medskip

\textbf{(b)} We must prove the equivalence $\left(  F\subseteq\mathbf{A}%
_{\sigma}\right)  \ \Longleftrightarrow\ \left(  \left.  \sigma\mid
_{V\setminus L}\right.  =f\right)  $. In order to do so, it clearly suffices
to prove the two implications $\left(  F\subseteq\mathbf{A}_{\sigma}\right)
\ \Longrightarrow\ \left(  \left.  \sigma\mid_{V\setminus L}\right.
=f\right)  $ and $\left(  F\subseteq\mathbf{A}_{\sigma}\right)
\ \Longleftarrow\ \left(  \left.  \sigma\mid_{V\setminus L}\right.  =f\right)
$. Let us do so:

\begin{proof}
[Proof of the implication $\left(  F\subseteq\mathbf{A}_{\sigma}\right)
\ \Longrightarrow\ \left(  \left.  \sigma\mid_{V\setminus L}\right.
=f\right)  $.]Assume that $F\subseteq\mathbf{A}_{\sigma}$ holds. We must show
that $\left.  \sigma\mid_{V\setminus L}\right.  =f$ holds.

Indeed, let $w\in V\setminus L$. Thus, $f\left(  w\right)  $ is well-defined.

Let $q$ be the path $\operatorname*{path}\left(  C,w\right)  $. Thus, $q$ is
the unique path $p\in C$ that contains $w$ (by the definition of
$\operatorname*{path}\left(  C,w\right)  $). Hence, $q\in C$ is a path that
contains $w$.

Write the path $q$ as $q=\left(  q_{1},q_{2},\ldots,q_{k}\right)  $. Then,
$w=q_{i}$ for some $i\in\left[  k\right]  $ (since $q$ contains $w$). Consider
this $i$.

The definition of $f$ yields
\begin{align*}
f\left(  w\right)   &  =\operatorname*{next}\left(  q,w\right) \\
&  =\left(  \text{the next entry after }w\text{ in the tuple }q\right) \\
&  \ \ \ \ \ \ \ \ \ \ \ \ \ \ \ \ \ \ \ \ \left(
\begin{array}
[c]{c}%
\text{since }\operatorname*{next}\left(  q,w\right)  \text{ is defined to be
the}\\
\text{next entry after }w\text{ in the tuple }q
\end{array}
\right) \\
&  =\left(  \text{the next entry after }q_{i}\text{ in the tuple }q\right)
\ \ \ \ \ \ \ \ \ \ \left(  \text{since }w=q_{i}\right) \\
&  =q_{i+1}\ \ \ \ \ \ \ \ \ \ \left(
\begin{array}
[c]{c}%
\text{since }q=\left(  q_{1},q_{2},\ldots,q_{k}\right)  \text{, so that the
entry }q_{i}\\
\text{is followed by }q_{i+1}\text{ in the tuple }q
\end{array}
\right)  .
\end{align*}
In particular, $q_{i+1}$ is well-defined, so that $i+1\in\left[  k\right]  $.
Hence, $i\in\left\{  0,1,\ldots,k-1\right\}  $. Since $i$ is positive, we thus
conclude that $i\in\left[  k-1\right]  $.

Now,
\begin{align*}
F  &  =\operatorname*{Arcs}C=\bigcup_{v\in C}\operatorname*{Arcs}%
v\ \ \ \ \ \ \ \ \ \ \left(  \text{by the definition of }\operatorname*{Arcs}%
C\right) \\
&  \supseteq\operatorname*{Arcs}q\ \ \ \ \ \ \ \ \ \ \left(
\begin{array}
[c]{c}%
\text{since }\operatorname*{Arcs}q\text{ is one of the terms in the union
}\bigcup_{v\in C}\operatorname*{Arcs}v\\
\text{(because }q\in C\text{)}%
\end{array}
\right) \\
&  =\operatorname*{Arcs}\left(  q_{1},q_{2},\ldots,q_{k}\right)
\ \ \ \ \ \ \ \ \ \ \left(  \text{since }q=\left(  q_{1},q_{2},\ldots
,q_{k}\right)  \right) \\
&  =\left\{  \left(  q_{1},q_{2}\right)  ,\ \left(  q_{2},q_{3}\right)
,\ \ldots,\ \left(  q_{k-1},q_{k}\right)  \right\}
\ \ \ \ \ \ \ \ \ \ \left(  \text{by (\ref{eq.def.Arcs-and-Carcs.a.2}),
applied to }v=q\text{ and }v_{j}=q_{j}\right)  .
\end{align*}

However, from $w=q_{i}$ and $f\left(  w\right)  =q_{i+1}$, we obtain
\begin{align*}
\left(  w,f\left(  w\right)  \right)   &  =\left(  q_{i},q_{i+1}\right) \\
&  \in\left\{  \left(  q_{1},q_{2}\right)  ,\ \left(  q_{2},q_{3}\right)
,\ \ldots,\ \left(  q_{k-1},q_{k}\right)  \right\}
\ \ \ \ \ \ \ \ \ \ \left(  \text{since }i\in\left[  k-1\right]  \right) \\
&  \subseteq F\ \ \ \ \ \ \ \ \ \ \left(  \text{since }F\supseteq\left\{
\left(  q_{1},q_{2}\right)  ,\ \left(  q_{2},q_{3}\right)  ,\ \ldots,\ \left(
q_{k-1},q_{k}\right)  \right\}  \right) \\
&  \subseteq\mathbf{A}_{\sigma}\\
&  =\left\{  \left(  v,\sigma\left(  v\right)  \right)  \ \mid\ v\in
V\right\}  \ \ \ \ \ \ \ \ \ \ \left(  \text{by the definition of }%
\mathbf{A}_{\sigma}\right)  .
\end{align*}
In other words, $\left(  w,f\left(  w\right)  \right)  =\left(  v,\sigma
\left(  v\right)  \right)  $ for some $v\in V$. Consider this $v$. From
$\left(  w,f\left(  w\right)  \right)  =\left(  v,\sigma\left(  v\right)
\right)  $, we obtain $w=v$ and $f\left(  w\right)  =\sigma\left(  v\right)  $.

Now, $w\in V\setminus L$, so that $\left(  \sigma\mid_{V\setminus L}\right)
\left(  w\right)  =\sigma\left(  \underbrace{w}_{=v}\right)  =\sigma\left(
v\right)  =f\left(  w\right)  $ (since $f\left(  w\right)  =\sigma\left(
v\right)  $).

Forget that we fixed $w$. We thus have shown that $\left(  \sigma
\mid_{V\setminus L}\right)  \left(  w\right)  =f\left(  w\right)  $ for each
$w\in V\setminus L$. In other words, $\left.  \sigma\mid_{V\setminus
L}\right.  =f$.

Altogether, we have now proved that $\left.  \sigma\mid_{V\setminus L}\right.
=f$ under the assumption that $F\subseteq\mathbf{A}_{\sigma}$. In other words,
we have proved the implication $\left(  F\subseteq\mathbf{A}_{\sigma}\right)
\ \Longrightarrow\ \left(  \left.  \sigma\mid_{V\setminus L}\right.
=f\right)  $.
\end{proof}

\begin{proof}
[Proof of the implication $\left(  F\subseteq\mathbf{A}_{\sigma}\right)
\ \Longleftarrow\ \left(  \left.  \sigma\mid_{V\setminus L}\right.  =f\right)
$.]Assume that $\left.  \sigma\mid_{V\setminus L}\right.  =f$ holds. We must
show that $F\subseteq\mathbf{A}_{\sigma}$ holds.

Indeed, let $a\in F$. Then,
\begin{align*}
a  &  \in F=\operatorname*{Arcs}C=\bigcup_{v\in C}\operatorname*{Arcs}%
v\ \ \ \ \ \ \ \ \ \ \left(  \text{by the definition of }\operatorname*{Arcs}%
C\right) \\
&  =\bigcup_{q\in C}\operatorname*{Arcs}q\ \ \ \ \ \ \ \ \ \ \left(
\text{here, we have renamed the index }v\text{ as }q\right)  .
\end{align*}

In other words, $a\in\operatorname*{Arcs}q$ for some $q\in C$. Consider this
$q$.

Recall that $C$ is a path cover of $V$. In other words, $C$ is a set of paths
of $V$ such that each $v\in V$ belongs to exactly one of these paths (by the
definition of a path cover). Hence, $C$ is a set of paths of $V$. Thus, $q$ is
a path of $V$ (since $q\in C$).

Write the path $q$ as $q=\left(  q_{1},q_{2},\ldots,q_{k}\right)  $. Thus,%
\begin{align*}
&  \operatorname*{Arcs}q\\
&  =\operatorname*{Arcs}\left(  q_{1},q_{2},\ldots,q_{k}\right) \\
&  =\left\{  \left(  q_{1},q_{2}\right)  ,\ \left(  q_{2},q_{3}\right)
,\ \ldots,\ \left(  q_{k-1},q_{k}\right)  \right\}
\ \ \ \ \ \ \ \ \ \ \left(  \text{by (\ref{eq.def.Arcs-and-Carcs.a.2}),
applied to }v=q\text{ and }v_{j}=q_{j}\right)  .
\end{align*}
Hence,%
\[
a\in\operatorname*{Arcs}q=\left\{  \left(  q_{1},q_{2}\right)  ,\ \left(
q_{2},q_{3}\right)  ,\ \ldots,\ \left(  q_{k-1},q_{k}\right)  \right\}  .
\]
In other words, $a=\left(  q_{i},q_{i+1}\right)  $ for some $i\in\left[
k-1\right]  $. Consider this $i$.

Clearly, $q_{i}$ is an entry of $q$. In other words, the path $q$ contains
$q_{i}$.

Recall that $\operatorname*{path}\left(  C,q_{i}\right)  $ is defined as the
unique path $p\in C$ that contains $q_{i}$. Hence, if $p\in C$ is a path that
contains $q_{i}$, then $p=\operatorname*{path}\left(  C,q_{i}\right)  $.
Applying this to $p=q$, we conclude that $q=\operatorname*{path}\left(
C,q_{i}\right)  $ (since $q\in C$ is a path that contains $q_{i}$).

We shall next show that $q_{i}\notin L$.

Indeed, assume the contrary. Thus, $q_{i}\in L=\left\{
p_{\operatorname*{last}}\ \mid\ p\in C\right\}  $ (by the definition of $L$).
In other words, $q_{i}=p_{\operatorname*{last}}$ for some $p\in C$. Consider
this $p$. Clearly, $p_{\operatorname*{last}}$ is the last entry of $p$ (by the
definition of $p_{\operatorname*{last}}$), and thus belongs to $p$. In other
words, $q_{i}$ belongs to $p$ (since $q_{i}=p_{\operatorname*{last}}$). But
$q_{i}$ also belongs to $q$ (since $q_{i}$ is an entry of $q$).

Recall that $C$ is a set of paths of $V$ such that each $v\in V$ belongs to
exactly one of these paths. Hence, in particular, each $v\in V$ belongs to
exactly one of the paths in $C$. Applying this to $v=q_{i}$, we conclude that
$q_{i}$ belongs to exactly one of the paths in $C$. However, both $p$ and $q$
are paths in $C$. Thus, if the paths $p$ and $q$ were distinct, then $q_{i}$
would belong to (at least) two distinct paths in $C$ (since $q_{i}$ belongs to
both $p$ and $q$), which would contradict the fact that $q_{i}$ belongs to
exactly one of the paths in $C$. Hence, the paths $p$ and $q$ cannot be
distinct. In other words, $p=q$.

Thus,
\begin{align*}
p_{\operatorname*{last}}  &  =q_{\operatorname*{last}}=\left(  \text{the last
entry of }q\right)  \ \ \ \ \ \ \ \ \ \ \left(  \text{by the definition of
}q_{\operatorname*{last}}\right) \\
&  =q_{k}\ \ \ \ \ \ \ \ \ \ \left(  \text{since }q=\left(  q_{1},q_{2}%
,\ldots,q_{k}\right)  \right)  .
\end{align*}
In other words, $q_{i}=q_{k}$ (since $q_{i}=p_{\operatorname*{last}}$).

However, $q$ is a path of $V$. In other words, $q$ is a nonempty tuple of
distinct elements of $V$ (by the definition of a path of $V$). Hence, in
particular, the entries of $q$ are distinct. In other words, $q_{1}%
,q_{2},\ldots,q_{k}$ are distinct (since $q_{1},q_{2},\ldots,q_{k}$ are the
entries of $q$ (because $q=\left(  q_{1},q_{2},\ldots,q_{k}\right)  $)). In
other words, if $b$ and $c$ are two elements of $\left[  k\right]  $
satisfying $q_{b}=q_{c}$, then $b=c$. Applying this to $b=i$ and $c=k$, we
conclude that $i=k$ (since $q_{i}=q_{k}$). Hence, $i=k\notin\left[
k-1\right]  $ (since $k>k-1$). But this contradicts $i\in\left[  k-1\right]  $.

This contradiction shows that our assumption was false. Hence, $q_{i}\notin L$
is proved.

Combining $q_{i}\in V$ with $q_{i}\notin L$, we obtain $q_{i}\in V\setminus
L$. Hence, $f\left(  q_{i}\right)  $ is well-defined.

Now, Claim 1 (applied to $w=q_{i}$) yields $f\left(  q_{i}\right)
=\operatorname*{next}\left(  \operatorname*{path}\left(  C,q_{i}\right)
,\ q_{i}\right)  $ and $\operatorname*{path}\left(  C,f\left(  q_{i}\right)
\right)  =\operatorname*{path}\left(  C,q_{i}\right)  $. Hence,%
\begin{align*}
f\left(  q_{i}\right)   &  =\operatorname*{next}\left(
\underbrace{\operatorname*{path}\left(  C,q_{i}\right)  }%
_{\substack{=q\\\text{(since }q=\operatorname*{path}\left(  C,q_{i}\right)
\text{)}}},\ q_{i}\right)  =\operatorname*{next}\left(  q,q_{i}\right) \\
&  =\left(  \text{the next entry after }q_{i}\text{ in the tuple }q\right) \\
&  \ \ \ \ \ \ \ \ \ \ \ \ \ \ \ \ \ \ \ \ \left(
\begin{array}
[c]{c}%
\text{since }\operatorname*{next}\left(  q,q_{i}\right)  \text{ is defined to
be the}\\
\text{next entry after }q_{i}\text{ in the tuple }q
\end{array}
\right) \\
&  =q_{i+1}\ \ \ \ \ \ \ \ \ \ \left(
\begin{array}
[c]{c}%
\text{since }q=\left(  q_{1},q_{2},\ldots,q_{k}\right)  \text{, so that the
entry }q_{i}\\
\text{is followed by }q_{i+1}\text{ in the tuple }q
\end{array}
\right)  .
\end{align*}

However, we assumed that $\left.  \sigma\mid_{V\setminus L}\right.  =f$ holds.
Thus, $\left(  \sigma\mid_{V\setminus L}\right)  \left(  q_{i}\right)
=f\left(  q_{i}\right)  =q_{i+1}$. Therefore,%
\[
q_{i+1}=\left(  \sigma\mid_{V\setminus L}\right)  \left(  q_{i}\right)
=\sigma\left(  q_{i}\right)  .
\]

Now,
\begin{align*}
a  &  =\left(  q_{i},\underbrace{q_{i+1}}_{=\sigma\left(  q_{i}\right)
}\right)  =\left(  q_{i},\sigma\left(  q_{i}\right)  \right) \\
&  \in\left\{  \left(  v,\sigma\left(  v\right)  \right)  \ \mid\ v\in
V\right\}  \ \ \ \ \ \ \ \ \ \ \left(  \text{since }q_{i}\in V\right) \\
&  =\mathbf{A}_{\sigma}\ \ \ \ \ \ \ \ \ \ \left(  \text{since }%
\mathbf{A}_{\sigma}\text{ is defined to be }\left\{  \left(  v,\sigma\left(
v\right)  \right)  \ \mid\ v\in V\right\}  \right)  .
\end{align*}

Forget that we fixed $a$. We thus have shown that $a\in\mathbf{A}_{\sigma}$
for each $a\in F$. In other words, $F\subseteq\mathbf{A}_{\sigma}$.

Altogether, we have now proved that $F\subseteq\mathbf{A}_{\sigma}$ under the
assumption that $\left.  \sigma\mid_{V\setminus L}\right.  =f$. In other
words, we have proved the implication $\left(  F\subseteq\mathbf{A}_{\sigma
}\right)  \ \Longleftarrow\ \left(  \left.  \sigma\mid_{V\setminus L}\right.
=f\right)  $.
\end{proof}

We have now proved the two implications $\left(  F\subseteq\mathbf{A}_{\sigma
}\right)  \ \Longrightarrow\ \left(  \left.  \sigma\mid_{V\setminus L}\right.
=f\right)  $ and $\left(  F\subseteq\mathbf{A}_{\sigma}\right)
\ \Longleftarrow\ \left(  \left.  \sigma\mid_{V\setminus L}\right.  =f\right)
$. Combining them, we obtain the equivalence
\[
\left(  F\subseteq\mathbf{A}_{\sigma}\right)  \ \Longleftrightarrow\ \left(
\left.  \sigma\mid_{V\setminus L}\right.  =f\right)  .
\]
Thus, Lemma \ref{lem.linear.Aisg-num.1} \textbf{(b)} is proved.
\end{proof}

We are now ready to prove Proposition \ref{prop.linear.Aisg-num}:

\begin{proof}
[Proof of Proposition \ref{prop.linear.Aisg-num}.]Assume that
$F=\operatorname*{Arcs}C$ for some path cover $C$ of $V$. Consider this path
cover $C$. Define the set $L$ and the map $f:V\setminus L\rightarrow V$ as in
Lemma \ref{lem.linear.Aisg-num.1}. Clearly, $L$ is a subset of $V$. Hence,
$\left\vert V\setminus L\right\vert =\left\vert V\right\vert -\left\vert
L\right\vert $. Also, $V\setminus L$ is a subset of $V$.

It is now easy to show that $\left\vert L\right\vert =\left\vert C\right\vert
$. Indeed:

\begin{proof}
[Proof of $\left\vert L\right\vert =\left\vert C\right\vert $.]Let $p$ and $q$
be two distinct paths in $C$. We shall show that $p_{\operatorname*{last}}\neq
q_{\operatorname*{last}}$.

Indeed, assume the contrary. Thus, $p_{\operatorname*{last}}%
=q_{\operatorname*{last}}$.

Note that $p_{\operatorname*{last}}$ is the last entry of $p$ (by the
definition of $p_{\operatorname*{last}}$). Hence, in particular,
$p_{\operatorname*{last}}$ belongs to $p$. Similarly, $q_{\operatorname*{last}%
}$ belongs to $q$. In other words, $p_{\operatorname*{last}}$ belongs to $q$
(since $p_{\operatorname*{last}}=q_{\operatorname*{last}}$).

Now, we know that $p_{\operatorname*{last}}$ belongs to both $p$ and $q$.
Since $p$ and $q$ are two distinct paths in $C$, we thus conclude that
$p_{\operatorname*{last}}$ belongs to (at least) two distinct paths in $C$.

However, $C$ is a path cover of $V$. In other words, $C$ is a set of paths of
$V$ such that each $v\in V$ belongs to exactly one of these paths (by the
definition of a path cover). In particular, each $v\in V$ belongs to exactly
one of the paths in $C$. Applying this to $v=p_{\operatorname*{last}}$, we
conclude that $p_{\operatorname*{last}}$ belongs to exactly one of the paths
in $C$. But this contradicts the fact that $p_{\operatorname*{last}}$ belongs
to (at least) two distinct paths in $C$.

This contradiction shows that our assumption was false. Hence,
$p_{\operatorname*{last}}\neq q_{\operatorname*{last}}$ is proved.

Forget that we fixed $p$ and $q$. We thus have shown that if $p$ and $q$ are
two distinct paths in $C$, then $p_{\operatorname*{last}}\neq
q_{\operatorname*{last}}$. In other words, the elements
$p_{\operatorname*{last}}$ for all $p\in C$ are distinct. Hence, there are
$\left\vert C\right\vert $ many such elements in total. In other words,
$\left\vert \left\{  p_{\operatorname*{last}}\ \mid\ p\in C\right\}
\right\vert =\left\vert C\right\vert $. Since $L=\left\{
p_{\operatorname*{last}}\ \mid\ p\in C\right\}  $ (by the definition of $L$),
we can rewrite this as $\left\vert L\right\vert =\left\vert C\right\vert $.
Thus, $\left\vert L\right\vert =\left\vert C\right\vert $ is proved.
\end{proof}

Now, Lemma \ref{lem.linear.Aisg-num.1} \textbf{(a)} yields that the map $f$ is
an injection. On the other hand, if $\sigma\in\mathfrak{S}_{V}$ is a
permutation, then the statement \textquotedblleft$F\subseteq\mathbf{A}%
_{\sigma}$\textquotedblright\ is equivalent to \textquotedblleft$\left.
\sigma\mid_{V\setminus L}\right.  =f$\textquotedblright\ (by Lemma
\ref{lem.linear.Aisg-num.1} \textbf{(b)}). Hence,%
\begin{align*}
&  \left(  \text{\# of permutations }\sigma\in\mathfrak{S}_{V}\text{
satisfying }F\subseteq\mathbf{A}_{\sigma}\right) \\
&  =\left(  \text{\# of permutations }\sigma\in\mathfrak{S}_{V}\text{
satisfying }\left.  \sigma\mid_{V\setminus L}\right.  =f\right) \\
&  =\left(  \text{\# of permutations }\sigma\in\mathfrak{S}_{V}\text{ such
that }\left.  \sigma\mid_{V\setminus L}\right.  =f\right) \\
&  =\left(  \left\vert V\right\vert -\underbrace{\left\vert V\setminus
L\right\vert }_{=\left\vert V\right\vert -\left\vert L\right\vert }\right)
!\ \ \ \ \ \ \ \ \ \ \left(  \text{by Corollary
\ref{cor.count.bij-maps.extend}, applied to }X=V\text{ and }Y=V\setminus
L\right) \\
&  =\left(  \underbrace{\left\vert V\right\vert -\left(  \left\vert
V\right\vert -\left\vert L\right\vert \right)  }_{=\left\vert L\right\vert
}\right)  !=\underbrace{\left\vert L\right\vert }_{=\left\vert C\right\vert
}!=\left\vert C\right\vert !.
\end{align*}

In other words, there are exactly $\left\vert C\right\vert !$ many
permutations $\sigma\in\mathfrak{S}_{V}$ satisfying $F\subseteq\mathbf{A}%
_{\sigma}$. This proves Proposition \ref{prop.linear.Aisg-num}.
\end{proof}
\end{verlong}

\subsection{Counting hamps by inclusion-exclusion}

Our next lemma will be about counting Hamiltonian paths -- which we abbreviate
as \textquotedblleft hamps\textquotedblright. Here is how they are defined:

\begin{definition}
\label{def.hamp}Let $D$ be a digraph. A \emph{hamp} of $D$ means a $D$-path
that contains each vertex of $D$. (The word \textquotedblleft
hamp\textquotedblright\ is short for \textquotedblleft Hamiltonian
path\textquotedblright.)
\end{definition}

For a digraph $D=\left(  V,A\right)  $, there is an obvious connection between
the linear subsets of $A$ and the hamps of $D$: If $v$ is a hamp of $D$, then
$\operatorname*{Arcs}v$ is a maximum-size linear subset of $A$ (and this
maximum size is $\left\vert V\right\vert -1$ if $V$ is nonempty). More
interestingly, there is a far less obvious connection between the linear
subsets of $A$ and the hamps of the complement $\overline{D}$:

\begin{lemma}
\label{lem.hamps-by-lin}Let $D=\left(  V,A\right)  $ be a digraph with
$V\neq\varnothing$. Then,%
\[
\sum_{F\subseteq A\text{ is linear}}\left(  -1\right)  ^{\left\vert
F\right\vert }\cdot\left(  \text{\# of }\sigma\in\mathfrak{S}_{V}\text{
satisfying }F\subseteq\mathbf{A}_{\sigma}\right)  =\left(  \text{\# of hamps
of }\overline{D}\right)  .
\]
(We are using Convention \ref{conv.number} here.)
\end{lemma}

\begin{vershort}

\begin{proof}
We will use the Iverson bracket notation (as in Convention \ref{conv.iverson}%
). We have%
\begin{align*}
\sum_{v\text{ is a }V\text{-listing}}\ \ \underbrace{\sum
_{\substack{F\subseteq A;\\F\subseteq\operatorname*{Arcs}v}}\left(  -1\right)
^{\left\vert F\right\vert }}_{\substack{=\sum_{F\subseteq A\cap
\operatorname*{Arcs}v}\left(  -1\right)  ^{\left\vert F\right\vert }\\=\left[
A\cap\operatorname*{Arcs}v=\varnothing\right]  \\\text{(by Lemma
\ref{lem.cancel})}}}  &  =\sum_{v\text{ is a }V\text{-listing}}\left[
A\cap\operatorname*{Arcs}v=\varnothing\right] \\
&  =\left(  \text{\# of }V\text{-listings }v\text{ satisfying }A\cap
\operatorname*{Arcs}v=\varnothing\right) \\
&  =\left(  \text{\# of hamps of }\overline{D}\right)
\end{align*}
(since the hamps of $\overline{D}$ are precisely the $V$-listings $v$ such
that $A\cap\operatorname*{Arcs}v=\varnothing$). Thus,%
\begin{align}
\left(  \text{\# of hamps of }\overline{D}\right)   &  =\sum_{v\text{ is a
}V\text{-listing}}\ \ \sum_{\substack{F\subseteq A;\\F\subseteq
\operatorname*{Arcs}v}}\left(  -1\right)  ^{\left\vert F\right\vert
}\nonumber\\
&  =\sum_{F\subseteq A}\ \ \sum_{\substack{v\text{ is a }V\text{-listing;}%
\\F\subseteq\operatorname*{Arcs}v}}\left(  -1\right)  ^{\left\vert
F\right\vert }\nonumber\\
&  =\sum_{F\subseteq A\text{ is linear}}\ \ \sum_{\substack{v\text{ is a
}V\text{-listing;}\\F\subseteq\operatorname*{Arcs}v}}\left(  -1\right)
^{\left\vert F\right\vert } \label{pf.lem.hamps-by-lin.short.4}%
\end{align}
(here, we have restricted the outer sum to only the linear subsets $F$ of $A$,
because if a subset $F$ of $A$ is not linear, then the inner sum
$\sum_{\substack{v\text{ is a }V\text{-listing;}\\F\subseteq
\operatorname*{Arcs}v}}\left(  -1\right)  ^{\left\vert F\right\vert }$ is
empty\footnote{by Proposition \ref{prop.linear.listing} \textbf{(a)}}).

Now, let $F$ be a linear subset of $A$. Thus, $F=\operatorname*{Arcs}C$ for
some path cover $C$ of $V$. Consider this $C$. Then, Proposition
\ref{prop.linear.Aisg-num} yields that%
\begin{equation}
\left(  \text{\# of }\sigma\in\mathfrak{S}_{V}\text{ satisfying }%
F\subseteq\mathbf{A}_{\sigma}\right)  =\left\vert C\right\vert !.
\label{pf.lem.hamps-by-lin.short.5}%
\end{equation}
On the other hand, Proposition \ref{prop.linear.listing} \textbf{(b)} yields
that
\[
\left(  \text{\# of }V\text{-listings }v\text{ satisfying }F\subseteq
\operatorname*{Arcs}v\right)  =\left\vert C\right\vert !.
\]
Hence,%
\begin{align}
\sum_{\substack{v\text{ is a }V\text{-listing;}\\F\subseteq
\operatorname*{Arcs}v}}\left(  -1\right)  ^{\left\vert F\right\vert }  &
=\underbrace{\left(  \text{\# of }V\text{-listings }v\text{ satisfying
}F\subseteq\operatorname*{Arcs}v\right)  }_{\substack{=\left\vert C\right\vert
!\\=\left(  \text{\# of }\sigma\in\mathfrak{S}_{V}\text{ satisfying
}F\subseteq\mathbf{A}_{\sigma}\right)  \\\text{(by
(\ref{pf.lem.hamps-by-lin.short.5}))}}}\cdot\left(  -1\right)  ^{\left\vert
F\right\vert }\nonumber\\
&  =\left(  \text{\# of }\sigma\in\mathfrak{S}_{V}\text{ satisfying
}F\subseteq\mathbf{A}_{\sigma}\right)  \cdot\left(  -1\right)  ^{\left\vert
F\right\vert }\nonumber\\
&  =\left(  -1\right)  ^{\left\vert F\right\vert }\cdot\left(  \text{\# of
}\sigma\in\mathfrak{S}_{V}\text{ satisfying }F\subseteq\mathbf{A}_{\sigma
}\right)  . \label{pf.lem.hamps-by-lin.short.6}%
\end{align}

Forget that we fixed $F$. We thus have proved
(\ref{pf.lem.hamps-by-lin.short.6}) for each linear subset $F$ of $A$. Now,
(\ref{pf.lem.hamps-by-lin.short.4}) becomes%
\begin{align*}
\left(  \text{\# of hamps of }\overline{D}\right)   &  =\sum_{F\subseteq
A\text{ is linear}}\ \ \underbrace{\sum_{\substack{v\text{ is a }%
V\text{-listing;}\\F\subseteq\operatorname*{Arcs}v}}\left(  -1\right)
^{\left\vert F\right\vert }}_{\substack{=\left(  -1\right)  ^{\left\vert
F\right\vert }\cdot\left(  \text{\# of }\sigma\in\mathfrak{S}_{V}\text{
satisfying }F\subseteq\mathbf{A}_{\sigma}\right)  \\\text{(by
(\ref{pf.lem.hamps-by-lin.short.6}))}}}\\
&  =\sum_{F\subseteq A\text{ is linear}}\left(  -1\right)  ^{\left\vert
F\right\vert }\cdot\left(  \text{\# of }\sigma\in\mathfrak{S}_{V}\text{
satisfying }F\subseteq\mathbf{A}_{\sigma}\right)  .
\end{align*}
This proves Lemma \ref{lem.hamps-by-lin}.
\end{proof}
\end{vershort}

\begin{verlong}

\begin{proof}
The hamps of $D$ are precisely the $D$-paths that contain each vertex of $D$.
In other words, the hamps of $D$ are precisely the $D$-paths that are
$V$-listings (because a $D$-path contains each vertex of $D$ if and only if it
is a $V$-listing). In other words, the hamps of $D$ are precisely the
$V$-listings that are $D$-paths. In other words, the hamps of $D$ are
precisely the nonempty $V$-listings that are $D$-paths (since every
$V$-listing is nonempty\footnote{This is because $V\neq\varnothing$.}). In
other words, the hamps of $D$ are precisely the $V$-listings $v$ that satisfy
$\operatorname*{Arcs}v\subseteq A$ (since a nonempty $V$-listing $v$ is a
$D$-path if and only if it satisfies $\operatorname*{Arcs}v\subseteq A$).

Applying the same reasoning to the digraph $\overline{D}=\left(  V,\ \left(
V\times V\right)  \setminus A\right)  $ instead of the digraph $D=\left(
V,A\right)  $, we obtain the following: The hamps of $\overline{D}$ are
precisely the $V$-listings $v$ that satisfy $\operatorname*{Arcs}%
v\subseteq\left(  V\times V\right)  \setminus A$. In other words, the hamps of
$\overline{D}$ are precisely the $V$-listings $v$ that satisfy $A\cap
\operatorname*{Arcs}v=\varnothing$ (because if $v$ is a $V$-listing, then
$\operatorname*{Arcs}v$ is a subset of $V\times V$, and thus the statement
\textquotedblleft$\operatorname*{Arcs}v\subseteq\left(  V\times V\right)
\setminus A$\textquotedblright\ is equivalent to \textquotedblleft%
$A\cap\operatorname*{Arcs}v=\varnothing$\textquotedblright). Hence,%
\begin{align}
&  \left(  \text{\# of hamps of }\overline{D}\right) \nonumber\\
&  =\left(  \text{\# of }V\text{-listings }v\text{ that satisfy }%
A\cap\operatorname*{Arcs}v=\varnothing\right)  . \label{pf.lem.hamps-by-lin.1}%
\end{align}

We will use the Iverson bracket notation (as in Convention \ref{conv.iverson}%
). We have%
\begin{align*}
&  \sum_{v\text{ is a }V\text{-listing}}\ \ \underbrace{\sum
_{\substack{F\subseteq A;\\F\subseteq\operatorname*{Arcs}v}}\left(  -1\right)
^{\left\vert F\right\vert }}_{\substack{=\sum_{F\subseteq A\cap
\operatorname*{Arcs}v}\left(  -1\right)  ^{\left\vert F\right\vert }\\=\left[
A\cap\operatorname*{Arcs}v=\varnothing\right]  \\\text{(by Lemma
\ref{lem.cancel})}}}\\
&  =\sum_{v\text{ is a }V\text{-listing}}\left[  A\cap\operatorname*{Arcs}%
v=\varnothing\right] \\
&  =\sum_{\substack{v\text{ is a }V\text{-listing;}\\A\cap\operatorname*{Arcs}%
v=\varnothing}}\underbrace{\left[  A\cap\operatorname*{Arcs}v=\varnothing
\right]  }_{\substack{=1\\\text{(since }A\cap\operatorname*{Arcs}%
v=\varnothing\text{)}}}+\sum_{\substack{v\text{ is a }V\text{-listing;}%
\\\text{we don't have }A\cap\operatorname*{Arcs}v=\varnothing}%
}\underbrace{\left[  A\cap\operatorname*{Arcs}v=\varnothing\right]
}_{\substack{=0\\\text{(since we don't}\\\text{have }A\cap\operatorname*{Arcs}%
v=\varnothing\text{)}}}\\
&  =\sum_{\substack{v\text{ is a }V\text{-listing;}\\A\cap\operatorname*{Arcs}%
v=\varnothing}}1+\underbrace{\sum_{\substack{v\text{ is a }V\text{-listing;}%
\\\text{we don't have }A\cap\operatorname*{Arcs}v=\varnothing}}0}_{=0}%
=\sum_{\substack{v\text{ is a }V\text{-listing;}\\A\cap\operatorname*{Arcs}%
v=\varnothing}}1\\
&  =\left(  \text{\# of }V\text{-listings }v\text{ that satisfy }%
A\cap\operatorname*{Arcs}v=\varnothing\right)  .
\end{align*}
Comparing this with (\ref{pf.lem.hamps-by-lin.1}), we obtain%
\begin{align}
&  \left(  \text{\# of hamps of }\overline{D}\right) \nonumber\\
&  =\sum_{v\text{ is a }V\text{-listing}}\ \ \sum_{\substack{F\subseteq
A;\\F\subseteq\operatorname*{Arcs}v}}\left(  -1\right)  ^{\left\vert
F\right\vert }\nonumber\\
&  =\sum_{F\subseteq A}\ \ \sum_{\substack{v\text{ is a }V\text{-listing;}%
\\F\subseteq\operatorname*{Arcs}v}}\left(  -1\right)  ^{\left\vert
F\right\vert }\nonumber\\
&  =\sum_{F\subseteq A\text{ is linear}}\ \ \sum_{\substack{v\text{ is a
}V\text{-listing;}\\F\subseteq\operatorname*{Arcs}v}}\left(  -1\right)
^{\left\vert F\right\vert }+\sum_{F\subseteq A\text{ is not linear}%
}\ \ \underbrace{\sum_{\substack{v\text{ is a }V\text{-listing;}%
\\F\subseteq\operatorname*{Arcs}v}}\left(  -1\right)  ^{\left\vert
F\right\vert }}_{\substack{=0\\\text{(because Proposition
\ref{prop.linear.listing} \textbf{(a)}}\\\text{shows that this sum is empty)}%
}}\nonumber\\
&  =\sum_{F\subseteq A\text{ is linear}}\ \ \sum_{\substack{v\text{ is a
}V\text{-listing;}\\F\subseteq\operatorname*{Arcs}v}}\left(  -1\right)
^{\left\vert F\right\vert }+\underbrace{\sum_{F\subseteq A\text{ is not
linear}}\ \ 0}_{=0}\nonumber\\
&  =\sum_{F\subseteq A\text{ is linear}}\ \ \sum_{\substack{v\text{ is a
}V\text{-listing;}\\F\subseteq\operatorname*{Arcs}v}}\left(  -1\right)
^{\left\vert F\right\vert }. \label{pf.lem.hamps-by-lin.4}%
\end{align}

Now, let $F$ be a linear subset of $A$. Thus, $F=\operatorname*{Arcs}C$ for
some path cover $C$ of $V$. Consider this $C$. Then, Proposition
\ref{prop.linear.Aisg-num} yields that there are exactly $\left\vert
C\right\vert !$ many permutations $\sigma\in\mathfrak{S}_{V}$ satisfying
$F\subseteq\mathbf{A}_{\sigma}$. In other words, we have%
\begin{equation}
\left(  \text{\# of }\sigma\in\mathfrak{S}_{V}\text{ satisfying }%
F\subseteq\mathbf{A}_{\sigma}\right)  =\left\vert C\right\vert !.
\label{pf.lem.hamps-by-lin.5}%
\end{equation}
On the other hand, Proposition \ref{prop.linear.listing} \textbf{(b)} yields
that there are exactly $\left\vert C\right\vert !$ many $V$-listings $v$
satisfying $F\subseteq\operatorname*{Arcs}v$. In other words, we have
\[
\left(  \text{\# of }V\text{-listings }v\text{ satisfying }F\subseteq
\operatorname*{Arcs}v\right)  =\left\vert C\right\vert !.
\]
Comparing this with (\ref{pf.lem.hamps-by-lin.5}), we find%
\begin{align}
&  \left(  \text{\# of }V\text{-listings }v\text{ satisfying }F\subseteq
\operatorname*{Arcs}v\right) \nonumber\\
&  =\left(  \text{\# of }\sigma\in\mathfrak{S}_{V}\text{ satisfying
}F\subseteq\mathbf{A}_{\sigma}\right)  . \label{pf.lem.hamps-by-lin.8}%
\end{align}

Hence,%
\begin{align}
\sum_{\substack{v\text{ is a }V\text{-listing;}\\F\subseteq
\operatorname*{Arcs}v}}\left(  -1\right)  ^{\left\vert F\right\vert }  &
=\underbrace{\left(  \text{\# of }V\text{-listings }v\text{ satisfying
}F\subseteq\operatorname*{Arcs}v\right)  }_{\substack{=\left(  \text{\# of
}\sigma\in\mathfrak{S}_{V}\text{ satisfying }F\subseteq\mathbf{A}_{\sigma
}\right)  \\\text{(by (\ref{pf.lem.hamps-by-lin.8}))}}}\cdot\left(  -1\right)
^{\left\vert F\right\vert }\nonumber\\
&  =\left(  \text{\# of }\sigma\in\mathfrak{S}_{V}\text{ satisfying
}F\subseteq\mathbf{A}_{\sigma}\right)  \cdot\left(  -1\right)  ^{\left\vert
F\right\vert }\nonumber\\
&  =\left(  -1\right)  ^{\left\vert F\right\vert }\cdot\left(  \text{\# of
}\sigma\in\mathfrak{S}_{V}\text{ satisfying }F\subseteq\mathbf{A}_{\sigma
}\right)  . \label{pf.lem.hamps-by-lin.6}%
\end{align}

Forget that we fixed $F$. We thus have proved (\ref{pf.lem.hamps-by-lin.6})
for each linear subset $F$ of $A$. Now, (\ref{pf.lem.hamps-by-lin.4}) becomes%
\begin{align*}
\left(  \text{\# of hamps of }\overline{D}\right)   &  =\sum_{F\subseteq
A\text{ is linear}}\ \ \underbrace{\sum_{\substack{v\text{ is a }%
V\text{-listing;}\\F\subseteq\operatorname*{Arcs}v}}\left(  -1\right)
^{\left\vert F\right\vert }}_{\substack{=\left(  -1\right)  ^{\left\vert
F\right\vert }\cdot\left(  \text{\# of }\sigma\in\mathfrak{S}_{V}\text{
satisfying }F\subseteq\mathbf{A}_{\sigma}\right)  \\\text{(by
(\ref{pf.lem.hamps-by-lin.6}))}}}\\
&  =\sum_{F\subseteq A\text{ is linear}}\left(  -1\right)  ^{\left\vert
F\right\vert }\cdot\left(  \text{\# of }\sigma\in\mathfrak{S}_{V}\text{
satisfying }F\subseteq\mathbf{A}_{\sigma}\right)  .
\end{align*}
This proves Lemma \ref{lem.hamps-by-lin}.
\end{proof}
\end{verlong}

\subsection{Level decomposition and maps $f$ satisfying $f\circ\sigma=f$}

\begin{vershort}
This entire subsection is devoted to building up some language that will only
ever be used in the proof of Lemma \ref{lem.hamps-by-lin-f}. All proofs are
omitted, as they are straightforward exercises in understanding the underlying
definitions. (They can be found in the detailed version \cite{verlong}, too.)
\end{vershort}

\begin{verlong}
This entire subsection is devoted to building up some language that will only
ever be used in the proof of Lemma \ref{lem.hamps-by-lin-f}. A reader familiar
with combinatorial tropes should be able to skip all proofs in this
subsection, along with many of the statements; nothing substantial is being
done here, and all hindrances being surmounted are notational. We would not be
surprised if the entire argument could be simplified or made slicker using
some algebraic notions, but we have not been able to find such notions.
\end{verlong}

We shall study what happens when a function $f:V\rightarrow\mathbb{P}$ is
introduced into a digraph $D=\left(  V,A\right)  $. The nonempty fibers of $f$
(i.e., the sets $f^{-1}\left(  j\right)  $ for all $j\in f\left(  V\right)  $)
partition the vertex set $V$, and this leads to a decomposition of $D$ into
subdigraphs. Let us introduce some notation for this, starting with the case
of an arbitrary set $V$ (we will later specialize to digraphs):

\begin{definition}
\label{def.level.levsets}Let $V$ be any set. Let $f:V\rightarrow\mathbb{P}$ be
any map.

\begin{enumerate}
\item[\textbf{(a)}] For each $v\in V$, we will refer to the number $f\left(
v\right)  $ as the \emph{level} of $v$ (with respect to $f$).

\item[\textbf{(b)}] For each $j\in\mathbb{P}$, the subset $f^{-1}\left(
j\right)  $ of $V$ shall be called the $j$\emph{-th level set} of $f$.
\end{enumerate}
\end{definition}

\begin{example}
Let $V=\left\{  1,2,3\right\}  $. Let $f:V\rightarrow\mathbb{P}$ be given by
$f\left(  1\right)  =1$, $f\left(  2\right)  =4$ and $f\left(  3\right)  =1$.
Then, the level sets of $f$ are
\begin{align*}
f^{-1}\left(  1\right)   &  =\left\{  1,3\right\}  ,\ \ \ \ \ \ \ \ \ \ f^{-1}%
\left(  4\right)  =\left\{  2\right\}  ,\ \ \ \ \ \ \ \ \ \ \text{and}\\
f^{-1}\left(  j\right)   &  =\varnothing\text{ for all }j\in\mathbb{P}%
\setminus\left\{  1,4\right\}  .
\end{align*}

\end{example}

\begin{remark}
Let $V$ be any set. Let $f:V\rightarrow\mathbb{P}$ be any map. Let
$j\in\mathbb{P}$. Then, the $j$-th level set $f^{-1}\left(  j\right)  $ is
empty if and only if $j\notin f\left(  V\right)  $. Hence, the nonempty level
sets of $f$ correspond to the elements of $f\left(  V\right)  $.
\end{remark}

\begin{definition}
\label{def.level.levsubdig}Let $D=\left(  V,A\right)  $ be a digraph. Let
$f:V\rightarrow\mathbb{P}$ be any map.

\begin{enumerate}
\item[\textbf{(a)}] For each $j\in\mathbb{P}$, we define a subset $A_{j}$ of
$A$ by%
\begin{align}
A_{j}:  &  =\left\{  \left(  u,v\right)  \in A\ \mid\ u,v\in f^{-1}\left(
j\right)  \right\} \label{eq.def.level.levsubdig.a.Aj=1}\\
&  =\left\{  \left(  u,v\right)  \in A\ \mid\ f\left(  u\right)  =f\left(
v\right)  =j\right\} \label{eq.def.level.levsubdig.a.Aj=2}\\
&  =A\cap\left(  f^{-1}\left(  j\right)  \times f^{-1}\left(  j\right)
\right)  . \label{eq.def.level.levsubdig.a.Aj=3}%
\end{align}
This set $A_{j}$ is also a subset of $f^{-1}\left(  j\right)  \times
f^{-1}\left(  j\right)  $, of course.

\item[\textbf{(b)}] We let $A_{f}$ denote the subset
\[
\left\{  \left(  u,v\right)  \in A\ \mid\ f\left(  u\right)  =f\left(
v\right)  \right\}
\]
of $A$.

\item[\textbf{(c)}] For each $j\in\mathbb{P}$, we let $D_{j}$ denote the
digraph $\left(  f^{-1}\left(  j\right)  ,\ A_{j}\right)  $. This digraph
$D_{j}$ is the restriction of the digraph $D$ to the subset $f^{-1}\left(
j\right)  $ (that is, the digraph obtained from $D$ by removing all vertices
that don't belong to $f^{-1}\left(  j\right)  $ and all arcs that contain any
of these vertices).

This digraph $D_{j}$ will be called the $j$\emph{-th level subdigraph} of $D$
with respect to $f$. (We should properly call it $D_{j,f}$ instead of $D_{j}$,
but we will usually keep $f$ fixed when we study it.)
\end{enumerate}
\end{definition}

\begin{example}
Let $D$ be as in Example \ref{exa.complement.3verts}. Let $f:V\rightarrow
\mathbb{P}$ be given by $f\left(  1\right)  =1$, $f\left(  2\right)  =4$ and
$f\left(  3\right)  =1$. Then,
\begin{align*}
A_{1}  &  =\left\{  \left(  3,3\right)  \right\}  ,\ \ \ \ \ \ \ \ \ \ A_{4}%
=\left\{  \left(  2,2\right)  \right\}  ,\\
A_{j}  &  =\varnothing\text{ for all }j\in\mathbb{P}\setminus\left\{
1,4\right\}  ,
\end{align*}
and%
\[
A_{f}=\left\{  \left(  3,3\right)  ,\ \left(  2,2\right)  \right\}  .
\]
The level subdigraphs of $D$ are the two digraphs%
\[
D_{1}=\left(  \left\{  1,3\right\}  ,\ \left\{  \left(  3,3\right)  \right\}
\right)  \ \ \ \ \ \ \ \ \ \ \text{and}\ \ \ \ \ \ \ \ \ \ D_{4}=\left(
\left\{  2\right\}  ,\ \left\{  \left(  2,2\right)  \right\}  \right)
\]
(as well as the infinitely many empty digraphs $D_{j}$ for all $j\in
\mathbb{P}\setminus\left\{  1,4\right\}  $). Note that the arc $\left(
3,3\right)  $ of $D$ is contained in $D_{1}$, and the arc $\left(  2,2\right)
$ is contained in $D_{4}$, but the arc $\left(  1,2\right)  $ is contained in
none of the level subdigraphs (since its two endpoints $1$ and $2$ have
different levels).
\end{example}

\begin{remark}
Let $D=\left(  V,A\right)  $ be a digraph. Let $f:V\rightarrow\mathbb{P}$ be
any map. Let $j\in\mathbb{P}$. Then, the $j$-th level subdigraph $D_{j}$ and
its arc set $A_{j}$ are empty if $j\notin f\left(  V\right)  $. (However,
$A_{j}$ can be empty even if $j$ does belong to $f\left(  V\right)  $.)
\end{remark}

In the following, the symbols \textquotedblleft$\sqcup$\textquotedblright\ and
\textquotedblleft$\bigsqcup$\textquotedblright\ stand for unions of disjoint
sets. Thus, for example, \textquotedblleft$A_{1}\sqcup A_{2}\sqcup A_{3}%
\sqcup\cdots$\textquotedblright\ will mean the union of some (pairwise)
disjoint sets $A_{1},A_{2},A_{3},\ldots$.

\begin{proposition}
\label{prop.path-covers-djun}Let $V$ and $J$ be two finite sets. Let $V_{j}$
be a subset of $V$ for each $j\in J$. Assume that the sets $V_{j}$ for
different $j\in J$ are disjoint. Let $C_{j}$ be a path cover of $V_{j}$ for
each $j\in J$. Then:

\begin{enumerate}
\item[\textbf{(a)}] The sets $C_{j}$ for different $j\in J$ are disjoint.

\item[\textbf{(b)}] Their union $\bigsqcup_{j\in J}C_{j}$ is a path cover of
$\bigsqcup_{j\in J}V_{j}$, and its arc set is $\operatorname*{Arcs}\left(
\bigsqcup_{j\in J}C_{j}\right)  =\bigsqcup_{j\in J}\operatorname*{Arcs}\left(
C_{j}\right)  $.
\end{enumerate}
\end{proposition}

\begin{verlong}

\begin{proof}
\textbf{(a)} It suffices to show that if $A$ and $B$ are two disjoint finite
sets, then any path cover of $A$ is disjoint from any path cover of $B$. But
this is clear, since the elements of a path cover of $A$ are paths of $A$,
whereas the elements of a path cover of $B$ are paths of $B$, and clearly a
path of $A$ cannot be a path of $B$ (since $A$ and $B$ are disjoint).

\textbf{(b)} This is obvious from the definitions of path covers and arc sets.
\end{proof}
\end{verlong}

\begin{corollary}
\label{cor.linear-djun}Let $V$ and $J$ be two finite sets. Let $V_{j}$ be a
subset of $V$ for each $j\in J$. Assume that the sets $V_{j}$ for different
$j\in J$ are disjoint. For each $j\in J$, let $F_{j}$ be a linear subset of
$V_{j}\times V_{j}$. Then, the union $\bigcup_{j\in J}F_{j}$ is a linear
subset of $V\times V$.
\end{corollary}

\begin{verlong}

\begin{proof}
Let $W$ be the union $\bigcup_{j\in J}V_{j}$. This union $W=\bigcup_{j\in
J}V_{j}$ is a subset of $V$ (since $V_{j}$ is a subset of $V$ for each $j\in
J$).

The sets $V_{j}$ for different $j\in J$ are disjoint (by assumption). Thus,
their union $\bigcup_{j\in J}V_{j}$ is a disjoint union. In other words,
$\bigcup_{j\in J}V_{j}=\bigsqcup_{j\in J}V_{j}$. In other words,
$W=\bigsqcup_{j\in J}V_{j}$ (since $W=\bigcup_{j\in J}V_{j}$).

For each $j\in J$, the set $F_{j}$ is a linear subset of $V_{j}\times V_{j}$
(by assumption), and thus is the arc set of some path cover $C_{j}$ of $V_{j}$
(by the definition of \textquotedblleft linear\textquotedblright). In other
words, for each $j\in J$, there exists a path cover $C_{j}$ of $V_{j}$ such
that
\begin{equation}
F_{j}=\operatorname*{Arcs}\left(  C_{j}\right)  . \label{pf.cor.linear-djun.1}%
\end{equation}
Consider these path covers $C_{j}$.

Proposition \ref{prop.path-covers-djun} \textbf{(a)} shows that these path
covers $C_{j}$ for different $j\in J$ are disjoint. Hence, their union
$\bigcup_{j\in J}C_{j}$ is a disjoint union. In other words, $\bigcup_{j\in
J}C_{j}=\bigsqcup_{j\in J}C_{j}$.

Proposition \ref{prop.path-covers-djun} \textbf{(b)} shows that their union
$\bigsqcup_{j\in J}C_{j}$ is a path cover of $\bigsqcup_{j\in J}V_{j}$, and
its arc set is $\operatorname*{Arcs}\left(  \bigsqcup_{j\in J}C_{j}\right)
=\bigsqcup_{j\in J}\operatorname*{Arcs}\left(  C_{j}\right)  $.

In particular, $\bigsqcup_{j\in J}C_{j}$ is a path cover of $\bigsqcup_{j\in
J}V_{j}=W$. The arc set of this path cover is%
\[
\operatorname*{Arcs}\left(  \bigsqcup_{j\in J}C_{j}\right)  =\bigsqcup_{j\in
J}\underbrace{\operatorname*{Arcs}\left(  C_{j}\right)  }_{\substack{=F_{j}%
\\\text{(by (\ref{pf.cor.linear-djun.1}))}}}=\bigsqcup_{j\in J}F_{j}%
=\bigcup_{j\in J}F_{j}.
\]
Hence, $\bigcup_{j\in J}F_{j}$ is the arc set of some path cover of $W$
(namely, of the path cover $\bigsqcup_{j\in J}C_{j}$). In other words, $F$ is
linear as a subset of $W\times W$ (by the definition of \textquotedblleft
linear\textquotedblright). Therefore, $F$ is linear as a subset of $V\times W$
(by Proposition \ref{prop.linear.VW}). This proves Corollary
\ref{cor.linear-djun}.
\end{proof}
\end{verlong}

\begin{proposition}
\label{prop.level.Af=sum}Let $D=\left(  V,A\right)  $ be a digraph. Let
$f:V\rightarrow\mathbb{P}$ be any map. Then, the sets $A_{1},A_{2}%
,A_{3},\ldots$ are disjoint, and their union is
\[
A_{1}\sqcup A_{2}\sqcup A_{3}\sqcup\cdots=\bigsqcup_{j\in f\left(  V\right)
}A_{j}=A_{f}.
\]

\end{proposition}

\begin{verlong}

\begin{proof}
The sets%
\begin{equation}
A_{j}=\left\{  \left(  u,v\right)  \in A\ \mid\ f\left(  u\right)  =f\left(
v\right)  =j\right\}  \label{pf.prop.level.Af=sum.Aj}%
\end{equation}
for different $j\in\mathbb{P}$ are clearly disjoint, because a pair $\left(
u,v\right)  \in A$ cannot satisfy $f\left(  u\right)  =f\left(  v\right)  =j$
for two different values of $j$ at the same time. In other words, the sets
$A_{1},A_{2},A_{3},\ldots$ are disjoint. Hence, their union is%
\begin{align*}
A_{1}\sqcup A_{2}\sqcup A_{3}\sqcup\cdots &  =\bigsqcup_{j\in\mathbb{P}}%
A_{j}\\
&  =\bigsqcup_{j\in\mathbb{P}}\left\{  \left(  u,v\right)  \in A\ \mid
\ f\left(  u\right)  =f\left(  v\right)  =j\right\}
\ \ \ \ \ \ \ \ \ \ \left(  \text{by (\ref{pf.prop.level.Af=sum.Aj})}\right)
\\
&  =\left\{  \left(  u,v\right)  \in A\ \mid\ f\left(  u\right)  =f\left(
v\right)  =j\text{ for some }j\in\mathbb{P}\right\} \\
&  =\left\{  \left(  u,v\right)  \in A\ \mid\ f\left(  u\right)  =f\left(
v\right)  \right\} \\
&  =A_{f}\ \ \ \ \ \ \ \ \ \ \left(  \text{by the definition of }A_{f}\right)
.
\end{align*}

It remains to observe that $A_{1}\sqcup A_{2}\sqcup A_{3}\sqcup\cdots
=\bigsqcup_{j\in f\left(  V\right)  }A_{j}$ (since $A_{j}$ is empty whenever
$j\notin f\left(  V\right)  $).
\end{proof}
\end{verlong}

Let us now connect the level decomposition to linear sets:

\begin{proposition}
\label{prop.level.lin-sub-Af}Let $D=\left(  V,A\right)  $ be a digraph. Let
$f:V\rightarrow\mathbb{P}$ be any map. Let $F$ be any set. Then:

\begin{enumerate}
\item[\textbf{(a)}] The set $F$ is a linear subset of $A_{f}$ if and only if
$F$ can be written as $F=\bigsqcup_{j\in f\left(  V\right)  }F_{j}$, where
each $F_{j}$ is a linear subset of $A_{j}$.

\item[\textbf{(b)}] In this case, the subsets $F_{j}$ are uniquely determined
by $F$ (namely, $F_{j}=F\cap A_{j}$ for each $j\in f\left(  V\right)  $).
\end{enumerate}
\end{proposition}

\begin{verlong}

\begin{proof}
\textbf{(a)} $\Longleftarrow:$ Assume that $F$ can be written as
$F=\bigsqcup_{j\in f\left(  V\right)  }F_{j}$, where each $F_{j}$ is a linear
subset of $A_{j}$. Consider these linear subsets $F_{j}$.

Each $F_{j}$ is a linear subset of $A_{j}$ (by assumption) and therefore is a
linear subset of $f^{-1}\left(  j\right)  \times f^{-1}\left(  j\right)  $ as
well (since $A_{j}\subseteq f^{-1}\left(  j\right)  \times f^{-1}\left(
j\right)  $). The level sets $f^{-1}\left(  j\right)  $ for different $j$'s
are disjoint. Thus, Corollary \ref{cor.linear-djun} (applied to $J=f\left(
V\right)  $ and $V_{j}=f^{-1}\left(  j\right)  $) shows that the union
$\bigcup_{j\in f\left(  V\right)  }F_{j}$ is a linear subset of $V\times V$.
In other words, $F$ is a linear subset of $V\times V$ (since $F=\bigsqcup
_{j\in f\left(  V\right)  }F_{j}=\bigcup_{j\in f\left(  V\right)  }F_{j}$).
Furthermore,%
\[
F=\bigsqcup_{j\in f\left(  V\right)  }\underbrace{F_{j}}_{\substack{\subseteq
A_{j}\\\text{(since }F_{j}\text{ is a linear subset of }A_{j}\\\text{(by
assumption))}}}\subseteq\bigsqcup_{j\in f\left(  V\right)  }A_{j}=A_{f}%
\]
(by Proposition \ref{prop.level.Af=sum}). Hence, $F$ is a subset of $A_{f}$.
This shows that $F$ is a linear subset of $A_{f}$ (since $F$ is linear). This
proves the \textquotedblleft$\Longleftarrow$\textquotedblright\ direction of
Proposition \ref{prop.level.lin-sub-Af} \textbf{(a)}.

$\Longrightarrow:$ Assume that $F$ is a linear subset of $A_{f}$. In
particular, $F$ is linear. Thus, $F$ is the arc set of a path cover $C$ of
$V$. Consider this $C$. Thus, $F=\operatorname*{Arcs}C$.

We say that a path of $V$ is \emph{level} if all entries of this path have the
same level (with respect to $f$). If $p$ is a level path of $V$, then the
\emph{level} of $p$ will mean the level of each entry of $p$.

We claim that each path in $C$ is level. Indeed, let $v=\left(  v_{1}%
,v_{2},\ldots,v_{k}\right)  $ be a path in $C$. Then, the pairs $\left(
v_{1},v_{2}\right)  ,\ \left(  v_{2},v_{3}\right)  ,\ \ldots,\ \left(
v_{k-1},v_{k}\right)  $ are arcs of $v$, thus belong to $\operatorname*{Arcs}%
v$, therefore belong to $\operatorname*{Arcs}C$ (since $v$ is a path in $C$,
and thus we have $\operatorname*{Arcs}v\subseteq\operatorname*{Arcs}C$). In
other words, each $i\in\left[  k-1\right]  $ satisfies $\left(  v_{i}%
,v_{i+1}\right)  \in\operatorname*{Arcs}C$. Hence, each $i\in\left[
k-1\right]  $ satisfies $\left(  v_{i},v_{i+1}\right)  \in\operatorname*{Arcs}%
C=F\subseteq A_{f}$ and therefore $f\left(  v_{i}\right)  =f\left(
v_{i+1}\right)  $ (by the definition of $A_{f}$). In other words, $f\left(
v_{1}\right)  =f\left(  v_{2}\right)  =\cdots=f\left(  v_{k}\right)  $. In
other words, all entries $v_{1},v_{2},\ldots,v_{k}$ of $v$ have the same
level. In other words, $v$ is level. Forget now that we fixed $v$. We thus
have shown that each path $v=\left(  v_{1},v_{2},\ldots,v_{k}\right)  $ in $C$
is level. In other words, each path in $C$ is level.

Hence, each path in $C$ has a (unique) level. Set%
\[
C_{j}:=\left\{  \text{all paths of level }j\text{ in }C\right\}
\ \ \ \ \ \ \ \ \ \ \text{for each }j\in\mathbb{P}.
\]
Then, the sets $C_{1},C_{2},C_{3},\ldots$ are disjoint (since a path cannot
have two different levels simultaneously), and their union $C_{1}\sqcup
C_{2}\sqcup C_{3}\sqcup\cdots$ is $C$ (since each path in $C$ has a level). In
particular, $C=C_{1}\sqcup C_{2}\sqcup C_{3}\sqcup\cdots$.

Let $j\in\mathbb{P}$. Then, $C_{j}$ is a path cover of $f^{-1}\left(
j\right)  $\ \ \ \ \footnote{\textit{Proof.} Let $p\in C_{j}$ be a path. Then,
$p$ is a path of level $j$ in $C$ (since $p\in C_{j}=\left\{  \text{all paths
of level }j\text{ in }C\right\}  $). Therefore, all entries of $p$ have level
$j$. In other words, all entries of $p$ belong to $f^{-1}\left(  j\right)  $.
Hence, $p$ is a path of $f^{-1}\left(  j\right)  $ (not just a path of $V$).
\par
Forget that we fixed $p$. Thus, we have shown that each path $p\in C_{j}$ is a
path of $f^{-1}\left(  j\right)  $. In other words, $C_{j}$ is a set of paths
of $f^{-1}\left(  j\right)  $.
\par
Any element $v\in f^{-1}\left(  j\right)  $ belongs to $V$, and therefore must
belong to a unique path in $C$ (since $C$ is a path cover of $V$). This latter
path must have level $j$ (since $v$ has level $j$) and therefore belong to
$C_{j}$ (by the definition of $C_{j}$). Hence, we conclude that any element
$v\in f^{-1}\left(  j\right)  $ belongs to a unique path in $C_{j}$. This
shows that $C_{j}$ is a path cover of $f^{-1}\left(  j\right)  $ (since
$C_{j}$ is a set of paths of $f^{-1}\left(  j\right)  $).}. Hence,
$\operatorname*{Arcs}\left(  C_{j}\right)  $ is a linear subset of
$f^{-1}\left(  j\right)  \times f^{-1}\left(  j\right)  $ (by the definition
of \textquotedblleft linear\textquotedblright). Furthermore, $C_{j}\subseteq
C$ (by the definition of $C_{j}$) and therefore $\operatorname*{Arcs}\left(
C_{j}\right)  \subseteq\operatorname*{Arcs}C=F\subseteq A_{f}\subseteq A$.
Combining this with $\operatorname*{Arcs}\left(  C_{j}\right)  \subseteq
f^{-1}\left(  j\right)  \times f^{-1}\left(  j\right)  $, we obtain
\begin{align*}
\operatorname*{Arcs}\left(  C_{j}\right)   &  \subseteq A\cap\left(
f^{-1}\left(  j\right)  \times f^{-1}\left(  j\right)  \right) \\
&  =A_{j}\ \ \ \ \ \ \ \ \ \ \left(  \text{since }A_{j}\text{ is defined to be
}A\cap\left(  f^{-1}\left(  j\right)  \times f^{-1}\left(  j\right)  \right)
\right)  .
\end{align*}
Thus, $\operatorname*{Arcs}\left(  C_{j}\right)  $ is a linear subset of
$A_{j}$ (since $\operatorname*{Arcs}\left(  C_{j}\right)  $ is linear).

Forget that we fixed $j$. We thus have shown that $\operatorname*{Arcs}\left(
C_{j}\right)  $ is a linear subset of $A_{j}$ for each $j\in\mathbb{P}$.

In other words, the sets $\operatorname*{Arcs}\left(  C_{1}\right)
,\ \operatorname*{Arcs}\left(  C_{2}\right)  ,\ \operatorname*{Arcs}\left(
C_{3}\right)  ,\ \ldots$ are subsets of the sets $A_{1},A_{2},A_{3},\ldots$,
respectively. Since the latter sets $A_{1},A_{2},A_{3},\ldots$ are disjoint
(by Proposition \ref{prop.level.Af=sum}), we thus conclude that their subsets
$\operatorname*{Arcs}\left(  C_{1}\right)  ,\ \operatorname*{Arcs}\left(
C_{2}\right)  ,\ \operatorname*{Arcs}\left(  C_{3}\right)  ,\ \ldots$ are
disjoint as well.

Moreover, each positive integer $j\notin f\left(  V\right)  $ satisfies
$\operatorname*{Arcs}\left(  C_{j}\right)  =\varnothing$%
\ \ \ \ \footnote{\textit{Proof.} Let $j$ be a positive integer such that
$j\notin f\left(  V\right)  $. Then, $f^{-1}\left(  j\right)  =\varnothing$.
However, we have shown above that $\operatorname*{Arcs}\left(  C_{j}\right)  $
is a linear subset of $A_{j}$. Hence,%
\begin{align*}
\operatorname*{Arcs}\left(  C_{j}\right)   &  \subseteq A_{j}=A\cap\left(
f^{-1}\left(  j\right)  \times f^{-1}\left(  j\right)  \right)
\ \ \ \ \ \ \ \ \ \ \left(  \text{by the definition of }A_{j}\right) \\
&  \subseteq\underbrace{f^{-1}\left(  j\right)  }_{=\varnothing}%
\times\underbrace{f^{-1}\left(  j\right)  }_{=\varnothing}=\varnothing
\times\varnothing=\varnothing,
\end{align*}
so that $\operatorname*{Arcs}\left(  C_{j}\right)  =\varnothing$.}.

From $C=C_{1}\sqcup C_{2}\sqcup C_{3}\sqcup\cdots=\bigsqcup_{j\in\mathbb{P}%
}C_{j}=\bigcup_{j\in\mathbb{P}}C_{j}$, we obtain%
\begin{align}
\operatorname*{Arcs}C  &  =\operatorname*{Arcs}\left(  \bigcup_{j\in
\mathbb{P}}C_{j}\right)  =\bigcup_{j\in\mathbb{P}}\operatorname*{Arcs}\left(
C_{j}\right) \nonumber\\
&  =\bigsqcup_{j\in\mathbb{P}}\operatorname*{Arcs}\left(  C_{j}\right)
\ \ \ \ \ \ \ \ \ \ \left(
\begin{array}
[c]{c}%
\text{since the}\\
\text{sets }\operatorname*{Arcs}\left(  C_{1}\right)  ,\ \operatorname*{Arcs}%
\left(  C_{2}\right)  ,\ \operatorname*{Arcs}\left(  C_{3}\right)  ,\ \ldots\\
\text{are disjoint}%
\end{array}
\right) \nonumber\\
&  =\bigsqcup_{j\in f\left(  V\right)  }\operatorname*{Arcs}\left(
C_{j}\right)  \label{pf.prop.level.lin-sub-Af.5}%
\end{align}
(since each $j\notin f\left(  V\right)  $ satisfies $\operatorname*{Arcs}%
\left(  C_{j}\right)  =\varnothing$).

Thus, $\operatorname*{Arcs}C$ can be written as $\bigsqcup_{j\in f\left(
V\right)  }F_{j}$, where each $F_{j}$ is a linear subset of $A_{j}$ (since
$\operatorname*{Arcs}\left(  C_{j}\right)  $ is a linear subset of $A_{j}$ for
each $j\in\mathbb{P}$). In other words, $F$ can be written in this way (since
$F=\operatorname*{Arcs}C$). This proves the \textquotedblleft$\Longrightarrow
$\textquotedblright\ direction of Proposition \ref{prop.level.lin-sub-Af}
\textbf{(a)}.

\textbf{(b)} Assume that $F$ is written as $F=\bigsqcup_{j\in f\left(
V\right)  }F_{j}$, where each $F_{j}$ is a linear subset of $A_{j}$. We must
show that $F_{j}=F\cap A_{j}$ for each $j\in f\left(  V\right)  $.

Indeed, we have $F=\bigsqcup_{j\in f\left(  V\right)  }F_{j}=\bigsqcup_{i\in
f\left(  V\right)  }F_{i}$.

Now, let $j\in f\left(  V\right)  $. Then, $F_{j}$ is a subset of $F$ (since
$F=\bigsqcup_{i\in f\left(  V\right)  }F_{i}$) and also a subset of $A_{j}$
(by definition of $F_{j}$). In other words, $F_{j}$ is a subset of both $F$
and $A_{j}$. Thus, $F_{j}$ is a subset of the intersection $F\cap A_{j}$ as
well. Let us now show that $F\cap A_{j}$ is a subset of $F_{j}$.

Indeed, let $\alpha\in F\cap A_{j}$. Then, $\alpha\in F\cap A_{j}\subseteq
F=\bigsqcup_{i\in f\left(  V\right)  }F_{i}$, so that $\alpha\in F_{i}$ for
some $i\in f\left(  V\right)  $. Consider this $i$. Then, $\alpha\in
F_{i}\subseteq A_{i}$ (by the definition of $F_{i}$). However, $\alpha\in
F\cap A_{j}\subseteq A_{j}$. Thus, the element $\alpha$ belongs to both sets
$A_{i}$ and $A_{j}$. Therefore, the sets $A_{i}$ and $A_{j}$ are not disjoint.
However, Proposition \ref{prop.level.Af=sum} shows that the sets $A_{1}%
,A_{2},A_{3},\ldots$ are disjoint. The only way to reconcile the previous two
sentences is when $i=j$.

Thus, we obtain $i=j$. Hence, $\alpha\in F_{j}$ (since $\alpha\in F_{i}$).

Forget that we fixed $\alpha$. We thus have shown that $\alpha\in F_{j}$ for
each $\alpha\in F\cap A_{j}$. In other words, $F\cap A_{j}\subseteq F_{j}$.
Since $F_{j}$ is (in turn) a subset of $F\cap A_{j}$, we thus conclude that
$F_{j}=F\cap A_{j}$. This completes the proof of Proposition
\ref{prop.level.lin-sub-Af} \textbf{(b)}.
\end{proof}
\end{verlong}

Next, we return to studying permutations.

When a set $V$ is a union of two disjoint subsets $A$ and $B$, and we are
given a permutation $\sigma_{A}$ of $A$ and a permutation $\sigma_{B}$ of $B$,
then we can combine these two permutations to obtain a permutation $\sigma
_{A}\oplus\sigma_{B}$ of $V$: Namely, this latter permutation sends each $a\in
A$ to $\sigma_{A}\left(  a\right)  $, and sends each $b\in B$ to $\sigma
_{B}\left(  b\right)  $. That is, this permutation $\sigma_{A}\oplus\sigma
_{B}$ is \textquotedblleft acting as $\sigma_{A}$\textquotedblright\ on the
subset $A$ and \textquotedblleft acting as $\sigma_{B}$\textquotedblright\ on
the subset $B$.

The same construction can be performed when $V$ is a union of more than two
disjoint subsets (and we are given a permutation of each of these subsets). We
will encounter this situation when a map $f:V\rightarrow\mathbb{P}$ subdivides
the set $V$ into its level sets $f^{-1}\left(  1\right)  ,\ f^{-1}\left(
2\right)  ,\ f^{-1}\left(  3\right)  ,\ \ldots$, and we are given a
permutation $\sigma_{j}\in\mathfrak{S}_{f^{-1}\left(  j\right)  }$ of each
level set $f^{-1}\left(  j\right)  $ (to be more precise, we only need
$\sigma_{j}$ to be given when $j\in f\left(  V\right)  $, since the level set
$f^{-1}\left(  j\right)  $ is empty otherwise). The permutation of $V$
obtained by combining these permutations $\sigma_{j}$ will then be denoted by
$\bigoplus_{j\in f\left(  V\right)  }\sigma_{j}$. Here is its explicit definition:

\begin{definition}
\label{def.level.decompose-perm}Let $V$ be any set. Let $f:V\rightarrow
\mathbb{P}$ be any map.

For each $j\in f\left(  V\right)  $, let $\sigma_{j}\in\mathfrak{S}%
_{f^{-1}\left(  j\right)  }$ be a permutation of the level set $f^{-1}\left(
j\right)  $.

Then, $\bigoplus_{j\in f\left(  V\right)  }\sigma_{j}$ shall denote the
permutation of $V$ that sends each $v\in V$ to $\sigma_{f\left(  v\right)
}\left(  v\right)  $. This is the permutation that acts as $\sigma_{j}$ on
each level set $f^{-1}\left(  j\right)  $.
\end{definition}

\begin{proposition}
\label{prop.level.decompose-perm}Let $V$ be any set. Let $f:V\rightarrow
\mathbb{P}$ be any map. Let $\sigma\in\mathfrak{S}_{V}$ be any permutation. Then:

\begin{enumerate}
\item[\textbf{(a)}] We have $f\circ\sigma=f$ if and only if $\sigma$ can be
written in the form $\sigma=\bigoplus_{j\in f\left(  V\right)  }\sigma_{j}$,
where $\sigma_{j}\in\mathfrak{S}_{f^{-1}\left(  j\right)  }$ for each $j\in
f\left(  V\right)  $.

\item[\textbf{(b)}] In this case, the permutations $\sigma_{j}$ for all $j\in
f\left(  V\right)  $ are uniquely determined by $\sigma$ (namely, $\sigma_{j}$
is the restriction of $\sigma$ to the subset $f^{-1}\left(  j\right)  $ for
each $j\in f\left(  V\right)  $).
\end{enumerate}
\end{proposition}

\begin{verlong}

\begin{proof}
\textbf{(a)} $\Longrightarrow:$ Assume that $f\circ\sigma=f$.

Let $j\in f\left(  V\right)  $. Let $v\in f^{-1}\left(  j\right)  $. Then,
$f\left(  v\right)  =j$. However, $f\left(  \sigma\left(  v\right)  \right)
=\underbrace{\left(  f\circ\sigma\right)  }_{=f}\left(  v\right)  =f\left(
v\right)  =j$, so that $\sigma\left(  v\right)  \in f^{-1}\left(  j\right)  $.

Forget that we fixed $v$. We thus have shown that $\sigma\left(  v\right)  \in
f^{-1}\left(  j\right)  $ for each $v\in f^{-1}\left(  j\right)  $. Hence, the
map%
\begin{align*}
f^{-1}\left(  j\right)   &  \rightarrow f^{-1}\left(  j\right)  ,\\
v  &  \mapsto\sigma\left(  v\right)
\end{align*}
is well-defined. Let us denote this map by $\sigma_{j}$. This map $\sigma_{j}$
is the restriction of $\sigma$ to the subset $f^{-1}\left(  j\right)  $ of $V$.

The map $\sigma$ is a permutation, thus has an inverse $\sigma^{-1}$. From
$f\circ\sigma=f$, we obtain $\underbrace{f}_{=f\circ\sigma}\circ\sigma
^{-1}=f\circ\underbrace{\sigma\circ\sigma^{-1}}_{=\operatorname*{id}}=f$.
Hence, just as we have constructed a map $\sigma_{j}:f^{-1}\left(  j\right)
\rightarrow f^{-1}\left(  j\right)  $ by restricting the map $\sigma$ to
$f^{-1}\left(  j\right)  $, we can likewise construct a map $\left(
\sigma^{-1}\right)  _{j}:f^{-1}\left(  j\right)  \rightarrow f^{-1}\left(
j\right)  $ by restricting the map $\sigma^{-1}$ to $f^{-1}\left(  j\right)
$. These two maps $\sigma_{j}$ and $\left(  \sigma^{-1}\right)  _{j}$ are
mutually inverse (since they are restrictions of the mutually inverse maps
$\sigma$ and $\sigma^{-1}$). Hence, the map $\sigma_{j}$ is invertible, i.e.,
is a permutation of $f^{-1}\left(  j\right)  $. In other words, $\sigma_{j}%
\in\mathfrak{S}_{f^{-1}\left(  j\right)  }$.

Forget now that we fixed $j$. Thus, for each $j\in f\left(  V\right)  $, we
have constructed a permutation $\sigma_{j}\in\mathfrak{S}_{f^{-1}\left(
j\right)  }$ by restricting the map $\sigma$ to $f^{-1}\left(  j\right)  $.
These permutations clearly satisfy $\sigma=\bigoplus_{j\in f\left(  V\right)
}\sigma_{j}$ (since $V=\bigsqcup_{j\in f\left(  V\right)  }f^{-1}\left(
j\right)  $). This proves the \textquotedblleft$\Longrightarrow$%
\textquotedblright\ direction of Proposition \ref{prop.level.decompose-perm}
\textbf{(a)}.

$\Longleftarrow:$ Assume that $\sigma$ can be written in the form
$\sigma=\bigoplus_{j\in f\left(  V\right)  }\sigma_{j}$, where $\sigma_{j}%
\in\mathfrak{S}_{f^{-1}\left(  j\right)  }$ for each $j\in f\left(  V\right)
$.

Let $v\in V$. Let $i=f\left(  v\right)  $. Thus, $i\in f\left(  V\right)  $.
From $\sigma=\bigoplus_{j\in f\left(  V\right)  }\sigma_{j}$, we obtain
\begin{align*}
\sigma\left(  v\right)   &  =\sigma_{f\left(  v\right)  }\left(  v\right)
\ \ \ \ \ \ \ \ \ \ \left(  \text{by the definition of }\bigoplus_{j\in
f\left(  V\right)  }\sigma_{j}\right) \\
&  =\sigma_{i}\left(  v\right)  \ \ \ \ \ \ \ \ \ \ \left(  \text{since
}f\left(  v\right)  =i\right) \\
&  \in f^{-1}\left(  i\right)  \ \ \ \ \ \ \ \ \ \ \left(  \text{since }%
\sigma_{i}\in\mathfrak{S}_{f^{-1}\left(  i\right)  }\text{ is a map from
}f^{-1}\left(  i\right)  \text{ to }f^{-1}\left(  i\right)  \right)  .
\end{align*}
In other words, $f\left(  \sigma\left(  v\right)  \right)  =i$. Hence,
$\left(  f\circ\sigma\right)  \left(  v\right)  =f\left(  \sigma\left(
v\right)  \right)  =i=f\left(  v\right)  $.

Forget that we fixed $v$. We thus have shown that $\left(  f\circ
\sigma\right)  \left(  v\right)  =f\left(  v\right)  $ for each $v\in V$. In
other words, $f\circ\sigma=f$. We thus have proved the \textquotedblleft%
$\Longleftarrow$\textquotedblright\ direction of Proposition
\ref{prop.level.decompose-perm} \textbf{(a)}.

\textbf{(b)} This is obvious.
\end{proof}
\end{verlong}

Now, we recall the set $\mathbf{A}_{\sigma}$ defined in Definition
\ref{def.Asigma} for any finite set $V$ and any permutation $\sigma$ of $V$.

\begin{proposition}
\label{prop.level.Asig-f2}Let $V$ be a finite set. Let $f:V\rightarrow
\mathbb{P}$ be any map. Let $\sigma\in\mathfrak{S}_{V}$ be a permutation
satisfying $f\circ\sigma=f$. Write $\sigma$ in the form $\sigma=\bigoplus
_{j\in f\left(  V\right)  }\sigma_{j}$, where $\sigma_{j}\in\mathfrak{S}%
_{f^{-1}\left(  j\right)  }$ for each $j\in f\left(  V\right)  $. (This can be
done, because of Proposition \ref{prop.level.decompose-perm} \textbf{(a)}.)
Then,
\[
\mathbf{A}_{\sigma}=\bigsqcup_{j\in f\left(  V\right)  }\mathbf{A}_{\sigma
_{j}}.
\]

\end{proposition}

\begin{verlong}

\begin{proof}
We have $\sigma=\bigoplus_{j\in f\left(  V\right)  }\sigma_{j}$. Thus, for
each $j\in f\left(  V\right)  $ and each $v\in f^{-1}\left(  j\right)  $, we
have%
\begin{equation}
\sigma\left(  v\right)  =\sigma_{f\left(  v\right)  }\left(  v\right)
=\sigma_{j}\left(  v\right)  \label{pf.prop.level.Asig-f2.1}%
\end{equation}
(since $f\left(  v\right)  =j$ (because $v\in f^{-1}\left(  j\right)  $)).

It is easy to see that the sets $\mathbf{A}_{\sigma_{j}}$ for different $j\in
f\left(  V\right)  $ are disjoint\footnote{\textit{Proof.} Let $i$ and $j$ be
two distinct elements of $f\left(  V\right)  $. We must prove that
$\mathbf{A}_{\sigma_{i}}$ and $\mathbf{A}_{\sigma_{j}}$ are disjoint.
\par
Indeed, the sets $f^{-1}\left(  i\right)  $ and $f^{-1}\left(  j\right)  $ are
disjoint (since $i$ and $j$ are distinct). In other words, $f^{-1}\left(
i\right)  \cap f^{-1}\left(  j\right)  =\varnothing$. However, $\mathbf{A}%
_{\sigma_{j}}\subseteq f^{-1}\left(  j\right)  \times f^{-1}\left(  j\right)
$ (since $\sigma_{j}$ is a permutation of $f^{-1}\left(  j\right)  $) and
$\mathbf{A}_{\sigma_{i}}\subseteq f^{-1}\left(  i\right)  \times f^{-1}\left(
i\right)  $ (likewise). Hence,%
\begin{align*}
\underbrace{\mathbf{A}_{\sigma_{i}}}_{\subseteq f^{-1}\left(  i\right)  \times
f^{-1}\left(  i\right)  }\cap\underbrace{\mathbf{A}_{\sigma_{j}}}_{\subseteq
f^{-1}\left(  j\right)  \times f^{-1}\left(  j\right)  }  &  \subseteq\left(
f^{-1}\left(  i\right)  \times f^{-1}\left(  i\right)  \right)  \cap\left(
f^{-1}\left(  j\right)  \times f^{-1}\left(  j\right)  \right) \\
&  =\underbrace{\left(  f^{-1}\left(  i\right)  \cap f^{-1}\left(  j\right)
\right)  }_{=\varnothing}\times\underbrace{\left(  f^{-1}\left(  i\right)
\cap f^{-1}\left(  j\right)  \right)  }_{=\varnothing}\\
&  =\varnothing\times\varnothing=\varnothing.
\end{align*}
Hence, $\mathbf{A}_{\sigma_{i}}\cap\mathbf{A}_{\sigma_{j}}=\varnothing$. In
other words, $\mathbf{A}_{\sigma_{i}}$ and $\mathbf{A}_{\sigma_{j}}$ are
disjoint. This completes our proof.}. Hence, the union of these sets is a
disjoint union. That is,%
\[
\bigcup\limits_{j\in f\left(  V\right)  }\mathbf{A}_{\sigma_{j}}%
=\bigsqcup_{j\in f\left(  V\right)  }\mathbf{A}_{\sigma_{j}}.
\]

The definition of $\mathbf{A}_{\sigma}$ yields%
\begin{align*}
\mathbf{A}_{\sigma}  &  =\left\{  \left(  v,\sigma\left(  v\right)  \right)
\ \mid\ v\in V\right\} \\
&  =\bigcup_{j\in f\left(  V\right)  }\underbrace{\left\{  \left(
v,\sigma\left(  v\right)  \right)  \ \mid\ v\in V\text{ and }f\left(
v\right)  =j\right\}  }_{\substack{=\left\{  \left(  v,\sigma\left(  v\right)
\right)  \ \mid\ v\in f^{-1}\left(  j\right)  \right\}  \\\text{(since the
elements }v\in V\text{ satisfying }f\left(  v\right)  =j\\\text{are precisely
the elements of }f^{-1}\left(  j\right)  \text{)}}}\\
&  \ \ \ \ \ \ \ \ \ \ \ \ \ \ \ \ \ \ \ \ \left(  \text{since each }v\in
V\text{ satisfies }f\left(  v\right)  =j\text{ for some }j\in f\left(
V\right)  \right) \\
&  =\bigcup_{j\in f\left(  V\right)  }\left\{  \left(  v,\underbrace{\sigma
\left(  v\right)  }_{\substack{=\sigma_{j}\left(  v\right)  \\\text{(by
(\ref{pf.prop.level.Asig-f2.1}))}}}\right)  \ \mid\ v\in f^{-1}\left(
j\right)  \right\} \\
&  =\bigcup_{j\in f\left(  V\right)  }\underbrace{\left\{  \left(
v,\sigma_{j}\left(  v\right)  \right)  \ \mid\ v\in f^{-1}\left(  j\right)
\right\}  }_{\substack{=\mathbf{A}_{\sigma_{j}}\\\text{(since }\mathbf{A}%
_{\sigma_{j}}\text{ is defined as }\left\{  \left(  v,\sigma_{j}\left(
v\right)  \right)  \ \mid\ v\in f^{-1}\left(  j\right)  \right\}  \text{)}}}\\
&  =\bigcup_{j\in f\left(  V\right)  }\mathbf{A}_{\sigma_{j}}=\bigsqcup_{j\in
f\left(  V\right)  }\mathbf{A}_{\sigma_{j}}.
\end{align*}
This proves Proposition \ref{prop.level.Asig-f2}.
\end{proof}
\end{verlong}

Next, we connect the above construction with the level subdigraphs of a digraph:

\begin{proposition}
\label{prop.level.Asig-f}Let $D=\left(  V,A\right)  $ be a digraph. Let
$f:V\rightarrow\mathbb{P}$ be any map. Let $\sigma\in\mathfrak{S}_{V}$ be a
permutation satisfying $f\circ\sigma=f$. Then,
\[
\mathbf{A}_{\sigma}\cap A\subseteq A_{f}.
\]

\end{proposition}

\begin{verlong}

\begin{proof}
Let $\alpha\in\mathbf{A}_{\sigma}\cap A$. Thus, $\alpha\in\mathbf{A}_{\sigma}$
and $\alpha\in A$. In particular, $\alpha\in\mathbf{A}_{\sigma}=\left\{
\left(  v,\sigma\left(  v\right)  \right)  \ \mid\ v\in V\right\}  $ (by the
definition of $\mathbf{A}_{\sigma}$). Hence, $\alpha=\left(  v,\sigma\left(
v\right)  \right)  $ for some $v\in V$. Consider this $v$. We have $f\left(
\sigma\left(  v\right)  \right)  =\underbrace{\left(  f\circ\sigma\right)
}_{=f}\left(  v\right)  =f\left(  v\right)  $. In other words, $f\left(
v\right)  =f\left(  \sigma\left(  v\right)  \right)  $. Hence, $\left(
v,\sigma\left(  v\right)  \right)  \in A_{f}$ (by the definition of $A_{f}$,
since $\left(  v,\sigma\left(  v\right)  \right)  =\alpha\in A$). In other
words, $\alpha\in A_{f}$ (since $\alpha=\left(  v,\sigma\left(  v\right)
\right)  $).

Forget that we fixed $\alpha$. We thus have proved that $\alpha\in A_{f}$ for
each $\alpha\in\mathbf{A}_{\sigma}\cap A$. In other words, $\mathbf{A}%
_{\sigma}\cap A\subseteq A_{f}$.
\end{proof}
\end{verlong}

Our last result in this section is the following trivial yet complex-looking
lemma, which will be used in the proof after it:

\begin{lemma}
\label{lem.level.djunex}Let $D=\left(  V,A\right)  $ be a digraph. Let
$f:V\rightarrow\mathbb{P}$ be any map. Let $\sigma_{j}\in\mathfrak{S}%
_{f^{-1}\left(  j\right)  }$ be a permutation for each $j\in f\left(
V\right)  $. Let $F_{j}$ be a subset of $A_{j}$ for each $j\in f\left(
V\right)  $. Then, we have the following logical equivalence:%
\[
\ \left(  \bigsqcup_{j\in f\left(  V\right)  }F_{j}\subseteq\bigsqcup_{j\in
f\left(  V\right)  }\mathbf{A}_{\sigma_{j}}\right)  \ \Longleftrightarrow
\ \left(  F_{j}\subseteq\mathbf{A}_{\sigma_{j}}\text{ for each }j\in f\left(
V\right)  \right)  .
\]

\end{lemma}

\begin{verlong}

\begin{proof}
The sets $f^{-1}\left(  j\right)  $ for different $j\in f\left(  V\right)  $
are clearly disjoint. Hence, the sets $f^{-1}\left(  j\right)  \times
f^{-1}\left(  j\right)  $ for different $j\in f\left(  V\right)  $ are
disjoint as well\footnote{\textit{Proof.} Let $r$ and $s$ be two distinct
elements of $f\left(  V\right)  $. We must prove that the sets $f^{-1}\left(
r\right)  \times f^{-1}\left(  r\right)  $ and $f^{-1}\left(  s\right)  \times
f^{-1}\left(  s\right)  $ are disjoint.
\par
Indeed, $r$ and $s$ are distinct. Hence, $f^{-1}\left(  r\right)  \cap
f^{-1}\left(  s\right)  =\varnothing$ (since the sets $f^{-1}\left(  j\right)
$ for different $j\in f\left(  V\right)  $ are disjoint). Now,%
\begin{align*}
\left(  f^{-1}\left(  r\right)  \times f^{-1}\left(  r\right)  \right)
\cap\left(  f^{-1}\left(  s\right)  \times f^{-1}\left(  s\right)  \right)
&  =\underbrace{\left(  f^{-1}\left(  r\right)  \cap f^{-1}\left(  s\right)
\right)  }_{=\varnothing}\times\underbrace{\left(  f^{-1}\left(  r\right)
\cap f^{-1}\left(  s\right)  \right)  }_{=\varnothing}\\
&  =\varnothing\times\varnothing=\varnothing.
\end{align*}
In other words, the sets $f^{-1}\left(  r\right)  \times f^{-1}\left(
r\right)  $ and $f^{-1}\left(  s\right)  \times f^{-1}\left(  s\right)  $ are
disjoint. Qed.}.

For each $j\in f\left(  V\right)  $, we have%
\begin{align}
F_{j}  &  \subseteq A_{j}\ \ \ \ \ \ \ \ \ \ \left(  \text{by the definition
of }F_{j}\right) \nonumber\\
&  =A\cap\left(  f^{-1}\left(  j\right)  \times f^{-1}\left(  j\right)
\right)  \ \ \ \ \ \ \ \ \ \ \left(  \text{by the definition of }A_{j}\right)
\nonumber\\
&  \subseteq f^{-1}\left(  j\right)  \times f^{-1}\left(  j\right)  .
\label{pf.lem.level.djunex.1}%
\end{align}
In other words, for each $j\in f\left(  V\right)  $, the set $F_{j}$ is a
subset of $f^{-1}\left(  j\right)  \times f^{-1}\left(  j\right)  $. Hence,
the sets $F_{j}$ for different $j\in f\left(  V\right)  $ are
disjoint\footnote{\textit{Proof.} Let $r$ and $s$ be two distinct elements of
$f\left(  V\right)  $. We must prove that the sets $F_{r}$ and $F_{s}$ are
disjoint.
\par
Recall that the sets $f^{-1}\left(  j\right)  \times f^{-1}\left(  j\right)  $
for different $j\in f\left(  V\right)  $ are disjoint. Hence, the sets
$f^{-1}\left(  r\right)  \times f^{-1}\left(  r\right)  $ and $f^{-1}\left(
s\right)  \times f^{-1}\left(  s\right)  $ are disjoint (since $r$ and $s$ are
distinct elements of $f\left(  V\right)  $). In other words, $\left(
f^{-1}\left(  r\right)  \times f^{-1}\left(  r\right)  \right)  \cap\left(
f^{-1}\left(  s\right)  \times f^{-1}\left(  s\right)  \right)  =\varnothing$.
\par
We have%
\[
\underbrace{F_{r}}_{\substack{\subseteq f^{-1}\left(  r\right)  \times
f^{-1}\left(  r\right)  \\\text{(by (\ref{pf.lem.level.djunex.1}%
),}\\\text{applied to }j=r\text{)}}}\cap\underbrace{F_{s}}%
_{\substack{\subseteq f^{-1}\left(  s\right)  \times f^{-1}\left(  s\right)
\\\text{(by (\ref{pf.lem.level.djunex.1}),}\\\text{applied to }j=s\text{)}%
}}\subseteq\left(  f^{-1}\left(  r\right)  \times f^{-1}\left(  r\right)
\right)  \cap\left(  f^{-1}\left(  s\right)  \times f^{-1}\left(  s\right)
\right)  =\varnothing.
\]
Hence, $F_{r}\cap F_{s}=\varnothing$. In other words, the sets $F_{r}$ and
$F_{s}$ are disjoint. Qed.}. The disjoint union $\bigsqcup_{j\in f\left(
V\right)  }F_{j}$ thus is well-defined.

For each $j\in f\left(  V\right)  $, the set $\mathbf{A}_{\sigma_{j}}$ is a
subset of $f^{-1}\left(  j\right)  \times f^{-1}\left(  j\right)  $ (since
$\sigma_{j}$ is a permutation of $f^{-1}\left(  j\right)  $). In other words,
for each $j\in f\left(  V\right)  $, we have%
\begin{equation}
\mathbf{A}_{\sigma_{j}}\subseteq f^{-1}\left(  j\right)  \times f^{-1}\left(
j\right)  . \label{pf.lem.level.djunex.2}%
\end{equation}
Hence, the sets $\mathbf{A}_{\sigma_{j}}$ for different $j\in f\left(
V\right)  $ are disjoint\footnote{\textit{Proof.} Let $r$ and $s$ be two
distinct elements of $f\left(  V\right)  $. We must prove that the sets
$\mathbf{A}_{\sigma_{r}}$ and $\mathbf{A}_{\sigma_{s}}$ are disjoint.
\par
Recall that the sets $f^{-1}\left(  j\right)  \times f^{-1}\left(  j\right)  $
for different $j\in f\left(  V\right)  $ are disjoint. Hence, the sets
$f^{-1}\left(  r\right)  \times f^{-1}\left(  r\right)  $ and $f^{-1}\left(
s\right)  \times f^{-1}\left(  s\right)  $ are disjoint (since $r$ and $s$ are
distinct elements of $f\left(  V\right)  $). In other words, $\left(
f^{-1}\left(  r\right)  \times f^{-1}\left(  r\right)  \right)  \cap\left(
f^{-1}\left(  s\right)  \times f^{-1}\left(  s\right)  \right)  =\varnothing$.
\par
We have%
\[
\underbrace{\mathbf{A}_{\sigma_{r}}}_{\substack{\subseteq f^{-1}\left(
r\right)  \times f^{-1}\left(  r\right)  \\\text{(by
(\ref{pf.lem.level.djunex.2}),}\\\text{applied to }j=r\text{)}}}\cap
\underbrace{\mathbf{A}_{\sigma_{s}}}_{\substack{\subseteq f^{-1}\left(
s\right)  \times f^{-1}\left(  s\right)  \\\text{(by
(\ref{pf.lem.level.djunex.2}),}\\\text{applied to }j=s\text{)}}}\subseteq
\left(  f^{-1}\left(  r\right)  \times f^{-1}\left(  r\right)  \right)
\cap\left(  f^{-1}\left(  s\right)  \times f^{-1}\left(  s\right)  \right)
=\varnothing.
\]
Hence, $\mathbf{A}_{\sigma_{r}}\cap\mathbf{A}_{\sigma_{s}}=\varnothing$. In
other words, the sets $\mathbf{A}_{\sigma_{r}}$ and $\mathbf{A}_{\sigma_{s}}$
are disjoint. Qed.}. The disjoint union $\bigsqcup_{j\in f\left(  V\right)
}\mathbf{A}_{\sigma_{j}}$ thus is well-defined.

Our goal is to prove the equivalence%
\[
\ \left(  \bigsqcup_{j\in f\left(  V\right)  }F_{j}\subseteq\bigsqcup_{j\in
f\left(  V\right)  }\mathbf{A}_{\sigma_{j}}\right)  \ \Longleftrightarrow
\ \left(  F_{j}\subseteq\mathbf{A}_{\sigma_{j}}\text{ for each }j\in f\left(
V\right)  \right)  .
\]
The \textquotedblleft$\Longleftarrow$\textquotedblright\ direction of this
equivalence is obvious. Thus, we only need to prove the \textquotedblleft%
$\Longrightarrow$\textquotedblright\ direction.

Let us do this. We assume that $\bigsqcup_{j\in f\left(  V\right)  }%
F_{j}\subseteq\bigsqcup_{j\in f\left(  V\right)  }\mathbf{A}_{\sigma_{j}}$. We
must prove that $F_{j}\subseteq\mathbf{A}_{\sigma_{j}}$ for each $j\in
f\left(  V\right)  $.

Let $i\in f\left(  V\right)  $. Let $\alpha\in F_{i}$. Then,%
\[
\alpha\in F_{i}\subseteq\bigsqcup_{j\in f\left(  V\right)  }F_{j}%
\subseteq\bigsqcup_{j\in f\left(  V\right)  }\mathbf{A}_{\sigma_{j}}.
\]
In other words, $\alpha\in\mathbf{A}_{\sigma_{k}}$ for some $k\in f\left(
V\right)  $. Consider this $k$. The set $F_{i}$ is a subset of $f^{-1}\left(
i\right)  \times f^{-1}\left(  i\right)  $ (because for each $j\in f\left(
V\right)  $, the set $F_{j}$ is a subset of $f^{-1}\left(  j\right)  \times
f^{-1}\left(  j\right)  $). Thus, $F_{i}\subseteq f^{-1}\left(  i\right)
\times f^{-1}\left(  i\right)  $, so that $\alpha\in F_{i}\subseteq
f^{-1}\left(  i\right)  \times f^{-1}\left(  i\right)  $.

However, $\mathbf{A}_{\sigma_{k}}$ is a subset of $f^{-1}\left(  k\right)
\times f^{-1}\left(  k\right)  $ (because for each $j\in f\left(  V\right)  $,
the set $\mathbf{A}_{\sigma_{j}}$ is a subset of $f^{-1}\left(  j\right)
\times f^{-1}\left(  j\right)  $). In other words, $\mathbf{A}_{\sigma_{k}%
}\subseteq f^{-1}\left(  k\right)  \times f^{-1}\left(  k\right)  $. Hence,
$\alpha\in\mathbf{A}_{\sigma_{k}}\subseteq f^{-1}\left(  k\right)  \times
f^{-1}\left(  k\right)  $.

Thus, the two sets $f^{-1}\left(  i\right)  \times f^{-1}\left(  i\right)  $
and $f^{-1}\left(  k\right)  \times f^{-1}\left(  k\right)  $ both contain the
element $\alpha$. However, if we had $i\neq k$, then these two sets would be
disjoint (since the sets $f^{-1}\left(  j\right)  \times f^{-1}\left(
j\right)  $ for different $j\in f\left(  V\right)  $ are disjoint), which
would contradict the previous sentence. Thus, we have $i=k$. Hence, we can
rewrite $\alpha\in\mathbf{A}_{\sigma_{k}}$ (which we know to be true) as
$\alpha\in\mathbf{A}_{\sigma_{i}}$.

Forget that we fixed $\alpha$. We thus have shown that $\alpha\in
\mathbf{A}_{\sigma_{i}}$ for each $\alpha\in F_{i}$. In other words,
$F_{i}\subseteq\mathbf{A}_{\sigma_{i}}$.

Forget that we fixed $i$. We thus have proved that $F_{i}\subseteq
\mathbf{A}_{\sigma_{i}}$ for each $i\in f\left(  V\right)  $. In other words,
$F_{j}\subseteq\mathbf{A}_{\sigma_{j}}$ for each $j\in f\left(  V\right)  $.
This proves the \textquotedblleft$\Longrightarrow$\textquotedblright%
\ direction of the above equivalence. Thus, the proof of Lemma
\ref{lem.level.djunex} is complete.
\end{proof}
\end{verlong}

\subsection{An alternating sum involving permutations $\sigma$ with
$f\circ\sigma=f$}

Now, we come to a crucial lemma, which generalizes Lemma
\ref{lem.hamps-by-lin} to the case of a digraph $D=\left(  V,A\right)  $
\textquotedblleft shattered\textquotedblright\ by a map $f:V\rightarrow
\mathbb{P}$:

\begin{lemma}
\label{lem.hamps-by-lin-f}Let $D=\left(  V,A\right)  $ be a digraph. Let
$f:V\rightarrow\mathbb{P}$ be any map. For each $j\in\mathbb{P}$, we define a
digraph $D_{j}$ as in Definition \ref{def.level.levsubdig} \textbf{(c)}. Then,%
\[
\sum_{\substack{\sigma\in\mathfrak{S}_{V};\\f\circ\sigma=f}}\ \ \sum
_{\substack{F\subseteq\mathbf{A}_{\sigma}\cap A\\\text{is linear}}}\left(
-1\right)  ^{\left\vert F\right\vert }=\prod_{j\in f\left(  V\right)  }\left(
\text{\# of hamps of }\overline{D_{j}}\right)  .
\]

\end{lemma}

\begin{vershort}

\begin{proof}
We shall use the notations from Definition \ref{def.level.levsets}, Definition
\ref{def.level.levsubdig} and Definition \ref{def.level.decompose-perm}. We
recall that every $j\in\mathbb{P}$ satisfies $D_{j}=\left(  f^{-1}\left(
j\right)  ,\ A_{j}\right)  $ (by the definition of $D_{j}$).

Let $\sigma\in\mathfrak{S}_{V}$ be a permutation satisfying $f\circ\sigma=f$.
Then, Proposition \ref{prop.level.Asig-f} yields $\mathbf{A}_{\sigma}\cap
A\subseteq A_{f}$. Hence, $\mathbf{A}_{\sigma}\cap A=\mathbf{A}_{\sigma}\cap
A_{f}$ (because $\underbrace{\mathbf{A}_{\sigma}}_{=\mathbf{A}_{\sigma}%
\cap\mathbf{A}_{\sigma}}\cap A=\mathbf{A}_{\sigma}\cap\underbrace{\mathbf{A}%
_{\sigma}\cap A}_{\subseteq A_{f}}\subseteq\mathbf{A}_{\sigma}\cap A_{f}$ and
conversely $\mathbf{A}_{\sigma}\cap\underbrace{A_{f}}_{\subseteq A}%
\subseteq\mathbf{A}_{\sigma}\cap A$). Thus,%
\begin{align}
\sum_{\substack{F\subseteq\mathbf{A}_{\sigma}\cap A\\\text{is linear}}}\left(
-1\right)  ^{\left\vert F\right\vert }  &  =\sum_{\substack{F\subseteq
\mathbf{A}_{\sigma}\cap A_{f}\\\text{is linear}}}\left(  -1\right)
^{\left\vert F\right\vert }\nonumber\\
&  =\sum_{\substack{F\subseteq A_{f}\text{ is linear};\\F\subseteq
\mathbf{A}_{\sigma}}}\left(  -1\right)  ^{\left\vert F\right\vert }
\label{pf.lem.hamps-by-lin-f.short.1}%
\end{align}
(since a subset of $\mathbf{A}_{\sigma}\cap A_{f}$ is the same thing as a
subset $F$ of $A_{f}$ that satisfies $F\subseteq\mathbf{A}_{\sigma}$).

Forget that we fixed $\sigma$. We thus have proved
(\ref{pf.lem.hamps-by-lin-f.short.1}) for every $\sigma\in\mathfrak{S}_{V}$
satisfying $f\circ\sigma=f$.

Summing up the equality (\ref{pf.lem.hamps-by-lin-f.short.1}) over all
permutations $\sigma\in\mathfrak{S}_{V}$ satisfying $f\circ\sigma=f$, we
obtain
\begin{align}
&  \sum_{\substack{\sigma\in\mathfrak{S}_{V};\\f\circ\sigma=f}}\ \ \sum
_{\substack{F\subseteq\mathbf{A}_{\sigma}\cap A\\\text{is linear}}}\left(
-1\right)  ^{\left\vert F\right\vert }\nonumber\\
&  =\sum_{\substack{\sigma\in\mathfrak{S}_{V};\\f\circ\sigma=f}}\ \ \sum
_{\substack{F\subseteq A_{f}\text{ is linear};\\F\subseteq\mathbf{A}_{\sigma}%
}}\left(  -1\right)  ^{\left\vert F\right\vert }\nonumber\\
&  =\sum_{F\subseteq A_{f}\text{ is linear}}\ \ \underbrace{\sum
_{\substack{\sigma\in\mathfrak{S}_{V};\\F\subseteq\mathbf{A}_{\sigma}%
;\\f\circ\sigma=f}}\left(  -1\right)  ^{\left\vert F\right\vert }}_{=\left(
-1\right)  ^{\left\vert F\right\vert }\cdot\left(  \text{\# of }\sigma
\in\mathfrak{S}_{V}\text{ satisfying }F\subseteq\mathbf{A}_{\sigma}\text{ and
}f\circ\sigma=f\right)  }\nonumber\\
&  =\sum_{F\subseteq A_{f}\text{ is linear}}\left(  -1\right)  ^{\left\vert
F\right\vert }\cdot\left(  \text{\# of }\sigma\in\mathfrak{S}_{V}\text{
satisfying }F\subseteq\mathbf{A}_{\sigma}\text{ and }f\circ\sigma=f\right)  .
\label{pf.lem.hamps-by-lin-f.short.2}%
\end{align}

Now, a linear subset $F$ of $A_{f}$ is the same as a set $F$ of the form
$\bigsqcup_{j\in f\left(  V\right)  }F_{j}$, where each $F_{j}$ is a linear
subset of $A_{j}$ (by Proposition \ref{prop.level.lin-sub-Af} \textbf{(a)});
furthermore, if $F$ is such a subset, then the latter subsets $F_{j}$
satisfying $F=\bigsqcup_{j\in f\left(  V\right)  }F_{j}$ are uniquely
determined by $F$ (by Proposition \ref{prop.level.lin-sub-Af} \textbf{(b)}).
Hence, we can substitute $\bigsqcup_{j\in f\left(  V\right)  }F_{j}$ for $F$
in the sum on the right hand side of (\ref{pf.lem.hamps-by-lin-f.short.2}). We
thus obtain%
\begin{align}
&  \sum_{F\subseteq A_{f}\text{ is linear}}\left(  -1\right)  ^{\left\vert
F\right\vert }\cdot\left(  \text{\# of }\sigma\in\mathfrak{S}_{V}\text{
satisfying }F\subseteq\mathbf{A}_{\sigma}\text{ and }f\circ\sigma=f\right)
\nonumber\\
&  =\sum_{\substack{\left(  F_{j}\right)  _{j\in f\left(  V\right)  }\text{ is
a family}\\\text{of linear subsets }F_{j}\subseteq A_{j}}}\left(  -1\right)
^{\left\vert \bigsqcup_{j\in f\left(  V\right)  }F_{j}\right\vert }\nonumber\\
&  \ \ \ \ \ \ \ \ \ \ \cdot\left(  \text{\# of }\sigma\in\mathfrak{S}%
_{V}\text{ satisfying }\bigsqcup_{j\in f\left(  V\right)  }F_{j}%
\subseteq\mathbf{A}_{\sigma}\text{ and }f\circ\sigma=f\right)  .
\label{pf.lem.hamps-by-lin-f.short.5}%
\end{align}

Furthermore, a permutation $\sigma\in\mathfrak{S}_{V}$ satisfies $f\circ
\sigma=f$ if and only if it can be written in the form $\sigma=\bigoplus_{j\in
f\left(  V\right)  }\sigma_{j}$, where $\sigma_{j}\in\mathfrak{S}%
_{f^{-1}\left(  j\right)  }$ for each $j\in f\left(  V\right)  $ (by
Proposition \ref{prop.level.decompose-perm} \textbf{(a)}). Moreover, if
$\sigma$ is written in this way, then the permutations $\sigma_{j}$ are
uniquely determined by $\sigma$ (by Proposition
\ref{prop.level.decompose-perm} \textbf{(b)}), and we have $\mathbf{A}%
_{\sigma}=\bigsqcup_{j\in f\left(  V\right)  }\mathbf{A}_{\sigma_{j}}$ (by
Proposition \ref{prop.level.Asig-f2}). Hence, for each family $\left(
F_{j}\right)  _{j\in f\left(  V\right)  }$ of linear subsets $F_{j}\subseteq
A_{j}$, we have%
\begin{align}
&  \left(  \text{\# of }\sigma\in\mathfrak{S}_{V}\text{ satisfying }%
\bigsqcup_{j\in f\left(  V\right)  }F_{j}\subseteq\mathbf{A}_{\sigma}\text{
and }f\circ\sigma=f\right) \nonumber\\
&  =\left(  \text{\# of families }\left(  \sigma_{j}\right)  _{j\in f\left(
V\right)  }\in\prod_{j\in f\left(  V\right)  }\mathfrak{S}_{f^{-1}\left(
j\right)  }\text{ satisfying }\bigsqcup_{j\in f\left(  V\right)  }%
F_{j}\subseteq\bigsqcup_{j\in f\left(  V\right)  }\mathbf{A}_{\sigma_{j}%
}\right) \nonumber\\
&  =\left(  \text{\# of families }\left(  \sigma_{j}\right)  _{j\in f\left(
V\right)  }\in\prod_{j\in f\left(  V\right)  }\mathfrak{S}_{f^{-1}\left(
j\right)  }\text{ satisfying }F_{j}\subseteq\mathbf{A}_{\sigma_{j}}\text{ for
each }j\in f\left(  V\right)  \right) \nonumber\\
&  \ \ \ \ \ \ \ \ \ \ \ \ \ \ \ \ \ \ \ \ \left(
\begin{array}
[c]{c}%
\text{since the condition \textquotedblleft}\bigsqcup_{j\in f\left(  V\right)
}F_{j}\subseteq\bigsqcup_{j\in f\left(  V\right)  }\mathbf{A}_{\sigma_{j}%
}\text{\textquotedblright}\\
\text{is equivalent to \textquotedblleft}F_{j}\subseteq\mathbf{A}_{\sigma_{j}%
}\text{ for each }j\in f\left(  V\right)  \text{\textquotedblright}\\
\text{(by Lemma \ref{lem.level.djunex})}%
\end{array}
\right) \nonumber\\
&  =\prod_{j\in f\left(  V\right)  }\left(  \text{\# of }\sigma_{j}%
\in\mathfrak{S}_{f^{-1}\left(  j\right)  }\text{ satisfying }F_{j}%
\subseteq\mathbf{A}_{\sigma_{j}}\right) \nonumber\\
&  =\prod_{j\in f\left(  V\right)  }\left(  \text{\# of }\sigma\in
\mathfrak{S}_{f^{-1}\left(  j\right)  }\text{ satisfying }F_{j}\subseteq
\mathbf{A}_{\sigma}\right) \label{pf.lem.hamps-by-lin-f.short.6}\\
&  \ \ \ \ \ \ \ \ \ \ \ \ \ \ \ \ \ \ \ \ \left(  \text{here, we have renamed
the index }\sigma_{j}\text{ as }\sigma\right)  .\nonumber
\end{align}
Thus, (\ref{pf.lem.hamps-by-lin-f.short.5}) becomes%
\begin{align*}
&  \sum_{F\subseteq A_{f}\text{ is linear}}\left(  -1\right)  ^{\left\vert
F\right\vert }\cdot\left(  \text{\# of }\sigma\in\mathfrak{S}_{V}\text{
satisfying }F\subseteq\mathbf{A}_{\sigma}\text{ and }f\circ\sigma=f\right) \\
&  =\sum_{\substack{\left(  F_{j}\right)  _{j\in f\left(  V\right)  }\text{ is
a family}\\\text{of linear subsets }F_{j}\subseteq A_{j}}%
}\ \ \underbrace{\left(  -1\right)  ^{\left\vert \bigsqcup_{j\in f\left(
V\right)  }F_{j}\right\vert }}_{\substack{=\left(  -1\right)  ^{\sum_{j\in
f\left(  V\right)  }\left\vert F_{j}\right\vert }\\=\prod_{j\in f\left(
V\right)  }\left(  -1\right)  ^{\left\vert F_{j}\right\vert }}}\\
&  \ \ \ \ \ \ \ \ \ \ \cdot\underbrace{\left(  \text{\# of }\sigma
\in\mathfrak{S}_{V}\text{ satisfying }\bigsqcup_{j\in f\left(  V\right)
}F_{j}\subseteq\mathbf{A}_{\sigma}\text{ and }f\circ\sigma=f\right)
}_{\substack{=\prod_{j\in f\left(  V\right)  }\left(  \text{\# of }\sigma
\in\mathfrak{S}_{f^{-1}\left(  j\right)  }\text{ satisfying }F_{j}%
\subseteq\mathbf{A}_{\sigma}\right)  \\\text{(by
(\ref{pf.lem.hamps-by-lin-f.short.6}))}}}\\
&  =\sum_{\substack{\left(  F_{j}\right)  _{j\in f\left(  V\right)  }\text{ is
a family}\\\text{of linear subsets }F_{j}\subseteq A_{j}}}\underbrace{\left(
\prod_{j\in f\left(  V\right)  }\left(  -1\right)  ^{\left\vert F_{j}%
\right\vert }\right)  \cdot\prod_{j\in f\left(  V\right)  }\left(  \text{\# of
}\sigma\in\mathfrak{S}_{f^{-1}\left(  j\right)  }\text{ satisfying }%
F_{j}\subseteq\mathbf{A}_{\sigma}\right)  }_{=\prod_{j\in f\left(  V\right)
}\left(  \left(  -1\right)  ^{\left\vert F_{j}\right\vert }\cdot\left(
\text{\# of }\sigma\in\mathfrak{S}_{f^{-1}\left(  j\right)  }\text{ satisfying
}F_{j}\subseteq\mathbf{A}_{\sigma}\right)  \right)  }\\
&  =\sum_{\substack{\left(  F_{j}\right)  _{j\in f\left(  V\right)  }\text{ is
a family}\\\text{of linear subsets }F_{j}\subseteq A_{j}}}\prod_{j\in f\left(
V\right)  }\left(  \left(  -1\right)  ^{\left\vert F_{j}\right\vert }%
\cdot\left(  \text{\# of }\sigma\in\mathfrak{S}_{f^{-1}\left(  j\right)
}\text{ satisfying }F_{j}\subseteq\mathbf{A}_{\sigma}\right)  \right) \\
&  =\prod_{j\in f\left(  V\right)  }\ \ \sum_{\substack{F_{j}\subseteq
A_{j}\text{ is linear}}}\left(  -1\right)  ^{\left\vert F_{j}\right\vert
}\cdot\left(  \text{\# of }\sigma\in\mathfrak{S}_{f^{-1}\left(  j\right)
}\text{ satisfying }F_{j}\subseteq\mathbf{A}_{\sigma}\right) \\
&  \ \ \ \ \ \ \ \ \ \ \ \ \ \ \ \ \ \ \ \ \left(  \text{by the product
rule}\right) \\
&  =\prod_{j\in f\left(  V\right)  }\ \ \underbrace{\sum_{\substack{F\subseteq
A_{j}\text{ is linear}}}\left(  -1\right)  ^{\left\vert F\right\vert }%
\cdot\left(  \text{\# of }\sigma\in\mathfrak{S}_{f^{-1}\left(  j\right)
}\text{ satisfying }F\subseteq\mathbf{A}_{\sigma}\right)  }%
_{\substack{=\left(  \text{\# of hamps of }\overline{D_{j}}\right)
\\\text{(by Lemma \ref{lem.hamps-by-lin}, applied to }D_{j}=\left(
f^{-1}\left(  j\right)  ,\ A_{j}\right)  \text{ instead of }D=\left(
V,A\right)  \text{)}}}\\
&  \ \ \ \ \ \ \ \ \ \ \ \ \ \ \ \ \ \ \ \ \left(  \text{here, we have renamed
the summation index }F_{j}\text{ as }F\right) \\
&  =\prod_{j\in f\left(  V\right)  }\left(  \text{\# of hamps of }%
\overline{D_{j}}\right)  .
\end{align*}
In view of this, we can rewrite (\ref{pf.lem.hamps-by-lin-f.short.2}) as%
\[
\sum_{\substack{\sigma\in\mathfrak{S}_{V};\\f\circ\sigma=f}}\ \ \sum
_{\substack{F\subseteq\mathbf{A}_{\sigma}\cap A\\\text{is linear}}}\left(
-1\right)  ^{\left\vert F\right\vert }=\prod_{j\in f\left(  V\right)  }\left(
\text{\# of hamps of }\overline{D_{j}}\right)  .
\]
This proves Lemma \ref{lem.hamps-by-lin-f}.
\end{proof}
\end{vershort}

\begin{verlong}

\begin{proof}
We shall use the notations from Definition \ref{def.level.levsets}, Definition
\ref{def.level.levsubdig} and Definition \ref{def.level.decompose-perm}. We
recall that every $j\in\mathbb{P}$ satisfies $D_{j}=\left(  f^{-1}\left(
j\right)  ,\ A_{j}\right)  $ (by definition of $D_{j}$).

Let $\sigma\in\mathfrak{S}_{V}$ be a permutation satisfying $f\circ\sigma=f$.
Then, Proposition \ref{prop.level.Asig-f} yields $\mathbf{A}_{\sigma}\cap
A\subseteq A_{f}$. Hence, $\mathbf{A}_{\sigma}\cap A=\mathbf{A}_{\sigma}\cap
A_{f}$ (because $\underbrace{\mathbf{A}_{\sigma}}_{=\mathbf{A}_{\sigma}%
\cap\mathbf{A}_{\sigma}}\cap A=\mathbf{A}_{\sigma}\cap\underbrace{\mathbf{A}%
_{\sigma}\cap A}_{\subseteq A_{f}}\subseteq\mathbf{A}_{\sigma}\cap A_{f}$ and
conversely $\mathbf{A}_{\sigma}\cap\underbrace{A_{f}}_{\subseteq A}%
\subseteq\mathbf{A}_{\sigma}\cap A$). Hence,%
\begin{align}
\sum_{\substack{F\subseteq\mathbf{A}_{\sigma}\cap A\\\text{is linear}}}\left(
-1\right)  ^{\left\vert F\right\vert }  &  =\sum_{\substack{F\subseteq
\mathbf{A}_{\sigma}\cap A_{f}\\\text{is linear}}}\left(  -1\right)
^{\left\vert F\right\vert }\nonumber\\
&  =\sum_{\substack{F\subseteq A_{f}\text{ is linear};\\F\subseteq
\mathbf{A}_{\sigma}}}\left(  -1\right)  ^{\left\vert F\right\vert }
\label{pf.lem.hamps-by-lin-f.1}%
\end{align}
(since a subset of $\mathbf{A}_{\sigma}\cap A_{f}$ is the same thing as a
subset $F$ of $A_{f}$ that satisfies $F\subseteq\mathbf{A}_{\sigma}$).

Forget that we fixed $\sigma$. We thus have proved
(\ref{pf.lem.hamps-by-lin-f.1}) for every $\sigma\in\mathfrak{S}_{V}$
satisfying $f\circ\sigma=f$.

Summing up the equality (\ref{pf.lem.hamps-by-lin-f.1}) over all permutations
$\sigma\in\mathfrak{S}_{V}$ satisfying $f\circ\sigma=f$, we obtain
\begin{align}
&  \sum_{\substack{\sigma\in\mathfrak{S}_{V};\\f\circ\sigma=f}}\ \ \sum
_{\substack{F\subseteq\mathbf{A}_{\sigma}\cap A\\\text{is linear}}}\left(
-1\right)  ^{\left\vert F\right\vert }\nonumber\\
&  =\sum_{\substack{\sigma\in\mathfrak{S}_{V};\\f\circ\sigma=f}}\ \ \sum
_{\substack{F\subseteq A_{f}\text{ is linear};\\F\subseteq\mathbf{A}_{\sigma}%
}}\left(  -1\right)  ^{\left\vert F\right\vert }\nonumber\\
&  =\sum_{F\subseteq A_{f}\text{ is linear}}\ \ \underbrace{\sum
_{\substack{\sigma\in\mathfrak{S}_{V};\\F\subseteq\mathbf{A}_{\sigma}%
;\\f\circ\sigma=f}}\left(  -1\right)  ^{\left\vert F\right\vert }}_{=\left(
-1\right)  ^{\left\vert F\right\vert }\cdot\left(  \text{\# of }\sigma
\in\mathfrak{S}_{V}\text{ satisfying }F\subseteq\mathbf{A}_{\sigma}\text{ and
}f\circ\sigma=f\right)  }\nonumber\\
&  =\sum_{F\subseteq A_{f}\text{ is linear}}\left(  -1\right)  ^{\left\vert
F\right\vert }\cdot\left(  \text{\# of }\sigma\in\mathfrak{S}_{V}\text{
satisfying }F\subseteq\mathbf{A}_{\sigma}\text{ and }f\circ\sigma=f\right)  .
\label{pf.lem.hamps-by-lin-f.2}%
\end{align}

Now, we observe the following: If $F_{j}$ is a linear subset of $A_{j}$ for
each $j\in f\left(  V\right)  $, then the disjoint union $\bigsqcup_{j\in
f\left(  V\right)  }F_{j}$ is well-defined\footnote{\textit{Proof.} Let
$F_{j}$ be a linear subset of $A_{j}$ for each $j\in f\left(  V\right)  $. The
sets $A_{j}$ for different $j\in f\left(  V\right)  $ are disjoint (since
Proposition \ref{prop.level.Af=sum} yields that the sets $A_{1},A_{2}%
,A_{3},\ldots$ are disjoint). Hence, their subsets $F_{j}$ must be disjoint as
well (since $F_{j}$ is a subset of $A_{j}$ for each $j\in f\left(  V\right)
$). Thus, the disjoint union $\bigsqcup_{j\in f\left(  V\right)  }F_{j}$ is
well-defined.}, and is a linear subset of $A_{f}$ (by the \textquotedblleft%
$\Longleftarrow$\textquotedblright\ direction of Proposition
\ref{prop.level.lin-sub-Af} \textbf{(a)}, applied to $F=\bigsqcup_{j\in
f\left(  V\right)  }F_{j}$). Hence, the map%
\begin{align*}
&  \text{from }\left\{  \text{families }\left(  F_{j}\right)  _{j\in f\left(
V\right)  }\text{, where each }F_{j}\text{ is a linear subset of }%
A_{j}\right\} \\
&  \text{to }\left\{  \text{linear subsets of }A_{f}\right\} \\
&  \text{that sends each family }\left(  F_{j}\right)  _{j\in f\left(
V\right)  }\text{ to }\bigsqcup_{j\in f\left(  V\right)  }F_{j}%
\end{align*}
is well-defined. Moreover, this map is injective (since Proposition
\ref{prop.level.lin-sub-Af} \textbf{(b)} shows that the sets $F_{j}$ are
uniquely determined by their union $\bigsqcup_{j\in f\left(  V\right)  }F_{j}%
$) and surjective (by the \textquotedblleft$\Longrightarrow$\textquotedblright%
\ direction of Proposition \ref{prop.level.lin-sub-Af} \textbf{(a)}). Thus, it
is bijective. Hence, we can substitute $\bigsqcup_{j\in f\left(  V\right)
}F_{j}$ for $F$ in the sum on the right hand side of
(\ref{pf.lem.hamps-by-lin-f.2}). We thus obtain%
\begin{align}
&  \sum_{F\subseteq A_{f}\text{ is linear}}\left(  -1\right)  ^{\left\vert
F\right\vert }\cdot\left(  \text{\# of }\sigma\in\mathfrak{S}_{V}\text{
satisfying }F\subseteq\mathbf{A}_{\sigma}\text{ and }f\circ\sigma=f\right)
\nonumber\\
&  =\sum_{\substack{\left(  F_{j}\right)  _{j\in f\left(  V\right)  }\text{ is
a family}\\\text{of linear subsets }F_{j}\subseteq A_{j}}}\left(  -1\right)
^{\left\vert \bigsqcup_{j\in f\left(  V\right)  }F_{j}\right\vert }\nonumber\\
&  \ \ \ \ \ \ \ \ \ \ \cdot\left(  \text{\# of }\sigma\in\mathfrak{S}%
_{V}\text{ satisfying }\bigsqcup_{j\in f\left(  V\right)  }F_{j}%
\subseteq\mathbf{A}_{\sigma}\text{ and }f\circ\sigma=f\right)  .
\label{pf.lem.hamps-by-lin-f.5}%
\end{align}

Now, fix a family $\left(  F_{j}\right)  _{j\in f\left(  V\right)  }$ of
linear subsets $F_{j}\subseteq A_{j}$.

A permutation $\sigma\in\mathfrak{S}_{V}$ satisfies $f\circ\sigma=f$ if and
only if it can be written in the form $\sigma=\bigoplus_{j\in f\left(
V\right)  }\sigma_{j}$, where $\sigma_{j}\in\mathfrak{S}_{f^{-1}\left(
j\right)  }$ for each $j\in f\left(  V\right)  $ (by Proposition
\ref{prop.level.decompose-perm} \textbf{(a)}). Moreover, if $\sigma$ is
written in this way, then we have $\mathbf{A}_{\sigma}=\bigsqcup_{j\in
f\left(  V\right)  }\mathbf{A}_{\sigma_{j}}$ (by Proposition
\ref{prop.level.Asig-f2}).

Hence, if $\left(  \sigma_{j}\right)  _{j\in f\left(  V\right)  }\in
\prod_{j\in f\left(  V\right)  }\mathfrak{S}_{f^{-1}\left(  j\right)  }$ is a
family of permutations (i.e., if we are given a permutation $\sigma_{j}%
\in\mathfrak{S}_{f^{-1}\left(  j\right)  }$ for each $j\in f\left(  V\right)
$) satisfying $\bigsqcup_{j\in f\left(  V\right)  }F_{j}\subseteq
\bigsqcup_{j\in f\left(  V\right)  }\mathbf{A}_{\sigma_{j}}$, then
$\bigoplus_{j\in f\left(  V\right)  }\sigma_{j}$ is a permutation $\sigma
\in\mathfrak{S}_{V}$ satisfying $\bigsqcup_{j\in f\left(  V\right)  }%
F_{j}\subseteq\mathbf{A}_{\sigma}$ and $f\circ\sigma=f$%
\ \ \ \ \footnote{\textit{Proof.} Let $\left(  \sigma_{j}\right)  _{j\in
f\left(  V\right)  }\in\prod_{j\in f\left(  V\right)  }\mathfrak{S}%
_{f^{-1}\left(  j\right)  }$ be a family of permutations satisfying
$\bigsqcup_{j\in f\left(  V\right)  }F_{j}\subseteq\bigsqcup_{j\in f\left(
V\right)  }\mathbf{A}_{\sigma_{j}}$. Set $\sigma=\bigoplus_{j\in f\left(
V\right)  }\sigma_{j}$. Then, we must prove that $\bigsqcup_{j\in f\left(
V\right)  }F_{j}\subseteq\mathbf{A}_{\sigma}$ and $f\circ\sigma=f$.
\par
However, this is easy: The \textquotedblleft$\Longleftarrow$\textquotedblright%
\ direction of Proposition \ref{prop.level.decompose-perm} \textbf{(a)} yields
$f\circ\sigma=f$ (since $\sigma=\bigoplus_{j\in f\left(  V\right)  }\sigma
_{j}$ with $\sigma_{j}\in\mathfrak{S}_{f^{-1}\left(  j\right)  }$ for each
$j\in f\left(  V\right)  $). Thus, Proposition \ref{prop.level.Asig-f2} yields
$\mathbf{A}_{\sigma}=\bigsqcup_{j\in f\left(  V\right)  }\mathbf{A}%
_{\sigma_{j}}$. Hence, $\bigsqcup_{j\in f\left(  V\right)  }F_{j}%
\subseteq\bigsqcup_{j\in f\left(  V\right)  }\mathbf{A}_{\sigma_{j}%
}=\mathbf{A}_{\sigma}$. Thus, both $\bigsqcup_{j\in f\left(  V\right)  }%
F_{j}\subseteq\mathbf{A}_{\sigma}$ and $f\circ\sigma=f$ are proved.}. Thus,
the map
\begin{align*}
&  \text{from }\left\{  \text{families }\left(  \sigma_{j}\right)  _{j\in
f\left(  V\right)  }\in\prod_{j\in f\left(  V\right)  }\mathfrak{S}%
_{f^{-1}\left(  j\right)  }\text{ satisfying }\bigsqcup_{j\in f\left(
V\right)  }F_{j}\subseteq\bigsqcup_{j\in f\left(  V\right)  }\mathbf{A}%
_{\sigma_{j}}\right\} \\
&  \text{to }\left\{  \sigma\in\mathfrak{S}_{V}\text{ satisfying }%
\bigsqcup_{j\in f\left(  V\right)  }F_{j}\subseteq\mathbf{A}_{\sigma}\text{
and }f\circ\sigma=f\right\} \\
&  \text{that sends each family }\left(  \sigma_{j}\right)  _{j\in f\left(
V\right)  }\text{ to }\bigoplus_{j\in f\left(  V\right)  }\sigma_{j}%
\end{align*}
is well-defined. This map is furthermore surjective (this follows easily from
the \textquotedblleft$\Longrightarrow$\textquotedblright\ direction of
Proposition \ref{prop.level.decompose-perm} \textbf{(a)}%
\footnote{\textit{Proof.} Let $\sigma\in\mathfrak{S}_{V}$ be a permutation
satisfying $\bigsqcup_{j\in f\left(  V\right)  }F_{j}\subseteq\mathbf{A}%
_{\sigma}$ and $f\circ\sigma=f$. We must prove that $\sigma$ is an image of
some family $\left(  \sigma_{j}\right)  _{j\in f\left(  V\right)  }\in
\prod_{j\in f\left(  V\right)  }\mathfrak{S}_{f^{-1}\left(  j\right)  }$ under
the map we just constructed. In other words, we must prove that there exists a
family $\left(  \sigma_{j}\right)  _{j\in f\left(  V\right)  }\in\prod_{j\in
f\left(  V\right)  }\mathfrak{S}_{f^{-1}\left(  j\right)  }$ satisfying
$\bigsqcup_{j\in f\left(  V\right)  }F_{j}\subseteq\bigsqcup_{j\in f\left(
V\right)  }\mathbf{A}_{\sigma_{j}}$ such that $\sigma=\bigoplus_{j\in f\left(
V\right)  }\sigma_{j}$.
\par
However, the \textquotedblleft$\Longrightarrow$\textquotedblright\ direction
of Proposition \ref{prop.level.decompose-perm} \textbf{(a)} yields that
$\sigma$ can be written in the form $\sigma=\bigoplus_{j\in f\left(  V\right)
}\sigma_{j}$, where $\sigma_{j}\in\mathfrak{S}_{f^{-1}\left(  j\right)  }$ for
each $j\in f\left(  V\right)  $. In other words, there exists a family
$\left(  \sigma_{j}\right)  _{j\in f\left(  V\right)  }\in\prod_{j\in f\left(
V\right)  }\mathfrak{S}_{f^{-1}\left(  j\right)  }$ such that $\sigma
=\bigoplus_{j\in f\left(  V\right)  }\sigma_{j}$. This family $\left(
\sigma_{j}\right)  _{j\in f\left(  V\right)  }$ furthermore satisfies
$\bigsqcup_{j\in f\left(  V\right)  }F_{j}\subseteq\mathbf{A}_{\sigma
}=\bigsqcup_{j\in f\left(  V\right)  }\mathbf{A}_{\sigma_{j}}$ (by Proposition
\ref{prop.level.Asig-f2}). Thus, we have proved that there exists a family
$\left(  \sigma_{j}\right)  _{j\in f\left(  V\right)  }\in\prod_{j\in f\left(
V\right)  }\mathfrak{S}_{f^{-1}\left(  j\right)  }$ satisfying $\bigsqcup
_{j\in f\left(  V\right)  }F_{j}\subseteq\bigsqcup_{j\in f\left(  V\right)
}\mathbf{A}_{\sigma_{j}}$ such that $\sigma=\bigoplus_{j\in f\left(  V\right)
}\sigma_{j}$.}) and injective (since Proposition
\ref{prop.level.decompose-perm} \textbf{(b)} shows that the permutations
$\sigma_{j}$ are uniquely determined by $\sigma$ when $\sigma=\bigoplus_{j\in
f\left(  V\right)  }\sigma_{j}$). Thus, this map is bijective.

Hence, by the bijection principle, we have%
\begin{align}
&  \left(  \text{\# of }\sigma\in\mathfrak{S}_{V}\text{ satisfying }%
\bigsqcup_{j\in f\left(  V\right)  }F_{j}\subseteq\mathbf{A}_{\sigma}\text{
and }f\circ\sigma=f\right) \nonumber\\
&  =\left(  \text{\# of families }\left(  \sigma_{j}\right)  _{j\in f\left(
V\right)  }\in\prod_{j\in f\left(  V\right)  }\mathfrak{S}_{f^{-1}\left(
j\right)  }\text{ satisfying }\bigsqcup_{j\in f\left(  V\right)  }%
F_{j}\subseteq\bigsqcup_{j\in f\left(  V\right)  }\mathbf{A}_{\sigma_{j}%
}\right) \nonumber\\
&  =\left(  \text{\# of families }\left(  \sigma_{j}\right)  _{j\in f\left(
V\right)  }\in\prod_{j\in f\left(  V\right)  }\mathfrak{S}_{f^{-1}\left(
j\right)  }\text{ satisfying }F_{j}\subseteq\mathbf{A}_{\sigma_{j}}\text{ for
each }j\in f\left(  V\right)  \right) \nonumber\\
&  \ \ \ \ \ \ \ \ \ \ \ \ \ \ \ \ \ \ \ \ \left(
\begin{array}
[c]{c}%
\text{since the condition \textquotedblleft}\bigsqcup_{j\in f\left(  V\right)
}F_{j}\subseteq\bigsqcup_{j\in f\left(  V\right)  }\mathbf{A}_{\sigma_{j}%
}\text{\textquotedblright}\\
\text{is equivalent to \textquotedblleft}F_{j}\subseteq\mathbf{A}_{\sigma_{j}%
}\text{ for each }j\in f\left(  V\right)  \text{\textquotedblright}\\
\text{(by Lemma \ref{lem.level.djunex})}%
\end{array}
\right) \nonumber\\
&  =\prod_{j\in f\left(  V\right)  }\left(  \text{\# of }\sigma_{j}%
\in\mathfrak{S}_{f^{-1}\left(  j\right)  }\text{ satisfying }F_{j}%
\subseteq\mathbf{A}_{\sigma_{j}}\right)  \ \ \ \ \ \ \ \ \ \ \left(  \text{by
the product rule}\right) \nonumber\\
&  =\prod_{j\in f\left(  V\right)  }\left(  \text{\# of }\sigma\in
\mathfrak{S}_{f^{-1}\left(  j\right)  }\text{ satisfying }F_{j}\subseteq
\mathbf{A}_{\sigma}\right) \label{pf.lem.hamps-by-lin-f.6}\\
&  \ \ \ \ \ \ \ \ \ \ \ \ \ \ \ \ \ \ \ \ \left(  \text{here, we have renamed
the index }\sigma_{j}\text{ as }\sigma\right)  .\nonumber
\end{align}

Also, we have $\left\vert \bigsqcup_{j\in f\left(  V\right)  }F_{j}\right\vert
=\sum_{j\in f\left(  V\right)  }\left\vert F_{j}\right\vert $ (by the sum
rule), and thus%
\begin{equation}
\left(  -1\right)  ^{\left\vert \bigsqcup_{j\in f\left(  V\right)  }%
F_{j}\right\vert }=\left(  -1\right)  ^{\sum_{j\in f\left(  V\right)
}\left\vert F_{j}\right\vert }=\prod_{j\in f\left(  V\right)  }\left(
-1\right)  ^{\left\vert F_{j}\right\vert }. \label{pf.lem.hamps-by-lin-f.7a}%
\end{equation}

Multiplying the equalities (\ref{pf.lem.hamps-by-lin-f.7a}) and
(\ref{pf.lem.hamps-by-lin-f.6}), we obtain%
\begin{align}
&  \left(  -1\right)  ^{\left\vert \bigsqcup_{j\in f\left(  V\right)  }%
F_{j}\right\vert }\cdot\left(  \text{\# of }\sigma\in\mathfrak{S}_{V}\text{
satisfying }\bigsqcup_{j\in f\left(  V\right)  }F_{j}\subseteq\mathbf{A}%
_{\sigma}\text{ and }f\circ\sigma=f\right) \nonumber\\
&  =\left(  \prod_{j\in f\left(  V\right)  }\left(  -1\right)  ^{\left\vert
F_{j}\right\vert }\right)  \cdot\prod_{j\in f\left(  V\right)  }\left(
\text{\# of }\sigma\in\mathfrak{S}_{f^{-1}\left(  j\right)  }\text{ satisfying
}F_{j}\subseteq\mathbf{A}_{\sigma}\right) \nonumber\\
&  =\prod_{j\in f\left(  V\right)  }\left(  \left(  -1\right)  ^{\left\vert
F_{j}\right\vert }\cdot\left(  \text{\# of }\sigma\in\mathfrak{S}%
_{f^{-1}\left(  j\right)  }\text{ satisfying }F_{j}\subseteq\mathbf{A}%
_{\sigma}\right)  \right)  . \label{pf.lem.hamps-by-lin-f.7b}%
\end{align}

Forget that we fixed $\left(  F_{j}\right)  _{j\in f\left(  V\right)  }$. We
thus have proved the equality (\ref{pf.lem.hamps-by-lin-f.7b}) for any family
$\left(  F_{j}\right)  _{j\in f\left(  V\right)  }$ of linear subsets
$F_{j}\subseteq A_{j}$.

Now, (\ref{pf.lem.hamps-by-lin-f.5}) becomes%
\begin{align}
&  \sum_{F\subseteq A_{f}\text{ is linear}}\left(  -1\right)  ^{\left\vert
F\right\vert }\cdot\left(  \text{\# of }\sigma\in\mathfrak{S}_{V}\text{
satisfying }F\subseteq\mathbf{A}_{\sigma}\text{ and }f\circ\sigma=f\right)
\nonumber\\
&  =\sum_{\substack{\left(  F_{j}\right)  _{j\in f\left(  V\right)  }\text{ is
a family}\\\text{of linear subsets }F_{j}\subseteq A_{j}}}\nonumber\\
&  \ \ \ \ \ \ \ \ \ \ \cdot\underbrace{\left(  -1\right)  ^{\left\vert
\bigsqcup_{j\in f\left(  V\right)  }F_{j}\right\vert }\cdot\left(  \text{\# of
}\sigma\in\mathfrak{S}_{V}\text{ satisfying }\bigsqcup_{j\in f\left(
V\right)  }F_{j}\subseteq\mathbf{A}_{\sigma}\text{ and }f\circ\sigma=f\right)
}_{\substack{=\prod_{j\in f\left(  V\right)  }\left(  \left(  -1\right)
^{\left\vert F_{j}\right\vert }\cdot\left(  \text{\# of }\sigma\in
\mathfrak{S}_{f^{-1}\left(  j\right)  }\text{ satisfying }F_{j}\subseteq
\mathbf{A}_{\sigma}\right)  \right)  \\\text{(by
(\ref{pf.lem.hamps-by-lin-f.7b}))}}}\nonumber\\
&  =\sum_{\substack{\left(  F_{j}\right)  _{j\in f\left(  V\right)  }\text{ is
a family}\\\text{of linear subsets }F_{j}\subseteq A_{j}}}\prod_{j\in f\left(
V\right)  }\left(  \left(  -1\right)  ^{\left\vert F_{j}\right\vert }%
\cdot\left(  \text{\# of }\sigma\in\mathfrak{S}_{f^{-1}\left(  j\right)
}\text{ satisfying }F_{j}\subseteq\mathbf{A}_{\sigma}\right)  \right)
\nonumber\\
&  =\prod_{j\in f\left(  V\right)  }\ \ \sum_{\substack{F_{j}\subseteq
A_{j}\text{ is linear}}}\left(  -1\right)  ^{\left\vert F_{j}\right\vert
}\cdot\left(  \text{\# of }\sigma\in\mathfrak{S}_{f^{-1}\left(  j\right)
}\text{ satisfying }F_{j}\subseteq\mathbf{A}_{\sigma}\right) \nonumber\\
&  \ \ \ \ \ \ \ \ \ \ \ \ \ \ \ \ \ \ \ \ \left(  \text{by the product
rule}\right) \nonumber\\
&  =\prod_{j\in f\left(  V\right)  }\ \ \underbrace{\sum_{\substack{F\subseteq
A_{j}\text{ is linear}}}\left(  -1\right)  ^{\left\vert F\right\vert }%
\cdot\left(  \text{\# of }\sigma\in\mathfrak{S}_{f^{-1}\left(  j\right)
}\text{ satisfying }F\subseteq\mathbf{A}_{\sigma}\right)  }%
_{\substack{=\left(  \text{\# of hamps of }\overline{D_{j}}\right)
\\\text{(by Lemma \ref{lem.hamps-by-lin}, applied to }D_{j}=\left(
f^{-1}\left(  j\right)  ,\ A_{j}\right)  \text{ instead of }D=\left(
V,A\right)  \\\text{(since }f^{-1}\left(  j\right)  \neq\varnothing\text{
(because }j\in f\left(  V\right)  \text{)))}}}\nonumber\\
&  \ \ \ \ \ \ \ \ \ \ \ \ \ \ \ \ \ \ \ \ \left(  \text{here, we have renamed
the summation index }F_{j}\text{ as }F\right) \nonumber\\
&  =\prod_{j\in f\left(  V\right)  }\left(  \text{\# of hamps of }%
\overline{D_{j}}\right)  . \label{pf.lem.hamps-by-lin-f.15}%
\end{align}
Now, (\ref{pf.lem.hamps-by-lin-f.2}) becomes%
\begin{align*}
&  \sum_{\substack{\sigma\in\mathfrak{S}_{V};\\f\circ\sigma=f}}\ \ \sum
_{\substack{F\subseteq\mathbf{A}_{\sigma}\cap A\\\text{is linear}}}\left(
-1\right)  ^{\left\vert F\right\vert }\\
&  =\sum_{F\subseteq A_{f}\text{ is linear}}\left(  -1\right)  ^{\left\vert
F\right\vert }\cdot\left(  \text{\# of }\sigma\in\mathfrak{S}_{V}\text{
satisfying }F\subseteq\mathbf{A}_{\sigma}\text{ and }f\circ\sigma=f\right) \\
&  =\prod_{j\in f\left(  V\right)  }\left(  \text{\# of hamps of }%
\overline{D_{j}}\right)  \ \ \ \ \ \ \ \ \ \ \left(  \text{by
(\ref{pf.lem.hamps-by-lin-f.15})}\right)  .
\end{align*}
This proves Lemma \ref{lem.hamps-by-lin-f}.
\end{proof}
\end{verlong}

\begin{noncompile}
TODO: Deriving Lemma \ref{lem.hamps-by-lin-f} from Lemma
\ref{lem.hamps-by-lin} might have been an instance of the noncommutative
Cauchy kernel or Aguiar--Bergeron--Sottile universality; in this case, the
handwaving about disjoint unions could perhaps be avoided.
\end{noncompile}

\subsection{$\left(  f,D\right)  $-friendly $V$-listings}

The following restatement of Lemma \ref{lem.hamps-by-lin-f} will be useful for us:

\begin{lemma}
\label{lem.friendlies-by-f}Let $D=\left(  V,A\right)  $ be a digraph. Let
$f:V\rightarrow\mathbb{P}$ be any map. A $V$-listing $v=\left(  v_{1}%
,v_{2},\ldots,v_{n}\right)  $ will be called $\left(  f,D\right)
$\emph{-friendly} if it has the properties that $f\left(  v_{1}\right)  \leq
f\left(  v_{2}\right)  \leq\cdots\leq f\left(  v_{n}\right)  $ and that%
\begin{equation}
f\left(  v_{p}\right)  <f\left(  v_{p+1}\right)  \text{ for each }p\in\left[
n-1\right]  \text{ satisfying }\left(  v_{p},v_{p+1}\right)  \in A.
\label{eq.lem.friendlies-by-f.less}%
\end{equation}
Then,%
\[
\sum_{\substack{\sigma\in\mathfrak{S}_{V};\\f\circ\sigma=f}}\ \ \sum
_{\substack{F\subseteq\mathbf{A}_{\sigma}\cap A\\\text{is linear}}}\left(
-1\right)  ^{\left\vert F\right\vert }=\left(  \text{\# of }\left(
f,D\right)  \text{-friendly }V\text{-listings}\right)  .
\]

\end{lemma}

\begin{vershort}

\begin{proof}
For each $j\in f\left(  V\right)  $, we define a digraph $D_{j}$ as in
Definition \ref{def.level.levsubdig} \textbf{(c)}. This digraph $D_{j}$ is the
restriction of the digraph $D$ to the subset $f^{-1}\left(  j\right)  $. In
particular, its vertex set is $f^{-1}\left(  j\right)  $. In other words, its
vertices are precisely those vertices of $D$ that have level $j$ (with respect
to $f$).

Clearly, a $V$-listing $v=\left(  v_{1},v_{2},\ldots,v_{n}\right)  $ satisfies
$f\left(  v_{1}\right)  \leq f\left(  v_{2}\right)  \leq\cdots\leq f\left(
v_{n}\right)  $ if and only if it lists the vertices of $D$ in the order of
increasing level, i.e., if it first lists the vertices of the smallest level,
then the vertices of the second-smallest level, and so on.

In other words, a $V$-listing $v=\left(  v_{1},v_{2},\ldots,v_{n}\right)  $
satisfies $f\left(  v_{1}\right)  \leq f\left(  v_{2}\right)  \leq\cdots\leq
f\left(  v_{n}\right)  $ if and only if it can be constructed by choosing an
$f^{-1}\left(  j\right)  $-listing $v^{\left(  j\right)  }$ for each $j\in
f\left(  V\right)  $ and concatenating all these $f^{-1}\left(  j\right)
$-listings $v^{\left(  j\right)  }$ in the order of increasing $j$. Moreover,
if it can be constructed in this way, then its construction is unique (i.e.,
each $v^{\left(  j\right)  }$ is determined uniquely by $v$). Finally, for a
$V$-listing $v$ that is written as a concatenation of such $f^{-1}\left(
j\right)  $-listings $v^{\left(  j\right)  }$, the condition
(\ref{eq.lem.friendlies-by-f.less}) is equivalent to the condition that each
$v^{\left(  j\right)  }$ is a hamp of $\overline{D_{j}}$ (indeed, this is
easiest to see by rewriting (\ref{eq.lem.friendlies-by-f.less}) in the
contrapositive form \textquotedblleft if $p\in\left[  n-1\right]  $ satisfies
$f\left(  v_{p}\right)  =f\left(  v_{p+1}\right)  $, then $\left(
v_{p},v_{p+1}\right)  \notin A$\textquotedblright). Thus, a $V$-listing
$v=\left(  v_{1},v_{2},\ldots,v_{n}\right)  $ satisfies both $f\left(
v_{1}\right)  \leq f\left(  v_{2}\right)  \leq\cdots\leq f\left(
v_{n}\right)  $ and (\ref{eq.lem.friendlies-by-f.less}) at the same time if
and only if it can be constructed by choosing a hamp of $\overline{D_{j}}$ for
each $j\in f\left(  V\right)  $ and concatenating all these hamps in the order
of increasing $j$. In other words, a $V$-listing $v$ is $\left(  f,D\right)
$-friendly if and only if it can be constructed in this way. Since this
construction is unique, we thus have%
\[
\left(  \text{\# of }\left(  f,D\right)  \text{-friendly }V\text{-listings}%
\right)  =\prod_{j\in f\left(  V\right)  }\left(  \text{\# of hamps of
}\overline{D_{j}}\right)  .
\]
Thus, Lemma \ref{lem.friendlies-by-f} follows from Lemma
\ref{lem.hamps-by-lin-f}.
\end{proof}
\end{vershort}

\begin{verlong}

\begin{proof}
For each $j\in f\left(  V\right)  $, we define a digraph $D_{j}$ as in
Definition \ref{def.level.levsubdig} \textbf{(c)}. The vertex set of this
digraph $D_{j}$ is $f^{-1}\left(  j\right)  $. In other words, its vertices
are precisely those vertices of $D$ that have level $j$ (with respect to $f$).

Clearly, a $V$-listing $v=\left(  v_{1},v_{2},\ldots,v_{n}\right)  $ satisfies
$f\left(  v_{1}\right)  \leq f\left(  v_{2}\right)  \leq\cdots\leq f\left(
v_{n}\right)  $ if and only if it lists the vertices of $D$ in the order of
increasing level, i.e., if it first lists the vertices of the smallest level,
then the vertices of the second-smallest level, and so on.

If $\left(  a_{j}\right)  _{j\in f\left(  V\right)  }$ is a family of (finite)
lists (one list $a_{j}$ for each $j\in f\left(  V\right)  $), then
$\bigotimes\limits_{j\in f\left(  V\right)  }a_{j}$ shall denote the
concatenation of these lists $a_{j}$ in the order of increasing $j$ (that is,
the list starting with the entries of $a_{j}$ for the smallest $j\in f\left(
V\right)  $, then continuing with the entries of $a_{j}$ for the
second-smallest $j\in f\left(  V\right)  $, and so on). For instance, if
$f\left(  V\right)  =\left\{  2,3,5\right\}  $ and $a_{2}=\left(  u,v\right)
$ and $a_{3}=\left(  x,y,z\right)  $ and $a_{5}=\left(  p\right)  $, then
$\bigotimes\limits_{j\in f\left(  V\right)  }a_{j}=\left(  u,v,x,y,z,p\right)
$. The lists $a_{j}$ are called the \emph{factors} of the concatenation
$\bigotimes\limits_{j\in f\left(  V\right)  }a_{j}$.

Now, we shall prove five claims:

\begin{statement}
\textit{Claim 1:} Let $\left(  v^{\left(  j\right)  }\right)  _{j\in f\left(
V\right)  }$ be a family of lists, where each $v^{\left(  j\right)  }$ is an
$f^{-1}\left(  j\right)  $-listing. Write the concatenation $\bigotimes
\limits_{j\in f\left(  V\right)  }v^{\left(  j\right)  }$ in the form%
\[
\bigotimes\limits_{j\in f\left(  V\right)  }v^{\left(  j\right)  }=v=\left(
v_{1},v_{2},\ldots,v_{n}\right)  .
\]
Then, $v$ is a $V$-listing and satisfies $f\left(  v_{1}\right)  \leq f\left(
v_{2}\right)  \leq\cdots\leq f\left(  v_{n}\right)  $.
\end{statement}

[\textit{Proof of Claim 1:} For each $j\in f\left(  V\right)  $, the list
$v^{\left(  j\right)  }$ is an $f^{-1}\left(  j\right)  $-listing, thus a list
of elements of $f^{-1}\left(  j\right)  $, therefore a list of elements of $V$
(since $f^{-1}\left(  j\right)  \subseteq V$). Hence, the concatenation
$\bigotimes\limits_{j\in f\left(  V\right)  }v^{\left(  j\right)  }$ of these
lists $v^{\left(  j\right)  }$ is a list of elements of $V$ as well. In other
words, $v$ is a list of elements of $V$ (since $\bigotimes\limits_{j\in
f\left(  V\right)  }v^{\left(  j\right)  }=v$).

Moreover, each element of $V$ is contained exactly once in $v$%
\ \ \ \ \footnote{\textit{Proof.} Let $p$ be an arbitrary element of $V$. We
must then show that $p$ is contained exactly once in $v$.
\par
Indeed, let $i=f\left(  p\right)  $. Then, $p$ is an element of $f^{-1}\left(
i\right)  $. Also, $i=f\left(  p\right)  \in f\left(  V\right)  $.
\par
Moreover, the list $v^{\left(  i\right)  }$ is an $f^{-1}\left(  i\right)
$-listing (because for each $j\in f\left(  V\right)  $, the list $v^{\left(
j\right)  }$ is an $f^{-1}\left(  j\right)  $-listing). Hence, this list
$v^{\left(  i\right)  }$ contains each element of $f^{-1}\left(  i\right)  $
exactly once. In particular, this shows that $v^{\left(  i\right)  }$ contains
$p$ exactly once (since $p$ is an element of $f^{-1}\left(  i\right)  $). In
other words, $p$ appears exactly once in the list $v^{\left(  i\right)  }$.
\par
Now, let $j\in f\left(  V\right)  $ be distinct from $i$. Then, $j\neq
i=f\left(  p\right)  $, so that $f\left(  p\right)  \neq j$. Therefore, $p$ is
not an element of $f^{-1}\left(  j\right)  $. However, the list $v^{\left(
j\right)  }$ is an $f^{-1}\left(  j\right)  $-listing (by assumption), and
thus is a list of elements of $f^{-1}\left(  j\right)  $. Hence, this list
$v^{\left(  j\right)  }$ does not contain $p$ (since $p$ is not an element of
$f^{-1}\left(  j\right)  $). In other words, $p$ does not appear in this list
$v^{\left(  j\right)  }$.
\par
Forget that we fixed $j$. We thus have shown that if $j\in f\left(  V\right)
$ is distinct from $i$, then $p$ does not appear in the list $v^{\left(
j\right)  }$.
\par
Altogether, we now know that $p$ appears exactly once in the list $v^{\left(
i\right)  }$ but does not appear in the list $v^{\left(  j\right)  }$ for any
$j\in f\left(  V\right)  $ that is distinct from $i$. In other words, $p$
appears exactly once in the list $v^{\left(  i\right)  }$ and does not appear
in any other list $v^{\left(  j\right)  }$ with $j\neq i$. Consequently, $p$
appears exactly once in the concatenation $\bigotimes\limits_{j\in f\left(
V\right)  }v^{\left(  j\right)  }$. In other words, $p$ appears exactly once
in $v$ (since $v=\bigotimes\limits_{j\in f\left(  V\right)  }v^{\left(
j\right)  }$). In other words, $p$ is contained exactly once in $v$. Qed.}.
Hence, $v$ is a $V$-listing (since $v$ is a list of elements of $V$).

It remains to show that $f\left(  v_{1}\right)  \leq f\left(  v_{2}\right)
\leq\cdots\leq f\left(  v_{n}\right)  $.

We fix $p\in\left[  n-1\right]  $. We shall show that $f\left(  v_{p}\right)
\leq f\left(  v_{p+1}\right)  $.

Indeed, assume the contrary. Thus, $f\left(  v_{p}\right)  >f\left(
v_{p+1}\right)  $. Set $\alpha=f\left(  v_{p}\right)  $ and $\beta=f\left(
v_{p+1}\right)  $. Thus, $\alpha=f\left(  v_{p}\right)  >f\left(
v_{p+1}\right)  =\beta$. Moreover, $\alpha=f\left(  v_{p}\right)  \in f\left(
V\right)  $ and $\beta=f\left(  v_{p+1}\right)  \in f\left(  V\right)  $.
Furthermore, from $f\left(  v_{p}\right)  =\alpha$, we see that $v_{p}$ is an
element of $f^{-1}\left(  \alpha\right)  $. From $f\left(  v_{p+1}\right)
=\beta$, we see that $v_{p+1}$ is an element of $f^{-1}\left(  \beta\right)  $.

We recall that $v$ is a $V$-listing. Thus, each entry of $v$ appears only once
in $v$.

Recall that for each $j\in f\left(  V\right)  $, the list $v^{\left(
j\right)  }$ is an $f^{-1}\left(  j\right)  $-listing. Hence, in particular,
$v^{\left(  \alpha\right)  }$ is an $f^{-1}\left(  \alpha\right)  $-listing.
Hence, each element of $f^{-1}\left(  \alpha\right)  $ appears in $v^{\left(
\alpha\right)  }$. Thus, in particular, $v_{p}$ appears in $v^{\left(
\alpha\right)  }$ (since $v_{p}$ is an element of $f^{-1}\left(
\alpha\right)  $). The same argument (applied to $p+1$ and $\beta$ instead of
$p$ and $\alpha$) shows that $v_{p+1}$ appears in $v^{\left(  \beta\right)  }$.

However, $\beta<\alpha$ (since $\alpha>\beta$). Thus, in the concatenation
$\bigotimes\limits_{j\in f\left(  V\right)  }v^{\left(  j\right)  }$, the
factor $v^{\left(  \beta\right)  }$ appears to the left of the factor
$v^{\left(  \alpha\right)  }$ (since this concatenation is concatenating the
lists $v^{\left(  j\right)  }$ in the order of increasing $j$). Hence, in
particular, in this concatenation $\bigotimes\limits_{j\in f\left(  V\right)
}v^{\left(  j\right)  }$, the entry $v_{p+1}$ appears to the left of the entry
$v_{p}$ (since $v_{p+1}$ appears in $v^{\left(  \beta\right)  }$, whereas
$v_{p}$ appears in $v^{\left(  \alpha\right)  }$). In other words, in the list
$v$, the entry $v_{p+1}$ appears to the left of the entry $v_{p}$ (since
$v=\bigotimes\limits_{j\in f\left(  V\right)  }v^{\left(  j\right)  }$). On
the other hand, it is clear that the entry $v_{p+1}$ appears to the right of
the entry $v_{p}$ in the list $v$ (since $v=\left(  v_{1},v_{2},\ldots
,v_{n}\right)  $ and $p+1>p$).

Thus, we have shown that in the list $v$, the entry $v_{p+1}$ appears both to
the left and to the right of the entry $v_{p}$. Clearly, this is only possible
if one of these entries appears more than once in the list $v$. We thus
conclude that one of these entries appears more than once in the list $v$.
However, this contradicts the fact that each entry of $v$ appears only once in
$v$.

This contradiction shows that our assumption was false. Hence, $f\left(
v_{p}\right)  \leq f\left(  v_{p+1}\right)  $ is proved.

Forget that we fixed $p$. We thus have shown that $f\left(  v_{p}\right)  \leq
f\left(  v_{p+1}\right)  $ for each $p\in\left[  n-1\right]  $. In other
words, $f\left(  v_{1}\right)  \leq f\left(  v_{2}\right)  \leq\cdots\leq
f\left(  v_{n}\right)  $. This completes the proof of Claim 1.]

\begin{statement}
\textit{Claim 2:} Let $\left(  v^{\left(  j\right)  }\right)  _{j\in f\left(
V\right)  }$ be a family of lists, where each $v^{\left(  j\right)  }$ is an
$f^{-1}\left(  j\right)  $-listing. Then, this family $\left(  v^{\left(
j\right)  }\right)  _{j\in f\left(  V\right)  }$ can be uniquely reconstructed
from the concatenation $\bigotimes\limits_{j\in f\left(  V\right)  }v^{\left(
j\right)  }$.
\end{statement}

[\textit{Proof of Claim 2:} Fix $i\in f\left(  V\right)  $. Then, $v^{\left(
i\right)  }$ is an $f^{-1}\left(  i\right)  $-listing (since each $v^{\left(
j\right)  }$ is an $f^{-1}\left(  j\right)  $-listing). This $f^{-1}\left(
i\right)  $-listing $v^{\left(  i\right)  }$ is a factor of the concatenation
$\bigotimes\limits_{j\in f\left(  V\right)  }v^{\left(  j\right)  }$. This
factor $v^{\left(  i\right)  }$ consists entirely of elements of
$f^{-1}\left(  i\right)  $ (since it is an $f^{-1}\left(  i\right)
$-listing), whereas all the other factors $v^{\left(  j\right)  }$ of the
concatenation $\bigotimes\limits_{j\in f\left(  V\right)  }v^{\left(
j\right)  }$ contain no elements of $f^{-1}\left(  i\right)  $
whatsoever\footnote{\textit{Proof.} We must prove that if $j\in f\left(
V\right)  $ is distinct from $i$, then $v^{\left(  j\right)  }$ contains no
elements of $f^{-1}\left(  i\right)  $ whatsoever.
\par
So let $j\in f\left(  V\right)  $ be distinct from $i$. We must prove that
$v^{\left(  j\right)  }$ contains no elements of $f^{-1}\left(  i\right)  $
whatsoever.
\par
Assume the contrary. Thus, $v^{\left(  j\right)  }$ contains some element of
$f^{-1}\left(  i\right)  $. Let $p$ be this element. Then, $f\left(  p\right)
=i$ (since $p$ is an element of $f^{-1}\left(  i\right)  $). However,
$v^{\left(  j\right)  }$ is an $f^{-1}\left(  j\right)  $-listing (by
assumption), and thus is a list of elements of $f^{-1}\left(  j\right)  $.
Hence, each entry of $v^{\left(  j\right)  }$ belongs to $f^{-1}\left(
j\right)  $. Since $p$ is an entry of $v^{\left(  j\right)  }$ (because
$v^{\left(  j\right)  }$ contains $p$), we thus conclude that $p$ belongs to
$f^{-1}\left(  j\right)  $. In other words, $f\left(  p\right)  =j$.
Therefore, $j=f\left(  p\right)  =i$, which contradicts the fact that $j$ is
distinct from $i$.
\par
This contradiction shows that our assumption was false. Hence, we have shown
that $v^{\left(  j\right)  }$ contains no elements of $f^{-1}\left(  i\right)
$ whatsoever.}. Thus, we can reconstruct $v^{\left(  i\right)  }$ from
$\bigotimes\limits_{j\in f\left(  V\right)  }v^{\left(  j\right)  }$ by
removing all entries that don't belong to $f^{-1}\left(  i\right)  $ (since
this removal preserves the factor $v^{\left(  i\right)  }$ but makes all the
other factors $v^{\left(  j\right)  }$ disappear).

Forget that we fixed $i$. Thus, we have shown that for each $i\in f\left(
V\right)  $, we can reconstruct $v^{\left(  i\right)  }$ from $\bigotimes
\limits_{j\in f\left(  V\right)  }v^{\left(  j\right)  }$. In other words, we
can reconstruct the family $\left(  v^{\left(  i\right)  }\right)  _{i\in
f\left(  V\right)  }$ from $\bigotimes\limits_{j\in f\left(  V\right)
}v^{\left(  j\right)  }$. In other words, we can reconstruct the family
$\left(  v^{\left(  j\right)  }\right)  _{j\in f\left(  V\right)  }$ from
$\bigotimes\limits_{j\in f\left(  V\right)  }v^{\left(  j\right)  }$ (since
$\left(  v^{\left(  i\right)  }\right)  _{i\in f\left(  V\right)  }=\left(
v^{\left(  j\right)  }\right)  _{j\in f\left(  V\right)  }$). This proves
Claim 2.]

\begin{statement}
\textit{Claim 3:} Let $\left(  v^{\left(  j\right)  }\right)  _{j\in f\left(
V\right)  }$ be a family of lists, where each $v^{\left(  j\right)  }$ is a
hamp of $\overline{D_{j}}$. Then, the concatenation $\bigotimes\limits_{j\in
f\left(  V\right)  }v^{\left(  j\right)  }$ is an $\left(  f,D\right)
$-friendly $V$-listing.
\end{statement}

[\textit{Proof of Claim 3:} Let us write the concatenation $\bigotimes
\limits_{j\in f\left(  V\right)  }v^{\left(  j\right)  }$ in the form%
\[
\bigotimes\limits_{j\in f\left(  V\right)  }v^{\left(  j\right)  }=v=\left(
v_{1},v_{2},\ldots,v_{n}\right)  .
\]
For each $j\in f\left(  V\right)  $, the list $v^{\left(  j\right)  }$ is a
hamp of $\overline{D_{j}}$ and thus an $f^{-1}\left(  j\right)  $%
-listing\footnote{\textit{Proof.} Let $j\in f\left(  V\right)  $. Then, the
list $v^{\left(  j\right)  }$ is a hamp of $\overline{D_{j}}$. In other words,
this list $v^{\left(  j\right)  }$ is a $\overline{D_{j}}$-path that contains
each vertex of $\overline{D_{j}}$ (by the definition of a \textquotedblleft
hamp\textquotedblright).
\par
We recall that the digraph $D_{j}$ was defined to be $\left(  f^{-1}\left(
j\right)  ,\ A_{j}\right)  $. Hence, its complement $\overline{D_{j}}$ is
$\left(  f^{-1}\left(  j\right)  ,\ \left(  f^{-1}\left(  j\right)  \times
f^{-1}\left(  j\right)  \right)  \setminus A_{j}\right)  $ (by the definition
of the complement of a digraph). Thus, the vertices of $\overline{D_{j}}$ are
the elements of $f^{-1}\left(  j\right)  $, whereas the arcs of $\overline
{D_{j}}$ are the elements of $\left(  f^{-1}\left(  j\right)  \times
f^{-1}\left(  j\right)  \right)  \setminus A_{j}$.
\par
Now, recall that the list $v^{\left(  j\right)  }$ is a $\overline{D_{j}}%
$-path. In other words, $v^{\left(  j\right)  }$ is a nonempty tuple of
distinct elements of $f^{-1}\left(  j\right)  $ such that
$\operatorname*{Arcs}\left(  v^{\left(  j\right)  }\right)  \subseteq\left(
f^{-1}\left(  j\right)  \times f^{-1}\left(  j\right)  \right)  \setminus
A_{j}$ (by the definition of a \textquotedblleft$\overline{D_{j}}%
$-path\textquotedblright, since $\overline{D_{j}}=\left(  f^{-1}\left(
j\right)  ,\ \left(  f^{-1}\left(  j\right)  \times f^{-1}\left(  j\right)
\right)  \setminus A_{j}\right)  $). Thus, in particular, $v^{\left(
j\right)  }$ is a tuple of distinct elements of $f^{-1}\left(  j\right)  $.
Hence, $v^{\left(  j\right)  }$ contains no entry more than once.
\par
Furthermore, recall that the tuple $v^{\left(  j\right)  }$ contains each
vertex of $\overline{D_{j}}$. In other words, $v^{\left(  j\right)  }$
contains each element of $f^{-1}\left(  j\right)  $ (since the vertices of
$\overline{D_{j}}$ are the elements of $f^{-1}\left(  j\right)  $). Hence,
$v^{\left(  j\right)  }$ contains each element of $f^{-1}\left(  j\right)  $
exactly once (since $v^{\left(  j\right)  }$ contains no entry more than
once).
\par
Thus, we know that $v^{\left(  j\right)  }$ is a list of elements of
$f^{-1}\left(  j\right)  $ that contains each element of $f^{-1}\left(
j\right)  $ exactly once. In other words, $v^{\left(  j\right)  }$ is an
$f^{-1}\left(  j\right)  $-listing. Qed.}. Hence, Claim 1 shows that $v$ is a
$V$-listing and satisfies $f\left(  v_{1}\right)  \leq f\left(  v_{2}\right)
\leq\cdots\leq f\left(  v_{n}\right)  $.

We shall now show that this $V$-listing $v$ is $\left(  f,D\right)
$-friendly. Indeed, as we just proved, it satisfies $f\left(  v_{1}\right)
\leq f\left(  v_{2}\right)  \leq\cdots\leq f\left(  v_{n}\right)  $. In order
to prove that it is $\left(  f,D\right)  $-friendly, it thus suffices to show
that it satisfies%
\[
f\left(  v_{p}\right)  <f\left(  v_{p+1}\right)  \text{ for each }p\in\left[
n-1\right]  \text{ satisfying }\left(  v_{p},v_{p+1}\right)  \in A.
\]

So let us do this. Let $p\in\left[  n-1\right]  $ be such that $\left(
v_{p},v_{p+1}\right)  \in A$. We must prove that $f\left(  v_{p}\right)
<f\left(  v_{p+1}\right)  $.

Assume the contrary. Thus, $f\left(  v_{p}\right)  \geq f\left(
v_{p+1}\right)  $. Combining this with $f\left(  v_{p}\right)  \leq f\left(
v_{p+1}\right)  $ (which follows from $f\left(  v_{1}\right)  \leq f\left(
v_{2}\right)  \leq\cdots\leq f\left(  v_{n}\right)  $), we obtain $f\left(
v_{p}\right)  =f\left(  v_{p+1}\right)  $. Set $i=f\left(  v_{p}\right)  $.
Then, $i=f\left(  v_{p}\right)  \in f\left(  V\right)  $ and $i=f\left(
v_{p}\right)  =f\left(  v_{p+1}\right)  $. We have $v_{p}\in f^{-1}\left(
i\right)  $ (since $f\left(  v_{p}\right)  =i$) and $v_{p+1}\in f^{-1}\left(
i\right)  $ (since $f\left(  v_{p+1}\right)  =i$). In other words, $v_{p}$ and
$v_{p+1}$ belong to $f^{-1}\left(  i\right)  $.

In the above proof of Claim 2, we noticed the following: The factor
$v^{\left(  i\right)  }$ of the concatenation $\bigotimes\limits_{j\in
f\left(  V\right)  }v^{\left(  j\right)  }$ consists entirely of elements of
$f^{-1}\left(  i\right)  $, whereas all the other factors $v^{\left(
j\right)  }$ of this concatenation contain no elements of $f^{-1}\left(
i\right)  $ whatsoever. Hence, any entry of the concatenation $\bigotimes
\limits_{j\in f\left(  V\right)  }v^{\left(  j\right)  }$ that belongs to
$f^{-1}\left(  i\right)  $ must appear only in the $v^{\left(  i\right)  }$
factor of this concatenation.

However, $v_{p}$ and $v_{p+1}$ are consecutive entries of the concatenation
$\bigotimes\limits_{j\in f\left(  V\right)  }v^{\left(  j\right)  }$ (since
$\bigotimes\limits_{j\in f\left(  V\right)  }v^{\left(  j\right)  }=v=\left(
v_{1},v_{2},\ldots,v_{n}\right)  $). Since these two entries $v_{p}$ and
$v_{p+1}$ belong to $f^{-1}\left(  i\right)  $, we thus conclude that $v_{p}$
and $v_{p+1}$ appear only in the $v^{\left(  i\right)  }$ factor of this
concatenation (since any entry of the concatenation $\bigotimes\limits_{j\in
f\left(  V\right)  }v^{\left(  j\right)  }$ that belongs to $f^{-1}\left(
i\right)  $ must appear only in the $v^{\left(  i\right)  }$ factor of this
concatenation). Therefore, $v_{p}$ and $v_{p+1}$ are two consecutive entries
of $v^{\left(  i\right)  }$ (since $v_{p}$ and $v_{p+1}$ are consecutive
entries of the concatenation $\bigotimes\limits_{j\in f\left(  V\right)
}v^{\left(  j\right)  }$). In other words, the list $v^{\left(  i\right)  }$
has the form $\left(  \ldots,v_{p},v_{p+1},\ldots\right)  $ (where each
\textquotedblleft$\ldots$\textquotedblright\ stands for some number of
entries). Therefore, the pair $\left(  v_{p},v_{p+1}\right)  $ is an arc of
$v^{\left(  i\right)  }$. In other words, $\left(  v_{p},v_{p+1}\right)
\in\operatorname*{Arcs}\left(  v^{\left(  i\right)  }\right)  $.

However, $v^{\left(  i\right)  }$ is a hamp of $\overline{D_{i}}$ (since each
$v^{\left(  j\right)  }$ is a hamp of $\overline{D_{j}}$). In other words,
$v^{\left(  i\right)  }$ is a $\overline{D_{i}}$-path that contains each
vertex of $\overline{D_{i}}$ (by the definition of a \textquotedblleft
hamp\textquotedblright).

We recall that the digraph $D_{i}$ was defined to be $\left(  f^{-1}\left(
i\right)  ,\ A_{i}\right)  $. Hence, its complement $\overline{D_{i}}$ is
$\left(  f^{-1}\left(  i\right)  ,\ \left(  f^{-1}\left(  i\right)  \times
f^{-1}\left(  i\right)  \right)  \setminus A_{i}\right)  $ (by the definition
of the complement of a digraph). Since $v^{\left(  i\right)  }$ is a
$\overline{D_{i}}$-path, we thus have $\operatorname*{Arcs}\left(  v^{\left(
i\right)  }\right)  \subseteq\left(  f^{-1}\left(  i\right)  \times
f^{-1}\left(  i\right)  \right)  \setminus A_{i}$ (by the definition of a
\textquotedblleft$\overline{D_{i}}$-path\textquotedblright).

Hence, $\left(  v_{p},v_{p+1}\right)  \in\operatorname*{Arcs}\left(
v^{\left(  i\right)  }\right)  \subseteq\left(  f^{-1}\left(  i\right)  \times
f^{-1}\left(  i\right)  \right)  \setminus A_{i}$. In other words, $\left(
v_{p},v_{p+1}\right)  \in f^{-1}\left(  i\right)  \times f^{-1}\left(
i\right)  $ and $\left(  v_{p},v_{p+1}\right)  \notin A_{i}$.

Combining $\left(  v_{p},v_{p+1}\right)  \in A$ with $\left(  v_{p}%
,v_{p+1}\right)  \in f^{-1}\left(  i\right)  \times f^{-1}\left(  i\right)  $,
we obtain%
\[
\left(  v_{p},v_{p+1}\right)  \in A\cap\left(  f^{-1}\left(  i\right)  \times
f^{-1}\left(  i\right)  \right)  =A_{i}%
\]
(since (\ref{eq.def.level.levsubdig.a.Aj=3}) (applied to $j=i$) yields
$A_{i}=A\cap\left(  f^{-1}\left(  i\right)  \times f^{-1}\left(  i\right)
\right)  $). But this contradicts $\left(  v_{p},v_{p+1}\right)  \notin A_{i}%
$. This contradiction shows that our assumption was false. Hence, we have
shown that $f\left(  v_{p}\right)  <f\left(  v_{p+1}\right)  $.

Forget that we fixed $p$. We thus have proved that%
\[
f\left(  v_{p}\right)  <f\left(  v_{p+1}\right)  \text{ for each }p\in\left[
n-1\right]  \text{ satisfying }\left(  v_{p},v_{p+1}\right)  \in A.
\]
Since we furthermore know that $f\left(  v_{1}\right)  \leq f\left(
v_{2}\right)  \leq\cdots\leq f\left(  v_{n}\right)  $, we thus conclude that
the $V$-listing $v=\left(  v_{1},v_{2},\ldots,v_{n}\right)  $ is $\left(
f,D\right)  $-friendly (by the definition of \textquotedblleft$\left(
f,D\right)  $-friendly\textquotedblright). Hence, $v$ is an $\left(
f,D\right)  $-friendly $V$-listing. In other words, $\bigotimes\limits_{j\in
f\left(  V\right)  }v^{\left(  j\right)  }$ is an $\left(  f,D\right)
$-friendly $V$-listing (since $v=\bigotimes\limits_{j\in f\left(  V\right)
}v^{\left(  j\right)  }$). This proves Claim 3.]

\begin{statement}
\textit{Claim 4:} Let $v$ be an $\left(  f,D\right)  $-friendly $V$-listing.
Then, $v$ can be written in the form $v=\bigotimes\limits_{j\in f\left(
V\right)  }v^{\left(  j\right)  }$, where $v^{\left(  j\right)  }$ is a hamp
of $\overline{D_{j}}$ for each $j\in f\left(  V\right)  $.
\end{statement}

[\textit{Proof of Claim 4:} Let $j_{1},j_{2},\ldots,j_{q}$ be the elements of
$f\left(  V\right)  $, listed in increasing order (so that $j_{1}<j_{2}%
<\cdots<j_{q}$). Thus, $j_{1},j_{2},\ldots,j_{q}$ are all the levels (with
respect to $f$) that a vertex of $D$ can have, listed in increasing order.

Write the $V$-listing $v$ in the form $v=\left(  v_{1},v_{2},\ldots
,v_{n}\right)  $. As we assumed, this $V$-listing $v$ is $\left(  f,D\right)
$-friendly. In other words, it has the properties that $f\left(  v_{1}\right)
\leq f\left(  v_{2}\right)  \leq\cdots\leq f\left(  v_{n}\right)  $ and that%
\begin{equation}
f\left(  v_{p}\right)  <f\left(  v_{p+1}\right)  \text{ for each }p\in\left[
n-1\right]  \text{ satisfying }\left(  v_{p},v_{p+1}\right)  \in A
\label{pf.lem.friendlies-by-f.c4.1}%
\end{equation}
(by the definition of \textquotedblleft$\left(  f,D\right)  $%
-friendly\textquotedblright).

In particular, we have $f\left(  v_{1}\right)  \leq f\left(  v_{2}\right)
\leq\cdots\leq f\left(  v_{n}\right)  $. In other words, the entries of the
$V$-listing $v$ appear in $v$ in the order of increasing level. In other
words, the $V$-listing $v$ first lists the vertices of the smallest level,
then the vertices of the second-smallest level, and so on. Since $v$ is a
$V$-listing (i.e., contains each element of $V$ exactly once), we can restate
this as follows: The $V$-listing $v$ first lists each vertex of the smallest
level exactly once, then lists each vertex of the second-smallest level
exactly once, and so on. In other words, the $V$-listing $v$ first lists each
vertex of level $j_{1}$ exactly once, then lists each vertex of level $j_{2}$
exactly once, and so on (since $j_{1},j_{2},\ldots,j_{q}$ are all the levels
that a vertex of $D$ can have, listed in increasing order). In other words,
the $V$-listing $v$ first lists each element of $f^{-1}\left(  j_{1}\right)  $
exactly once, then lists each element of $f^{-1}\left(  j_{2}\right)  $
exactly once, and so on. In other words, the $V$-listing $v$ can be written as
a concatenation of an $f^{-1}\left(  j_{1}\right)  $-listing, an
$f^{-1}\left(  j_{2}\right)  $-listing, and so on (in this order).

In other words, $v$ can be written as a concatenation $\bigotimes\limits_{j\in
f\left(  V\right)  }v^{\left(  j\right)  }$, where $v^{\left(  j\right)  }$ is
an $f^{-1}\left(  j\right)  $-listing for each $j\in f\left(  V\right)  $
(since $j_{1},j_{2},\ldots,j_{q}$ are the elements of $f\left(  V\right)  $,
listed in increasing order).

Let us write $v$ in this way. Thus, $v^{\left(  j\right)  }$ is an
$f^{-1}\left(  j\right)  $-listing for each $j\in f\left(  V\right)  $, and we
have $v=\bigotimes\limits_{j\in f\left(  V\right)  }v^{\left(  j\right)  }$.

We shall now show that $v^{\left(  j\right)  }$ is a hamp of $\overline{D_{j}%
}$ for each $j\in f\left(  V\right)  $.

Indeed, let $i\in f\left(  V\right)  $ be arbitrary. We recall that the
digraph $D_{i}$ was defined to be $\left(  f^{-1}\left(  i\right)
,\ A_{i}\right)  $. Hence, its complement $\overline{D_{i}}$ is $\left(
f^{-1}\left(  i\right)  ,\ \left(  f^{-1}\left(  i\right)  \times
f^{-1}\left(  i\right)  \right)  \setminus A_{i}\right)  $ (by the definition
of the complement of a digraph). In particular, the vertices of $\overline
{D_{i}}$ are the elements of $f^{-1}\left(  i\right)  $.

Note that $v^{\left(  i\right)  }$ is an $f^{-1}\left(  i\right)  $-listing
(since $v^{\left(  j\right)  }$ is an $f^{-1}\left(  j\right)  $-listing for
each $j\in f\left(  V\right)  $). Thus, $v^{\left(  i\right)  }$ is a list of
all elements of $f^{-1}\left(  i\right)  $. In particular, all entries of the
list $v^{\left(  i\right)  }$ belong to $f^{-1}\left(  i\right)  $. Note that
the set $f^{-1}\left(  i\right)  $ is nonempty (since $i\in f\left(  V\right)
$), so that any $f^{-1}\left(  i\right)  $-listing must also be nonempty.
Hence, $v^{\left(  i\right)  }$ is nonempty (since $v^{\left(  i\right)  }$ is
an $f^{-1}\left(  i\right)  $-listing). Furthermore, $v^{\left(  i\right)  }$
is a tuple of distinct elements of $f^{-1}\left(  i\right)  $ (since
$v^{\left(  i\right)  }$ is an $f^{-1}\left(  i\right)  $-listing).

Clearly, $v^{\left(  i\right)  }$ is a factor of the concatenation
$\bigotimes\limits_{j\in f\left(  V\right)  }v^{\left(  j\right)  }$. Thus,
$v^{\left(  i\right)  }$ is a contiguous block of the list $\bigotimes
\limits_{j\in f\left(  V\right)  }v^{\left(  j\right)  }$. In other words,
$v^{\left(  i\right)  }$ is a contiguous block of the list $v$ (since
$v=\bigotimes\limits_{j\in f\left(  V\right)  }v^{\left(  j\right)  }$). In
other words, $v^{\left(  i\right)  }=\left(  v_{k},v_{k+1},\ldots,v_{\ell
}\right)  $ for some two elements $k$ and $\ell$ of $\left[  n\right]  $
(since $v=\left(  v_{1},v_{2},\ldots,v_{n}\right)  $). Consider these $k$ and
$\ell$. From $v^{\left(  i\right)  }=\left(  v_{k},v_{k+1},\ldots,v_{\ell
}\right)  $, we obtain%
\begin{align}
\operatorname*{Arcs}\left(  v^{\left(  i\right)  }\right)   &
=\operatorname*{Arcs}\left(  \left(  v_{k},v_{k+1},\ldots,v_{\ell}\right)
\right) \nonumber\\
&  =\left\{  \left(  v_{k},v_{k+1}\right)  ,\ \left(  v_{k+1},v_{k+2}\right)
,\ \ldots,\ \left(  v_{\ell-1},v_{\ell}\right)  \right\} \nonumber\\
&  =\left\{  \left(  v_{p},v_{p+1}\right)  \ \mid\ p\in\left\{  k,k+1,\ldots
,\ell-1\right\}  \right\}  . \label{pf.lem.friendlies-by-f.c4.Arcs}%
\end{align}

Now, let $p\in\left\{  k,k+1,\ldots,\ell-1\right\}  $. We shall show that
$\left(  v_{p},v_{p+1}\right)  \in\left(  f^{-1}\left(  i\right)  \times
f^{-1}\left(  i\right)  \right)  \setminus A_{i}$.

Indeed, $p\in\left\{  k,k+1,\ldots,\ell-1\right\}  \subseteq\left\{
1,2,\ldots,n-1\right\}  $ (since $k\geq1$ and $\underbrace{\ell}_{\leq
n}-1\leq n-1$). In other words, $p\in\left[  n-1\right]  $ (since $\left[
n-1\right]  =\left\{  1,2,\ldots,n-1\right\}  $).

Also, from $p\in\left\{  k,k+1,\ldots,\ell-1\right\}  $, we see that both $p$
and $p+1$ belong to the set $\left\{  k,k+1,\ldots,\ell\right\}  $. Hence,
both $v_{p}$ and $v_{p+1}$ are entries of the list $\left(  v_{k}%
,v_{k+1},\ldots,v_{\ell}\right)  $. In other words, both $v_{p}$ and $v_{p+1}$
are entries of the list $v^{\left(  i\right)  }$ (since $v^{\left(  i\right)
}=\left(  v_{k},v_{k+1},\ldots,v_{\ell}\right)  $). Hence, both $v_{p}$ and
$v_{p+1}$ belong to $f^{-1}\left(  i\right)  $ (since all entries of the list
$v^{\left(  i\right)  }$ belong to $f^{-1}\left(  i\right)  $). Therefore,
$\left(  v_{p},v_{p+1}\right)  \in f^{-1}\left(  i\right)  \times
f^{-1}\left(  i\right)  $.

Now, we shall show that $\left(  v_{p},v_{p+1}\right)  \notin A_{i}$. Indeed,
assume the contrary. Thus,
\begin{align*}
\left(  v_{p},v_{p+1}\right)   &  \in A_{i}=A\cap\left(  f^{-1}\left(
i\right)  \times f^{-1}\left(  i\right)  \right)  \ \ \ \ \ \ \ \ \ \ \left(
\text{by (\ref{eq.def.level.levsubdig.a.Aj=3}), applied to }j=i\right) \\
&  \subseteq A
\end{align*}
and therefore $f\left(  v_{p}\right)  <f\left(  v_{p+1}\right)  $ (by
(\ref{pf.lem.friendlies-by-f.c4.1})). However, $f\left(  v_{p}\right)  =i$
(since $v_{p}$ belongs to $f^{-1}\left(  i\right)  $) and $f\left(
v_{p+1}\right)  =i$ (since $v_{p+1}$ belongs to $f^{-1}\left(  i\right)  $),
so that $f\left(  v_{p}\right)  =i=f\left(  v_{p+1}\right)  $. This
contradicts $f\left(  v_{p}\right)  <f\left(  v_{p+1}\right)  $. This
contradiction shows that our assumption was false. Hence, $\left(
v_{p},v_{p+1}\right)  \notin A_{i}$ is proved.

Combining $\left(  v_{p},v_{p+1}\right)  \in f^{-1}\left(  i\right)  \times
f^{-1}\left(  i\right)  $ with $\left(  v_{p},v_{p+1}\right)  \notin A_{i}$,
we obtain $\left(  v_{p},v_{p+1}\right)  \in\left(  f^{-1}\left(  i\right)
\times f^{-1}\left(  i\right)  \right)  \setminus A_{i}$.

Forget that we fixed $p$. We thus have proved that $\left(  v_{p}%
,v_{p+1}\right)  \in\left(  f^{-1}\left(  i\right)  \times f^{-1}\left(
i\right)  \right)  \setminus A_{i}$ for each $p\in\left\{  k,k+1,\ldots
,\ell-1\right\}  $. In other words,%
\[
\left\{  \left(  v_{p},v_{p+1}\right)  \ \mid\ p\in\left\{  k,k+1,\ldots
,\ell-1\right\}  \right\}  \subseteq\left(  f^{-1}\left(  i\right)  \times
f^{-1}\left(  i\right)  \right)  \setminus A_{i}.
\]
In view of (\ref{pf.lem.friendlies-by-f.c4.Arcs}), we can rewrite this as
\[
\operatorname*{Arcs}\left(  v^{\left(  i\right)  }\right)  \subseteq\left(
f^{-1}\left(  i\right)  \times f^{-1}\left(  i\right)  \right)  \setminus
A_{i}.
\]

Now, we know that $v^{\left(  i\right)  }$ is a nonempty tuple of distinct
elements of $f^{-1}\left(  i\right)  $ and has the property that
$\operatorname*{Arcs}\left(  v^{\left(  i\right)  }\right)  \subseteq\left(
f^{-1}\left(  i\right)  \times f^{-1}\left(  i\right)  \right)  \setminus
A_{i}$. In other words, $v^{\left(  i\right)  }$ is a $\overline{D_{i}}$-path
(by the definition of a \textquotedblleft$\overline{D_{i}}$%
-path\textquotedblright, since the digraph $\overline{D_{i}}$ is $\left(
f^{-1}\left(  i\right)  ,\ \left(  f^{-1}\left(  i\right)  \times
f^{-1}\left(  i\right)  \right)  \setminus A_{i}\right)  $). This
$\overline{D_{i}}$-path $v^{\left(  i\right)  }$ furthermore contains each
element of $f^{-1}\left(  i\right)  $ (since it is an $f^{-1}\left(  i\right)
$-listing). In other words, this $\overline{D_{i}}$-path $v^{\left(  i\right)
}$ contains each vertex of $\overline{D_{i}}$ (since the vertices of
$\overline{D_{i}}$ are the elements of $f^{-1}\left(  i\right)  $).

In other words, $v^{\left(  i\right)  }$ is a hamp of $\overline{D_{i}}$ (by
the definition of a hamp).

Forget that we fixed $i$. We thus have shown that $v^{\left(  i\right)  }$ is
a hamp of $\overline{D_{i}}$ for each $i\in f\left(  V\right)  $. Renaming the
variable $i$ as $j$ in this sentence, we obtain the following: $v^{\left(
j\right)  }$ is a hamp of $\overline{D_{j}}$ for each $j\in f\left(  V\right)
$.

We have thus written $v$ in the form $v=\bigotimes\limits_{j\in f\left(
V\right)  }v^{\left(  j\right)  }$, where $v^{\left(  j\right)  }$ is a hamp
of $\overline{D_{j}}$ for each $j\in f\left(  V\right)  $. This shows that $v$
can be written in this form. Claim 4 is thus proven.]

\begin{statement}
\textit{Claim 5:} Let $\left(  v^{\left(  j\right)  }\right)  _{j\in f\left(
V\right)  }$ be a family of lists, where each $v^{\left(  j\right)  }$ is a
hamp of $\overline{D_{j}}$. Then, this family $\left(  v^{\left(  j\right)
}\right)  _{j\in f\left(  V\right)  }$ can be uniquely reconstructed from the
concatenation $\bigotimes\limits_{j\in f\left(  V\right)  }v^{\left(
j\right)  }$.
\end{statement}

[\textit{Proof of Claim 5:} For each $j\in f\left(  V\right)  $, the list
$v^{\left(  j\right)  }$ is a hamp of $\overline{D_{j}}$ and thus an
$f^{-1}\left(  j\right)  $-listing\footnote{This can be shown just as in the
proof of Claim 3.}. Hence, Claim 2 shows that the family $\left(  v^{\left(
j\right)  }\right)  _{j\in f\left(  V\right)  }$ can be uniquely reconstructed
from the concatenation $\bigotimes\limits_{j\in f\left(  V\right)  }v^{\left(
j\right)  }$. This proves Claim 5.]

\medskip

Now, if $\left(  v^{\left(  j\right)  }\right)  _{j\in f\left(  V\right)  }%
\in\prod_{j\in f\left(  V\right)  }\left\{  \text{hamps of }\overline{D_{j}%
}\right\}  $ is any family (i.e., if $\left(  v^{\left(  j\right)  }\right)
_{j\in f\left(  V\right)  }$ is any family of lists such that each $v^{\left(
j\right)  }$ is a hamp of $\overline{D_{j}}$), then the concatenation
$\bigotimes\limits_{j\in f\left(  V\right)  }v^{\left(  j\right)  }$ is an
$\left(  f,D\right)  $-friendly $V$-listing (by Claim 3). Hence, the map%
\begin{align*}
\prod_{j\in f\left(  V\right)  }\left\{  \text{hamps of }\overline{D_{j}%
}\right\}   &  \rightarrow\left\{  \left(  f,D\right)  \text{-friendly
}V\text{-listings}\right\}  ,\\
\left(  v^{\left(  j\right)  }\right)  _{j\in f\left(  V\right)  }  &
\mapsto\bigotimes\limits_{j\in f\left(  V\right)  }v^{\left(  j\right)  }%
\end{align*}
is well-defined. This map is furthermore injective (since Claim 5 shows that a
family $\left(  v^{\left(  j\right)  }\right)  _{j\in f\left(  V\right)  }%
\in\prod_{j\in f\left(  V\right)  }\left\{  \text{hamps of }\overline{D_{j}%
}\right\}  $ can be uniquely reconstructed from the concatenation
$\bigotimes\limits_{j\in f\left(  V\right)  }v^{\left(  j\right)  }$) and
surjective (since Claim 4 says that any $\left(  f,D\right)  $-friendly
$V$-listing $v$ can be written in the form $v=\bigotimes\limits_{j\in f\left(
V\right)  }v^{\left(  j\right)  }$ for some family $\left(  v^{\left(
j\right)  }\right)  _{j\in f\left(  V\right)  }\in\prod_{j\in f\left(
V\right)  }\left\{  \text{hamps of }\overline{D_{j}}\right\}  $). Therefore,
this map is bijective. The bijection principle thus yields%
\begin{align*}
&  \left\vert \left\{  \left(  f,D\right)  \text{-friendly }V\text{-listings}%
\right\}  \right\vert \\
&  =\left\vert \prod_{j\in f\left(  V\right)  }\left\{  \text{hamps of
}\overline{D_{j}}\right\}  \right\vert =\prod_{j\in f\left(  V\right)
}\underbrace{\left\vert \left\{  \text{hamps of }\overline{D_{j}}\right\}
\right\vert }_{=\left(  \text{\# of hamps of }\overline{D_{j}}\right)  }\\
&  =\prod_{j\in f\left(  V\right)  }\left(  \text{\# of hamps of }%
\overline{D_{j}}\right)  .
\end{align*}
However, Lemma \ref{lem.hamps-by-lin-f} yields%
\[
\sum_{\substack{\sigma\in\mathfrak{S}_{V};\\f\circ\sigma=f}}\ \ \sum
_{\substack{F\subseteq\mathbf{A}_{\sigma}\cap A\\\text{is linear}}}\left(
-1\right)  ^{\left\vert F\right\vert }=\prod_{j\in f\left(  V\right)  }\left(
\text{\# of hamps of }\overline{D_{j}}\right)  .
\]
Comparing these two equalities, we obtain%
\begin{align*}
\sum_{\substack{\sigma\in\mathfrak{S}_{V};\\f\circ\sigma=f}}\ \ \sum
_{\substack{F\subseteq\mathbf{A}_{\sigma}\cap A\\\text{is linear}}}\left(
-1\right)  ^{\left\vert F\right\vert }  &  =\left\vert \left\{  \left(
f,D\right)  \text{-friendly }V\text{-listings}\right\}  \right\vert \\
&  =\left(  \text{\# of }\left(  f,D\right)  \text{-friendly }%
V\text{-listings}\right)  .
\end{align*}
This proves Lemma \ref{lem.friendlies-by-f}.
\end{proof}
\end{verlong}

\subsection{A bit of P\'{o}lya counting}

The following lemma is well-known, e.g., from the theory of P\'{o}lya enumeration:

\begin{lemma}
\label{lem.ptype-as-sum}Let $V$ be a finite set. Let $\sigma\in\mathfrak{S}%
_{V}$ be a permutation of $V$. Then,%
\[
\sum_{\substack{f:V\rightarrow\mathbb{P}\text{;}\\f\circ\sigma=f}%
}\ \ \prod_{v\in V}x_{f\left(  v\right)  }=p_{\operatorname*{type}\sigma}.
\]

\end{lemma}

\begin{vershort}

\begin{proof}
Let $\gamma_{1},\gamma_{2},\ldots,\gamma_{k}$ be the cycles of $\sigma$,
listed with no repetition. For each $i\in\left[  k\right]  $, let $V_{i}$ be
the set of entries of the cycle $\gamma_{i}$. Thus, $V=V_{1}\sqcup V_{2}%
\sqcup\cdots\sqcup V_{k}$. For each $i\in\left[  k\right]  $, the set $V_{i}$
is the set of entries of a cycle of $\sigma$ (namely, of $\gamma_{i}$), and
thus can be written as $\left\{  \sigma^{j}\left(  v_{i}\right)  \ \mid
\ j\in\mathbb{N}\right\}  $ for some $v_{i}\in V_{i}$.

Hence, a map $f:V\rightarrow\mathbb{P}$ satisfies $f\circ\sigma=f$ if and only
if $f$ is constant on each of the $k$ sets $V_{1},V_{2},\ldots,V_{k}$. Hence,
in order to construct a map $f:V\rightarrow\mathbb{P}$ that satisfies
$f\circ\sigma=f$, we only need to choose the values $a_{1},a_{2},\ldots,a_{k}$
that it takes on these $k$ sets (i.e., for each $i\in\left[  k\right]  $, we
need to choose the value $a_{i}$ that $f$ takes on all elements of $V_{i}$).

Let us be more precise: For each $k$-tuple $\left(  a_{1},a_{2},\ldots
,a_{k}\right)  \in\mathbb{P}^{k}$, there is a unique map $\Gamma\left(
a_{1},a_{2},\ldots,a_{k}\right)  :V\rightarrow\mathbb{P}$ that sends each
element of $V_{1}$ to $a_{1}$, each element of $V_{2}$ to $a_{2}$, and so on
(since $V=V_{1}\sqcup V_{2}\sqcup\cdots\sqcup V_{k}$). The latter map
$\Gamma\left(  a_{1},a_{2},\ldots,a_{k}\right)  $ is a map $f:V\rightarrow
\mathbb{P}$ that satisfies $f\circ\sigma=f$ (by the preceding paragraph, since
it is constant on each of the $k$ sets $V_{1},V_{2},\ldots,V_{k}$). Thus, we
obtain a map%
\begin{align*}
\Gamma:\mathbb{P}^{k}  &  \rightarrow\left\{  f:V\rightarrow\mathbb{P}%
\ \mid\ f\circ\sigma=f\right\}  ,\\
\left(  a_{1},a_{2},\ldots,a_{k}\right)   &  \mapsto\Gamma\left(  a_{1}%
,a_{2},\ldots,a_{k}\right)  .
\end{align*}
This map $\Gamma$ is easily seen to be injective (since $V_{1},V_{2}%
,\ldots,V_{k}$ are nonempty) and surjective (again by the previous paragraph).
Hence, it is bijective. Thus, substituting $\Gamma\left(  a_{1},a_{2}%
,\ldots,a_{k}\right)  $ for $f$ in the sum $\sum_{\substack{f:V\rightarrow
\mathbb{P}\text{;}\\f\circ\sigma=f}}\ \ \prod_{v\in V}x_{f\left(  v\right)  }%
$, we obtain%
\begin{align}
\sum_{\substack{f:V\rightarrow\mathbb{P}\text{;}\\f\circ\sigma=f}%
}\ \ \prod_{v\in V}x_{f\left(  v\right)  }  &  =\sum_{\left(  a_{1}%
,a_{2},\ldots,a_{k}\right)  \in\mathbb{P}^{k}}\ \ \underbrace{\prod_{v\in V}%
}_{\substack{=\prod_{i=1}^{k}\ \ \prod_{v\in V_{i}}\\\text{(since }%
V=V_{1}\sqcup V_{2}\sqcup\cdots\sqcup V_{k}\text{)}}}x_{\left(  \Gamma\left(
a_{1},a_{2},\ldots,a_{k}\right)  \right)  \left(  v\right)  }\nonumber\\
&  =\sum_{\left(  a_{1},a_{2},\ldots,a_{k}\right)  \in\mathbb{P}^{k}}%
\ \ \prod_{i=1}^{k}\ \ \prod_{v\in V_{i}}\underbrace{x_{\left(  \Gamma\left(
a_{1},a_{2},\ldots,a_{k}\right)  \right)  \left(  v\right)  }}%
_{\substack{_{\substack{=x_{a_{i}}}}\\\text{(since the map }\Gamma\left(
a_{1},a_{2},\ldots,a_{k}\right)  \\\text{sends each element of }V_{i}\text{ to
}a_{i}\text{)}}}\nonumber\\
&  =\sum_{\left(  a_{1},a_{2},\ldots,a_{k}\right)  \in\mathbb{P}^{k}}%
\ \ \prod_{i=1}^{k}\ \ \underbrace{\prod_{v\in V_{i}}x_{a_{i}}}_{=x_{a_{i}%
}^{\left\vert V_{i}\right\vert }}=\sum_{\left(  a_{1},a_{2},\ldots
,a_{k}\right)  \in\mathbb{P}^{k}}\ \ \prod_{i=1}^{k}x_{a_{i}}^{\left\vert
V_{i}\right\vert }\nonumber\\
&  =\prod_{i=1}^{k}\ \ \underbrace{\sum_{a\in\mathbb{P}}x_{a}^{\left\vert
V_{i}\right\vert }}_{\substack{=x_{1}^{\left\vert V_{i}\right\vert }%
+x_{2}^{\left\vert V_{i}\right\vert }+x_{3}^{\left\vert V_{i}\right\vert
}+\cdots\\=p_{\left\vert V_{i}\right\vert }\\\text{(by the definition of
}p_{\left\vert V_{i}\right\vert }\text{)}}}\ \ \ \ \ \ \ \ \ \ \left(
\text{by the product rule}\right) \nonumber\\
&  =\prod_{i=1}^{k}p_{\left\vert V_{i}\right\vert }.
\label{pf.lem.ptype-as-sum.4}%
\end{align}

However, the permutation $\sigma$ has cycles $\gamma_{1},\gamma_{2}%
,\ldots,\gamma_{k}$, and their respective sets of entries are $V_{1}%
,V_{2},\ldots,V_{k}$. Thus, the lengths of the cycles of $\sigma$ are
$\left\vert V_{1}\right\vert ,\left\vert V_{2}\right\vert ,\ldots,\left\vert
V_{k}\right\vert $. But the entries of the partition $\operatorname*{type}%
\sigma$ are precisely the lengths of the cycles of $\sigma$ (by the definition
of $\operatorname*{type}\sigma$), and thus must be $\left\vert V_{1}%
\right\vert ,\left\vert V_{2}\right\vert ,\ldots,\left\vert V_{k}\right\vert $
in some order (by the preceding sentence). Therefore, the partition
$\operatorname*{type}\sigma$ can be obtained from the $k$-tuple $\left(
\left\vert V_{1}\right\vert ,\left\vert V_{2}\right\vert ,\ldots,\left\vert
V_{k}\right\vert \right)  $ by sorting the entries in weakly decreasing order.
Hence,%
\[
p_{\operatorname*{type}\sigma}=\prod_{i=1}^{k}p_{\left\vert V_{i}\right\vert
}.
\]
Comparing this with (\ref{pf.lem.ptype-as-sum.4}), we obtain
\[
\sum_{\substack{f:V\rightarrow\mathbb{P}\text{;}\\f\circ\sigma=f}%
}\ \ \prod_{v\in V}x_{f\left(  v\right)  }=p_{\operatorname*{type}\sigma}.
\]
This proves Lemma \ref{lem.ptype-as-sum}.
\end{proof}
\end{vershort}

\begin{verlong}

\begin{proof}
Let $\gamma_{1},\gamma_{2},\ldots,\gamma_{k}$ be the cycles of $\sigma$,
listed with no repetition\footnote{Keep in mind that a cycle is a
rotation-equivalence class. Thus, \textquotedblleft listed with no
repetition\textquotedblright\ means that no two of $\gamma_{1},\gamma
_{2},\ldots,\gamma_{k}$ are the same rotation-equivalence class. For example,
if $\gamma_{1}$ is $\left(  1,2\right)  _{\sim}$, then $\gamma_{2}$ cannot be
$\left(  2,1\right)  _{\sim}$.}. For each $i\in\left[  k\right]  $, let
$V_{i}$ be the set of entries of the cycle $\gamma_{i}$. We shall now collect
some basic properties of these cycles $\gamma_{i}$ and the corresponding sets
$V_{i}$:

\begin{statement}
\textit{Claim 1:} Let $v\in V$. Then, there exists a unique $i\in\left[
k\right]  $ such that $v\in V_{i}$.
\end{statement}

[\textit{Proof of Claim 1:} We know that $\sigma$ is a permutation of $V$.
Hence, each element of $V$ belongs to exactly one cycle of $\sigma$. In
particular, $v$ belongs to exactly one cycle of $\sigma$. In other words,
there exists exactly one cycle of $\sigma$ such that $v$ is an entry of this
cycle. In other words, there exists a unique $i\in\left[  k\right]  $ such
that $v$ is an entry of $\gamma_{i}$ (since $\gamma_{1},\gamma_{2}%
,\ldots,\gamma_{k}$ are the cycles of $\sigma$, listed with no repetition). In
other words, there exists a unique $i\in\left[  k\right]  $ such that $v\in
V_{i}$ (since the statement \textquotedblleft$v\in V_{i}$\textquotedblright%
\ is equivalent to the statement \textquotedblleft$v$ is an entry of
$\gamma_{i}$\textquotedblright\ \ \ \ \footnote{because $V_{i}$ is the set of
entries of $\gamma_{i}$}). This proves Claim 1.] \medskip

\begin{statement}
\textit{Claim 2:} Let $i\in\left[  k\right]  $. Then:

\textbf{(a)} We have $\sigma\left(  V_{i}\right)  \subseteq V_{i}$.

\textbf{(b)} There exists an element $v_{i}\in V_{i}$ such that%
\[
V_{i}=\left\{  \sigma^{j}\left(  v_{i}\right)  \ \mid\ j\in\mathbb{N}\right\}
.
\]

\end{statement}

[\textit{Proof of Claim 2:} It is well-known (from Definition
\ref{def.perm.cycs} \textbf{(a)}) that each cycle of $\sigma$ has the form
\[
\left(  \sigma^{0}\left(  w\right)  ,\ \sigma^{1}\left(  w\right)
,\ \sigma^{2}\left(  w\right)  ,\ \ldots,\ \sigma^{q-1}\left(  w\right)
\right)  ,
\]
where $w$ is an element of $V$ and where $q$ is the smallest positive integer
satisfying $\sigma^{q}\left(  w\right)  =w$. Thus, in particular, $\gamma_{i}$
has this form (since $\gamma_{i}$ is a cycle of $\sigma$). In other words,
there exists an element $w$ of $V$ such that%
\[
\gamma_{i}=\left(  \sigma^{0}\left(  w\right)  ,\ \sigma^{1}\left(  w\right)
,\ \sigma^{2}\left(  w\right)  ,\ \ldots,\ \sigma^{q-1}\left(  w\right)
\right)  ,
\]
where $q$ is the smallest positive integer satisfying $\sigma^{q}\left(
w\right)  =w$. Let us consider this $w$ and this $q$.

We have $\gamma_{i}=\left(  \sigma^{0}\left(  w\right)  ,\ \sigma^{1}\left(
w\right)  ,\ \sigma^{2}\left(  w\right)  ,\ \ldots,\ \sigma^{q-1}\left(
w\right)  \right)  $. Thus, the entries of the cycle $\gamma_{i}$ are
$\sigma^{0}\left(  w\right)  ,\ \sigma^{1}\left(  w\right)  ,\ \sigma
^{2}\left(  w\right)  ,\ \ldots,\ \sigma^{q-1}\left(  w\right)  $.

The set $V_{i}$ was defined as the set of entries of the cycle $\gamma_{i}$.
Thus,%
\[
V_{i}=\left\{  \sigma^{0}\left(  w\right)  ,\ \sigma^{1}\left(  w\right)
,\ \sigma^{2}\left(  w\right)  ,\ \ldots,\ \sigma^{q-1}\left(  w\right)
\right\}
\]
(since the entries of the cycle $\gamma_{i}$ are $\sigma^{0}\left(  w\right)
,\ \sigma^{1}\left(  w\right)  ,\ \sigma^{2}\left(  w\right)  ,\ \ldots
,\ \sigma^{q-1}\left(  w\right)  $). Since $q$ is a positive integer, we have
$q-1\in\mathbb{N}$ and thus $0\in\left\{  0,1,\ldots,q-1\right\}  $. Hence,%
\[
\sigma^{0}\left(  w\right)  \in\left\{  \sigma^{0}\left(  w\right)
,\ \sigma^{1}\left(  w\right)  ,\ \sigma^{2}\left(  w\right)  ,\ \ldots
,\ \sigma^{q-1}\left(  w\right)  \right\}  =V_{i}.
\]
In other words, $w\in V_{i}$ (since $\underbrace{\sigma^{0}}%
_{=\operatorname*{id}}\left(  w\right)  =\operatorname*{id}\left(  w\right)
=w$).

\textbf{(a)} Let $g\in V_{i}$. We shall prove that $\sigma\left(  g\right)
\in V_{i}$.

Indeed, $g\in V_{i}=\left\{  \sigma^{0}\left(  w\right)  ,\ \sigma^{1}\left(
w\right)  ,\ \sigma^{2}\left(  w\right)  ,\ \ldots,\ \sigma^{q-1}\left(
w\right)  \right\}  $. In other words, $g=\sigma^{r}\left(  w\right)  $ for
some $r\in\left\{  0,1,\ldots,q-1\right\}  $. Consider this $r$. Applying the
map $\sigma$ to both sides of $g=\sigma^{r}\left(  w\right)  $, we obtain
$\sigma\left(  g\right)  =\sigma\left(  \sigma^{r}\left(  w\right)  \right)
=\underbrace{\left(  \sigma\circ\sigma^{r}\right)  }_{=\sigma^{r+1}}\left(
w\right)  =\sigma^{r+1}\left(  w\right)  $.

We are in one of the following two cases:

\textit{Case 1:} We have $r\neq q-1$.

\textit{Case 2:} We have $r=q-1$.

Let us first consider Case 1. In this case, we have $r\neq q-1$. Combining
this with $r\in\left\{  0,1,\ldots,q-1\right\}  $, we obtain $r\in\left\{
0,1,\ldots,q-1\right\}  \setminus\left\{  q-1\right\}  =\left\{
0,1,\ldots,q-2\right\}  $. Hence, $r+1\in\left\{  1,2,\ldots,q-1\right\}
\subseteq\left\{  0,1,\ldots,q-1\right\}  $. Thus,
\[
\sigma^{r+1}\left(  w\right)  \in\left\{  \sigma^{0}\left(  w\right)
,\ \sigma^{1}\left(  w\right)  ,\ \sigma^{2}\left(  w\right)  ,\ \ldots
,\ \sigma^{q-1}\left(  w\right)  \right\}  =V_{i}.
\]

Now, $\sigma\left(  g\right)  =\sigma^{r+1}\left(  w\right)  \in V_{i}$.
Hence, $\sigma\left(  g\right)  \in V_{i}$ is proved in Case 1.

Let us next consider Case 2. In this case, we have $r=q-1$. Hence, $r+1=q$.
Now, $\sigma\left(  g\right)  =\sigma^{r+1}\left(  w\right)  =\sigma
^{q}\left(  w\right)  $ (since $r+1=q$), so that $\sigma\left(  g\right)
=\sigma^{q}\left(  w\right)  =w\in V_{i}$. Hence, $\sigma\left(  g\right)  \in
V_{i}$ is proved in Case 2.

We have now proved $\sigma\left(  g\right)  \in V_{i}$ in both Cases 1 and 2.
Hence, $\sigma\left(  g\right)  \in V_{i}$ always holds.

Forget that we fixed $g$. We thus have shown that $\sigma\left(  g\right)  \in
V_{i}$ for each $g\in V_{i}$. In other words, $\sigma\left(  V_{i}\right)
\subseteq V_{i}$. This proves Claim 2 \textbf{(a)}.

\textbf{(b)} We shall first show that $\sigma^{j}\left(  w\right)  \in V_{i}$
for each $j\in\mathbb{N}$.

Indeed, we shall prove this by induction on $j$:

\textit{Base case:} Our claim $\sigma^{j}\left(  w\right)  \in V_{i}$ holds
for $j=0$, since $\sigma^{0}\left(  w\right)  \in V_{i}$.

\textit{Induction step:} Let $s\in\mathbb{N}$. Assume (as the induction
hypothesis) that $\sigma^{j}\left(  w\right)  \in V_{i}$ holds for $j=s$. We
must prove that $\sigma^{j}\left(  w\right)  \in V_{i}$ holds for $j=s+1$.

We have assumed that $\sigma^{j}\left(  w\right)  \in V_{i}$ holds for $j=s$.
In other words, $\sigma^{s}\left(  w\right)  \in V_{i}$. Now,
\[
\underbrace{\sigma^{s+1}}_{=\sigma\circ\sigma^{s}}\left(  w\right)  =\left(
\sigma\circ\sigma^{s}\right)  \left(  w\right)  =\sigma\left(
\underbrace{\sigma^{s}\left(  w\right)  }_{\in V_{i}}\right)  \in\sigma\left(
V_{i}\right)  \subseteq V_{i}%
\]
(by Claim 2 \textbf{(a)}). In other words, $\sigma^{j}\left(  w\right)  \in
V_{i}$ holds for $j=s+1$. This completes the induction step.

Thus, we have proved that $\sigma^{j}\left(  w\right)  \in V_{i}$ for each
$j\in\mathbb{N}$. In other words,%
\[
\left\{  \sigma^{j}\left(  w\right)  \ \mid\ j\in\mathbb{N}\right\}  \subseteq
V_{i}.
\]
Combining this with%
\begin{align*}
V_{i}  &  =\left\{  \sigma^{0}\left(  w\right)  ,\ \sigma^{1}\left(  w\right)
,\ \sigma^{2}\left(  w\right)  ,\ \ldots,\ \sigma^{q-1}\left(  w\right)
\right\} \\
&  =\left\{  \sigma^{j}\left(  w\right)  \ \mid\ j\in\left\{  0,1,\ldots
,q-1\right\}  \right\} \\
&  \subseteq\left\{  \sigma^{j}\left(  w\right)  \ \mid\ j\in\mathbb{N}%
\right\}  \ \ \ \ \ \ \ \ \ \ \left(  \text{since }\left\{  0,1,\ldots
,q-1\right\}  \subseteq\mathbb{N}\right)  ,
\end{align*}
we obtain $V_{i}=\left\{  \sigma^{j}\left(  w\right)  \ \mid\ j\in
\mathbb{N}\right\}  $. Since we know that $w\in V_{i}$, we can thus conclude
that there exists an element $v_{i}\in V_{i}$ such that $V_{i}=\left\{
\sigma^{j}\left(  v_{i}\right)  \ \mid\ j\in\mathbb{N}\right\}  $ (namely,
$v_{i}=w$). This proves Claim 2 \textbf{(b)}.] \medskip

\begin{statement}
\textit{Claim 3:} The entries of the partition $\operatorname*{type}\sigma$
are the numbers $\left\vert V_{1}\right\vert ,\left\vert V_{2}\right\vert
,\ldots,\left\vert V_{k}\right\vert $ in some order.
\end{statement}

[\textit{Proof of Claim 3:} For each $i\in\left[  k\right]  $, we have
\begin{equation}
\left(  \text{the length of the cycle }\gamma_{i}\right)  =\left\vert
V_{i}\right\vert \label{pf.lem.ptype-as-sum.c3.pf.1}%
\end{equation}
\footnote{\textit{Proof.} Let $i\in\left[  k\right]  $. Then, $V_{i}$ is the
set of entries of $\gamma_{i}$ (by the definition of $V_{i}$). Hence, the
elements of $V_{i}$ are the entries of $\gamma_{i}$.
\par
The cycle $\gamma_{i}$ of $\sigma$ is a rotation-equivalence class of tuples
of distinct elements (since any cycle of any permutation is such a class).
Hence, its entries are distinct.
\par
Now,%
\begin{align*}
\left\vert V_{i}\right\vert  &  =\left(  \text{the number of distinct elements
of }V_{i}\right) \\
&  =\left(  \text{the number of distinct entries of }\gamma_{i}\right) \\
&  \ \ \ \ \ \ \ \ \ \ \ \ \ \ \ \ \ \ \ \ \left(  \text{since the elements of
}V_{i}\text{ are the entries of }\gamma_{i}\right) \\
&  =\left(  \text{the number of entries of }\gamma_{i}\right) \\
&  \ \ \ \ \ \ \ \ \ \ \ \ \ \ \ \ \ \ \ \ \left(  \text{since the entries of
}\gamma_{i}\text{ are distinct}\right) \\
&  =\left(  \text{the length of the cycle }\gamma_{i}\right)  .
\end{align*}
Therefore, $\left(  \text{the length of the cycle }\gamma_{i}\right)
=\left\vert V_{i}\right\vert $.}. However, we defined $\operatorname*{type}%
\sigma$ to be the partition whose entries are the lengths of the cycles of
$\sigma$. Thus,%
\begin{align*}
&  \left(  \text{the entries of }\operatorname*{type}\sigma\right) \\
&  =\left(  \text{the lengths of the cycles of }\sigma\right) \\
&  \ \ \ \ \ \ \ \ \ \ \ \ \ \ \ \ \ \ \ \ \left(  \text{where we disregard
the order of the entries}\right) \\
&  =\left(  \text{the lengths of the cycles }\gamma_{1},\gamma_{2}%
,\ldots,\gamma_{k}\right) \\
&  \ \ \ \ \ \ \ \ \ \ \ \ \ \ \ \ \ \ \ \ \left(
\begin{array}
[c]{c}%
\text{since the cycles of }\sigma\text{ are }\gamma_{1},\gamma_{2}%
,\ldots,\gamma_{k}\\
\text{(listed without repetition)}%
\end{array}
\right) \\
&  =\left(  \text{the numbers }\left\vert V_{1}\right\vert ,\left\vert
V_{2}\right\vert ,\ldots,\left\vert V_{k}\right\vert \right)
\ \ \ \ \ \ \ \ \ \ \left(  \text{by (\ref{pf.lem.ptype-as-sum.c3.pf.1}%
)}\right)  .
\end{align*}
In other words, the entries of $\operatorname*{type}\sigma$ are the numbers
$\left\vert V_{1}\right\vert ,\left\vert V_{2}\right\vert ,\ldots,\left\vert
V_{k}\right\vert $ in some order. This proves Claim 3.] \medskip

Now, we introduce two crucial pieces of notation:

\begin{itemize}
\item For each $v\in V$, we let $\operatorname*{ind}v$ denote the unique
$i\in\left[  k\right]  $ such that $v\in V_{i}$. (This is well-defined, since
Claim 1 shows that there indeed exists a unique $i\in\left[  k\right]  $ such
that $v\in V_{i}$.)

\item For any $k$-tuple $\left(  a_{1},a_{2},\ldots,a_{k}\right)
\in\mathbb{P}^{k}$, we define a map%
\[
g\left[  a_{1},a_{2},\ldots,a_{k}\right]  :V\rightarrow\mathbb{P}%
\]
by setting
\[
\left(  \left(  g\left[  a_{1},a_{2},\ldots,a_{k}\right]  \right)  \left(
v\right)  :=a_{\operatorname*{ind}v}\ \ \ \ \ \ \ \ \ \ \text{for each }v\in
V\right)  .
\]

\end{itemize}

We observe the following:

\begin{statement}
\textit{Claim 4:} Let $j\in\left[  k\right]  $, and let $v\in V_{j}$. Then,
$\operatorname*{ind}v=j$.
\end{statement}

[\textit{Proof of Claim 4:} Recall that $\operatorname*{ind}v$ is defined as
the unique $i\in\left[  k\right]  $ such that $v\in V_{i}$. Hence, if some
$i\in\left[  k\right]  $ satisfies $v\in V_{i}$, then $\operatorname*{ind}%
v=i$. Applying this to $i=j$, we obtain $\operatorname*{ind}v=j$ (since $v\in
V_{j}$). This proves Claim 4.]

\begin{statement}
\textit{Claim 5:} For any $k$-tuple $\left(  a_{1},a_{2},\ldots,a_{k}\right)
\in\mathbb{P}^{k}$, we have $g\left[  a_{1},a_{2},\ldots,a_{k}\right]
\in\left\{  f:V\rightarrow\mathbb{P}\ \mid\ f\circ\sigma=f\right\}  $.
\end{statement}

[\textit{Proof of Claim 5:} Let $\left(  a_{1},a_{2},\ldots,a_{k}\right)
\in\mathbb{P}^{k}$ be a $k$-tuple. Then, $g\left[  a_{1},a_{2},\ldots
,a_{k}\right]  $ is a map from $V$ to $\mathbb{P}$. We shall now show that
$\left(  g\left[  a_{1},a_{2},\ldots,a_{k}\right]  \right)  \circ
\sigma=g\left[  a_{1},a_{2},\ldots,a_{k}\right]  $.

Indeed, let $v\in V$ be arbitrary. Recall that $\operatorname*{ind}v$ is
defined as the unique $i\in\left[  k\right]  $ such that $v\in V_{i}$. Hence,
$\operatorname*{ind}v\in\left[  k\right]  $ and $v\in V_{\operatorname*{ind}%
v}$. Thus, Claim 2 \textbf{(a)} (applied to $i=\operatorname*{ind}v$) yields
$\sigma\left(  V_{\operatorname*{ind}v}\right)  \subseteq
V_{\operatorname*{ind}v}$. From $v\in V_{\operatorname*{ind}v}$, we obtain
$\sigma\left(  v\right)  \in\sigma\left(  V_{\operatorname*{ind}v}\right)
\subseteq V_{\operatorname*{ind}v}$.

Therefore, Claim 4 (applied to $\operatorname*{ind}v$ and $\sigma\left(
v\right)  $ instead of $j$ and $v$) yields $\operatorname*{ind}\left(
\sigma\left(  v\right)  \right)  =\operatorname*{ind}v$ (since $\sigma\left(
v\right)  \in V_{\operatorname*{ind}v}$).

The definition of $g\left[  a_{1},a_{2},\ldots,a_{k}\right]  $ yields%
\begin{align*}
\left(  g\left[  a_{1},a_{2},\ldots,a_{k}\right]  \right)  \left(  v\right)
&  =a_{\operatorname*{ind}v}\ \ \ \ \ \ \ \ \ \ \text{and}\\
\left(  g\left[  a_{1},a_{2},\ldots,a_{k}\right]  \right)  \left(
\sigma\left(  v\right)  \right)   &  =a_{\operatorname*{ind}\left(
\sigma\left(  v\right)  \right)  }=a_{\operatorname*{ind}v}%
\ \ \ \ \ \ \ \ \ \ \left(  \text{since }\operatorname*{ind}\left(
\sigma\left(  v\right)  \right)  =\operatorname*{ind}v\right)  .
\end{align*}
Comparing these two equalities, we obtain%
\[
\left(  g\left[  a_{1},a_{2},\ldots,a_{k}\right]  \right)  \left(  v\right)
=\left(  g\left[  a_{1},a_{2},\ldots,a_{k}\right]  \right)  \left(
\sigma\left(  v\right)  \right)  =\left(  \left(  g\left[  a_{1},a_{2}%
,\ldots,a_{k}\right]  \right)  \circ\sigma\right)  \left(  v\right)  .
\]

Forget that we fixed $v$. We thus have proved that \newline$\left(  g\left[
a_{1},a_{2},\ldots,a_{k}\right]  \right)  \left(  v\right)  =\left(  \left(
g\left[  a_{1},a_{2},\ldots,a_{k}\right]  \right)  \circ\sigma\right)  \left(
v\right)  $ for each $v\in V$. In other words, $g\left[  a_{1},a_{2}%
,\ldots,a_{k}\right]  =\left(  g\left[  a_{1},a_{2},\ldots,a_{k}\right]
\right)  \circ\sigma$. In other words, $\left(  g\left[  a_{1},a_{2}%
,\ldots,a_{k}\right]  \right)  \circ\sigma=g\left[  a_{1},a_{2},\ldots
,a_{k}\right]  $.

Thus, $g\left[  a_{1},a_{2},\ldots,a_{k}\right]  $ is a map $f:V\rightarrow
\mathbb{P}$ satisfying $f\circ\sigma=f$. In other words, $g\left[  a_{1}%
,a_{2},\ldots,a_{k}\right]  \in\left\{  f:V\rightarrow\mathbb{P}\ \mid
\ f\circ\sigma=f\right\}  $. This proves Claim 5.] \medskip

Claim 5 allows us to define a map%
\begin{align*}
\Gamma:\mathbb{P}^{k}  &  \rightarrow\left\{  f:V\rightarrow\mathbb{P}%
\ \mid\ f\circ\sigma=f\right\}  ,\\
\left(  a_{1},a_{2},\ldots,a_{k}\right)   &  \mapsto g\left[  a_{1}%
,a_{2},\ldots,a_{k}\right]  .
\end{align*}
Consider this map $\Gamma$. Next, we claim:

\begin{statement}
\textit{Claim 6:} The map $\Gamma$ is injective.
\end{statement}

[\textit{Proof of Claim 6:} Let $\left(  a_{1},a_{2},\ldots,a_{k}\right)  $
and $\left(  b_{1},b_{2},\ldots,b_{k}\right)  $ be two elements of
$\mathbb{P}^{k}$ satisfying $\Gamma\left(  a_{1},a_{2},\ldots,a_{k}\right)
=\Gamma\left(  b_{1},b_{2},\ldots,b_{k}\right)  $. We shall show that $\left(
a_{1},a_{2},\ldots,a_{k}\right)  =\left(  b_{1},b_{2},\ldots,b_{k}\right)  $.

Indeed, let us fix $i\in\left[  k\right]  $. Claim 2 \textbf{(b)} shows that
there exists an element $v_{i}\in V_{i}$ such that%
\[
V_{i}=\left\{  \sigma^{j}\left(  v_{i}\right)  \ \mid\ j\in\mathbb{N}\right\}
.
\]
Consider this $v_{i}$. Claim 4 (applied to $j=i$ and $v=v_{i}$) yields
$\operatorname*{ind}\left(  v_{i}\right)  =i$ (since $v_{i}\in V_{i}$). The
definition of $g\left[  a_{1},a_{2},\ldots,a_{k}\right]  $ yields%
\[
\left(  g\left[  a_{1},a_{2},\ldots,a_{k}\right]  \right)  \left(
v_{i}\right)  =a_{\operatorname*{ind}\left(  v_{i}\right)  }=a_{i}%
\ \ \ \ \ \ \ \ \ \ \left(  \text{since }\operatorname*{ind}\left(
v_{i}\right)  =i\right)  .
\]
The definition of $\Gamma$ yields $\Gamma\left(  a_{1},a_{2},\ldots
,a_{k}\right)  =g\left[  a_{1},a_{2},\ldots,a_{k}\right]  $. Thus,%
\begin{equation}
\underbrace{\left(  \Gamma\left(  a_{1},a_{2},\ldots,a_{k}\right)  \right)
}_{=g\left[  a_{1},a_{2},\ldots,a_{k}\right]  }\left(  v_{i}\right)  =\left(
g\left[  a_{1},a_{2},\ldots,a_{k}\right]  \right)  \left(  v_{i}\right)
=a_{i}. \label{pf.lem.ptype-as-sum.c6.pf.1}%
\end{equation}
The same argument (applied to $\left(  b_{1},b_{2},\ldots,b_{k}\right)  $
instead of $\left(  a_{1},a_{2},\ldots,a_{k}\right)  $) yields%
\begin{equation}
\left(  \Gamma\left(  b_{1},b_{2},\ldots,b_{k}\right)  \right)  \left(
v_{i}\right)  =b_{i}. \label{pf.lem.ptype-as-sum.c6.pf.2}%
\end{equation}

However, (\ref{pf.lem.ptype-as-sum.c6.pf.1}) yields%
\[
a_{i}=\underbrace{\left(  \Gamma\left(  a_{1},a_{2},\ldots,a_{k}\right)
\right)  }_{=\Gamma\left(  b_{1},b_{2},\ldots,b_{k}\right)  }\left(
v_{i}\right)  =\left(  \Gamma\left(  b_{1},b_{2},\ldots,b_{k}\right)  \right)
\left(  v_{i}\right)  =b_{i}\ \ \ \ \ \ \ \ \ \ \left(  \text{by
(\ref{pf.lem.ptype-as-sum.c6.pf.2})}\right)  .
\]

Forget that we fixed $i$. We thus have proved that $a_{i}=b_{i}$ for each
$i\in\left[  k\right]  $. In other words, $\left(  a_{1},a_{2},\ldots
,a_{k}\right)  =\left(  b_{1},b_{2},\ldots,b_{k}\right)  $.

Forget that we fixed $\left(  a_{1},a_{2},\ldots,a_{k}\right)  $ and $\left(
b_{1},b_{2},\ldots,b_{k}\right)  $. We thus have shown that if $\left(
a_{1},a_{2},\ldots,a_{k}\right)  $ and $\left(  b_{1},b_{2},\ldots
,b_{k}\right)  $ are two elements of $\mathbb{P}^{k}$ satisfying
$\Gamma\left(  a_{1},a_{2},\ldots,a_{k}\right)  =\Gamma\left(  b_{1}%
,b_{2},\ldots,b_{k}\right)  $, then $\left(  a_{1},a_{2},\ldots,a_{k}\right)
=\left(  b_{1},b_{2},\ldots,b_{k}\right)  $. In other words, the map $\Gamma$
is injective. This proves Claim 6.]

\begin{statement}
\textit{Claim 7:} The map $\Gamma$ is surjective.
\end{statement}

[\textit{Proof of Claim 7:} Let $h\in\left\{  f:V\rightarrow\mathbb{P}%
\ \mid\ f\circ\sigma=f\right\}  $. We shall construct a $k$-tuple $\left(
a_{1},a_{2},\ldots,a_{k}\right)  \in\mathbb{P}^{k}$ such that $h=\Gamma\left(
a_{1},a_{2},\ldots,a_{k}\right)  $.

Indeed, we have $h\in\left\{  f:V\rightarrow\mathbb{P}\ \mid\ f\circ
\sigma=f\right\}  $. In other words, $h$ is a map $f:V\rightarrow\mathbb{P}$
satisfying $f\circ\sigma=f$. In other words, $h$ is a map from $V$ to
$\mathbb{P}$ and satisfies $h\circ\sigma=h$. Hence,
\begin{equation}
h\circ\sigma^{j}=h \label{pf.lem.ptype-as-sum.c7.pf.1}%
\end{equation}
for each $j\in\mathbb{N}$\ \ \ \ \footnote{\textit{Proof of
(\ref{pf.lem.ptype-as-sum.c7.pf.1}):} We prove
(\ref{pf.lem.ptype-as-sum.c7.pf.1}) by induction on $j$:
\par
\textit{Base case:} We have $h\circ\underbrace{\sigma^{0}}%
_{=\operatorname*{id}}=h\circ\operatorname*{id}=h$. In other words,
(\ref{pf.lem.ptype-as-sum.c7.pf.1}) holds for $j=0$.
\par
\textit{Induction step:} Let $i\in\mathbb{N}$. Assume (as the induction
hypothesis) that (\ref{pf.lem.ptype-as-sum.c7.pf.1}) holds for $j=i$. We must
prove that (\ref{pf.lem.ptype-as-sum.c7.pf.1}) holds for $j=i+1$.
\par
We have assumed that (\ref{pf.lem.ptype-as-sum.c7.pf.1}) holds for $j=i$. In
other words, $h\circ\sigma^{i}=h$. Now,%
\[
h\circ\underbrace{\sigma^{i+1}}_{=\sigma\circ\sigma^{i}}=\underbrace{h\circ
\sigma}_{=h}\circ\sigma^{i}=h\circ\sigma^{i}=h.
\]
In other words, (\ref{pf.lem.ptype-as-sum.c7.pf.1}) holds for $j=i+1$. This
completes the induction step. Thus, we have proved
(\ref{pf.lem.ptype-as-sum.c7.pf.1}) by induction.}.

For each $i\in\left[  k\right]  $, there exists an element $v_{i}\in V_{i}$
such that%
\begin{equation}
V_{i}=\left\{  \sigma^{j}\left(  v_{i}\right)  \ \mid\ j\in\mathbb{N}\right\}
\label{pf.lem.ptype-as-sum.c7.pf.2}%
\end{equation}
(by Claim 2 \textbf{(b)}). Consider such a $v_{i}$ for each $i\in\left[
k\right]  $. For each $i\in\left[  k\right]  $, we set
\[
a_{i}:=h\left(  v_{i}\right)  .
\]
Thus, we have defined $k$ elements $a_{1},a_{2},\ldots,a_{k}\in\mathbb{P}$.
Hence, $\left(  a_{1},a_{2},\ldots,a_{k}\right)  \in\mathbb{P}^{k}$.

Now, we shall show that $h=\Gamma\left(  a_{1},a_{2},\ldots,a_{k}\right)  $.

Indeed, let $v\in V$ be arbitrary. Recall that $\operatorname*{ind}v$ is
defined as the unique $i\in\left[  k\right]  $ such that $v\in V_{i}$. Hence,
$\operatorname*{ind}v\in\left[  k\right]  $ and $v\in V_{\operatorname*{ind}%
v}$. Therefore,%
\[
v\in V_{\operatorname*{ind}v}=\left\{  \sigma^{j}\left(
v_{\operatorname*{ind}v}\right)  \ \mid\ j\in\mathbb{N}\right\}
\]
(by (\ref{pf.lem.ptype-as-sum.c7.pf.2}), applied to $i=\operatorname*{ind}v$).
In other words, $v=\sigma^{j}\left(  v_{\operatorname*{ind}v}\right)  $ for
some $j\in\mathbb{N}$. Consider this $j$. Now,%
\[
h\left(  \underbrace{v}_{=\sigma^{j}\left(  v_{\operatorname*{ind}v}\right)
}\right)  =h\left(  \sigma^{j}\left(  v_{\operatorname*{ind}v}\right)
\right)  =\underbrace{\left(  h\circ\sigma^{j}\right)  }%
_{\substack{=h\\\text{(by (\ref{pf.lem.ptype-as-sum.c7.pf.1}))}}}\left(
v_{\operatorname*{ind}v}\right)  =h\left(  v_{\operatorname*{ind}v}\right)  .
\]
Comparing this with%
\begin{align*}
\underbrace{\left(  \Gamma\left(  a_{1},a_{2},\ldots,a_{k}\right)  \right)
}_{\substack{=g\left[  a_{1},a_{2},\ldots,a_{k}\right]  \\\text{(by the
definition of }\Gamma\text{)}}}\left(  v\right)   &  =\left(  g\left[
a_{1},a_{2},\ldots,a_{k}\right]  \right)  \left(  v\right) \\
&  =a_{\operatorname*{ind}v}\ \ \ \ \ \ \ \ \ \ \left(  \text{by the
definition of }g\left[  a_{1},a_{2},\ldots,a_{k}\right]  \right) \\
&  =h\left(  v_{\operatorname*{ind}v}\right)  \ \ \ \ \ \ \ \ \ \ \left(
\text{by the definition of }a_{\operatorname*{ind}v}\right)  ,
\end{align*}
we obtain%
\[
h\left(  v\right)  =\left(  \Gamma\left(  a_{1},a_{2},\ldots,a_{k}\right)
\right)  \left(  v\right)  .
\]

Forget that we fixed $v$. We thus have proved that $h\left(  v\right)
=\left(  \Gamma\left(  a_{1},a_{2},\ldots,a_{k}\right)  \right)  \left(
v\right)  $ for each $v\in V$. In other words, $h=\Gamma\left(  a_{1}%
,a_{2},\ldots,a_{k}\right)  $. Hence, $h\in\Gamma\left(  \mathbb{P}%
^{k}\right)  $ (since $\left(  a_{1},a_{2},\ldots,a_{k}\right)  \in
\mathbb{P}^{k}$).

Forget that we fixed $h$. We thus have proved that $h\in\Gamma\left(
\mathbb{P}^{k}\right)  $ for each $h\in\left\{  f:V\rightarrow\mathbb{P}%
\ \mid\ f\circ\sigma=f\right\}  $. In other words, $\left\{  f:V\rightarrow
\mathbb{P}\ \mid\ f\circ\sigma=f\right\}  \subseteq\Gamma\left(
\mathbb{P}^{k}\right)  $. In other words, the map $\Gamma$ is surjective. This
proves Claim 7.] \medskip

\begin{statement}
\textit{Claim 8:} Let $\left(  a_{1},a_{2},\ldots,a_{k}\right)  \in
\mathbb{P}^{k}$. Then,%
\[
\prod_{v\in V}x_{\left(  \Gamma\left(  a_{1},a_{2},\ldots,a_{k}\right)
\right)  \left(  v\right)  }=\prod_{i=1}^{k}x_{a_{i}}^{\left\vert
V_{i}\right\vert }.
\]

\end{statement}

[\textit{Proof of Claim 8:} For each $v\in V$, we have%
\begin{align*}
\underbrace{\left(  \Gamma\left(  a_{1},a_{2},\ldots,a_{k}\right)  \right)
}_{\substack{=g\left[  a_{1},a_{2},\ldots,a_{k}\right]  \\\text{(by the
definition of }\Gamma\text{)}}}\left(  v\right)   &  =\left(  g\left[
a_{1},a_{2},\ldots,a_{k}\right]  \right)  \left(  v\right) \\
&  =a_{\operatorname*{ind}v}\ \ \ \ \ \ \ \ \ \ \left(  \text{by the
definition of }g\left[  a_{1},a_{2},\ldots,a_{k}\right]  \right)  .
\end{align*}
Thus, for each $v\in V$, we have%
\[
x_{\left(  \Gamma\left(  a_{1},a_{2},\ldots,a_{k}\right)  \right)  \left(
v\right)  }=x_{a_{\operatorname*{ind}v}}.
\]
Multiplying these equalities for all $v\in V$, we obtain%
\begin{align}
\prod_{v\in V}x_{\left(  \Gamma\left(  a_{1},a_{2},\ldots,a_{k}\right)
\right)  \left(  v\right)  }  &  =\prod_{v\in V}x_{a_{\operatorname*{ind}v}%
}=\prod_{j=1}^{k}\ \ \prod_{\substack{v\in V;\\\operatorname*{ind}%
v=j}}\underbrace{x_{a_{\operatorname*{ind}v}}}_{\substack{=x_{a_{j}%
}\\\text{(since }\operatorname*{ind}v=j\text{)}}}\nonumber\\
&  \ \ \ \ \ \ \ \ \ \ \ \ \ \ \ \ \ \ \ \ \left(
\begin{array}
[c]{c}%
\text{here, we have split the product}\\
\text{according to the value of }\operatorname*{ind}v
\end{array}
\right) \nonumber\\
&  =\prod_{j=1}^{k}\ \ \prod_{\substack{v\in V;\\\operatorname*{ind}%
v=j}}x_{a_{j}}. \label{pf.lem.ptype-as-sum.c8.pf.1}%
\end{align}

Now, fix $j\in\left[  k\right]  $. Then, $\left\{  v\in V\ \mid
\ \operatorname*{ind}v=j\right\}  =V_{j}$\ \ \ \ \footnote{\textit{Proof.} If
$v\in V_{j}$, then $v\in V$ (since $V_{j}$ is a subset of $V$) and
$\operatorname*{ind}v=j$ (by Claim 4). Thus, every element of $V_{j}$ is an
element $v\in V$ satisfying $\operatorname*{ind}v=j$. In other words,
$V_{j}\subseteq\left\{  v\in V\ \mid\ \operatorname*{ind}v=j\right\}  $.
\par
Let $v\in V$ satisfy $\operatorname*{ind}v=j$. Recall that
$\operatorname*{ind}v$ is defined as the unique $i\in\left[  k\right]  $ such
that $v\in V_{i}$. Hence, $\operatorname*{ind}v\in\left[  k\right]  $ and
$v\in V_{\operatorname*{ind}v}$. Thus, $v\in V_{\operatorname*{ind}v}=V_{j}$
(since $\operatorname*{ind}v=j$). In other words, $v$ belongs to $V_{j}$.
\par
Forget that we fixed $v$. We thus have shown that every $v\in V$ satisfying
$\operatorname*{ind}v=j$ must belong to $V_{j}$. In other words, $\left\{
v\in V\ \mid\ \operatorname*{ind}v=j\right\}  \subseteq V_{j}$. Combining this
with $V_{j}\subseteq\left\{  v\in V\ \mid\ \operatorname*{ind}v=j\right\}  $,
we obtain $\left\{  v\in V\ \mid\ \operatorname*{ind}v=j\right\}  =V_{j}$.}.
Hence, the product sign \textquotedblleft$\prod_{\substack{v\in
V;\\\operatorname*{ind}v=j}}$\textquotedblright\ can be rewritten as
\textquotedblleft$\prod_{v\in V_{j}}$\textquotedblright. Thus, in particular,%
\begin{equation}
\prod_{\substack{v\in V;\\\operatorname*{ind}v=j}}x_{a_{j}}=\prod_{v\in V_{j}%
}x_{a_{j}}=x_{a_{j}}^{\left\vert V_{j}\right\vert }.
\label{pf.lem.ptype-as-sum.c8.pf.2}%
\end{equation}

Now, forget that we fixed $j$. We thus have proved
(\ref{pf.lem.ptype-as-sum.c8.pf.2}) for each $j\in\left[  k\right]  $. Now,
(\ref{pf.lem.ptype-as-sum.c8.pf.1}) becomes%
\[
\prod_{v\in V}x_{\left(  \Gamma\left(  a_{1},a_{2},\ldots,a_{k}\right)
\right)  \left(  v\right)  }=\prod_{j=1}^{k}\ \ \underbrace{\prod
_{\substack{v\in V;\\\operatorname*{ind}v=j}}x_{a_{j}}}_{\substack{=x_{a_{j}%
}^{\left\vert V_{j}\right\vert }\\\text{(by (\ref{pf.lem.ptype-as-sum.c8.pf.2}%
))}}}=\prod_{j=1}^{k}x_{a_{j}}^{\left\vert V_{j}\right\vert }=\prod_{i=1}%
^{k}x_{a_{i}}^{\left\vert V_{i}\right\vert }%
\]
(here, we have renamed the index $j$ as $i$ in the product). This proves Claim
8.] \medskip

Now, our proof is almost complete. The map $\Gamma:\mathbb{P}^{k}%
\rightarrow\left\{  f:V\rightarrow\mathbb{P}\ \mid\ f\circ\sigma=f\right\}  $
is injective (by Claim 6) and surjective (by Claim 7); thus, it is bijective.
In other words, $\Gamma$ is a bijection. Hence, we can substitute
$\Gamma\left(  a_{1},a_{2},\ldots,a_{k}\right)  $ for $f$ in the sum
$\sum_{\substack{f:V\rightarrow\mathbb{P}\text{;}\\f\circ\sigma=f}%
}\ \ \prod_{v\in V}x_{f\left(  v\right)  }$. We thus obtain%
\begin{align}
\sum_{\substack{f:V\rightarrow\mathbb{P}\text{;}\\f\circ\sigma=f}%
}\ \ \prod_{v\in V}x_{f\left(  v\right)  }  &  =\sum_{\left(  a_{1}%
,a_{2},\ldots,a_{k}\right)  \in\mathbb{P}^{k}}\ \ \underbrace{\prod_{v\in
V}x_{\left(  \Gamma\left(  a_{1},a_{2},\ldots,a_{k}\right)  \right)  \left(
v\right)  }}_{\substack{=\prod_{i=1}^{k}x_{a_{i}}^{\left\vert V_{i}\right\vert
}\\\text{(by Claim 8)}}}\nonumber\\
&  =\sum_{\left(  a_{1},a_{2},\ldots,a_{k}\right)  \in\mathbb{P}^{k}}%
\ \ \prod_{i=1}^{k}x_{a_{i}}^{\left\vert V_{i}\right\vert }.
\label{pf.lem.ptype-as-sum.lhs=}%
\end{align}

On the other hand, Claim 3 shows that the entries of the partition
$\operatorname*{type}\sigma$ are the numbers $\left\vert V_{1}\right\vert
,\left\vert V_{2}\right\vert ,\ldots,\left\vert V_{k}\right\vert $ in some
order. In other words, there exists a permutation $\tau$ of $\left[  k\right]
$ such that $\operatorname*{type}\sigma=\left(  \left\vert V_{\tau\left(
1\right)  }\right\vert ,\left\vert V_{\tau\left(  2\right)  }\right\vert
,\ldots,\left\vert V_{\tau\left(  k\right)  }\right\vert \right)  $. Consider
this $\tau$. The map $\tau:\left[  k\right]  \rightarrow\left[  k\right]  $ is
a bijection (since it is a permutation of $\left[  k\right]  $). From
$\operatorname*{type}\sigma=\left(  \left\vert V_{\tau\left(  1\right)
}\right\vert ,\left\vert V_{\tau\left(  2\right)  }\right\vert ,\ldots
,\left\vert V_{\tau\left(  k\right)  }\right\vert \right)  $, we obtain%
\begin{align*}
p_{\operatorname*{type}\sigma}  &  =p_{\left(  \left\vert V_{\tau\left(
1\right)  }\right\vert ,\left\vert V_{\tau\left(  2\right)  }\right\vert
,\ldots,\left\vert V_{\tau\left(  k\right)  }\right\vert \right)  }\\
&  =p_{\left\vert V_{\tau\left(  1\right)  }\right\vert }p_{\left\vert
V_{\tau\left(  2\right)  }\right\vert }\cdots p_{\left\vert V_{\tau\left(
k\right)  }\right\vert }\ \ \ \ \ \ \ \ \ \ \left(  \text{by the definition of
}p_{\left(  \left\vert V_{\tau\left(  1\right)  }\right\vert ,\left\vert
V_{\tau\left(  2\right)  }\right\vert ,\ldots,\left\vert V_{\tau\left(
k\right)  }\right\vert \right)  }\right) \\
&  =\prod_{i\in\left[  k\right]  }p_{\left\vert V_{\tau\left(  i\right)
}\right\vert }\\
&  =\underbrace{\prod_{i\in\left[  k\right]  }}_{=\prod_{i=1}^{k}%
}\underbrace{p_{\left\vert V_{i}\right\vert }}_{\substack{=x_{1}^{\left\vert
V_{i}\right\vert }+x_{2}^{\left\vert V_{i}\right\vert }+x_{3}^{\left\vert
V_{i}\right\vert }+\cdots\\\text{(by the definition of }p_{\left\vert
V_{i}\right\vert }\text{)}}}\ \ \ \ \ \ \ \ \ \ \left(
\begin{array}
[c]{c}%
\text{here, we have substituted }i\text{ for }\tau\left(  i\right)  \text{ in
the}\\
\text{product, since }\tau:\left[  k\right]  \rightarrow\left[  k\right]
\text{ is a bijection}%
\end{array}
\right) \\
&  =\prod_{i=1}^{k}\underbrace{\left(  x_{1}^{\left\vert V_{i}\right\vert
}+x_{2}^{\left\vert V_{i}\right\vert }+x_{3}^{\left\vert V_{i}\right\vert
}+\cdots\right)  }_{=\sum_{a\in\mathbb{P}}x_{a}^{\left\vert V_{i}\right\vert
}}=\prod_{i=1}^{k}\ \ \sum_{a\in\mathbb{P}}x_{a}^{\left\vert V_{i}\right\vert
}\\
&  =\sum_{\left(  a_{1},a_{2},\ldots,a_{k}\right)  \in\mathbb{P}^{k}}%
\ \ \prod_{i=1}^{k}x_{a_{i}}^{\left\vert V_{i}\right\vert }%
\ \ \ \ \ \ \ \ \ \ \left(  \text{by the product rule}\right)  .
\end{align*}
Comparing this with (\ref{pf.lem.ptype-as-sum.lhs=}), we obtain%
\[
\sum_{\substack{f:V\rightarrow\mathbb{P}\text{;}\\f\circ\sigma=f}%
}\ \ \prod_{v\in V}x_{f\left(  v\right)  }=p_{\operatorname*{type}\sigma}.
\]
Thus, Lemma \ref{lem.ptype-as-sum} is proven.
\end{proof}
\end{verlong}

\subsection{A final alternating sum}

We need one more alternating-sum identity:

\begin{proposition}
\label{prop.linear-in-AsigA}Let $D=\left(  V,A\right)  $ be a digraph. Let
$\sigma\in\mathfrak{S}_{V}$ be a permutation of $V$. Then,%
\[
\sum_{\substack{F\subseteq\mathbf{A}_{\sigma}\cap A\\\text{is linear}}}\left(
-1\right)  ^{\left\vert F\right\vert }=%
\begin{cases}
\left(  -1\right)  ^{\varphi\left(  \sigma\right)  }, & \text{if }\sigma
\in\mathfrak{S}_{V}\left(  D,\overline{D}\right)  ;\\
0, & \text{else,}%
\end{cases}
\]
where we set%
\[
\varphi\left(  \sigma\right)  :=\sum_{\substack{\gamma\in\operatorname*{Cycs}%
\sigma;\\\gamma\text{ is a }D\text{-cycle}}}\left(  \ell\left(  \gamma\right)
-1\right)  .
\]

\end{proposition}

\begin{vershort}
\begin{proof}
Let us first prove the proposition in the case when $\sigma$ has only one
cycle. That is, we shall first prove the following claim:

\begin{statement}
\textit{Claim 1:} Assume that $\sigma$ has a unique cycle $\gamma$. Then,%
\[
\sum_{\substack{F\subseteq\mathbf{A}_{\sigma}\cap A\\\text{is linear}}}\left(
-1\right)  ^{\left\vert F\right\vert }=%
\begin{cases}
\left(  -1\right)  ^{\ell\left(  \gamma\right)  -1}, & \text{if }\gamma\text{
is a }D\text{-cycle;}\\
1, & \text{if }\gamma\text{ is a }\overline{D}\text{-cycle;}\\
0, & \text{else.}%
\end{cases}
\]

\end{statement}

[\textit{Proof of Claim 1:} We have assumed that $\sigma$ has a unique cycle
$\gamma$. Thus, $\mathbf{A}_{\sigma}=\operatorname*{CArcs}\gamma$. Hence, each
proper subset of $\mathbf{A}_{\sigma}$ is linear\footnote{since the removal of
any cyclic arc from a cycle turns the cycle into a path, and the removal of
any further arcs will break this path into smaller paths}, but $\mathbf{A}%
_{\sigma}$ itself is not\footnote{since the digraph $\left(  V,\mathbf{A}%
_{\sigma}\right)  $ has the cycle $\gamma$}. Furthermore, we have $\left\vert
\mathbf{A}_{\sigma}\right\vert =\left\vert \operatorname*{CArcs}%
\gamma\right\vert =\ell\left(  \gamma\right)  $.

The sum $\sum_{\substack{F\subseteq\mathbf{A}_{\sigma}\cap A\\\text{is
linear}}}\left(  -1\right)  ^{\left\vert F\right\vert }$ depends on whether
$\gamma$ is a $D$-cycle, a $\overline{D}$-cycle or neither:

\begin{itemize}
\item If $\gamma$ is a $D$-cycle, then all arcs in $\mathbf{A}_{\sigma}$
belong to $A$, and therefore we have $\mathbf{A}_{\sigma}\cap A=\mathbf{A}%
_{\sigma}$. Thus, in this case, we have%
\begin{align*}
&  \left\{  \text{linear subsets of }\mathbf{A}_{\sigma}\cap A\right\} \\
&  =\left\{  \text{linear subsets of }\mathbf{A}_{\sigma}\right\}  =\left\{
\text{proper subsets of }\mathbf{A}_{\sigma}\right\}
\end{align*}
(since each proper subset of $\mathbf{A}_{\sigma}$ is linear, but
$\mathbf{A}_{\sigma}$ itself is not). Hence, in this case, we have%
\begin{align*}
\sum_{\substack{F\subseteq\mathbf{A}_{\sigma}\cap A\\\text{is linear}}}\left(
-1\right)  ^{\left\vert F\right\vert }  &  =\sum_{\substack{F\subseteq
\mathbf{A}_{\sigma}\\\text{is a proper subset}}}\left(  -1\right)
^{\left\vert F\right\vert }=\underbrace{\sum_{F\subseteq\mathbf{A}_{\sigma}%
}\left(  -1\right)  ^{\left\vert F\right\vert }}_{\substack{=0\\\text{(by
Lemma \ref{lem.cancel},}\\\text{since }\mathbf{A}_{\sigma}\neq\varnothing
\text{)}}}-\left(  -1\right)  ^{\left\vert \mathbf{A}_{\sigma}\right\vert }\\
&  =-\left(  -1\right)  ^{\left\vert \mathbf{A}_{\sigma}\right\vert }=\left(
-1\right)  ^{\left\vert \mathbf{A}_{\sigma}\right\vert -1}\\
&  =\left(  -1\right)  ^{\ell\left(  \gamma\right)  -1}%
\ \ \ \ \ \ \ \ \ \ \left(  \text{since }\left\vert \mathbf{A}_{\sigma
}\right\vert =\ell\left(  \gamma\right)  \right)  .
\end{align*}

\item If $\gamma$ is a $\overline{D}$-cycle, then no arcs in $\mathbf{A}%
_{\sigma}$ belong to $A$, and therefore we have $\mathbf{A}_{\sigma}\cap
A=\varnothing$. Thus, in this case, we have%
\[
\left\{  \text{linear subsets of }\mathbf{A}_{\sigma}\cap A\right\}  =\left\{
\text{linear subsets of }\varnothing\right\}  =\left\{  \varnothing\right\}
.
\]
Hence, in this case, the sum $\sum_{\substack{F\subseteq\mathbf{A}_{\sigma
}\cap A\\\text{is linear}}}\left(  -1\right)  ^{\left\vert F\right\vert }$ has
only one addend, namely the addend corresponding to $F=\varnothing$. Thus,
this sum equals $1$ in this case.

\item If $\gamma$ is neither a $D$-cycle nor a $\overline{D}$-cycle, then we
have $\mathbf{A}_{\sigma}\cap A\neq\varnothing$ (since $\mathbf{A}_{\sigma
}\cap A=\varnothing$ would imply that $\mathbf{A}_{\sigma}\subseteq\left(
V\times V\right)  \setminus A$, whence $\operatorname*{CArcs}\gamma
=\mathbf{A}_{\sigma}\subseteq\left(  V\times V\right)  \setminus A$; but this
would mean that $\gamma$ is a $\overline{D}$-cycle) and $\mathbf{A}_{\sigma
}\not \subseteq A$ (since $\mathbf{A}_{\sigma}\subseteq A$ would mean that
$\gamma$ is a $D$-cycle). Hence, in this case, any $F\subseteq\mathbf{A}%
_{\sigma}\cap A$ is a proper subset of $\mathbf{A}_{\sigma}$ (since
$\mathbf{A}_{\sigma}\not \subseteq A$ shows that $\mathbf{A}_{\sigma}\cap A$
is a proper subset of $\mathbf{A}_{\sigma}$) and therefore linear (since each
proper subset of $\mathbf{A}_{\sigma}$ is linear). Thus, in this case, we
have
\begin{align*}
\sum_{\substack{F\subseteq\mathbf{A}_{\sigma}\cap A\\\text{is linear}}}\left(
-1\right)  ^{\left\vert F\right\vert }  &  =\sum_{F\subseteq\mathbf{A}%
_{\sigma}\cap A}\left(  -1\right)  ^{\left\vert F\right\vert }=\left[
\mathbf{A}_{\sigma}\cap A=\varnothing\right]  \ \ \ \ \ \ \ \ \ \ \left(
\text{by Lemma \ref{lem.cancel}}\right) \\
&  =0\ \ \ \ \ \ \ \ \ \ \left(  \text{since }\mathbf{A}_{\sigma}\cap
A\neq\varnothing\right)  .
\end{align*}

\end{itemize}

Combining these three cases, we see that
\[
\sum_{\substack{F\subseteq\mathbf{A}_{\sigma}\cap A\\\text{is linear}}}\left(
-1\right)  ^{\left\vert F\right\vert }=%
\begin{cases}
\left(  -1\right)  ^{\ell\left(  \gamma\right)  -1}, & \text{if }\gamma\text{
is a }D\text{-cycle;}\\
1, & \text{if }\gamma\text{ is a }\overline{D}\text{-cycle;}\\
0, & \text{else.}%
\end{cases}
\]
This proves Claim 1.] \medskip

Let us now prove Proposition \ref{prop.linear-in-AsigA} in the general case.
Let $\gamma_{1},\gamma_{2},\ldots,\gamma_{k}$ be the cycles of $\sigma$,
listed with no repetition. Thus, these cycles $\gamma_{1},\gamma_{2}%
,\ldots,\gamma_{k}$ are distinct, and we have $\operatorname*{Cycs}%
\sigma=\left\{  \gamma_{1},\gamma_{2},\ldots,\gamma_{k}\right\}  $.

For each $i\in\left[  k\right]  $, let $V_{i}$ be the set of entries of the
cycle $\gamma_{i}$. Thus, $V=V_{1}\sqcup V_{2}\sqcup\cdots\sqcup V_{k}$.

Furthermore, for each $i\in\left[  k\right]  $, we let $D_{i}$ be the digraph
obtained from $D$ by removing all vertices that are not in $V_{i}$, and we let
$A_{i}$ be the set of all arcs of $D_{i}$. Thus, $A_{i}=A\cap\left(
V_{i}\times V_{i}\right)  $ and $D_{i}=\left(  V_{i},A_{i}\right)  $.

For each $i\in\left[  k\right]  $, let $\sigma_{i}$ be the permutation of
$V_{i}$ obtained by restricting $\sigma$ to $V_{i}$ (this is well-defined,
since $V_{i}$ is the set of entries of a cycle of $\sigma$ and thus preserved
under $\sigma$). This permutation $\sigma_{i}$ has a unique cycle, namely
$\gamma_{i}$. Thus, for each $i\in\left[  k\right]  $, Claim 1 (applied to
$D_{i}=\left(  V_{i},A_{i}\right)  $ and $\sigma_{i}$ and $\gamma_{i}$ instead
of $D=\left(  V,A\right)  $ and $\sigma$ and $\gamma$) yields%
\begin{align}
\sum_{\substack{F\subseteq\mathbf{A}_{\sigma_{i}}\cap A_{i}\\\text{is linear}%
}}\left(  -1\right)  ^{\left\vert F\right\vert }  &  =%
\begin{cases}
\left(  -1\right)  ^{\ell\left(  \gamma_{i}\right)  -1}, & \text{if }%
\gamma_{i}\text{ is a }D_{i}\text{-cycle;}\\
1, & \text{if }\gamma_{i}\text{ is a }\overline{D_{i}}\text{-cycle;}\\
0, & \text{else}%
\end{cases}
\nonumber\\
&  =%
\begin{cases}
\left(  -1\right)  ^{\ell\left(  \gamma_{i}\right)  -1}, & \text{if }%
\gamma_{i}\text{ is a }D\text{-cycle;}\\
1, & \text{if }\gamma_{i}\text{ is a }\overline{D}\text{-cycle;}\\
0, & \text{else}%
\end{cases}
\label{pf.prop.linear-in-AsigA.7}%
\end{align}
(since the statement \textquotedblleft$\gamma_{i}$ is a $D_{i}$%
-cycle\textquotedblright\ is equivalent to \textquotedblleft$\gamma_{i}$ is a
$D$-cycle\textquotedblright, and the statement \textquotedblleft$\gamma_{i}$
is a $\overline{D_{i}}$-cycle\textquotedblright\ is equivalent to
\textquotedblleft$\gamma_{i}$ is a $\overline{D}$-cycle\textquotedblright).

It is easy to see that
\begin{equation}
\mathbf{A}_{\sigma_{i}}\cap A_{i}=\mathbf{A}_{\sigma_{i}}\cap A
\label{pf.prop.linear-in-AsigA.8}%
\end{equation}
for each $i\in\left[  k\right]  $\ \ \ \ \footnote{\textit{Proof.} Let
$i\in\left[  k\right]  $. Then, $\mathbf{A}_{\sigma_{i}}\subseteq V_{i}\times
V_{i}$ (since $\sigma_{i}$ is a permutation of $V_{i}$), so that
$\mathbf{A}_{\sigma_{i}}\cap\left(  V_{i}\times V_{i}\right)  =\mathbf{A}%
_{\sigma_{i}}$. Therefore,%
\[
\mathbf{A}_{\sigma_{i}}\cap\underbrace{A_{i}}_{=A\cap\left(  V_{i}\times
V_{i}\right)  }=\mathbf{A}_{\sigma_{i}}\cap A\cap\left(  V_{i}\times
V_{i}\right)  =\underbrace{\mathbf{A}_{\sigma_{i}}\cap\left(  V_{i}\times
V_{i}\right)  }_{=\mathbf{A}_{\sigma_{i}}}\cap A=\mathbf{A}_{\sigma_{i}}\cap
A.
\]
This proves (\ref{pf.prop.linear-in-AsigA.8}).}. Hence, we can rewrite
(\ref{pf.prop.linear-in-AsigA.7}) as follows: For each $i\in\left[  k\right]
$, we have%
\begin{equation}
\sum_{\substack{F\subseteq\mathbf{A}_{\sigma_{i}}\cap A\\\text{is linear}%
}}\left(  -1\right)  ^{\left\vert F\right\vert }=%
\begin{cases}
\left(  -1\right)  ^{\ell\left(  \gamma_{i}\right)  -1}, & \text{if }%
\gamma_{i}\text{ is a }D\text{-cycle;}\\
1, & \text{if }\gamma_{i}\text{ is a }\overline{D}\text{-cycle;}\\
0, & \text{else.}%
\end{cases}
\label{pf.prop.linear-in-AsigA.7b}%
\end{equation}

The definition of $\varphi\left(  \sigma\right)  $ yields%
\[
\varphi\left(  \sigma\right)  =\sum_{\substack{\gamma\in\operatorname*{Cycs}%
\sigma;\\\gamma\text{ is a }D\text{-cycle}}}\left(  \ell\left(  \gamma\right)
-1\right)  =\sum_{\substack{i\in\left[  k\right]  ;\\\gamma_{i}\text{ is a
}D\text{-cycle}}}\left(  \ell\left(  \gamma_{i}\right)  -1\right)
\]
(since $\gamma_{1},\gamma_{2},\ldots,\gamma_{k}$ are distinct, and
$\operatorname*{Cycs}\sigma=\left\{  \gamma_{1},\gamma_{2},\ldots,\gamma
_{k}\right\}  $). Therefore,%
\begin{align}
\left(  -1\right)  ^{\varphi\left(  \sigma\right)  }  &  =\left(  -1\right)
^{\sum_{\substack{i\in\left[  k\right]  ;\\\gamma_{i}\text{ is a
}D\text{-cycle}}}\left(  \ell\left(  \gamma_{i}\right)  -1\right)
}\nonumber\\
&  =\prod_{\substack{i\in\left[  k\right]  ;\\\gamma_{i}\text{ is a
}D\text{-cycle}}}\left(  -1\right)  ^{\ell\left(  \gamma_{i}\right)  -1}.
\label{pf.prop.linear-in-AsigA.sign}%
\end{align}

However, it is easy to see that $\mathbf{A}_{\sigma}=\mathbf{A}_{\sigma_{1}%
}\sqcup\mathbf{A}_{\sigma_{2}}\sqcup\cdots\sqcup\mathbf{A}_{\sigma_{k}}$
(since $\sigma_{1},\sigma_{2},\ldots,\sigma_{k}$ are the restrictions of
$\sigma$ to the subsets $V_{1},V_{2},\ldots,V_{k}$, which cover $V$ without
overlap). Thus,%
\begin{align*}
\mathbf{A}_{\sigma}\cap A  &  =\left(  \mathbf{A}_{\sigma_{1}}\sqcup
\mathbf{A}_{\sigma_{2}}\sqcup\cdots\sqcup\mathbf{A}_{\sigma_{k}}\right)  \cap
A\\
&  =\left(  \mathbf{A}_{\sigma_{1}}\cap A\right)  \sqcup\left(  \mathbf{A}%
_{\sigma_{2}}\cap A\right)  \sqcup\cdots\sqcup\left(  \mathbf{A}_{\sigma_{k}%
}\cap A\right)  .
\end{align*}
Hence, a subset $F$ of $\mathbf{A}_{\sigma}\cap A$ is the same thing as a
(necessarily disjoint) union $F_{1}\sqcup F_{2}\sqcup\cdots\sqcup F_{k}$ of
subsets $F_{i}\subseteq\mathbf{A}_{\sigma_{i}}\cap A$ for all $i\in\left[
k\right]  $\ \ \ \ \footnote{This sentence should be understood as follows:
\par
\begin{enumerate}
\item A subset $F$ of $\mathbf{A}_{\sigma}\cap A$ is the same thing as a union
$F_{1}\cup F_{2}\cup\cdots\cup F_{k}$ of subsets $F_{i}\subseteq
\mathbf{A}_{\sigma_{i}}\cap A$ for all $i\in\left[  k\right]  $.
\par
\item Any union of the latter form is a disjoint union (thus can be written as
$F_{1}\sqcup F_{2}\sqcup\cdots\sqcup F_{k}$).
\end{enumerate}
}. Moreover, in this case, the subsets $F_{i}$ for all $i\in\left[  k\right]
$ are uniquely determined by $F$ (namely, we have $F_{i}=F\cap\mathbf{A}%
_{\sigma_{i}}$ for each $i\in\left[  k\right]  $). Finally, the former subset
$F$ is linear if and only if all the latter subsets $F_{i}$ are
linear\footnote{\textit{Proof.} Let $F$ be a subset of $\mathbf{A}_{\sigma
}\cap A$, and let us assume that $F$ is written as a disjoint union
$F_{1}\sqcup F_{2}\sqcup\cdots\sqcup F_{k}$ of subsets $F_{i}\subseteq
\mathbf{A}_{\sigma_{i}}\cap A$ for all $i\in\left[  k\right]  $. We must prove
that $F$ is linear if and only if all the subsets $F_{i}$ are linear.
\par
For each $i\in\left[  k\right]  $, we have $F_{i}\subseteq\mathbf{A}%
_{\sigma_{i}}\cap A\subseteq\mathbf{A}_{\sigma_{i}}\subseteq V_{i}\times
V_{i}$ (because $\sigma_{i}$ is a permutation of $V_{i}$). In other words, for
each $i\in\left[  k\right]  $, the set $F_{i}$ is a subset of $V_{i}\times
V_{i}$. We have $F=F_{1}\sqcup F_{2}\sqcup\cdots\sqcup F_{k}=F_{1}\cup
F_{2}\cup\cdots\cup F_{k}$. Moreover, the sets $V_{1},V_{2},\ldots,V_{k}$ are
disjoint subsets of $V$ and satisfy $V=V_{1}\cup V_{2}\cup\cdots\cup V_{k}$.
Hence, Proposition \ref{prop.linear.Vi} shows that the set $F$ is linear if
and only if all the subsets $F_{i}$ for $i\in\left[  k\right]  $ are linear.
This completes our proof.}. Hence, we can substitute $F_{1}\sqcup F_{2}%
\sqcup\cdots\sqcup F_{k}$ for $F$ in the sum $\sum_{\substack{F\subseteq
\mathbf{A}_{\sigma}\cap A\\\text{is linear}}}\left(  -1\right)  ^{\left\vert
F\right\vert }$. We thus obtain
\begin{align*}
\sum_{\substack{F\subseteq\mathbf{A}_{\sigma}\cap A\\\text{is linear}}}\left(
-1\right)  ^{\left\vert F\right\vert }  &  =\sum_{\substack{\left(
F_{i}\right)  _{i\in\left[  k\right]  }\text{ is a family,}\\\text{where each
}F_{i}\text{ is a linear}\\\text{subset of }\mathbf{A}_{\sigma_{i}}\cap
A}}\ \ \underbrace{\left(  -1\right)  ^{\left\vert F_{1}\sqcup F_{2}%
\sqcup\cdots\sqcup F_{k}\right\vert }}_{=\left(  -1\right)  ^{\left\vert
F_{1}\right\vert +\left\vert F_{2}\right\vert +\cdots+\left\vert
F_{k}\right\vert }=\prod_{i=1}^{k}\left(  -1\right)  ^{\left\vert
F_{i}\right\vert }}\\
&  =\sum_{\substack{\left(  F_{i}\right)  _{i\in\left[  k\right]  }\text{ is a
family,}\\\text{where each }F_{i}\text{ is a linear}\\\text{subset of
}\mathbf{A}_{\sigma_{i}}\cap A}}\ \ \prod_{i=1}^{k}\left(  -1\right)
^{\left\vert F_{i}\right\vert }\\
&  =\prod_{i=1}^{k}\ \ \sum_{\substack{F_{i}\subseteq\mathbf{A}_{\sigma_{i}%
}\cap A\\\text{is linear}}}\left(  -1\right)  ^{\left\vert F_{i}\right\vert
}\ \ \ \ \ \ \ \ \ \ \left(  \text{by the product rule}\right) \\
&  =\prod_{i=1}^{k}\ \ \sum_{\substack{F\subseteq\mathbf{A}_{\sigma_{i}}\cap
A\\\text{is linear}}}\left(  -1\right)  ^{\left\vert F\right\vert
}\ \ \ \ \ \ \ \ \ \ \left(
\begin{array}
[c]{c}%
\text{here, we have renamed the}\\
\text{summation index }F_{i}\text{ as }F
\end{array}
\right) \\
&  =\prod_{i=1}^{k}%
\begin{cases}
\left(  -1\right)  ^{\ell\left(  \gamma_{i}\right)  -1}, & \text{if }%
\gamma_{i}\text{ is a }D\text{-cycle;}\\
1, & \text{if }\gamma_{i}\text{ is a }\overline{D}\text{-cycle;}\\
0, & \text{else}%
\end{cases}
\ \ \ \ \ \ \ \ \ \ \left(  \text{by (\ref{pf.prop.linear-in-AsigA.7b}%
)}\right)  .
\end{align*}
The right hand side of this equality is clearly $0$ unless each of the cycles
$\gamma_{1},\gamma_{2},\ldots,\gamma_{k}$ is a $D$-cycle or a $\overline{D}%
$-cycle; otherwise, it equals%
\[
\prod_{i=1}^{k}%
\begin{cases}
\left(  -1\right)  ^{\ell\left(  \gamma_{i}\right)  -1}, & \text{if }%
\gamma_{i}\text{ is a }D\text{-cycle;}\\
1, & \text{if }\gamma_{i}\text{ is a }\overline{D}\text{-cycle}%
\end{cases}
=\prod_{\substack{i\in\left[  k\right]  ;\\\gamma_{i}\text{ is a
}D\text{-cycle}}}\left(  -1\right)  ^{\ell\left(  \gamma_{i}\right)
-1}=\left(  -1\right)  ^{\varphi\left(  \sigma\right)  }%
\]
(by (\ref{pf.prop.linear-in-AsigA.sign})). Hence, in either case, it equals
\begin{align*}
&
\begin{cases}
\left(  -1\right)  ^{\varphi\left(  \sigma\right)  }, & \text{if each of
}\gamma_{1},\gamma_{2},\ldots,\gamma_{k}\text{ is a }D\text{-cycle or a
}\overline{D}\text{-cycle};\\
0, & \text{else}%
\end{cases}
\\
&  =%
\begin{cases}
\left(  -1\right)  ^{\varphi\left(  \sigma\right)  }, & \text{if each cycle of
}\sigma\text{ is a }D\text{-cycle or a }\overline{D}\text{-cycle};\\
0, & \text{else}%
\end{cases}
\\
&  \ \ \ \ \ \ \ \ \ \ \ \ \ \ \ \ \ \ \ \ \left(  \text{since }\gamma
_{1},\gamma_{2},\ldots,\gamma_{k}\text{ are the cycles of }\sigma\right) \\
&  =%
\begin{cases}
\left(  -1\right)  ^{\varphi\left(  \sigma\right)  }, & \text{if }\sigma
\in\mathfrak{S}_{V}\left(  D,\overline{D}\right)  ;\\
0, & \text{else}%
\end{cases}
\end{align*}
(by the definition of $\mathfrak{S}_{V}\left(  D,\overline{D}\right)  $). This
proves Proposition \ref{prop.linear-in-AsigA}.
\end{proof}
\end{vershort}

\begin{verlong}
Our proof of this proposition requires several auxiliary results. We begin by
proving some lemmas on the linearity of certain sets:

\begin{lemma}
\label{lem.linear.sub-of-path}Let $V$ be a finite set. Let $p$ be a path of
$V$. Then, $\operatorname*{Arcs}p$ is a linear subset of $V\times V$.
\end{lemma}

\begin{proof}
Recall that a path of $V$ means a nonempty tuple of distinct elements of $V$.
Hence, $p$ is a nonempty tuple of distinct elements of $V$ (since $p$ is a
path of $V$).

Let $W$ be the set of all entries of $p$. Then, $W$ is a subset of $V$ (since
$p$ is a tuple of elements of $V$). Moreover, the entries of $p$ are precisely
the elements of $W$ (since $W$ is the set of all entries of $p$). Thus, all
entries of $p$ belong to $W$. Hence, $p$ is a tuple of elements of $W$. Thus,
$p$ is a nonempty tuple of distinct elements of $W$ (since $p$ is a nonempty
tuple of distinct elements of $V$). In other words, $p$ is a path of $W$ (by
the definition of a \textquotedblleft path of $W$\textquotedblright).

Each $v\in W$ is an element of $W$. In other words, each $v\in W$ is an entry
of $p$ (since the entries of $p$ are precisely the elements of $W$). In other
words, each $v\in W$ belongs to the path $p$.

Now, we claim that the $1$-element set $\left\{  p\right\}  $ is a path cover
of $W$. Indeed, $\left\{  p\right\}  $ is clearly a set of paths of $W$ (since
$p$ is a path of $W$) and has the property that each $v\in W$ belongs to
exactly one of these paths (because each $v\in W$ belongs to the path $p$). In
other words, $\left\{  p\right\}  $ is a path cover of $W$ (by the definition
of a \textquotedblleft path cover\textquotedblright). The arc set
$\operatorname*{Arcs}\left\{  p\right\}  $ of this path cover is
\begin{align*}
\operatorname*{Arcs}\left\{  p\right\}   &  =\bigcup_{v\in\left\{  p\right\}
}\operatorname*{Arcs}v\ \ \ \ \ \ \ \ \ \ \left(  \text{by the definition of
}\operatorname*{Arcs}\left\{  p\right\}  \right) \\
&  =\operatorname*{Arcs}p.
\end{align*}
Thus, $\operatorname*{Arcs}p$ is the arc set of some path cover of $W$
(namely, of the path cover $\left\{  p\right\}  $).

It is furthermore easy to see that $\operatorname*{Arcs}p$ is a subset of
$W\times W$\ \ \ \ \footnote{\textit{Proof.} Write the path $p$ as $\left(
v_{1},v_{2},\ldots,v_{k}\right)  $. Then, the entries of $p$ are $v_{1}%
,v_{2},\ldots,v_{k}$. However, we know that all entries of $p$ belong to $W$.
In other words, all of $v_{1},v_{2},\ldots,v_{k}$ belong to $W$ (since the
entries of $p$ are $v_{1},v_{2},\ldots,v_{k}$). Hence, the pairs $\left(
v_{1},v_{2}\right)  ,\ \left(  v_{2},v_{3}\right)  ,\ \ldots,\ \left(
v_{k-1},v_{k}\right)  $ belong to $W\times W$.
\par
However, $p=\left(  v_{1},v_{2},\ldots,v_{k}\right)  $. Thus,%
\begin{align*}
\operatorname*{Arcs}p  &  =\operatorname*{Arcs}\left(  v_{1},v_{2}%
,\ldots,v_{k}\right) \\
&  =\left\{  \left(  v_{1},v_{2}\right)  ,\ \left(  v_{2},v_{3}\right)
,\ \ldots,\ \left(  v_{k-1},v_{k}\right)  \right\}
\ \ \ \ \ \ \ \ \ \ \left(  \text{by the definition of }\operatorname*{Arcs}%
\left(  v_{1},v_{2},\ldots,v_{k}\right)  \right) \\
&  \subseteq W\times W\ \ \ \ \ \ \ \ \ \ \left(  \text{since the pairs
}\left(  v_{1},v_{2}\right)  ,\ \left(  v_{2},v_{3}\right)  ,\ \ldots
,\ \left(  v_{k-1},v_{k}\right)  \text{ belong to }W\times W\right)  .
\end{align*}
In other words, $\operatorname*{Arcs}p$ is a subset of $W\times W$.}.

Now, recall that a subset $F$ of $W\times W$ is said to be linear if it is the
arc set of some path cover of $W$ (by the definition of \textquotedblleft
linear\textquotedblright). Hence, the subset $\operatorname*{Arcs}p$ of
$W\times W$ is linear (since it is the arc set of some path cover of $W$).

However, Proposition \ref{prop.linear.VW} (applied to $F=\operatorname*{Arcs}%
p$) shows that $\operatorname*{Arcs}p$ is linear as a subset of $W\times W$ if
and only if $\operatorname*{Arcs}p$ is linear as a subset of $V\times V$.
Thus, $\operatorname*{Arcs}p$ is linear as a subset of $V\times V$ (since
$\operatorname*{Arcs}p$ is linear as a subset of $W\times W$). This proves
Lemma \ref{lem.linear.sub-of-path}.
\end{proof}

\begin{lemma}
\label{lem.cycle.rotate}Let $V$ be a finite set. Let $\sigma\in\mathfrak{S}%
_{V}$ be a permutation of $V$. Let $\gamma$ be a cycle of $\sigma$. Let
$a\in\operatorname*{CArcs}\gamma$. Then, there exists a tuple $w=\left(
w_{1},w_{2},\ldots,w_{k}\right)  \in\gamma$ such that $a=\left(  w_{k}%
,w_{1}\right)  $.
\end{lemma}

\begin{proof}
The cycle $\gamma$ is a cycle of $\sigma$, and thus is a rotation-equivalence
class of nonempty tuples of distinct elements of $V$ (since any cycle of
$\sigma$ is such a class). Hence, we can write $\gamma$ in the form
$\gamma=v_{\sim}$, where $v$ is a nonempty tuple of distinct elements of $V$.
Consider this $v$.

Write this tuple $v$ as $v=\left(  v_{1},v_{2},\ldots,v_{k}\right)  $. Then,
$k\geq1$ (since $v$ is nonempty), and the entries $v_{1},v_{2},\ldots,v_{k}$
of $v$ are distinct (since $v$ is a tuple of distinct elements of $V$). Let us
set $v_{k+1}:=v_{1}$. We have%
\begin{align}
\operatorname*{CArcs}\gamma &  =\operatorname*{CArcs}\left(  v_{\sim}\right)
\ \ \ \ \ \ \ \ \ \ \left(  \text{since }\gamma=v_{\sim}\right) \nonumber\\
&  =\operatorname*{CArcs}v\ \ \ \ \ \ \ \ \ \ \left(  \text{by the definition
of }\operatorname*{CArcs}\left(  v_{\sim}\right)  \text{, since }v\in v_{\sim
}\right) \nonumber\\
&  =\left\{  \left(  v_{i},v_{i+1}\right)  \ \mid\ i\in\left[  k\right]
\right\}  \label{pf.lem.cycle.rotate.1}%
\end{align}
(by the definition of $\operatorname*{CArcs}v$, since $v=\left(  v_{1}%
,v_{2},\ldots,v_{k}\right)  $ and $v_{k+1}=v_{1}$). Thus,%
\begin{align}
\operatorname*{CArcs}\gamma &  =\left\{  \left(  v_{i},v_{i+1}\right)
\ \mid\ i\in\left[  k\right]  \right\} \nonumber\\
&  =\left\{  \left(  v_{1},v_{2}\right)  ,\ \left(  v_{2},v_{3}\right)
,\ \ldots,\ \left(  v_{k},v_{k+1}\right)  \right\}  .
\label{pf.lem.cycle.rotate.2}%
\end{align}

We have
\[
a\in\operatorname*{CArcs}\gamma=\left\{  \left(  v_{i},v_{i+1}\right)
\ \mid\ i\in\left[  k\right]  \right\}  \ \ \ \ \ \ \ \ \ \ \left(  \text{by
(\ref{pf.lem.cycle.rotate.1})}\right)  .
\]
In other words, there exists some $i\in\left[  k\right]  $ such that
$a=\left(  v_{i},v_{i+1}\right)  $. Consider this $i$.

We now define a $k$-tuple $w\in V^{k}$ by%
\[
w:=\left(  v_{i+1},v_{i+2},\ldots,v_{k},v_{1},v_{2},\ldots,v_{i}\right)  .
\]
Thus, $w$ can be obtained from $v$ by a cyclic rotation (specifically, by
cyclically rotating $v$ a total of $i$ steps to the left). Hence, $w$ is
rotation-equivalent to $v$. Thus, $w$ belongs to the same rotation-equivalence
class as $v$. In other words, $w\in v_{\sim}$ (since $v_{\sim}$ is the
rotation-equivalence class to which $v$ belongs). In other words, $w\in\gamma$
(since $\gamma=v_{\sim}$). In other words, $\left(  w_{1},w_{2},\ldots
,w_{k}\right)  \in\gamma$ (since $w=\left(  w_{1},w_{2},\ldots,w_{k}\right)  $).

Let us write the $k$-tuple $w\in V^{k}$ as $w=\left(  w_{1},w_{2},\ldots
,w_{k}\right)  $. Then, it is easy to see that $v_{i}=w_{k}$%
\ \ \ \ \footnote{\textit{Proof.} We have $w=\left(  v_{i+1},v_{i+2}%
,\ldots,v_{k},v_{1},v_{2},\ldots,v_{i}\right)  $. Thus,
\begin{align*}
&  \left(  \text{the last entry of the }k\text{-tuple }w\right) \\
&  =\left(  \text{the last entry of the }k\text{-tuple }\left(  v_{i+1}%
,v_{i+2},\ldots,v_{k},v_{1},v_{2},\ldots,v_{i}\right)  \right) \\
&  =v_{i}\ \ \ \ \ \ \ \ \ \ \left(  \text{since }i\in\left[  k\right]
\right)  .
\end{align*}
Hence,%
\[
v_{i}=\left(  \text{the last entry of the }k\text{-tuple }w\right)  =w_{k}%
\]
(since $w=\left(  w_{1},w_{2},\ldots,w_{k}\right)  $).} and $v_{i+1}=w_{1}%
$\ \ \ \ \footnote{\textit{Proof.} We are in one of the following two cases:
\par
\textit{Case 1:} We have $i\neq k$.
\par
\textit{Case 2:} We have $i=k$.
\par
Let us first consider Case 1. In this case, we have $i\neq k$. Combining
$i\in\left[  k\right]  $ with $i\neq k$, we find $i\in\left[  k\right]
\setminus\left\{  k\right\}  =\left[  k-1\right]  $. However, we have
$w=\left(  v_{i+1},v_{i+2},\ldots,v_{k},v_{1},v_{2},\ldots,v_{i}\right)  $.
Thus,
\begin{align*}
&  \left(  \text{the first entry of the }k\text{-tuple }w\right) \\
&  =\left(  \text{the first entry of the }k\text{-tuple }\left(
v_{i+1},v_{i+2},\ldots,v_{k},v_{1},v_{2},\ldots,v_{i}\right)  \right) \\
&  =v_{i+1}\ \ \ \ \ \ \ \ \ \ \left(  \text{since }i\in\left[  k-1\right]
\right)  .
\end{align*}
Hence, $v_{i+1}=\left(  \text{the first entry of the }k\text{-tuple }w\right)
=w_{1}$ (since $w=\left(  w_{1},w_{2},\ldots,w_{k}\right)  $). Thus,
$v_{i+1}=w_{1}$ is proved in Case 1.
\par
Let us now consider Case 2. In this case, we have $i=k$. Hence, $v_{i+1}%
=v_{k+1}=v_{1}$. However, we have
\begin{align*}
w  &  =\left(  v_{i+1},v_{i+2},\ldots,v_{k},v_{1},v_{2},\ldots,v_{i}\right) \\
&  =\left(  v_{k+1},v_{k+2},\ldots,v_{k},v_{1},v_{2},\ldots,v_{k}\right)
\ \ \ \ \ \ \ \ \ \ \left(  \text{since }i=k\right) \\
&  =\left(  v_{1},v_{2},\ldots,v_{k}\right)  .
\end{align*}
Hence,%
\[
\left(  \text{the first entry of the }k\text{-tuple }w\right)  =v_{1}%
=v_{i+1}\ \ \ \ \ \ \ \ \ \ \left(  \text{since }v_{i+1}=v_{1}\right)  .
\]
Thus, $v_{i+1}=\left(  \text{the first entry of the }k\text{-tuple }w\right)
=w_{1}$ (since $w=\left(  w_{1},w_{2},\ldots,w_{k}\right)  $). Thus,
$v_{i+1}=w_{1}$ is proved in Case 2.
\par
We have now proved $v_{i+1}=w_{1}$ in both Cases 1 and 2. Thus, $v_{i+1}%
=w_{1}$ always holds.}. Now,%
\[
a=\left(  \underbrace{v_{i}}_{=w_{k}},\underbrace{v_{i+1}}_{=w_{1}}\right)
=\left(  w_{k},w_{1}\right)  .
\]

We have thus found a tuple $w=\left(  w_{1},w_{2},\ldots,w_{k}\right)
\in\gamma$ such that $a=\left(  w_{k},w_{1}\right)  $. Hence, such a tuple
exists. This proves Lemma \ref{lem.cycle.rotate}.
\end{proof}

\begin{lemma}
\label{lem.carcs.linear1}Let $V$ be a finite set. Let $\sigma\in
\mathfrak{S}_{V}$ be a permutation of $V$. Let $\gamma$ be a cycle of $\sigma
$. Let $C=\operatorname*{CArcs}\gamma$. Then:

\begin{enumerate}
\item[\textbf{(a)}] The subset $C$ of $V\times V$ is not linear.

\item[\textbf{(b)}] Every proper subset of $C$ is linear.
\end{enumerate}
\end{lemma}

\begin{proof}
[Proof of Lemma \ref{lem.carcs.linear1}.]The cycle $\gamma$ is a cycle of
$\sigma$, and thus is a rotation-equivalence class of nonempty tuples of
distinct elements of $V$ (since any cycle of $\sigma$ is such a class). Hence,
we can write $\gamma$ in the form $\gamma=v_{\sim}$, where $v$ is a nonempty
tuple of distinct elements of $V$. Consider this $v$.

Write the tuple $v$ as $v=\left(  v_{1},v_{2},\ldots,v_{k}\right)  $. Then,
$k\geq1$ (since $v$ is nonempty), and the entries $v_{1},v_{2},\ldots,v_{k}$
of $v$ are distinct (since $v$ is a tuple of distinct elements of $V$). Let us
set $v_{k+1}:=v_{1}$. We have%
\begin{align}
C  &  =\operatorname*{CArcs}\gamma=\operatorname*{CArcs}\left(  v_{\sim
}\right)  \ \ \ \ \ \ \ \ \ \ \left(  \text{since }\gamma=v_{\sim}\right)
\nonumber\\
&  =\operatorname*{CArcs}v\ \ \ \ \ \ \ \ \ \ \left(  \text{by the definition
of }\operatorname*{CArcs}\left(  v_{\sim}\right)  \text{, since }v\in v_{\sim
}\right) \nonumber\\
&  =\left\{  \left(  v_{i},v_{i+1}\right)  \ \mid\ i\in\left[  k\right]
\right\}  \label{pf.lem.carcs.linear1.1}%
\end{align}
(by the definition of $\operatorname*{CArcs}v$, since $v=\left(  v_{1}%
,v_{2},\ldots,v_{k}\right)  $ and $v_{k+1}=v_{1}$). \medskip

\textbf{(a)} Assume the contrary. Thus, $C$ is linear. In other words, $C$ is
the arc set of some path cover of $V$ (by the definition of \textquotedblleft
linear\textquotedblright). Let $P$ be this path cover. Thus, $C$ is the arc
set of $P$. In other words, $C=\operatorname*{Arcs}P$.

We have $C=\operatorname*{Arcs}P=\bigcup_{q\in P}\operatorname*{Arcs}q$ (by
the definition of $\operatorname*{Arcs}P$).

Now, $1\in\left[  k\right]  $ (since $k\geq1$) and therefore
\begin{align*}
\left(  v_{1},v_{2}\right)   &  \in\left\{  \left(  v_{i},v_{i+1}\right)
\ \mid\ i\in\left[  k\right]  \right\}  =C\ \ \ \ \ \ \ \ \ \ \left(  \text{by
(\ref{pf.lem.carcs.linear1.1})}\right) \\
&  =\bigcup_{q\in P}\operatorname*{Arcs}q.
\end{align*}
In other words, there exists a path $q\in P$ such that $\left(  v_{1}%
,v_{2}\right)  \in\operatorname*{Arcs}q$. Consider such a path $q$, and denote
it by $w$. Thus, $w\in P$ and $\left(  v_{1},v_{2}\right)  \in
\operatorname*{Arcs}w$.

Let us write the path $w$ as $w=\left(  w_{1},w_{2},\ldots,w_{\ell}\right)  $.
Thus, the definition of $\operatorname*{Arcs}w$ yields%
\[
\operatorname*{Arcs}w=\left\{  \left(  w_{1},w_{2}\right)  ,\ \left(
w_{2},w_{3}\right)  ,\ \ldots,\ \left(  w_{\ell-1},w_{\ell}\right)  \right\}
.
\]
Hence, $\left(  v_{1},v_{2}\right)  \in\operatorname*{Arcs}w=\left\{  \left(
w_{1},w_{2}\right)  ,\ \left(  w_{2},w_{3}\right)  ,\ \ldots,\ \left(
w_{\ell-1},w_{\ell}\right)  \right\}  $. In other words, there exists a
$p\in\left[  \ell-1\right]  $ such that $\left(  v_{1},v_{2}\right)  =\left(
w_{p},w_{p+1}\right)  $. Consider this $p$.

From $\left(  v_{1},v_{2}\right)  =\left(  w_{p},w_{p+1}\right)  $, we obtain
$v_{1}=w_{p}$ and $v_{2}=w_{p+1}$.

Note that $w_{1},w_{2},\ldots,w_{\ell}$ are the entries of $w$ (because
$w=\left(  w_{1},w_{2},\ldots,w_{\ell}\right)  $).

We know that $w$ is a path of $V$ (since $w\in P$, but $P$ is a path cover of
$V$), thus a nonempty tuple of distinct elements of $V$ (since a path of $V$
is defined to be a nonempty tuple of distinct elements of $V$). Hence, the
entries of $w$ are distinct. In other words, $w_{1},w_{2},\ldots,w_{\ell}$ are
distinct (since $w_{1},w_{2},\ldots,w_{\ell}$ are the entries of $w$).

We now claim the following:

\begin{statement}
\textit{Claim 1:} For each $i\in\left[  k+1\right]  $, we have $p-1+i\in
\left[  \ell\right]  $ and $v_{i}=w_{p-1+i}$.
\end{statement}

[\textit{Proof of Claim 1:} We proceed by induction on $i$:

\textit{Base case:} We have $p-1+1=p\in\left[  \ell-1\right]  \subseteq\left[
\ell\right]  $ and $v_{1}=w_{p}=w_{p-1+1}$ (since $p=p-1+1$). In other words,
Claim 1 holds for $i=1$.

\textit{Induction step:} Let $j\in\left[  k\right]  $. Assume (as the
induction hypothesis) that Claim 1 holds for $i=j$. We must prove that Claim 1
holds for $i=j+1$ as well.

We have assumed that Claim 1 holds for $i=j$. In other words, we have
$p-1+j\in\left[  \ell\right]  $ and $v_{j}=w_{p-1+j}$.

We have $j\in\left[  k\right]  $. Thus,%
\[
\left(  v_{j},v_{j+1}\right)  \in\left\{  \left(  v_{i},v_{i+1}\right)
\ \mid\ i\in\left[  k\right]  \right\}  =\bigcup_{q\in P}\operatorname*{Arcs}%
q.
\]
Hence, there exists a path $q\in P$ such that $\left(  v_{j},v_{j+1}\right)
\in\operatorname*{Arcs}q$. Consider this $q$. From $\left(  v_{j}%
,v_{j+1}\right)  \in\operatorname*{Arcs}q$, it follows that $v_{j}$ and
$v_{j+1}$ are two entries of the path $q$ (since both entries of any pair in
$\operatorname*{Arcs}q$ are entries of $q$). Hence, in particular, $v_{j}$ is
an entry of $q$. Also, from $v_{j}=w_{p-1+j}$, it follows that $v_{j}$ is an
entry of the path $w$ (since $w_{1},w_{2},\ldots,w_{\ell}$ are the entries of
$w$).

Now, $v_{j}$ is both an entry of $q$ and an entry of $w$. In other words,
$v_{j}$ belongs to the paths $q$ and $w$.

Recall that $P$ is a path cover of $V$. Hence, each $v\in V$ belongs to
exactly one of the paths in $P$. In particular, $v_{j}$ belongs to exactly one
of the paths in $P$.

However, we know that $v_{j}$ belongs to the paths $q$ and $w$. If these paths
$q$ and $w$ were distinct, then this would mean that $v_{j}$ belongs to at
least two of the paths in $P$ (since both $q$ and $w$ are paths in $P$); but
this would contradict the fact that $v_{j}$ belongs to exactly one of the
paths in $P$. Hence, the paths $q$ and $w$ cannot be distinct.

In other words, $q=w$. However,
\[
\left(  v_{j},v_{j+1}\right)  \in\operatorname*{Arcs}\underbrace{q}%
_{=w}=\operatorname*{Arcs}w=\left\{  \left(  w_{1},w_{2}\right)  ,\ \left(
w_{2},w_{3}\right)  ,\ \ldots,\ \left(  w_{\ell-1},w_{\ell}\right)  \right\}
.
\]
In other words, there exists some $z\in\left[  \ell-1\right]  $ such that
$\left(  v_{j},v_{j+1}\right)  =\left(  w_{z},w_{z+1}\right)  $. Consider this
$z$.

From $\left(  v_{j},v_{j+1}\right)  =\left(  w_{z},w_{z+1}\right)  $, we
obtain $v_{j}=w_{z}$ and $v_{j+1}=w_{z+1}$. Comparing $v_{j}=w_{z}$ with
$v_{j}=w_{p-1+j}$, we obtain $w_{z}=w_{p-1+j}$. This entails $z=p-1+j$ (since
$w_{1},w_{2},\ldots,w_{\ell}$ are distinct). Hence, $p-1+j=z\in\left[
\ell-1\right]  $, so that $p-1+j\leq\ell-1$. Adding $1$ to both sides of this
inequality, we obtain $p+j\leq\ell$. In other words, $p-1+\left(  j+1\right)
\leq\ell$ (since $p-1+\left(  j+1\right)  =p+j$). Thus, $p-1+\left(
j+1\right)  \in\left[  \ell\right]  $. Moreover, we now know that
$v_{j+1}=w_{z+1}=w_{p-1+\left(  j+1\right)  }$ (since $\underbrace{z}%
_{=p-1+j}+1=p-1+j+1=p-1+\left(  j+1\right)  $).

Altogether, we have shown that $p-1+\left(  j+1\right)  \in\left[
\ell\right]  $ and $v_{j+1}=w_{p-1+\left(  j+1\right)  }$. In other words,
Claim 1 holds for $i=j+1$. This completes the induction step. Thus, Claim 1 is
proven.] \medskip

Now, $k+1\in\left[  k+1\right]  $ (since $k+1\geq k\geq1$). Hence, we can
apply Claim 1 to $i=k+1$, and thus obtain $p-1+\left(  k+1\right)  \in\left[
\ell\right]  $ and $v_{k+1}=w_{p-1+\left(  k+1\right)  }$. In other words,
$p+k\in\left[  \ell\right]  $ and $v_{k+1}=w_{p+k}$ (since $p-1+\left(
k+1\right)  =p+k$).

However, $v_{k+1}=v_{1}=w_{p}$. Hence, $w_{p}=v_{k+1}=w_{p+k}$. This entails
$p=p+k$ (since $w_{1},w_{2},\ldots,w_{\ell}$ are distinct). Hence, $k=0$,
which contradicts $k\geq1>0$. This contradiction shows that our assumption was
false. Thus, Lemma \ref{lem.carcs.linear1} \textbf{(a)} is proved. \medskip

\textbf{(b)} Let $D$ be a proper subset of $C$. We must show that $D$ is linear.

We have $C\setminus D\neq\varnothing$ (since $D$ is a proper subset of $C$).
Hence, there exists some $a\in C\setminus D$. Consider this $a$.

We have $a\in C\setminus D$. In other words, $a\in C$ and $a\notin D$.

We have $a\in C=\operatorname*{CArcs}\gamma$. Hence, Lemma
\ref{lem.cycle.rotate} shows that there exists a tuple $w=\left(  w_{1}%
,w_{2},\ldots,w_{k}\right)  \in\gamma$ such that $a=\left(  w_{k}%
,w_{1}\right)  $. Consider this tuple $w=\left(  w_{1},w_{2},\ldots
,w_{k}\right)  $. The entries of $w$ are $w_{1},w_{2},\ldots,w_{k}$ (since
$w=\left(  w_{1},w_{2},\ldots,w_{k}\right)  $).

The tuple $w$ belongs to $\gamma$ (since $w\in\gamma$), but $\gamma$ is a
rotation-equivalence class of nonempty tuples of distinct elements of $V$.
Hence, $w$ is a nonempty tuple of distinct elements of $V$. Hence, $w$ is a
path of $V$ (by the definition of a \textquotedblleft path of $V$%
\textquotedblright).

In particular, $w$ is a tuple of distinct elements of $V$. In other words, the
entries of $w$ are distinct. In other words, $w_{1},w_{2},\ldots,w_{k}$ are
distinct (since the entries of $w$ are $w_{1},w_{2},\ldots,w_{k}$). The arc
set of this path $w$ is%
\begin{equation}
\operatorname*{Arcs}w=\left\{  \left(  w_{1},w_{2}\right)  ,\ \left(
w_{2},w_{3}\right)  ,\ \ldots,\ \left(  w_{k-1},w_{k}\right)  \right\}
\label{pf.lem.carcs.linear1.b.1}%
\end{equation}
(by (\ref{eq.def.Arcs-and-Carcs.a.2}), applied to $w=\left(  w_{1}%
,w_{2},\ldots,w_{k}\right)  $ instead of $v=\left(  v_{1},v_{2},\ldots
,v_{k}\right)  $).

We have $a=\left(  w_{k},w_{1}\right)  $, thus $\left(  w_{k},w_{1}\right)
=a$.

We know that $\gamma$ is a rotation-equivalence class that contains $w$ (since
$w\in\gamma$). Hence, $\gamma$ is the rotation-equivalence class of $w$. In
other words, $\gamma=w_{\sim}$. Hence,
\begin{align*}
\operatorname*{CArcs}\gamma &  =\operatorname*{CArcs}\left(  w_{\sim}\right)
\\
&  =\operatorname*{CArcs}w\ \ \ \ \ \ \ \ \ \ \left(  \text{by the definition
of }\operatorname*{CArcs}\left(  w_{\sim}\right)  \text{, since }w\in w_{\sim
}\right) \\
&  =\left\{  \left(  w_{1},w_{2}\right)  ,\ \left(  w_{2},w_{3}\right)
,\ \ldots,\ \left(  w_{k-1},w_{k}\right)  ,\ \left(  w_{k},w_{1}\right)
\right\}
\end{align*}
(by (\ref{eq.def.Arcs-and-Carcs.b.2}), applied to $w=\left(  w_{1}%
,w_{2},\ldots,w_{k}\right)  $ instead of $v=\left(  v_{1},v_{2},\ldots
,v_{k}\right)  $). Hence,%
\begin{align*}
\operatorname*{CArcs}\gamma &  =\left\{  \left(  w_{1},w_{2}\right)
,\ \left(  w_{2},w_{3}\right)  ,\ \ldots,\ \left(  w_{k-1},w_{k}\right)
,\ \left(  w_{k},w_{1}\right)  \right\} \\
&  =\underbrace{\left\{  \left(  w_{1},w_{2}\right)  ,\ \left(  w_{2}%
,w_{3}\right)  ,\ \ldots,\ \left(  w_{k-1},w_{k}\right)  \right\}
}_{\substack{=\operatorname*{Arcs}w\\\text{(by (\ref{pf.lem.carcs.linear1.b.1}%
))}}}\cup\left\{  \underbrace{\left(  w_{k},w_{1}\right)  }_{=a}\right\} \\
&  =\left(  \operatorname*{Arcs}w\right)  \cup\left\{  a\right\}  .
\end{align*}

The $k$ pairs $\left(  w_{1},w_{2}\right)  ,\ \left(  w_{2},w_{3}\right)
,\ \ldots,\ \left(  w_{k-1},w_{k}\right)  ,\ \left(  w_{k},w_{1}\right)  $ are
distinct (since their first entries $w_{1},w_{2},\ldots,w_{k}$ are distinct).
Hence, in particular, the last of these $k$ pairs is not among the remaining
$k-1$ pairs. In other words,
\[
\left(  w_{k},w_{1}\right)  \notin\left\{  \left(  w_{1},w_{2}\right)
,\ \left(  w_{2},w_{3}\right)  ,\ \ldots,\ \left(  w_{k-1},w_{k}\right)
\right\}  =\operatorname*{Arcs}w
\]
(by (\ref{pf.lem.carcs.linear1.b.1})). Thus,%
\[
a=\left(  w_{k},w_{1}\right)  \notin\operatorname*{Arcs}w.
\]

Now,%
\begin{align*}
\underbrace{C}_{\substack{=\operatorname*{CArcs}\gamma\\=\left(
\operatorname*{Arcs}w\right)  \cup\left\{  a\right\}  }}\setminus\left\{
a\right\}   &  =\left(  \left(  \operatorname*{Arcs}w\right)  \cup\left\{
a\right\}  \right)  \setminus\left\{  a\right\}  =\left(  \operatorname*{Arcs}%
w\right)  \setminus\left\{  a\right\} \\
&  =\operatorname*{Arcs}w\ \ \ \ \ \ \ \ \ \ \left(  \text{since }%
a\notin\operatorname*{Arcs}w\right)  .
\end{align*}

However, $D$ is a subset of $C$ that does not contain $a$ (since $a\notin D$).
In other words, $D$ is a subset of $C\setminus\left\{  a\right\}  $. In other
words, $D$ is a subset of $\operatorname*{Arcs}w$ (since $C\setminus\left\{
a\right\}  =\operatorname*{Arcs}w$).

However, $w$ is a path of $V$. Hence, Lemma \ref{lem.linear.sub-of-path}
(applied to $p=w$) shows that $\operatorname*{Arcs}w$ is a linear subset of
$V\times V$. Thus, Proposition \ref{prop.linear.subset} (applied to
$F=\operatorname*{Arcs}w$) shows that any subset of $\operatorname*{Arcs}w$ is
linear as well. Thus, $D$ is linear (since $D$ is a subset of
$\operatorname*{Arcs}w$). This completes the proof of Lemma
\ref{lem.carcs.linear1} \textbf{(b)}.
\end{proof}

\begin{lemma}
\label{lem.carcs.len=num}Let $V$ be a finite set. Let $\sigma\in
\mathfrak{S}_{V}$ be a permutation of $V$. Let $\gamma$ be a cycle of $\sigma
$. Then, $\left\vert \operatorname*{CArcs}\gamma\right\vert =\ell\left(
\gamma\right)  \geq1$.
\end{lemma}

\begin{proof}
The cycle $\gamma$ is a cycle of $\sigma$, and thus is a rotation-equivalence
class of nonempty tuples of distinct elements of $V$ (since any cycle of
$\sigma$ is such a class). Hence, we can write $\gamma$ in the form
$\gamma=v_{\sim}$, where $v$ is a nonempty tuple of distinct elements of $V$.
Consider this $v$.

Write the tuple $v$ as $v=\left(  v_{1},v_{2},\ldots,v_{k}\right)  $. Then,
the entries $v_{1},v_{2},\ldots,v_{k}$ of $v$ are distinct (since $v$ is a
tuple of distinct elements of $V$). We have $\ell\left(  v\right)  =k$ (since
$v=\left(  v_{1},v_{2},\ldots,v_{k}\right)  $ is a $k$-tuple) and $\ell\left(
v\right)  \geq1$ (since $v$ is nonempty).

From $\gamma=v_{\sim}$, we obtain $\ell\left(  \gamma\right)  =\ell\left(
v_{\sim}\right)  =\ell\left(  v\right)  $ (by Definition
\ref{def.reqc.features} \textbf{(a)}). Hence, $\ell\left(  \gamma\right)
=\ell\left(  v\right)  =k$ and $\ell\left(  \gamma\right)  =\ell\left(
v\right)  \geq1$.

On the other hand, from $\gamma=v_{\sim}$, we obtain%
\begin{align*}
\operatorname*{CArcs}\gamma &  =\operatorname*{CArcs}\left(  v_{\sim}\right)
\\
&  =\operatorname*{CArcs}v\ \ \ \ \ \ \ \ \ \ \left(  \text{by the definition
of }\operatorname*{CArcs}\left(  v_{\sim}\right)  \text{, since }v\in v_{\sim
}\right) \\
&  =\left\{  \left(  v_{1},v_{2}\right)  ,\ \left(  v_{2},v_{3}\right)
,\ \ldots,\ \left(  v_{k-1},v_{k}\right)  ,\ \left(  v_{k},v_{1}\right)
\right\}
\end{align*}
(by (\ref{eq.def.Arcs-and-Carcs.b.2}), since $v=\left(  v_{1},v_{2}%
,\ldots,v_{k}\right)  $).

However, the $k$ pairs $\left(  v_{1},v_{2}\right)  ,\ \left(  v_{2}%
,v_{3}\right)  ,\ \ldots,\ \left(  v_{k-1},v_{k}\right)  ,\ \left(
v_{k},v_{1}\right)  $ are distinct (since their first entries $v_{1}%
,v_{2},\ldots,v_{k}$ are distinct). Hence, the set \newline$\left\{  \left(
v_{1},v_{2}\right)  ,\ \left(  v_{2},v_{3}\right)  ,\ \ldots,\ \left(
v_{k-1},v_{k}\right)  ,\ \left(  v_{k},v_{1}\right)  \right\}  $ of these $k$
pairs has size $k$. In other words,%
\[
\left\vert \left\{  \left(  v_{1},v_{2}\right)  ,\ \left(  v_{2},v_{3}\right)
,\ \ldots,\ \left(  v_{k-1},v_{k}\right)  ,\ \left(  v_{k},v_{1}\right)
\right\}  \right\vert =k.
\]
This rewrites as $\left\vert \operatorname*{CArcs}\gamma\right\vert
=\ell\left(  \gamma\right)  $ (since we have $\ell\left(  \gamma\right)  =k$
and \newline$\operatorname*{CArcs}\gamma=\left\{  \left(  v_{1},v_{2}\right)
,\ \left(  v_{2},v_{3}\right)  ,\ \ldots,\ \left(  v_{k-1},v_{k}\right)
,\ \left(  v_{k},v_{1}\right)  \right\}  $). Hence, $\left\vert
\operatorname*{CArcs}\gamma\right\vert =\ell\left(  \gamma\right)  \geq1$.
Thus, Lemma \ref{lem.carcs.len=num} is proved.
\end{proof}

\begin{lemma}
\label{lem.carcs.linear2}Let $D=\left(  V,A\right)  $ be a digraph. Let
$\sigma\in\mathfrak{S}_{V}$ be a permutation of $V$. Let $\gamma$ be a cycle
of $\sigma$. Let $C=\operatorname*{CArcs}\gamma$. Then:

\begin{enumerate}
\item[\textbf{(a)}] If $\gamma$ is a $D$-cycle, then $\sum
_{\substack{F\subseteq C\cap A\\\text{is linear}}}\left(  -1\right)
^{\left\vert F\right\vert }=\left(  -1\right)  ^{\ell\left(  \gamma\right)
-1}$.

\item[\textbf{(b)}] If $\gamma$ is a $\overline{D}$-cycle, then $\sum
_{\substack{F\subseteq C\cap A\\\text{is linear}}}\left(  -1\right)
^{\left\vert F\right\vert }=1$.

\item[\textbf{(c)}] If $\gamma$ is neither a $D$-cycle nor a $\overline{D}%
$-cycle, then $\sum_{\substack{F\subseteq C\cap A\\\text{is linear}}}\left(
-1\right)  ^{\left\vert F\right\vert }=0$.
\end{enumerate}
\end{lemma}

\begin{proof}
We know that $\gamma$ is a cycle of $\sigma$, thus a rotation-equivalence
class of nonempty tuples of distinct elements of $V$ (since any cycle of
$\sigma$ is such a rotation-equivalence class).

Lemma \ref{lem.carcs.len=num} yields $\left\vert \operatorname*{CArcs}%
\gamma\right\vert =\ell\left(  \gamma\right)  \geq1$. From
$C=\operatorname*{CArcs}\gamma$, we obtain $\left\vert C\right\vert
=\left\vert \operatorname*{CArcs}\gamma\right\vert =\ell\left(  \gamma\right)
\geq1>0$, so that the set $C$ is nonempty. In other words, $C\neq\varnothing$.
Also, $C=\operatorname*{CArcs}\gamma\subseteq V\times V$.\medskip

\textbf{(a)} Assume that $\gamma$ is a $D$-cycle. Thus, $\operatorname*{CArcs}%
\gamma\subseteq A$ (by the definition of a $D$-cycle). In other words,
$C\subseteq A$ (since $C=\operatorname*{CArcs}\gamma$). Therefore, $C\cap
A=C$. Hence,%
\begin{equation}
\sum_{\substack{F\subseteq C\cap A\\\text{is linear}}}\left(  -1\right)
^{\left\vert F\right\vert }=\sum_{\substack{F\subseteq C\\\text{is linear}%
}}\left(  -1\right)  ^{\left\vert F\right\vert }=\sum_{\substack{F\text{ is a
linear}\\\text{subset of }C}}\left(  -1\right)  ^{\left\vert F\right\vert }.
\label{pf.lem.carcs.linear2.1}%
\end{equation}

On the other hand, Lemma \ref{lem.cancel} (applied to $B=C$) yields
$\sum_{F\subseteq C}\left(  -1\right)  ^{\left\vert F\right\vert }=\left[
C=\varnothing\right]  =0$ (since $C\neq\varnothing$). Hence,%
\[
0=\sum_{F\subseteq C}\left(  -1\right)  ^{\left\vert F\right\vert }=\left(
-1\right)  ^{\left\vert C\right\vert }+\sum_{\substack{F\subseteq C;\\F\neq
C}}\left(  -1\right)  ^{\left\vert F\right\vert }%
\]
(here, we have split off the addend for $F=C$ from the sum). Therefore,%
\begin{equation}
\sum_{\substack{F\subseteq C;\\F\neq C}}\left(  -1\right)  ^{\left\vert
F\right\vert }=-\left(  -1\right)  ^{\left\vert C\right\vert }=\left(
-1\right)  ^{\left\vert C\right\vert -1}=\left(  -1\right)  ^{\ell\left(
\gamma\right)  -1} \label{pf.lem.carcs.linear2.2}%
\end{equation}
(since $\left\vert C\right\vert =\ell\left(  \gamma\right)  $).

However, from Lemma \ref{lem.carcs.linear1}, we easily obtain%
\[
\left\{  F\subseteq C\ \mid\ F\neq C\right\}  =\left\{  \text{linear subsets
of }C\right\}
\]
\footnote{\textit{Proof.} We shall first prove that $\left\{  F\subseteq
C\ \mid\ F\neq C\right\}  \subseteq\left\{  \text{linear subsets of
}C\right\}  $.
\par
Indeed, let $G\in\left\{  F\subseteq C\ \mid\ F\neq C\right\}  $. Thus, $G$ is
a subset $F\subseteq C$ satisfying $F\neq C$. In other words, $G$ is a subset
of $C$ such that $G\neq C$. In other words, $G$ is a proper subset of $C$.
Hence, $G$ is linear (since Lemma \ref{lem.carcs.linear1} \textbf{(b)} shows
that every proper subset of $C$ is linear). Therefore, $G$ is a linear subset
of $C$. Hence, $G\in\left\{  \text{linear subsets of }C\right\}  $.
\par
Forget that we fixed $G$. We thus have proved that each $G\in\left\{
F\subseteq C\ \mid\ F\neq C\right\}  $ satisfies $G\in\left\{  \text{linear
subsets of }C\right\}  $. In other words, we have%
\[
\left\{  F\subseteq C\ \mid\ F\neq C\right\}  \subseteq\left\{  \text{linear
subsets of }C\right\}  .
\]
\par
Let us now prove the reverse inclusion.
\par
Indeed, let $H\in\left\{  \text{linear subsets of }C\right\}  $. Thus, $H$ is
a linear subset of $C$. Hence, $H$ is a subset of $C$, so that $H\subseteq C$.
The set $H$ is linear, whereas the set $C$ is not (by Lemma
\ref{lem.carcs.linear1} \textbf{(a)}). Thus, $H$ cannot be identical with $C$.
In other words, $H\neq C$. Combining $H\subseteq C$ with $H\neq C$, we see
that $H$ is a subset $F$ of $C$ satisfying $F\neq C$. In other words,
$H\in\left\{  F\subseteq C\ \mid\ F\neq C\right\}  $.
\par
Forget that we fixed $H$. We thus have proved that each $H\in\left\{
\text{linear subsets of }C\right\}  $ satisfies $H\in\left\{  F\subseteq
C\ \mid\ F\neq C\right\}  $. In other words, we have%
\[
\left\{  \text{linear subsets of }C\right\}  \subseteq\left\{  F\subseteq
C\ \mid\ F\neq C\right\}  .
\]
Combining this with%
\[
\left\{  F\subseteq C\ \mid\ F\neq C\right\}  \subseteq\left\{  \text{linear
subsets of }C\right\}  ,
\]
we obtain
\[
\left\{  F\subseteq C\ \mid\ F\neq C\right\}  =\left\{  \text{linear subsets
of }C\right\}  .
\]
}. Hence, the summation sign \textquotedblleft$\sum_{\substack{F\subseteq
C;\\F\neq C}}$\textquotedblright\ can be rewritten as \textquotedblleft%
$\sum_{\substack{F\text{ is a linear}\\\text{subset of }C}}$\textquotedblright%
. Thus,
\[
\sum_{\substack{F\subseteq C;\\F\neq C}}\left(  -1\right)  ^{\left\vert
F\right\vert }=\sum_{\substack{F\text{ is a linear}\\\text{subset of }%
C}}\left(  -1\right)  ^{\left\vert F\right\vert }.
\]
Comparing this with (\ref{pf.lem.carcs.linear2.1}), we find%
\[
\sum_{\substack{F\subseteq C\cap A\\\text{is linear}}}\left(  -1\right)
^{\left\vert F\right\vert }=\sum_{\substack{F\subseteq C;\\F\neq C}}\left(
-1\right)  ^{\left\vert F\right\vert }=\left(  -1\right)  ^{\ell\left(
\gamma\right)  -1}\ \ \ \ \ \ \ \ \ \ \left(  \text{by
(\ref{pf.lem.carcs.linear2.2})}\right)  .
\]
This proves Lemma \ref{lem.carcs.linear2} \textbf{(a)}. \medskip

\textbf{(b)} Assume that $\gamma$ is a $\overline{D}$-cycle. Thus,
$\operatorname*{CArcs}\gamma\subseteq\left(  V\times V\right)  \setminus A$
(by the definition of a $\overline{D}$-cycle, since $\overline{D}=\left(
V,\ \left(  V\times V\right)  \setminus A\right)  $). In other words,
$C\subseteq\left(  V\times V\right)  \setminus A$ (since
$C=\operatorname*{CArcs}\gamma$). Therefore, $C\cap A=\varnothing
$\ \ \ \ \footnote{\textit{Proof.} Let $c\in C\cap A$. Thus, $c\in C\cap
A\subseteq C\subseteq\left(  V\times V\right)  \setminus A$, so that $c\notin
A$. But this contradicts $c\in C\cap A\subseteq A$.
\par
Forget that we fixed $c$. We thus have obtained a contradiction for each $c\in
C\cap A$. Hence, there is no $c\in C\cap A$. In other words, $C\cap
A=\varnothing$.}. Hence,%
\begin{equation}
\sum_{\substack{F\subseteq C\cap A\\\text{is linear}}}\left(  -1\right)
^{\left\vert F\right\vert }=\sum_{\substack{F\subseteq\varnothing\\\text{is
linear}}}\left(  -1\right)  ^{\left\vert F\right\vert }.
\label{pf.lem.carcs.linear2.4}%
\end{equation}
However, the set $\varnothing$ is linear (as a subset of $V\times
V$)\ \ \ \ \footnote{\textit{Proof.} We have $\varnothing\subseteq C$ and
$\varnothing\neq C$ (since $C\neq\varnothing$). Hence, $\varnothing$ is a
proper subset of $C$. However, Lemma \ref{lem.carcs.linear1} \textbf{(b)}
shows that every proper subset of $C$ is linear. Thus, $\varnothing$ is linear
(since $\varnothing$ is a proper subset of $C$).}. Thus, the only linear
subset of $\varnothing$ is $\varnothing$ itself (since $\varnothing$ is a
linear subset of $\varnothing$, and is clearly the only subset of
$\varnothing$). Therefore, the sum $\sum_{\substack{F\subseteq\varnothing
\\\text{is linear}}}\left(  -1\right)  ^{\left\vert F\right\vert }$ has only
one addend, namely the addend for $F=\varnothing$. Thus, this sum rewrites as
follows:%
\begin{align*}
\sum_{\substack{F\subseteq\varnothing\\\text{is linear}}}\left(  -1\right)
^{\left\vert F\right\vert }  &  =\left(  -1\right)  ^{\left\vert
\varnothing\right\vert }=\left(  -1\right)  ^{0}\ \ \ \ \ \ \ \ \ \ \left(
\text{since }\left\vert \varnothing\right\vert =0\right) \\
&  =1.
\end{align*}
Hence, (\ref{pf.lem.carcs.linear2.4}) rewrites as%
\[
\sum_{\substack{F\subseteq C\cap A\\\text{is linear}}}\left(  -1\right)
^{\left\vert F\right\vert }=1.
\]
This proves Lemma \ref{lem.carcs.linear2} \textbf{(b)}. \medskip

\textbf{(c)} Assume that $\gamma$ is neither a $D$-cycle nor a $\overline{D}%
$-cycle. Hence, $C\not \subseteq A$\ \ \ \ \footnote{\textit{Proof.} Assume
the contrary. Thus, $C\subseteq A$. In other words, $\operatorname*{CArcs}%
\gamma\subseteq A$ (since $C=\operatorname*{CArcs}\gamma$).
\par
Recall that $\gamma$ is a rotation-equivalence class of nonempty tuples of
distinct elements of $V$. Hence, from $\operatorname*{CArcs}\gamma\subseteq
A$, we conclude that $\gamma$ is a $D$-cycle (by the definition of a
$D$-cycle). This contradicts the fact that $\gamma$ is not a $D$-cycle. This
contradiction shows that our assumption was false, qed.}. Hence, $C\cap A\neq
C$. However, $C\cap A$ is clearly a subset of $C$. Thus, $C\cap A$ is a proper
subset of $C$ (since $C\cap A\neq C$).

Furthermore, $C\cap A\neq\varnothing$\ \ \ \ \footnote{\textit{Proof.} Assume
the contrary. Thus, $C\cap A=\varnothing$. However, we have $X\setminus
Y=X\setminus\left(  X\cap Y\right)  $ for any two sets $X$ and $Y$. Applying
this to $X=C$ and $Y=A$, we obtain $C\setminus A=C\setminus\underbrace{\left(
C\cap A\right)  }_{=\varnothing}=C\setminus\varnothing=C$. Therefore,
$C=\underbrace{C}_{\subseteq V\times V}\setminus A\subseteq\left(  V\times
V\right)  \setminus A$. In other words, $\operatorname*{CArcs}\gamma
\subseteq\left(  V\times V\right)  \setminus A$ (since
$C=\operatorname*{CArcs}\gamma$).
\par
Recall that $\gamma$ is a rotation-equivalence class of nonempty tuples of
distinct elements of $V$. Hence, from $\operatorname*{CArcs}\gamma
\subseteq\left(  V\times V\right)  \setminus A$, we conclude that $\gamma$ is
a $\overline{D}$-cycle (by the definition of a $\overline{D}$-cycle, since
$\overline{D}=\left(  V,\ \left(  V\times V\right)  \setminus A\right)  $).
This contradicts the fact that $\gamma$ is not a $\overline{D}$-cycle. This
contradiction shows that our assumption was false, qed.}.

Now, from Lemma \ref{lem.carcs.linear1}, we easily obtain%
\[
\left\{  F\subseteq C\cap A\right\}  =\left\{  \text{linear subsets of }C\cap
A\right\}
\]
\footnote{\textit{Proof.} We shall first prove that $\left\{  F\subseteq C\cap
A\right\}  \subseteq\left\{  \text{linear subsets of }C\cap A\right\}  $.
\par
Indeed, let $G\in\left\{  F\subseteq C\cap A\right\}  $. Thus, $G$ is a subset
of $C\cap A$. Hence, $G$ is a proper subset of $C$ (since $G$ is a subset of
$C\cap A$, but $C\cap A$ is a proper subset of $C$). Hence, $G$ is linear
(since Lemma \ref{lem.carcs.linear1} \textbf{(b)} shows that every proper
subset of $C$ is linear). Therefore, $G$ is a linear subset of $C\cap A$.
Hence, $G\in\left\{  \text{linear subsets of }C\cap A\right\}  $.
\par
Forget that we fixed $G$. We thus have proved that each $G\in\left\{
F\subseteq C\cap A\right\}  $ satisfies $G\in\left\{  \text{linear subsets of
}C\cap A\right\}  $. In other words, we have%
\[
\left\{  F\subseteq C\cap A\right\}  \subseteq\left\{  \text{linear subsets of
}C\cap A\right\}  .
\]
\par
Let us now prove the reverse inclusion.
\par
Indeed, let $H\in\left\{  \text{linear subsets of }C\cap A\right\}  $. Thus,
$H$ is a linear subset of $C\cap A$. Hence, $H$ is a subset of $C\cap A$, so
that $H\subseteq C\cap A$. In other words, $H\in\left\{  F\subseteq C\cap
A\right\}  $.
\par
Forget that we fixed $H$. We thus have proved that each $H\in\left\{
\text{linear subsets of }C\cap A\right\}  $ satisfies $H\in\left\{  F\subseteq
C\cap A\right\}  $. In other words, we have%
\[
\left\{  \text{linear subsets of }C\cap A\right\}  \subseteq\left\{
F\subseteq C\cap A\right\}  .
\]
Combining this with%
\[
\left\{  F\subseteq C\cap A\right\}  \subseteq\left\{  \text{linear subsets of
}C\cap A\right\}  ,
\]
we obtain
\[
\left\{  F\subseteq C\cap A\right\}  =\left\{  \text{linear subsets of }C\cap
A\right\}  .
\]
}. Hence, the summation sign \textquotedblleft$\sum_{F\subseteq C\cap A}%
$\textquotedblright\ can be rewritten as \textquotedblleft$\sum
_{\substack{F\text{ is a linear}\\\text{subset of }C\cap A}}$%
\textquotedblright. Thus,
\[
\sum_{F\subseteq C\cap A}\left(  -1\right)  ^{\left\vert F\right\vert }%
=\sum_{\substack{F\text{ is a linear}\\\text{subset of }C\cap A}}\left(
-1\right)  ^{\left\vert F\right\vert }=\sum_{\substack{F\subseteq C\cap
A\\\text{is linear}}}\left(  -1\right)  ^{\left\vert F\right\vert }.
\]
Therefore,%
\begin{align*}
\sum_{\substack{F\subseteq C\cap A\\\text{is linear}}}\left(  -1\right)
^{\left\vert F\right\vert }  &  =\sum_{F\subseteq C\cap A}\left(  -1\right)
^{\left\vert F\right\vert }\\
&  =\left[  C\cap A=\varnothing\right]  \ \ \ \ \ \ \ \ \ \ \left(  \text{by
Lemma \ref{lem.cancel}, applied to }B=C\cap A\right) \\
&  =0\ \ \ \ \ \ \ \ \ \ \left(  \text{since }C\cap A\neq\varnothing\right)  .
\end{align*}
This proves Lemma \ref{lem.carcs.linear2} \textbf{(c)}.
\end{proof}

Before we state the next lemma, let us again recall that the symbols
\textquotedblleft$\sqcup$\textquotedblright\ and \textquotedblleft$\bigsqcup
$\textquotedblright\ stand for unions of disjoint sets.

\begin{lemma}
\label{lem.carcs.linear-split}Let $D=\left(  V,A\right)  $ be a digraph. Let
$\sigma\in\mathfrak{S}_{V}$ be a permutation of $V$. Let $\gamma_{1}%
,\gamma_{2},\ldots,\gamma_{k}$ be the cycles of $\sigma$, listed with no
repetition\footnotemark. For each $i\in\left[  k\right]  $, let $C_{i}%
:=\operatorname*{CArcs}\left(  \gamma_{i}\right)  $.

Let $F$ be any set. Then:

\begin{enumerate}
\item[\textbf{(a)}] The set $F$ is a linear subset of $\mathbf{A}_{\sigma}\cap
A$ if and only if $F$ can be written as $F=\bigsqcup_{j\in\left[  k\right]
}F_{j}$, where each $F_{j}$ is a linear subset of $C_{j}\cap A$.

\item[\textbf{(b)}] In this case, the subsets $F_{j}$ are uniquely determined
by $F$ (namely, $F_{j}=F\cap C_{j}$ for each $j\in\left[  k\right]  $).
\end{enumerate}
\end{lemma}

\footnotetext{Keep in mind that a cycle is a rotation-equivalence class. Thus,
\textquotedblleft listed with no repetition\textquotedblright\ means that no
two of $\gamma_{1},\gamma_{2},\ldots,\gamma_{k}$ are the same
rotation-equivalence class. For example, if $\gamma_{1}$ is $\left(
1,2\right)  _{\sim}$, then $\gamma_{2}$ cannot be $\left(  2,1\right)  _{\sim
}$.}

\begin{proof}
Recall that $\gamma_{1},\gamma_{2},\ldots,\gamma_{k}$ are the cycles of
$\sigma$, listed with no repetition. Thus, these cycles $\gamma_{1},\gamma
_{2},\ldots,\gamma_{k}$ are distinct, and we have $\operatorname*{Cycs}%
\sigma=\left\{  \gamma_{1},\gamma_{2},\ldots,\gamma_{k}\right\}  $ (since
$\operatorname*{Cycs}\sigma$ is defined to be the set of all cycles of
$\sigma$).

We know that $\sigma$ is a permutation of $V$. Hence, each element of $V$
belongs to exactly one cycle of $\sigma$. In other words, each element of $V$
belongs to exactly one of the cycles $\gamma_{1},\gamma_{2},\ldots,\gamma_{k}$
(since $\gamma_{1},\gamma_{2},\ldots,\gamma_{k}$ are the cycles of $\sigma$,
listed with no repetition).

For each $i\in\left[  k\right]  $, let $V_{i}$ be the set of all entries of
the cycle $\gamma_{i}$. Thus, $V_{1},V_{2},\ldots,V_{k}$ are $k$ subsets of
$V$. These $k$ subsets $V_{1},V_{2},\ldots,V_{k}$ are furthermore
disjoint\footnote{\textit{Proof.} Let $i$ and $j$ be two distinct elements of
$\left[  k\right]  $. We shall show that $V_{i}\cap V_{j}=\varnothing$.
\par
Indeed, assume the contrary. Thus, $V_{i}\cap V_{j}\neq\varnothing$. Hence,
there exists some element $v\in V_{i}\cap V_{j}$. Consider this $v$.
\par
We have $v\in V_{i}\cap V_{j}\subseteq V_{i}$. In other words, $v$ is an entry
of the cycle $\gamma_{i}$ (since $V_{i}$ was defined as the set of all entries
of the cycle $\gamma_{i}$).
\par
We have $v\in V_{i}\cap V_{j}\subseteq V_{j}$. In other words, $v$ is an entry
of the cycle $\gamma_{j}$ (since $V_{j}$ was defined as the set of all entries
of the cycle $\gamma_{j}$).
\par
The two cycles $\gamma_{i}$ and $\gamma_{j}$ have the entry $v$ in common
(since $v$ is an entry of the cycle $\gamma_{i}$, and since $v$ is an entry of
the cycle $\gamma_{j}$).
\par
However, the cycles $\gamma_{1},\gamma_{2},\ldots,\gamma_{k}$ are distinct.
Hence, the two cycles $\gamma_{i}$ and $\gamma_{j}$ are distinct (since $i$
and $j$ are distinct). However, two distinct cycles of $\sigma$ cannot have
any entries in common (by the basic properties of the cycles of a
permutation). Thus, $\gamma_{i}$ and $\gamma_{j}$ have no entries in common
(since $\gamma_{i}$ and $\gamma_{j}$ are two distinct cycles of $\sigma$).
This contradicts the fact that the two cycles $\gamma_{i}$ and $\gamma_{j}$
have the entry $v$ in common. This contradiction shows that our assumption was
false. In other words, we have $V_{i}\cap V_{j}=\varnothing$.
\par
Forget that we fixed $i$ and $j$. We thus have proved that $V_{i}\cap
V_{j}=\varnothing$ whenever $i$ and $j$ are two distinct elements of $\left[
k\right]  $. In other words, the subsets $V_{1},V_{2},\ldots,V_{k}$ are
disjoint.}. In other words, the sets $V_{j}$ for different $j\in\left[
k\right]  $ are disjoint.

For any given $i\in\left[  k\right]  $ and any given $v\in V$, we have the
following logical equivalence:%
\begin{equation}
\left(  v\in V_{i}\right)  \Longleftrightarrow\left(  v\text{ is an entry of
the cycle }\gamma_{i}\right)  \label{pf.lem.carcs.linear-split.equiva}%
\end{equation}
(since $V_{i}$ is the set of all entries of the cycle $\gamma_{i}$).

We have $V=V_{1}\cup V_{2}\cup\cdots\cup V_{k}$%
\ \ \ \ \footnote{\textit{Proof.} Let $v\in V$. Then, $v$ belongs to exactly
one of the cycles $\gamma_{1},\gamma_{2},\ldots,\gamma_{k}$ (since each
element of $V$ belongs to exactly one of the cycles $\gamma_{1},\gamma
_{2},\ldots,\gamma_{k}$). In other words, there is exactly one $i\in\left[
k\right]  $ such that $v$ belongs to $\gamma_{i}$. Consider this $i$. Now, $v$
belongs to $\gamma_{i}$. In other words, $v$ is an entry of the cycle
$\gamma_{i}$. Hence, $v\in V_{i}$ (by (\ref{pf.lem.carcs.linear-split.equiva}%
)). Thus, $v\in V_{i}\subseteq V_{1}\cup V_{2}\cup\cdots\cup V_{k}$ (since
$V_{i}$ is one of the terms in the union $V_{1}\cup V_{2}\cup\cdots\cup V_{k}%
$).
\par
Forget that we fixed $v$. We thus have shown that $v\in V_{1}\cup V_{2}%
\cup\cdots\cup V_{k}$ for each $v\in V$. In other words, $V\subseteq V_{1}\cup
V_{2}\cup\cdots\cup V_{k}$. Combining this with $V_{1}\cup V_{2}\cup\cdots\cup
V_{k}\subseteq V$ (which is obvious, since $V_{1},V_{2},\ldots,V_{k}$ are
subsets of $V$), we obtain $V=V_{1}\cup V_{2}\cup\cdots\cup V_{k}$.}. Thus,%
\[
V=V_{1}\cup V_{2}\cup\cdots\cup V_{k}=\bigcup_{j\in\left[  k\right]  }%
V_{j}=\bigsqcup_{j\in\left[  k\right]  }V_{j}%
\]
(since the sets $V_{j}$ for different $j\in\left[  k\right]  $ are disjoint).

For each $i\in\left[  k\right]  $, we have $C_{i}\subseteq V_{i}\times V_{i}%
$\ \ \ \ \footnote{\textit{Proof.} Let $i\in\left[  k\right]  $. Then,
$C_{i}=\operatorname*{CArcs}\left(  \gamma_{i}\right)  $ (by the definition of
$C_{i}$). Write the cycle $\gamma_{i}$ in the form $\gamma_{i}=\left(
a_{1},a_{2},\ldots,a_{p}\right)  _{\sim}$. Thus, the entries of the cycle
$\gamma_{i}$ are $a_{1},a_{2},\ldots,a_{p}$. Hence, $V_{i}=\left\{
a_{1},a_{2},\ldots,a_{p}\right\}  $ (since $V_{i}$ was defined as the set of
all entries of the cycle $\gamma_{i}$). Therefore, $a_{1},a_{2},\ldots,a_{p}$
are elements of $V_{i}$. Hence, the pairs $\left(  a_{1},a_{2}\right)
,\ \left(  a_{2},a_{3}\right)  ,\ \ldots,\ \left(  a_{p-1},a_{p}\right)
,\ \left(  a_{p},a_{1}\right)  $ are elements of $V_{i}\times V_{i}$ (since
all their entries belong to $V_{i}$).
\par
However, from $\gamma_{i}=\left(  a_{1},a_{2},\ldots,a_{p}\right)  _{\sim}$,
we obtain%
\begin{align*}
\operatorname*{CArcs}\left(  \gamma_{i}\right)   &  =\operatorname*{CArcs}%
\left(  \left(  a_{1},a_{2},\ldots,a_{p}\right)  _{\sim}\right) \\
&  =\operatorname*{CArcs}\left(  a_{1},a_{2},\ldots,a_{p}\right)
\ \ \ \ \ \ \ \ \ \ \left(
\begin{array}
[c]{c}%
\text{since Definition \ref{def.reqc.features} \textbf{(b)} yields}\\
\text{that }\operatorname*{CArcs}\left(  w_{\sim}\right)
=\operatorname*{CArcs}w\text{ for each }w\in V^{p}%
\end{array}
\right) \\
&  =\left\{  \left(  a_{1},a_{2}\right)  ,\ \left(  a_{2},a_{3}\right)
,\ \ldots,\ \left(  a_{p-1},a_{p}\right)  ,\ \left(  a_{p},a_{1}\right)
\right\} \\
&  \ \ \ \ \ \ \ \ \ \ \ \ \ \ \ \ \ \ \ \ \left(
\begin{array}
[c]{c}%
\text{by (\ref{eq.def.Arcs-and-Carcs.b.2}), applied to }a=\left(  a_{1}%
,a_{2},\ldots,a_{p}\right) \\
\text{instead of }v=\left(  v_{1},v_{2},\ldots,v_{k}\right)
\end{array}
\right) \\
&  \subseteq V_{i}\times V_{i}%
\end{align*}
(since the pairs $\left(  a_{1},a_{2}\right)  ,\ \left(  a_{2},a_{3}\right)
,\ \ldots,\ \left(  a_{p-1},a_{p}\right)  ,\ \left(  a_{p},a_{1}\right)  $ are
elements of $V_{i}\times V_{i}$). Therefore,%
\[
C_{i}=\operatorname*{CArcs}\left(  \gamma_{i}\right)  \subseteq V_{i}\times
V_{i},
\]
qed.}. Renaming the variable $i$ as $j$ in this sentence, we obtain the
following: For each $j\in\left[  k\right]  $, we have%
\begin{equation}
C_{j}\subseteq V_{j}\times V_{j}. \label{pf.lem.carcs.linear-split.Cjsub}%
\end{equation}
Hence, the sets $C_{1},C_{2},\ldots,C_{k}$ are
disjoint\footnote{\textit{Proof.} Let $i$ and $j$ be two distinct elements of
$\left[  k\right]  $. We shall show that $C_{i}\cap C_{j}=\varnothing$.
\par
Indeed, $i$ and $j$ are distinct, and thus we have $V_{i}\cap V_{j}%
=\varnothing$ (since the $k$ subsets $V_{1},V_{2},\ldots,V_{k}$ are disjoint).
However, $C_{j}\subseteq V_{j}\times V_{j}$ (by
(\ref{pf.lem.carcs.linear-split.Cjsub})) and $C_{i}\subseteq V_{i}\times
V_{i}$ (similarly). Hence,%
\[
\underbrace{C_{i}}_{\subseteq V_{i}\times V_{i}}\cap\underbrace{C_{j}%
}_{\subseteq V_{j}\times V_{j}}\subseteq\left(  V_{i}\times V_{i}\right)
\cap\left(  V_{j}\times V_{j}\right)  =\underbrace{\left(  V_{i}\cap
V_{j}\right)  }_{=\varnothing}\times\underbrace{\left(  V_{i}\cap
V_{j}\right)  }_{=\varnothing}=\varnothing\times\varnothing=\varnothing.
\]
Therefore, $C_{i}\cap C_{j}=\varnothing$.
\par
Forget that we fixed $i$ and $j$. We thus have shown that $C_{i}\cap
C_{j}=\varnothing$ whenever $i$ and $j$ are two distinct elements of $\left[
k\right]  $. In other words, the sets $C_{1},C_{2},\ldots,C_{k}$ are
disjoint.}.

The definition of $\mathbf{A}_{\sigma}$ yields%
\begin{align}
\mathbf{A}_{\sigma}  &  =\bigcup_{c\in\operatorname*{Cycs}\sigma
}\operatorname*{CArcs}c\nonumber\\
&  =\left(  \operatorname*{CArcs}\left(  \gamma_{1}\right)  \right)
\cup\left(  \operatorname*{CArcs}\left(  \gamma_{2}\right)  \right)
\cup\cdots\cup\left(  \operatorname*{CArcs}\left(  \gamma_{k}\right)  \right)
\nonumber\\
&  \ \ \ \ \ \ \ \ \ \ \ \ \ \ \ \ \ \ \ \ \left(  \text{since }%
\operatorname*{Cycs}\sigma=\left\{  \gamma_{1},\gamma_{2},\ldots,\gamma
_{k}\right\}  \right) \nonumber\\
&  =\bigcup_{j\in\left[  k\right]  }\underbrace{\operatorname*{CArcs}\left(
\gamma_{j}\right)  }_{\substack{=C_{j}\\\text{(since }C_{j}\text{ was
defined}\\\text{to be }\operatorname*{CArcs}\left(  \gamma_{j}\right)
\text{)}}}\nonumber\\
&  =\bigcup_{j\in\left[  k\right]  }C_{j}.
\label{pf.lem.carcs.linear-split.Asig=un}%
\end{align}

\medskip

\textbf{(a)} We must prove that $F$ is a linear subset of $\mathbf{A}_{\sigma
}\cap A$ if and only if $F$ can be written as $F=\bigsqcup_{j\in\left[
k\right]  }F_{j}$, where each $F_{j}$ is a linear subset of $C_{j}\cap A$.

We shall prove the \textquotedblleft$\Longleftarrow$\textquotedblright\ and
\textquotedblleft$\Longrightarrow$\textquotedblright\ directions of this
equivalence separately:

$\Longleftarrow:$ Assume that $F$ can be written as $F=\bigsqcup_{j\in\left[
k\right]  }F_{j}$, where each $F_{j}$ is a linear subset of $C_{j}\cap A$.
Consider these subsets $F_{j}$.

Let $j\in\left[  k\right]  $. Then, $F_{j}$ is a linear subset of $C_{j}\cap
A$ (according to the preceding paragraph). Thus, $F_{j}\subseteq C_{j}\cap
A\subseteq C_{j}\subseteq V_{j}\times V_{j}$ (by
(\ref{pf.lem.carcs.linear-split.Cjsub})). Therefore, $F_{j}$ is a linear
subset of $V_{j}\times V_{j}$ (since $F_{j}$ is linear).

Forget that we fixed $j$. We thus have shown that for each $j\in\left[
k\right]  $, the set $F_{j}$ is a linear subset of $V_{j}\times V_{j}$. Hence,
Corollary \ref{cor.linear-djun} (applied to $J=\left[  k\right]  $) shows that
the union $\bigcup_{j\in J}F_{j}$ is a linear subset of $V\times V$. In other
words, $F$ is a linear subset of $V\times V$ (since $F=\bigsqcup_{j\in\left[
k\right]  }F_{j}=\bigcup_{j\in J}F_{j}$).

Finally,%
\begin{align*}
F  &  =\bigsqcup_{j\in\left[  k\right]  }F_{j}=\bigcup_{j\in\left[  k\right]
}\underbrace{F_{j}}_{\substack{\subseteq C_{j}\cap A\\\text{(since }%
F_{j}\text{ is a linear subset of }C_{j}\cap A\\\text{(by assumption))}%
}}\subseteq\bigcup_{j\in\left[  k\right]  }\left(  C_{j}\cap A\right) \\
&  =\underbrace{\left(  \bigcup_{j\in\left[  k\right]  }C_{j}\right)
}_{\substack{=\mathbf{A}_{\sigma}\\\text{(by
(\ref{pf.lem.carcs.linear-split.Asig=un}))}}}\cap A=\mathbf{A}_{\sigma}\cap A.
\end{align*}
Hence, $F$ is a subset of $\mathbf{A}_{\sigma}\cap A$. This shows that $F$ is
a linear subset of $\mathbf{A}_{\sigma}\cap A$ (since $F$ is linear). This
proves the \textquotedblleft$\Longleftarrow$\textquotedblright\ direction of
Lemma \ref{lem.carcs.linear-split} \textbf{(a)}.

$\Longrightarrow:$ Assume that $F$ is a linear subset of $\mathbf{A}_{\sigma
}\cap A$. Thus,%
\begin{align*}
F  &  \subseteq\mathbf{A}_{\sigma}\cap A\subseteq\mathbf{A}_{\sigma}%
=\bigcup_{j\in\left[  k\right]  }C_{j}\ \ \ \ \ \ \ \ \ \ \left(  \text{by
(\ref{pf.lem.carcs.linear-split.Asig=un})}\right) \\
&  =C_{1}\cup C_{2}\cup\cdots\cup C_{k}.
\end{align*}

For each $j\in\left[  k\right]  $, let us set $G_{j}:=F\cap C_{j}$. Thus,%
\begin{align*}
G_{1}\cup G_{2}\cup\cdots\cup G_{k}  &  =\left(  F\cap C_{1}\right)
\cup\left(  F\cap C_{2}\right)  \cup\cdots\cup\left(  F\cap C_{k}\right) \\
&  =F\cap\left(  C_{1}\cup C_{2}\cup\cdots\cup C_{k}\right)
\end{align*}
(since $\left(  X\cap Y_{1}\right)  \cup\left(  X\cap Y_{2}\right)  \cup
\cdots\cup\left(  X\cap Y_{n}\right)  =X\cap\left(  Y_{1}\cup Y_{2}\cup
\cdots\cup Y_{n}\right)  $ for any sets $X,Y_{1},Y_{2},\ldots,Y_{n}$). Hence,%
\[
G_{1}\cup G_{2}\cup\cdots\cup G_{k}=F\cap\left(  C_{1}\cup C_{2}\cup\cdots\cup
C_{k}\right)  =F
\]
(since $F\subseteq C_{1}\cup C_{2}\cup\cdots\cup C_{k}$).

Now, for each $j\in\left[  k\right]  $, the set $G_{j}$ is a linear subset of
$C_{j}\cap A$\ \ \ \ \footnote{\textit{Proof.} Let $j\in\left[  k\right]  $.
Then, $G_{j}=\underbrace{F}_{\subseteq\mathbf{A}_{\sigma}\cap A\subseteq
A}\cap C_{j}\subseteq A\cap C_{j}=C_{j}\cap A$. In other words, $G_{j}$ is a
subset of $C_{j}\cap A$. Furthermore, $G_{j}=F\cap C_{j}\subseteq F$, so that
$G_{j}$ is a subset of $F$. However, $F$ is a linear subset of $V\times V$
(since $F$ is linear, and $F\subseteq\mathbf{A}_{\sigma}\subseteq V\times V$).
Thus, Proposition \ref{prop.linear.subset} shows that any subset of $F$ is
linear as well. Therefore, $G_{j}$ is linear (since $G_{j}$ is a subset of
$F$). Hence, $G_{j}$ is a linear subset of $C_{j}\cap A$ (since $G_{j}$ is a
subset of $C_{j}\cap A$), qed.}. Moreover, the sets $G_{1},G_{2},\ldots,G_{k}$
are disjoint\footnote{\textit{Proof.} Let $i$ and $j$ be two distinct elements
of $\left[  k\right]  $. We shall show that $G_{i}\cap G_{j}=\varnothing$.
\par
Indeed, $i$ and $j$ are distinct, and thus we have $C_{i}\cap C_{j}%
=\varnothing$ (since the $k$ sets $C_{1},C_{2},\ldots,C_{k}$ are disjoint).
However, the definition of $G_{j}$ yields
\[
G_{j}=F\cap C_{j}\subseteq C_{j}.
\]
The same argument (applied to $i$ instead of $j$) yields $G_{i}\subseteq
C_{i}$. Hence,%
\[
\underbrace{G_{i}}_{\subseteq C_{i}}\cap\underbrace{G_{j}}_{\subseteq C_{j}%
}\subseteq C_{i}\cap C_{j}=\varnothing.
\]
Therefore, $G_{i}\cap G_{j}=\varnothing$.
\par
Forget that we fixed $i$ and $j$. We thus have shown that $G_{i}\cap
G_{j}=\varnothing$ whenever $i$ and $j$ are two distinct elements of $\left[
k\right]  $. In other words, the sets $G_{1},G_{2},\ldots,G_{k}$ are
disjoint.}. Thus, the disjoint union $G_{1}\sqcup G_{2}\sqcup\cdots\sqcup
G_{k}=\bigsqcup_{j\in\left[  k\right]  }G_{j}$ is well-defined. This disjoint
union is%
\[
\bigsqcup_{j\in\left[  k\right]  }G_{j}=G_{1}\sqcup G_{2}\sqcup\cdots\sqcup
G_{k}=G_{1}\cup G_{2}\cup\cdots\cup G_{k}=F.
\]
Thus, $F=\bigsqcup_{j\in\left[  k\right]  }G_{j}$.

Altogether, we have now shown that $F=\bigsqcup_{j\in\left[  k\right]  }G_{j}%
$, and that each $G_{j}$ is a linear subset of $C_{j}\cap A$. Hence, $F$ can
be written as $F=\bigsqcup_{j\in\left[  k\right]  }F_{j}$, where each $F_{j}$
is a linear subset of $C_{j}\cap A$ (namely, for $F_{j}=G_{j}$). This proves
the \textquotedblleft$\Longrightarrow$\textquotedblright\ direction of Lemma
\ref{lem.carcs.linear-split} \textbf{(a)}. \medskip

\textbf{(b)} Assume that $F$ is written as $F=\bigsqcup_{j\in\left[  k\right]
}F_{j}$, where each $F_{j}$ is a linear subset of $C_{j}\cap A$. We must show
that $F_{j}=F\cap C_{j}$ for each $j\in\left[  k\right]  $.

Indeed, we have $F=\bigsqcup_{j\in\left[  k\right]  }F_{j}=\bigsqcup
_{i\in\left[  k\right]  }F_{i}$.

Now, let $j\in\left[  k\right]  $. Then, $F_{j}$ is a subset of $F$ (since
$F=\bigsqcup_{i\in\left[  k\right]  }F_{i}$) and also a subset of $C_{j}\cap
A$ (since we required that each $F_{j}$ be a linear subset of $C_{j}\cap A$).
In other words, $F_{j}$ is a subset of both $F$ and $C_{j}\cap A$. Thus,
$F_{j}$ is a subset of the intersection $F\cap\left(  C_{j}\cap A\right)  $ as
well. In other words,%
\[
F_{j}\subseteq F\cap\underbrace{\left(  C_{j}\cap A\right)  }_{\subseteq
C_{j}}\subseteq F\cap C_{j}.
\]

Let us now show that $F\cap C_{j}\subseteq F_{j}$.

Indeed, let $\alpha\in F\cap C_{j}$. Then, $\alpha\in F\cap C_{j}\subseteq
F=\bigsqcup_{i\in\left[  k\right]  }F_{i}$. Hence, $\alpha\in F_{i}$ for some
$i\in\left[  k\right]  $. Consider this $i$. Recall that $F_{j}$ is a subset
of $C_{j}\cap A$; thus, $F_{j}\subseteq C_{j}\cap A\subseteq C_{j}$. The same
argument (applied to $i$ instead of $j$) yields $F_{i}\subseteq C_{i}$. Hence,
$\alpha\in F_{i}\subseteq C_{i}$. However, $\alpha\in F\cap C_{j}\subseteq
C_{j}$. Thus, the element $\alpha$ belongs to both sets $C_{i}$ and $C_{j}$.
Therefore, the sets $C_{i}$ and $C_{j}$ have at least one element in common.
In other words, the sets $C_{i}$ and $C_{j}$ are not disjoint. However, we
know that the sets $C_{1},C_{2},\ldots,C_{k}$ are disjoint. The only way to
reconcile the previous two sentences is when $i=j$.

Thus, we obtain $i=j$. Hence, $\alpha\in F_{i}=F_{j}$ (since $i=j$).

Forget that we fixed $\alpha$. We thus have shown that $\alpha\in F_{j}$ for
each $\alpha\in F\cap C_{j}$. In other words, $F\cap C_{j}\subseteq F_{j}$.
Combining this with $F_{j}\subseteq F\cap C_{j}$, we conclude that
$F_{j}=F\cap C_{j}$. This completes the proof of Lemma
\ref{lem.carcs.linear-split} \textbf{(b)}.
\end{proof}

We can now prove Proposition \ref{prop.linear-in-AsigA}:

\begin{proof}
[Proof of Proposition \ref{prop.linear-in-AsigA}.]Let $\gamma_{1},\gamma
_{2},\ldots,\gamma_{k}$ be the cycles of $\sigma$, listed with no repetition
(as in Lemma \ref{lem.carcs.linear-split}). Thus, these cycles $\gamma
_{1},\gamma_{2},\ldots,\gamma_{k}$ are distinct, and we have
$\operatorname*{Cycs}\sigma=\left\{  \gamma_{1},\gamma_{2},\ldots,\gamma
_{k}\right\}  $ (since $\operatorname*{Cycs}\sigma$ is defined to be the set
of all cycles of $\sigma$).

For each $i\in\left[  k\right]  $, let $C_{i}:=\operatorname*{CArcs}\left(
\gamma_{i}\right)  $. Then, the sets $C_{1},C_{2},\ldots,C_{k}$ are
disjoint\footnote{Indeed, this was shown during our above proof of Lemma
\ref{lem.carcs.linear-split}.}.

Now, we observe the following: If $F_{j}$ is a linear subset of $C_{j}\cap A$
for each $j\in\left[  k\right]  $, then the disjoint union $\bigsqcup
_{j\in\left[  k\right]  }F_{j}$ is well-defined\footnote{\textit{Proof.} Let
$F_{j}$ be a linear subset of $C_{j}\cap A$ for each $j\in\left[  k\right]  $.
Then, in particular, we have $F_{j}\subseteq C_{j}\cap A\subseteq C_{j}$ for
each $j\in\left[  k\right]  $. In other words, $F_{j}$ is a subset of $C_{j}$
for each $j\in\left[  k\right]  $. In other words, the sets $F_{1}%
,F_{2},\ldots,F_{k}$ are subsets of $C_{1},C_{2},\ldots,C_{k}$, respectively.
\par
However, the sets $C_{1},C_{2},\ldots,C_{k}$ are disjoint. Thus, their subsets
$F_{1},F_{2},\ldots,F_{k}$ are disjoint as well (since the sets $F_{1}%
,F_{2},\ldots,F_{k}$ are subsets of $C_{1},C_{2},\ldots,C_{k}$, respectively).
In other words, the sets $F_{j}$ for different $j\in\left[  k\right]  $ are
disjoint. Thus, the disjoint union $\bigsqcup_{j\in\left[  k\right]  }F_{j}$
is well-defined.}, and is a linear subset of $\mathbf{A}_{\sigma}\cap A$ (by
the \textquotedblleft$\Longleftarrow$\textquotedblright\ direction of Lemma
\ref{lem.carcs.linear-split} \textbf{(a)}, applied to $F=\bigsqcup
_{j\in\left[  k\right]  }F_{j}$). Hence, the map%
\begin{align*}
&  \text{from }\left\{  \text{families }\left(  F_{j}\right)  _{j\in\left[
k\right]  }\text{, where each }F_{j}\text{ is a linear subset of }C_{j}\cap
A\right\} \\
&  \text{to }\left\{  \text{linear subsets of }\mathbf{A}_{\sigma}\cap
A\right\} \\
&  \text{that sends each family }\left(  F_{j}\right)  _{j\in\left[  k\right]
}\text{ to }\bigsqcup_{j\in\left[  k\right]  }F_{j}%
\end{align*}
is well-defined. Moreover, this map is injective (since Lemma
\ref{lem.carcs.linear-split} \textbf{(b)} shows that the sets $F_{j}$ are
uniquely determined by their union $\bigsqcup_{j\in\left[  k\right]  }F_{j}$)
and surjective (by the \textquotedblleft$\Longrightarrow$\textquotedblright%
\ direction of Lemma \ref{lem.carcs.linear-split} \textbf{(a)}). Thus, it is
bijective. Hence, we can substitute $\bigsqcup_{j\in\left[  k\right]  }F_{j}$
for $F$ in the sum $\sum_{\substack{F\subseteq\mathbf{A}_{\sigma}\cap
A\\\text{is linear}}}\left(  -1\right)  ^{\left\vert F\right\vert }$. We thus
obtain%
\begin{align}
&  \sum_{\substack{F\subseteq\mathbf{A}_{\sigma}\cap A\\\text{is linear}%
}}\left(  -1\right)  ^{\left\vert F\right\vert }\nonumber\\
&  =\sum_{\substack{\left(  F_{j}\right)  _{j\in\left[  k\right]  }\text{ is a
family}\\\text{of linear subsets }F_{j}\subseteq C_{j}\cap A}}\left(
-1\right)  ^{\left\vert \bigsqcup_{j\in\left[  k\right]  }F_{j}\right\vert }.
\label{pf.prop.linear-in-AsigA.long.1}%
\end{align}

However, if $\left(  F_{j}\right)  _{j\in\left[  k\right]  }$ is a family of
linear subsets $F_{j}\subseteq C_{j}\cap A$, then%
\[
\left\vert \bigsqcup_{j\in\left[  k\right]  }F_{j}\right\vert =\sum
_{j\in\left[  k\right]  }\left\vert F_{j}\right\vert
\ \ \ \ \ \ \ \ \ \ \left(  \text{by the sum rule}\right)
\]
and thus%
\begin{equation}
\left(  -1\right)  ^{\left\vert \bigsqcup_{j\in\left[  k\right]  }%
F_{j}\right\vert }=\left(  -1\right)  ^{\sum_{j\in\left[  k\right]
}\left\vert F_{j}\right\vert }=\prod_{j\in\left[  k\right]  }\left(
-1\right)  ^{\left\vert F_{j}\right\vert }.
\label{pf.prop.linear-in-AsigA.long.2}%
\end{equation}

Hence, (\ref{pf.prop.linear-in-AsigA.long.1}) becomes%
\begin{align}
&  \sum_{\substack{F\subseteq\mathbf{A}_{\sigma}\cap A\\\text{is linear}%
}}\left(  -1\right)  ^{\left\vert F\right\vert }\nonumber\\
&  =\sum_{\substack{\left(  F_{j}\right)  _{j\in\left[  k\right]  }\text{ is a
family}\\\text{of linear subsets }F_{j}\subseteq C_{j}\cap A}%
}\underbrace{\left(  -1\right)  ^{\left\vert \bigsqcup_{j\in\left[  k\right]
}F_{j}\right\vert }}_{\substack{=\prod_{j\in\left[  k\right]  }\left(
-1\right)  ^{\left\vert F_{j}\right\vert }\\\text{(by
(\ref{pf.prop.linear-in-AsigA.long.2}))}}}\nonumber\\
&  =\sum_{\substack{\left(  F_{j}\right)  _{j\in\left[  k\right]  }\text{ is a
family}\\\text{of linear subsets }F_{j}\subseteq C_{j}\cap A}}\ \ \prod
_{j\in\left[  k\right]  }\left(  -1\right)  ^{\left\vert F_{j}\right\vert
}\nonumber\\
&  =\prod_{j\in\left[  k\right]  }\ \ \underbrace{\sum_{\substack{F_{j}%
\subseteq C_{j}\cap A\\\text{is linear}}}}_{\substack{=\sum_{\substack{F_{j}%
\subseteq\left(  \operatorname*{CArcs}\left(  \gamma_{j}\right)  \right)  \cap
A\\\text{is linear}}}\\\text{(since }C_{j}=\operatorname*{CArcs}\left(
\gamma_{j}\right)  \\\text{(by the definition of }C_{j}\text{))}}}\left(
-1\right)  ^{\left\vert F_{j}\right\vert }\ \ \ \ \ \ \ \ \ \ \left(  \text{by
the product rule}\right) \nonumber\\
&  =\prod_{j\in\left[  k\right]  }\ \ \sum_{\substack{F_{j}\subseteq\left(
\operatorname*{CArcs}\left(  \gamma_{j}\right)  \right)  \cap A\\\text{is
linear}}}\left(  -1\right)  ^{\left\vert F_{j}\right\vert }\nonumber\\
&  =\prod_{j\in\left[  k\right]  }\ \ \sum_{\substack{F\subseteq\left(
\operatorname*{CArcs}\left(  \gamma_{j}\right)  \right)  \cap A\\\text{is
linear}}}\left(  -1\right)  ^{\left\vert F\right\vert }
\label{pf.prop.linear-in-AsigA.long.3}%
\end{align}
(here, we have renamed the summation index $F_{j}$ as $F$).

On the other hand, recall that the cycles $\gamma_{1},\gamma_{2},\ldots
,\gamma_{k}$ are distinct, and that we have $\operatorname*{Cycs}%
\sigma=\left\{  \gamma_{1},\gamma_{2},\ldots,\gamma_{k}\right\}  $. Hence,
$\gamma_{1},\gamma_{2},\ldots,\gamma_{k}$ are the elements of the set
$\operatorname*{Cycs}\sigma$, listed with no repetition. In other words, the
map
\begin{align*}
\left[  k\right]   &  \rightarrow\operatorname*{Cycs}\sigma,\\
j  &  \mapsto\gamma_{j}%
\end{align*}
is a bijection. Hence, we can substitute $\gamma_{j}$ for $\gamma$ in the
product \newline$\prod_{\gamma\in\operatorname*{Cycs}\sigma}\ \ \sum
_{\substack{F\subseteq\left(  \operatorname*{CArcs}\gamma\right)  \cap
A\\\text{is linear}}}\left(  -1\right)  ^{\left\vert F\right\vert }$. We thus
obtain%
\[
\prod_{\gamma\in\operatorname*{Cycs}\sigma}\ \ \sum_{\substack{F\subseteq
\left(  \operatorname*{CArcs}\gamma\right)  \cap A\\\text{is linear}}}\left(
-1\right)  ^{\left\vert F\right\vert }=\prod_{j\in\left[  k\right]  }%
\ \ \sum_{\substack{F\subseteq\left(  \operatorname*{CArcs}\left(  \gamma
_{j}\right)  \right)  \cap A\\\text{is linear}}}\left(  -1\right)
^{\left\vert F\right\vert }.
\]
Comparing this with (\ref{pf.prop.linear-in-AsigA.long.3}), we obtain%
\begin{align}
&  \sum_{\substack{F\subseteq\mathbf{A}_{\sigma}\cap A\\\text{is linear}%
}}\left(  -1\right)  ^{\left\vert F\right\vert }\nonumber\\
&  =\prod_{\gamma\in\operatorname*{Cycs}\sigma}\ \ \sum_{\substack{F\subseteq
\left(  \operatorname*{CArcs}\gamma\right)  \cap A\\\text{is linear}}}\left(
-1\right)  ^{\left\vert F\right\vert }. \label{pf.prop.linear-in-AsigA.long.4}%
\end{align}

Next, the following three claims follow easily from Lemma
\ref{lem.carcs.linear2}:

\begin{statement}
\textit{Claim 1:} Let $\gamma\in\operatorname*{Cycs}\sigma$. Assume that
$\gamma$ is a $D$-cycle. Then,%
\[
\sum_{\substack{F\subseteq\left(  \operatorname*{CArcs}\gamma\right)  \cap
A\\\text{is linear}}}\left(  -1\right)  ^{\left\vert F\right\vert }=\left(
-1\right)  ^{\ell\left(  \gamma\right)  -1}.
\]

\end{statement}

\begin{statement}
\textit{Claim 2:} Let $\gamma\in\operatorname*{Cycs}\sigma$. Assume that
$\gamma$ is a $\overline{D}$-cycle. Then,%
\[
\sum_{\substack{F\subseteq\left(  \operatorname*{CArcs}\gamma\right)  \cap
A\\\text{is linear}}}\left(  -1\right)  ^{\left\vert F\right\vert }=1.
\]

\end{statement}

\begin{statement}
\textit{Claim 3:} Let $\gamma\in\operatorname*{Cycs}\sigma$. Assume that
$\gamma$ is neither a $D$-cycle nor a $\overline{D}$-cycle. Then,%
\[
\sum_{\substack{F\subseteq\left(  \operatorname*{CArcs}\gamma\right)  \cap
A\\\text{is linear}}}\left(  -1\right)  ^{\left\vert F\right\vert }=0.
\]

\end{statement}

\begin{proof}
[Proof of Claim 1.]We have $\gamma\in\operatorname*{Cycs}\sigma$. In other
words, $\gamma$ is a cycle of $\sigma$ (since $\operatorname*{Cycs}\sigma$ is
the set of all cycles of $\sigma$). Hence, Lemma \ref{lem.carcs.linear2}
\textbf{(a)} (applied to $C=\operatorname*{CArcs}\gamma$) yields
$\sum_{\substack{F\subseteq\left(  \operatorname*{CArcs}\gamma\right)  \cap
A\\\text{is linear}}}\left(  -1\right)  ^{\left\vert F\right\vert }=\left(
-1\right)  ^{\ell\left(  \gamma\right)  -1}$. This proves Claim 1.
\end{proof}

\begin{proof}
[Proof of Claim 2.]We have $\gamma\in\operatorname*{Cycs}\sigma$. In other
words, $\gamma$ is a cycle of $\sigma$ (since $\operatorname*{Cycs}\sigma$ is
the set of all cycles of $\sigma$). Hence, Lemma \ref{lem.carcs.linear2}
\textbf{(b)} (applied to $C=\operatorname*{CArcs}\gamma$) yields
$\sum_{\substack{F\subseteq\left(  \operatorname*{CArcs}\gamma\right)  \cap
A\\\text{is linear}}}\left(  -1\right)  ^{\left\vert F\right\vert }=1$. This
proves Claim 2.
\end{proof}

\begin{proof}
[Proof of Claim 3.]We have $\gamma\in\operatorname*{Cycs}\sigma$. In other
words, $\gamma$ is a cycle of $\sigma$ (since $\operatorname*{Cycs}\sigma$ is
the set of all cycles of $\sigma$). Hence, Lemma \ref{lem.carcs.linear2}
\textbf{(c)} (applied to $C=\operatorname*{CArcs}\gamma$) yields
$\sum_{\substack{F\subseteq\left(  \operatorname*{CArcs}\gamma\right)  \cap
A\\\text{is linear}}}\left(  -1\right)  ^{\left\vert F\right\vert }=0$. This
proves Claim 3.
\end{proof}

Now, we shall prove the following two claims:

\begin{statement}
\textit{Claim 4:} If $\sigma\in\mathfrak{S}_{V}\left(  D,\overline{D}\right)
$, then $\sum_{\substack{F\subseteq\mathbf{A}_{\sigma}\cap A\\\text{is
linear}}}\left(  -1\right)  ^{\left\vert F\right\vert }=\left(  -1\right)
^{\varphi\left(  \sigma\right)  }$.
\end{statement}

\begin{statement}
\textit{Claim 5:} If $\sigma\notin\mathfrak{S}_{V}\left(  D,\overline
{D}\right)  $, then $\sum_{\substack{F\subseteq\mathbf{A}_{\sigma}\cap
A\\\text{is linear}}}\left(  -1\right)  ^{\left\vert F\right\vert }=0$.
\end{statement}

\begin{proof}
[Proof of Claim 4.]Assume that $\sigma\in\mathfrak{S}_{V}\left(
D,\overline{D}\right)  $. Then, each cycle of $\sigma$ is a $D$-cycle or a
$\overline{D}$-cycle (by the definition of $\mathfrak{S}_{V}\left(
D,\overline{D}\right)  $). In other words, if $\gamma\in\operatorname*{Cycs}%
\sigma$ is not a $D$-cycle, then $\gamma$ is a $\overline{D}$%
-cycle\footnote{\textit{Proof.} Let $\gamma\in\operatorname*{Cycs}\sigma$ be
not a $D$-cycle. We must prove that $\gamma$ is a $\overline{D}$-cycle.
\par
We have $\gamma\in\operatorname*{Cycs}\sigma$. In other words, $\gamma$ is a
cycle of $\sigma$ (since $\operatorname*{Cycs}\sigma$ is the set of all cycles
of $\sigma$). Hence, $\gamma$ is a $D$-cycle or a $\overline{D}$-cycle (since
each cycle of $\sigma$ is a $D$-cycle or a $\overline{D}$-cycle). Since
$\gamma$ is not a $D$-cycle, we thus conclude that $\gamma$ is a $\overline
{D}$-cycle. Qed.} and thus satisfies
\begin{equation}
\sum_{\substack{F\subseteq\left(  \operatorname*{CArcs}\gamma\right)  \cap
A\\\text{is linear}}}\left(  -1\right)  ^{\left\vert F\right\vert }=1
\label{pf.prop.linear-in-AsigA.long.c4.pf.1}%
\end{equation}
(by Claim 2).

Now, (\ref{pf.prop.linear-in-AsigA.long.4}) becomes%
\begin{align*}
&  \sum_{\substack{F\subseteq\mathbf{A}_{\sigma}\cap A\\\text{is linear}%
}}\left(  -1\right)  ^{\left\vert F\right\vert }\\
&  =\prod_{\gamma\in\operatorname*{Cycs}\sigma}\ \ \sum_{\substack{F\subseteq
\left(  \operatorname*{CArcs}\gamma\right)  \cap A\\\text{is linear}}}\left(
-1\right)  ^{\left\vert F\right\vert }\\
&  =\left(  \prod_{\substack{\gamma\in\operatorname*{Cycs}\sigma
;\\\gamma\text{ is a }D\text{-cycle}}}\ \ \underbrace{\sum
_{\substack{F\subseteq\left(  \operatorname*{CArcs}\gamma\right)  \cap
A\\\text{is linear}}}\left(  -1\right)  ^{\left\vert F\right\vert }%
}_{\substack{=\left(  -1\right)  ^{\ell\left(  \gamma\right)  -1}\\\text{(by
Claim 1)}}}\right)  \cdot\left(  \prod_{\substack{\gamma\in
\operatorname*{Cycs}\sigma;\\\gamma\text{ is not a }D\text{-cycle}%
}}\ \ \underbrace{\sum_{\substack{F\subseteq\left(  \operatorname*{CArcs}%
\gamma\right)  \cap A\\\text{is linear}}}\left(  -1\right)  ^{\left\vert
F\right\vert }}_{\substack{=1\\\text{(by
(\ref{pf.prop.linear-in-AsigA.long.c4.pf.1}))}}}\right) \\
&  \ \ \ \ \ \ \ \ \ \ \ \ \ \ \ \ \ \ \ \ \left(
\begin{array}
[c]{c}%
\text{here, we have split our product in two: one}\\
\text{that contains all }\gamma\text{'s that are }D\text{-cycles, and}\\
\text{one that contain all other }\gamma\text{'s}%
\end{array}
\right) \\
&  =\left(  \prod_{\substack{\gamma\in\operatorname*{Cycs}\sigma
;\\\gamma\text{ is a }D\text{-cycle}}}\left(  -1\right)  ^{\ell\left(
\gamma\right)  -1}\right)  \cdot\underbrace{\left(  \prod_{\substack{\gamma
\in\operatorname*{Cycs}\sigma;\\\gamma\text{ is not a }D\text{-cycle}%
}}1\right)  }_{=1}\\
&  =\prod_{\substack{\gamma\in\operatorname*{Cycs}\sigma;\\\gamma\text{ is a
}D\text{-cycle}}}\left(  -1\right)  ^{\ell\left(  \gamma\right)  -1}.
\end{align*}
Comparing this with%
\begin{align*}
\left(  -1\right)  ^{\varphi\left(  \sigma\right)  }  &  =\left(  -1\right)
^{\sum_{\substack{\gamma\in\operatorname*{Cycs}\sigma;\\\gamma\text{ is a
}D\text{-cycle}}}\left(  \ell\left(  \gamma\right)  -1\right)  }%
\ \ \ \ \ \ \ \ \ \ \left(
\begin{array}
[c]{c}%
\text{since }\varphi\left(  \sigma\right)  =\sum_{\substack{\gamma
\in\operatorname*{Cycs}\sigma;\\\gamma\text{ is a }D\text{-cycle}}}\left(
\ell\left(  \gamma\right)  -1\right) \\
\text{(by the definition of }\varphi\left(  \sigma\right)  \text{)}%
\end{array}
\right) \\
&  =\prod_{\substack{\gamma\in\operatorname*{Cycs}\sigma;\\\gamma\text{ is a
}D\text{-cycle}}}\left(  -1\right)  ^{\ell\left(  \gamma\right)  -1},
\end{align*}
we obtain $\sum_{\substack{F\subseteq\mathbf{A}_{\sigma}\cap A\\\text{is
linear}}}\left(  -1\right)  ^{\left\vert F\right\vert }=\left(  -1\right)
^{\varphi\left(  \sigma\right)  }$. Thus, Claim 4 is proven.
\end{proof}

\begin{proof}
[Proof of Claim 5.]Assume that $\sigma\notin\mathfrak{S}_{V}\left(
D,\overline{D}\right)  $. Then, not each cycle of $\sigma$ is a $D$-cycle or a
$\overline{D}$-cycle (by the definition of $\mathfrak{S}_{V}\left(
D,\overline{D}\right)  $). In other words, there exists some cycle of $\sigma$
that is neither a $D$-cycle nor a $\overline{D}$-cycle. Let $\delta$ be such a
cycle. Then, $\delta$ is a cycle of $\sigma$. In other words, $\delta
\in\operatorname*{Cycs}\sigma$ (since $\operatorname*{Cycs}\sigma$ is the set
of all cycles of $\sigma$). We know that $\delta$ is neither a $D$-cycle nor a
$\overline{D}$-cycle (by its definition). Thus, Claim 3 (applied to
$\gamma=\delta$) yields%
\begin{equation}
\sum_{\substack{F\subseteq\left(  \operatorname*{CArcs}\delta\right)  \cap
A\\\text{is linear}}}\left(  -1\right)  ^{\left\vert F\right\vert }=0.
\label{pf.prop.linear-in-AsigA.long.c4.pf.5}%
\end{equation}

However, $\delta\in\operatorname*{Cycs}\sigma$. Thus, the sum $\sum
_{\substack{F\subseteq\left(  \operatorname*{CArcs}\delta\right)  \cap
A\\\text{is linear}}}\left(  -1\right)  ^{\left\vert F\right\vert }$ is one of
the factors of the product $\prod_{\gamma\in\operatorname*{Cycs}\sigma
}\ \ \sum_{\substack{F\subseteq\left(  \operatorname*{CArcs}\gamma\right)
\cap A\\\text{is linear}}}\left(  -1\right)  ^{\left\vert F\right\vert }$
(namely, the factor for $\gamma=\delta$). Since the former sum is $0$ (by
(\ref{pf.prop.linear-in-AsigA.long.c4.pf.5})), we can rewrite this as follows:
The number $0$ is one of the factors of the product $\prod_{\gamma
\in\operatorname*{Cycs}\sigma}\ \ \sum_{\substack{F\subseteq\left(
\operatorname*{CArcs}\gamma\right)  \cap A\\\text{is linear}}}\left(
-1\right)  ^{\left\vert F\right\vert }$. In other words, the latter product
has a factor equal to $0$. Therefore, this product must be $0$ (because if a
product has a factor equal to $0$, then this product must be $0$). In other
words,%
\[
\prod_{\gamma\in\operatorname*{Cycs}\sigma}\ \ \sum_{\substack{F\subseteq
\left(  \operatorname*{CArcs}\gamma\right)  \cap A\\\text{is linear}}}\left(
-1\right)  ^{\left\vert F\right\vert }=0.
\]
Now, (\ref{pf.prop.linear-in-AsigA.long.4}) becomes%
\[
\sum_{\substack{F\subseteq\mathbf{A}_{\sigma}\cap A\\\text{is linear}}}\left(
-1\right)  ^{\left\vert F\right\vert }=\prod_{\gamma\in\operatorname*{Cycs}%
\sigma}\ \ \sum_{\substack{F\subseteq\left(  \operatorname*{CArcs}%
\gamma\right)  \cap A\\\text{is linear}}}\left(  -1\right)  ^{\left\vert
F\right\vert }=0.
\]
This proves Claim 5.
\end{proof}

Combining Claim 4 with Claim 5, we obtain%
\[
\sum_{\substack{F\subseteq\mathbf{A}_{\sigma}\cap A\\\text{is linear}}}\left(
-1\right)  ^{\left\vert F\right\vert }=%
\begin{cases}
\left(  -1\right)  ^{\varphi\left(  \sigma\right)  }, & \text{if }\sigma
\in\mathfrak{S}_{V}\left(  D,\overline{D}\right)  ;\\
0, & \text{else.}%
\end{cases}
\]
This proves Proposition \ref{prop.linear-in-AsigA}.
\end{proof}
\end{verlong}

\subsection{A trivial lemma}

We need one more trivial \textquotedblleft data conversion\textquotedblright\ lemma:

\begin{lemma}
\label{lem.maps-vs-tups}Let $V$ be a finite set. Let $w=\left(  w_{1}%
,w_{2},\ldots,w_{n}\right)  $ be a $V$-listing. Then, the map%
\begin{align*}
\left\{  \text{maps }f:V\rightarrow\mathbb{P}\right\}   &  \rightarrow
\mathbb{P}^{n},\\
f  &  \mapsto\left(  f\left(  w_{1}\right)  ,f\left(  w_{2}\right)
,\ldots,f\left(  w_{n}\right)  \right)
\end{align*}
is well-defined and is a bijection.
\end{lemma}

\begin{vershort}

\begin{proof}
This is just saying that every map $f:V\rightarrow\mathbb{P}$ can be encoded
by its list of values $\left(  f\left(  w_{1}\right)  ,f\left(  w_{2}\right)
,\ldots,f\left(  w_{n}\right)  \right)  $, and conversely, that any list
$\left(  i_{1},i_{2},\ldots,i_{n}\right)  \in\mathbb{P}^{n}$ is the list of
values of a unique map $f:V\rightarrow\mathbb{P}$. Both of these claims are
clear, since $w_{1},w_{2},\ldots,w_{n}$ are the elements of $V$ (listed with
no repetitions).
\end{proof}
\end{vershort}

\begin{verlong}

\begin{proof}
Recall that $\left(  w_{1},w_{2},\ldots,w_{n}\right)  $ is a $V$-listing,
i.e., a list of elements of $V$ that contains each element of $V$ exactly once
(by the definition of a $V$-listing). Hence, in particular, $\left(
w_{1},w_{2},\ldots,w_{n}\right)  $ is a list of elements of $V$. In other
words, $w_{i}\in V$ for each $i\in\left[  n\right]  $.

Now, if $f:V\rightarrow\mathbb{P}$ is a map, then each $i\in\left[  n\right]
$ satisfies $f\left(  w_{i}\right)  \in\mathbb{P}$ (since $w_{i}\in V$ by the
preceding sentence), and thus the $n$-tuple $\left(  f\left(  w_{1}\right)
,f\left(  w_{2}\right)  ,\ldots,f\left(  w_{n}\right)  \right)  $ belongs to
$\mathbb{P}^{n}$. Thus, the map%
\begin{align*}
\left\{  \text{maps }f:V\rightarrow\mathbb{P}\right\}   &  \rightarrow
\mathbb{P}^{n},\\
f  &  \mapsto\left(  f\left(  w_{1}\right)  ,f\left(  w_{2}\right)
,\ldots,f\left(  w_{n}\right)  \right)
\end{align*}
is well-defined. Let us denote this map by $K$. It remains to prove that this
map $K$ is a bijection.

Let us first show that $K$ is injective. Indeed, let $f$ and $g$ be two maps
from $V$ to $\mathbb{P}$ that satisfy $K\left(  f\right)  =K\left(  g\right)
$. We shall show that $f=g$.

The definition of $K$ yields
\begin{align}
K\left(  f\right)   &  =\left(  f\left(  w_{1}\right)  ,f\left(  w_{2}\right)
,\ldots,f\left(  w_{n}\right)  \right) \label{pf.lem.maps-vs-tups.inj.f}\\
\text{and}\ \ \ \ \ \ \ \ \ \ K\left(  g\right)   &  =\left(  g\left(
w_{1}\right)  ,g\left(  w_{2}\right)  ,\ldots,g\left(  w_{n}\right)  \right)
. \label{pf.lem.maps-vs-tups.inj.g}%
\end{align}

We assumed that $K\left(  f\right)  =K\left(  g\right)  $. In view of
(\ref{pf.lem.maps-vs-tups.inj.f}) and (\ref{pf.lem.maps-vs-tups.inj.g}), we
can rewrite this as%
\[
\left(  f\left(  w_{1}\right)  ,f\left(  w_{2}\right)  ,\ldots,f\left(
w_{n}\right)  \right)  =\left(  g\left(  w_{1}\right)  ,g\left(  w_{2}\right)
,\ldots,g\left(  w_{n}\right)  \right)  .
\]
In other words,%
\begin{equation}
f\left(  w_{i}\right)  =g\left(  w_{i}\right)  \ \ \ \ \ \ \ \ \ \ \text{for
each }i\in\left[  n\right]  . \label{pf.lem.maps-vs-tups.inj.1}%
\end{equation}

Now, let $v\in V$ be arbitrary. Then, the list $\left(  w_{1},w_{2}%
,\ldots,w_{n}\right)  $ contains $v$ exactly once (since this list $\left(
w_{1},w_{2},\ldots,w_{n}\right)  $ contains each element of $V$ exactly once).
In other words, there is exactly one $i\in\left[  n\right]  $ such that
$w_{i}=v$. Consider this $i$. From (\ref{pf.lem.maps-vs-tups.inj.1}), we
obtain $f\left(  w_{i}\right)  =g\left(  w_{i}\right)  $. In view of $w_{i}%
=v$, we can rewrite this as $f\left(  v\right)  =g\left(  v\right)  $.

Forget that we fixed $v$. We thus have shown that $f\left(  v\right)
=g\left(  v\right)  $ for each $v\in V$. In other words, $f=g$.

Forget that we fixed $f$ and $g$. We have thus shown that if $f$ and $g$ are
two maps from $V$ to $\mathbb{P}$ that satisfy $K\left(  f\right)  =K\left(
g\right)  $, then $f=g$. In other words, the map $K$ is injective.

Now, let us prove that $K$ is surjective. Indeed, let $a\in\mathbb{P}^{n}$ be
arbitrary. We shall construct a map $f:V\rightarrow\mathbb{P}$ such that
$K\left(  f\right)  =a$.

According to Convention \ref{conv.wi}, we can write the $n$-tuple
$a\in\mathbb{P}^{n}$ as $a=\left(  a_{1},a_{2},\ldots,a_{n}\right)  $.

Recall that $\left(  w_{1},w_{2},\ldots,w_{n}\right)  $ is a $V$-listing,
i.e., a list of elements of $V$ that contains each element of $V$ exactly once
(by the definition of a $V$-listing).

Now, we define a map $f:V\rightarrow\mathbb{P}$ as follows:

Let $v\in V$. Then, the list $\left(  w_{1},w_{2},\ldots,w_{n}\right)  $
contains $v$ exactly once (since this list $\left(  w_{1},w_{2},\ldots
,w_{n}\right)  $ contains each element of $V$ exactly once). In other words,
there is exactly one $i\in\left[  n\right]  $ such that $w_{i}=v$. Consider
this $i$. Define $f\left(  v\right)  $ to be $a_{i}$. This is an element of
$\mathbb{P}$ (since it is an entry of the $n$-tuple $a\in\mathbb{P}^{n}$).

Thus, we have defined an element $f\left(  v\right)  $ of $\mathbb{P}$ for
each $v\in V$. In other words, we have defined a map $f:V\rightarrow
\mathbb{P}$. Its definition has the following consequence: If $v\in V$ is
arbitrary, and if $i\in\left[  n\right]  $ is an element satisfying $w_{i}=v$,
then%
\begin{equation}
f\left(  v\right)  =a_{i}. \label{pf.lem.maps-vs-tups.surj.3}%
\end{equation}

Now, we shall show that $K\left(  f\right)  =a$.

Indeed, the definition of $K$ yields $K\left(  f\right)  =\left(  f\left(
w_{1}\right)  ,f\left(  w_{2}\right)  ,\ldots,f\left(  w_{n}\right)  \right)
$. However, each $i\in\left[  n\right]  $ satisfies $w_{i}=w_{i}$ (obviously)
and thus $f\left(  w_{i}\right)  =a_{i}$ (by (\ref{pf.lem.maps-vs-tups.surj.3}%
), applied to $v=w_{i}$). In other words, we have%
\[
\left(  f\left(  w_{1}\right)  ,f\left(  w_{2}\right)  ,\ldots,f\left(
w_{n}\right)  \right)  =\left(  a_{1},a_{2},\ldots,a_{n}\right)  .
\]
In view of $K\left(  f\right)  =\left(  f\left(  w_{1}\right)  ,f\left(
w_{2}\right)  ,\ldots,f\left(  w_{n}\right)  \right)  $ and $a=\left(
a_{1},a_{2},\ldots,a_{n}\right)  $, we can rewrite this as $K\left(  f\right)
=a$. Thus, the map $K$ takes $a$ as a value (namely, at the input $f$).

Forget that we fixed $a$. We thus have shown that if $a\in\mathbb{P}^{n}$ is
arbitrary, then the map $K$ takes $a$ as a value. In other words, the map $K$
is surjective.

Now we know that the map $K$ is both injective and surjective. In other words,
$K$ is bijective. In other words, $K$ is a bijection.

So we have shown that the map $K$ is well-defined and is a bijection. In other
words, the map%
\begin{align*}
\left\{  \text{maps }f:V\rightarrow\mathbb{P}\right\}   &  \rightarrow
\mathbb{P}^{n},\\
f  &  \mapsto\left(  f\left(  w_{1}\right)  ,f\left(  w_{2}\right)
,\ldots,f\left(  w_{n}\right)  \right)
\end{align*}
is well-defined and is a bijection (since this map is what we called $K$).
This completes the proof of Lemma \ref{lem.maps-vs-tups}.
\end{proof}
\end{verlong}

\subsection{The proof of Theorem \ref{thm.UX.1}}

We are now ready to prove Theorem \ref{thm.UX.1}:

\begin{proof}
[Proof of Theorem \ref{thm.UX.1}.]Let $n=\left\vert V\right\vert $. Thus, the
digraph $D=\left(  V,A\right)  $ has $n$ vertices. Moreover, each $V$-listing
$w$ has $n$ entries (since $\left\vert V\right\vert =n$), thus satisfies
$w=\left(  w_{1},w_{2},\ldots,w_{n}\right)  $.

We will use a definition that we made back in Lemma \ref{lem.friendlies-by-f}:
If $f:V\rightarrow\mathbb{P}$ is a map, and if $v=\left(  v_{1},v_{2}%
,\ldots,v_{n}\right)  $ is a $V$-listing, then this $V$-listing $v$ will be
called $\left(  f,D\right)  $\emph{-friendly} if it has the properties that
$f\left(  v_{1}\right)  \leq f\left(  v_{2}\right)  \leq\cdots\leq f\left(
v_{n}\right)  $ and that%
\[
f\left(  v_{p}\right)  <f\left(  v_{p+1}\right)  \text{ for each }p\in\left[
n-1\right]  \text{ satisfying }\left(  v_{p},v_{p+1}\right)  \in A.
\]

The definition of $U_{D}$ yields
\[
U_{D}=\sum_{w\text{ is a }V\text{-listing}}L_{\operatorname*{Des}\left(
w,D\right)  ,\ n}.
\]
We shall now try to understand the addends in this sum better.

We fix a $V$-listing $w$. Then, $w$ has $n$ entries (since $\left\vert
V\right\vert =n$), and thus satisfies $w=\left(  w_{1},w_{2},\ldots
,w_{n}\right)  $. Moreover, the list $\left(  w_{1},w_{2},\ldots,w_{n}\right)
=w$ is a $V$-listing, i.e., consists of all elements of $V$ and contains each
of these elements exactly once. In other words, $\left(  w_{1},w_{2}%
,\ldots,w_{n}\right)  $ is a list of all elements of $V$, with no repetitions.
Hence, if we are given an element $c_{v}$ of $\mathbb{Z}\left[  \left[
x_{1},x_{2},x_{3},\ldots\right]  \right]  $ for each $v\in V$, then%
\begin{equation}
\prod_{v\in V}c_{v}=c_{w_{1}}c_{w_{2}}\cdots c_{w_{n}}. \label{pf.thm.UX.1.1}%
\end{equation}
Thus, if $f:V\rightarrow\mathbb{P}$ is any map, then
\begin{equation}
\prod_{v\in V}x_{f\left(  v\right)  }=x_{f\left(  w_{1}\right)  }x_{f\left(
w_{2}\right)  }\cdots x_{f\left(  w_{n}\right)  } \label{pf.thm.UX.1.1f}%
\end{equation}
(by (\ref{pf.thm.UX.1.1}), applied to $c_{v}=x_{f\left(  v\right)  }$).

However, the definition of $L_{\operatorname*{Des}\left(  w,D\right)  ,\ n}$
yields%
\begin{align}
L_{\operatorname*{Des}\left(  w,D\right)  ,\ n}  &  =\sum_{\substack{i_{1}\leq
i_{2}\leq\cdots\leq i_{n};\\i_{p}<i_{p+1}\text{ for each }p\in
\operatorname*{Des}\left(  w,D\right)  }}x_{i_{1}}x_{i_{2}}\cdots x_{i_{n}%
}\nonumber\\
&  =\sum_{\substack{\left(  i_{1},i_{2},\ldots,i_{n}\right)  \in\mathbb{P}%
^{n};\\i_{1}\leq i_{2}\leq\cdots\leq i_{n};\\i_{p}<i_{p+1}\text{ for each
}p\in\operatorname*{Des}\left(  w,D\right)  }}x_{i_{1}}x_{i_{2}}\cdots
x_{i_{n}} \label{pf.thm.UX.1.L1}%
\end{align}
(here, we have added the \textquotedblleft$\left(  i_{1},i_{2},\ldots
,i_{n}\right)  \in\mathbb{P}^{n}$\textquotedblright\ condition under the
summation sign, since this condition is tacitly implied when we sum over
$i_{1}\leq i_{2}\leq\cdots\leq i_{n}$).

We recall that $\operatorname*{Des}\left(  w,D\right)  $ is defined as the set
of all $D$-descents of $w$, but these $D$-descents are defined as the elements
$i\in\left[  n-1\right]  $ satisfying $\left(  w_{i},w_{i+1}\right)  \in A$.
Hence, $\operatorname*{Des}\left(  w,D\right)  $ is the set of all elements
$i\in\left[  n-1\right]  $ satisfying $\left(  w_{i},w_{i+1}\right)  \in A$.
Thus, an element of $\operatorname*{Des}\left(  w,D\right)  $ is the same
thing as an element $i\in\left[  n-1\right]  $ satisfying $\left(
w_{i},w_{i+1}\right)  \in A$. Renaming the variable $i$ as $p$ in this
sentence, we obtain the following: An element of $\operatorname*{Des}\left(
w,D\right)  $ is the same thing as an element $p\in\left[  n-1\right]  $
satisfying $\left(  w_{p},w_{p+1}\right)  \in A$.

Lemma \ref{lem.maps-vs-tups} yields that the map
\begin{align*}
\left\{  \text{maps }f:V\rightarrow\mathbb{P}\right\}   &  \rightarrow
\mathbb{P}^{n},\\
f  &  \mapsto\left(  f\left(  w_{1}\right)  ,f\left(  w_{2}\right)
,\ldots,f\left(  w_{n}\right)  \right)
\end{align*}
is well-defined and is a bijection. Hence, we can substitute $\left(  f\left(
w_{1}\right)  ,f\left(  w_{2}\right)  ,\ldots,f\left(  w_{n}\right)  \right)
$ for $\left(  i_{1},i_{2},\ldots,i_{n}\right)  $ in the sum on the right hand
side of (\ref{pf.thm.UX.1.L1}). We thus obtain
\begin{align*}
&  \sum_{\substack{\left(  i_{1},i_{2},\ldots,i_{n}\right)  \in\mathbb{P}%
^{n};\\i_{1}\leq i_{2}\leq\cdots\leq i_{n};\\i_{p}<i_{p+1}\text{ for each
}p\in\operatorname*{Des}\left(  w,D\right)  }}x_{i_{1}}x_{i_{2}}\cdots
x_{i_{n}}\\
&  =\sum_{\substack{f:V\rightarrow\mathbb{P}\text{ is a map};\\f\left(
w_{1}\right)  \leq f\left(  w_{2}\right)  \leq\cdots\leq f\left(
w_{n}\right)  ;\\f\left(  w_{p}\right)  <f\left(  w_{p+1}\right)  \text{ for
each }p\in\operatorname*{Des}\left(  w,D\right)  }}\underbrace{x_{f\left(
w_{1}\right)  }x_{f\left(  w_{2}\right)  }\cdots x_{f\left(  w_{n}\right)  }%
}_{\substack{=\prod_{v\in V}x_{f\left(  v\right)  }\\\text{(by
(\ref{pf.thm.UX.1.1f}))}}}\\
&  =\sum_{\substack{f:V\rightarrow\mathbb{P}\text{ is a map};\\f\left(
w_{1}\right)  \leq f\left(  w_{2}\right)  \leq\cdots\leq f\left(
w_{n}\right)  ;\\f\left(  w_{p}\right)  <f\left(  w_{p+1}\right)  \text{ for
each }p\in\operatorname*{Des}\left(  w,D\right)  }}\prod_{v\in V}x_{f\left(
v\right)  }\\
&  =\sum_{\substack{f:V\rightarrow\mathbb{P}\text{ is a map};\\f\left(
w_{1}\right)  \leq f\left(  w_{2}\right)  \leq\cdots\leq f\left(
w_{n}\right)  ;\\f\left(  w_{p}\right)  <f\left(  w_{p+1}\right)  \text{ for
each }p\in\left[  n-1\right]  \\\text{satisfying }\left(  w_{p},w_{p+1}%
\right)  \in A}}\prod_{v\in V}x_{f\left(  v\right)  }%
\end{align*}
(here, we have replaced the condition \textquotedblleft$p\in
\operatorname*{Des}\left(  w,D\right)  $\textquotedblright\ under the
summation sign by the equivalent condition \textquotedblleft$p\in\left[
n-1\right]  $ satisfying $\left(  w_{p},w_{p+1}\right)  \in A$%
\textquotedblright, because an element of $\operatorname*{Des}\left(
w,D\right)  $ is the same thing as an element $p\in\left[  n-1\right]  $
satisfying $\left(  w_{p},w_{p+1}\right)  \in A$). Thus, (\ref{pf.thm.UX.1.L1}%
) becomes%
\begin{align}
L_{\operatorname*{Des}\left(  w,D\right)  ,\ n}  &  =\sum_{\substack{\left(
i_{1},i_{2},\ldots,i_{n}\right)  \in\mathbb{P}^{n};\\i_{1}\leq i_{2}\leq
\cdots\leq i_{n};\\i_{p}<i_{p+1}\text{ for each }p\in\operatorname*{Des}%
\left(  w,D\right)  }}x_{i_{1}}x_{i_{2}}\cdots x_{i_{n}}\nonumber\\
&  =\sum_{\substack{f:V\rightarrow\mathbb{P}\text{ is a map};\\f\left(
w_{1}\right)  \leq f\left(  w_{2}\right)  \leq\cdots\leq f\left(
w_{n}\right)  ;\\f\left(  w_{p}\right)  <f\left(  w_{p+1}\right)  \text{ for
each }p\in\left[  n-1\right]  \\\text{satisfying }\left(  w_{p},w_{p+1}%
\right)  \in A}}\prod_{v\in V}x_{f\left(  v\right)  }. \label{pf.thm.UX.1.L2}%
\end{align}

The sum on the right hand side of (\ref{pf.thm.UX.1.L2}) ranges over all maps
$f:V\rightarrow\mathbb{P}$ that satisfy the condition
\begin{align*}
&  \ \ \text{\textquotedblleft}f\left(  w_{1}\right)  \leq f\left(
w_{2}\right)  \leq\cdots\leq f\left(  w_{n}\right)  \text{\textquotedblright%
}\\
\wedge &  \text{\ \ \textquotedblleft}f\left(  w_{p}\right)  <f\left(
w_{p+1}\right)  \text{ for each }p\in\left[  n-1\right]  \text{ satisfying
}\left(  w_{p},w_{p+1}\right)  \in A\text{\textquotedblright.}%
\end{align*}
However, this condition is equivalent to the condition \textquotedblleft the
$V$-listing $w$ is $\left(  f,D\right)  $-friendly\textquotedblright\ (because
this is how the notion of \textquotedblleft$\left(  f,D\right)  $%
-friendly\textquotedblright\ was defined). Therefore, we can replace the
former condition by the latter condition under the summation sign on the right
hand side of (\ref{pf.thm.UX.1.L2}). Thus, we can rewrite
(\ref{pf.thm.UX.1.L2}) as follows:%
\begin{equation}
L_{\operatorname*{Des}\left(  w,D\right)  ,\ n}=\sum_{\substack{f:V\rightarrow
\mathbb{P}\text{ is a map};\\\text{the }V\text{-listing }w\text{ is }\left(
f,D\right)  \text{-friendly}}}\prod_{v\in V}x_{f\left(  v\right)  }.
\label{pf.thm.UX.1.L3}%
\end{equation}

Now, forget that we fixed $w$. We thus have proved (\ref{pf.thm.UX.1.L3}) for
each $V$-listing $w$.

\begin{vershort}
Now,
\begin{align*}
U_{D}  &  =\sum_{w\text{ is a }V\text{-listing}}L_{\operatorname*{Des}\left(
w,D\right)  ,\ n}\\
&  =\underbrace{\sum_{w\text{ is a }V\text{-listing}}\ \ \sum
_{\substack{f:V\rightarrow\mathbb{P}\text{ is a map};\\\text{the
}V\text{-listing }w\text{ is }\left(  f,D\right)  \text{-friendly}}}}%
_{=\sum_{f:V\rightarrow\mathbb{P}}\ \ \sum_{w\text{ is an }\left(  f,D\right)
\text{-friendly }V\text{-listing}}}\prod_{v\in V}x_{f\left(  v\right)
}\ \ \ \ \ \ \ \ \ \ \left(  \text{by (\ref{pf.thm.UX.1.L3})}\right) \\
&  =\sum_{f:V\rightarrow\mathbb{P}}\ \ \underbrace{\sum_{w\text{ is an
}\left(  f,D\right)  \text{-friendly }V\text{-listing}}\ \ \prod_{v\in
V}x_{f\left(  v\right)  }}_{=\left(  \text{\# of }\left(  f,D\right)
\text{-friendly }V\text{-listings}\right)  \cdot\prod_{v\in V}x_{f\left(
v\right)  }}\\
&  =\sum_{f:V\rightarrow\mathbb{P}}\underbrace{\left(  \text{\# of }\left(
f,D\right)  \text{-friendly }V\text{-listings}\right)  }_{\substack{=\sum
_{\substack{\sigma\in\mathfrak{S}_{V};\\f\circ\sigma=f}}\ \ \sum
_{\substack{F\subseteq\mathbf{A}_{\sigma}\cap A\\\text{is linear}}}\left(
-1\right)  ^{\left\vert F\right\vert }\\\text{(by Lemma
\ref{lem.friendlies-by-f})}}}\cdot\prod_{v\in V}x_{f\left(  v\right)  }\\
&  =\sum_{f:V\rightarrow\mathbb{P}}\ \ \sum_{\substack{\sigma\in
\mathfrak{S}_{V};\\f\circ\sigma=f}}\ \ \sum_{\substack{F\subseteq
\mathbf{A}_{\sigma}\cap A\\\text{is linear}}}\left(  -1\right)  ^{\left\vert
F\right\vert }\cdot\prod_{v\in V}x_{f\left(  v\right)  }\\
&  =\sum_{\sigma\in\mathfrak{S}_{V}}\ \ \underbrace{\sum_{\substack{F\subseteq
\mathbf{A}_{\sigma}\cap A\\\text{is linear}}}\left(  -1\right)  ^{\left\vert
F\right\vert }}_{\substack{=%
\begin{cases}
\left(  -1\right)  ^{\varphi\left(  \sigma\right)  }, & \text{if }\sigma
\in\mathfrak{S}_{V}\left(  D,\overline{D}\right)  ;\\
0, & \text{else}%
\end{cases}
\\\text{(by Proposition \ref{prop.linear-in-AsigA})}}}\underbrace{\sum
_{\substack{f:V\rightarrow\mathbb{P}\text{;}\\f\circ\sigma=f}}\ \ \prod_{v\in
V}x_{f\left(  v\right)  }}_{\substack{=p_{\operatorname*{type}\sigma
}\\\text{(by Lemma \ref{lem.ptype-as-sum})}}}\\
&  =\sum_{\sigma\in\mathfrak{S}_{V}}%
\begin{cases}
\left(  -1\right)  ^{\varphi\left(  \sigma\right)  }, & \text{if }\sigma
\in\mathfrak{S}_{V}\left(  D,\overline{D}\right)  ;\\
0, & \text{else}%
\end{cases}
\ \ p_{\operatorname*{type}\sigma}\\
&  =\sum_{\sigma\in\mathfrak{S}_{V}\left(  D,\overline{D}\right)  }\left(
-1\right)  ^{\varphi\left(  \sigma\right)  }p_{\operatorname*{type}\sigma}.
\end{align*}

\end{vershort}

\begin{verlong}
Now,%
\begin{align*}
U_{D}  &  =\sum_{w\text{ is a }V\text{-listing}}L_{\operatorname*{Des}\left(
w,D\right)  ,\ n}\\
&  =\underbrace{\sum_{w\text{ is a }V\text{-listing}}\ \ \sum
_{\substack{f:V\rightarrow\mathbb{P}\text{ is a map};\\\text{the
}V\text{-listing }w\text{ is }\left(  f,D\right)  \text{-friendly}}}}%
_{=\sum_{f:V\rightarrow\mathbb{P}\text{ is a map}}\ \ \sum_{\substack{w\text{
is a }V\text{-listing;}\\\text{the }V\text{-listing }w\text{ is }\left(
f,D\right)  \text{-friendly}}}}\prod_{v\in V}x_{f\left(  v\right)
}\ \ \ \ \ \ \ \ \ \ \left(  \text{by (\ref{pf.thm.UX.1.L3})}\right) \\
&  =\underbrace{\sum_{f:V\rightarrow\mathbb{P}\text{ is a map}}}%
_{=\sum_{f:V\rightarrow\mathbb{P}}}\ \ \underbrace{\sum_{\substack{w\text{ is
a }V\text{-listing;}\\\text{the }V\text{-listing }w\text{ is }\left(
f,D\right)  \text{-friendly}}}}_{=\sum_{w\text{ is an }\left(  f,D\right)
\text{-friendly }V\text{-listing}}}\ \ \prod_{v\in V}x_{f\left(  v\right)  }\\
&  =\sum_{f:V\rightarrow\mathbb{P}}\ \ \underbrace{\sum_{w\text{ is an
}\left(  f,D\right)  \text{-friendly }V\text{-listing}}\ \ \prod_{v\in
V}x_{f\left(  v\right)  }}_{\substack{=\left(  \text{\# of }\left(
f,D\right)  \text{-friendly }V\text{-listings}\right)  \cdot\prod_{v\in
V}x_{f\left(  v\right)  }\\\text{(since all addends of this sum have the same
value }\prod_{v\in V}x_{f\left(  v\right)  }\text{)}}}\\
&  =\sum_{f:V\rightarrow\mathbb{P}}\underbrace{\left(  \text{\# of }\left(
f,D\right)  \text{-friendly }V\text{-listings}\right)  }_{\substack{=\sum
_{\substack{\sigma\in\mathfrak{S}_{V};\\f\circ\sigma=f}}\ \ \sum
_{\substack{F\subseteq\mathbf{A}_{\sigma}\cap A\\\text{is linear}}}\left(
-1\right)  ^{\left\vert F\right\vert }\\\text{(by Lemma
\ref{lem.friendlies-by-f})}}}\cdot\prod_{v\in V}x_{f\left(  v\right)  }\\
&  =\underbrace{\sum_{f:V\rightarrow\mathbb{P}}\ \ \sum_{\substack{\sigma
\in\mathfrak{S}_{V};\\f\circ\sigma=f}}\ \ \sum_{\substack{F\subseteq
\mathbf{A}_{\sigma}\cap A\\\text{is linear}}}}_{=\sum_{\sigma\in
\mathfrak{S}_{V}}\ \ \sum_{\substack{F\subseteq\mathbf{A}_{\sigma}\cap
A\\\text{is linear}}}\ \ \sum_{\substack{f:V\rightarrow\mathbb{P}%
\text{;}\\f\circ\sigma=f}}}\left(  -1\right)  ^{\left\vert F\right\vert }%
\cdot\prod_{v\in V}x_{f\left(  v\right)  }\\
&  =\sum_{\sigma\in\mathfrak{S}_{V}}\ \ \sum_{\substack{F\subseteq
\mathbf{A}_{\sigma}\cap A\\\text{is linear}}}\ \ \sum
_{\substack{f:V\rightarrow\mathbb{P}\text{;}\\f\circ\sigma=f}}\left(
-1\right)  ^{\left\vert F\right\vert }\cdot\prod_{v\in V}x_{f\left(  v\right)
}\\
&  =\sum_{\sigma\in\mathfrak{S}_{V}}\ \ \underbrace{\sum_{\substack{F\subseteq
\mathbf{A}_{\sigma}\cap A\\\text{is linear}}}\left(  -1\right)  ^{\left\vert
F\right\vert }}_{\substack{=%
\begin{cases}
\left(  -1\right)  ^{\varphi\left(  \sigma\right)  }, & \text{if }\sigma
\in\mathfrak{S}_{V}\left(  D,\overline{D}\right)  ;\\
0, & \text{else}%
\end{cases}
\\\text{(by Proposition \ref{prop.linear-in-AsigA})}}}\underbrace{\sum
_{\substack{f:V\rightarrow\mathbb{P}\text{;}\\f\circ\sigma=f}}\ \ \prod_{v\in
V}x_{f\left(  v\right)  }}_{\substack{=p_{\operatorname*{type}\sigma
}\\\text{(by Lemma \ref{lem.ptype-as-sum})}}}
\end{align*}%
\begin{align*}
&  =\sum_{\sigma\in\mathfrak{S}_{V}}%
\begin{cases}
\left(  -1\right)  ^{\varphi\left(  \sigma\right)  }, & \text{if }\sigma
\in\mathfrak{S}_{V}\left(  D,\overline{D}\right)  ;\\
0, & \text{else}%
\end{cases}
\ \ p_{\operatorname*{type}\sigma}\\
&  =\underbrace{\sum_{\substack{\sigma\in\mathfrak{S}_{V};\\\sigma
\in\mathfrak{S}_{V}\left(  D,\overline{D}\right)  }}}_{\substack{=\sum
_{\sigma\in\mathfrak{S}_{V}\left(  D,\overline{D}\right)  }\\\text{(since
}\mathfrak{S}_{V}\left(  D,\overline{D}\right)  \\\text{is a subset of
}\mathfrak{S}_{V}\text{)}}}\underbrace{%
\begin{cases}
\left(  -1\right)  ^{\varphi\left(  \sigma\right)  }, & \text{if }\sigma
\in\mathfrak{S}_{V}\left(  D,\overline{D}\right)  ;\\
0, & \text{else}%
\end{cases}
}_{\substack{=\left(  -1\right)  ^{\varphi\left(  \sigma\right)
}\\\text{(since we have }\sigma\in\mathfrak{S}_{V}\left(  D,\overline
{D}\right)  \text{)}}}\ \ p_{\operatorname*{type}\sigma}\\
&  \ \ \ \ \ \ \ \ \ \ +\sum_{\substack{\sigma\in\mathfrak{S}_{V};\\\text{we
don't have }\sigma\in\mathfrak{S}_{V}\left(  D,\overline{D}\right)
}}\underbrace{%
\begin{cases}
\left(  -1\right)  ^{\varphi\left(  \sigma\right)  }, & \text{if }\sigma
\in\mathfrak{S}_{V}\left(  D,\overline{D}\right)  ;\\
0, & \text{else}%
\end{cases}
}_{\substack{=0\\\text{(since we don't have }\sigma\in\mathfrak{S}_{V}\left(
D,\overline{D}\right)  \text{)}}}\ \ p_{\operatorname*{type}\sigma}\\
&  \ \ \ \ \ \ \ \ \ \ \ \ \ \ \ \ \ \ \ \ \left(  \text{since each }\sigma
\in\mathfrak{S}_{V}\text{ either satisfies }\sigma\in\mathfrak{S}_{V}\left(
D,\overline{D}\right)  \text{ or doesn't}\right) \\
&  =\sum_{\sigma\in\mathfrak{S}_{V}\left(  D,\overline{D}\right)  }\left(
-1\right)  ^{\varphi\left(  \sigma\right)  }p_{\operatorname*{type}\sigma
}+\underbrace{\sum_{\substack{\sigma\in\mathfrak{S}_{V};\\\text{we don't have
}\sigma\in\mathfrak{S}_{V}\left(  D,\overline{D}\right)  }%
}0p_{\operatorname*{type}\sigma}}_{=0}\\
&  =\sum_{\sigma\in\mathfrak{S}_{V}\left(  D,\overline{D}\right)  }\left(
-1\right)  ^{\varphi\left(  \sigma\right)  }p_{\operatorname*{type}\sigma}.
\end{align*}

\end{verlong}

\begin{vershort}
\noindent This proves Theorem \ref{thm.UX.1}.
\end{vershort}
\end{proof}

\section{\label{sec.pf.thm.UX.2}Proof of Theorem \ref{thm.UX.2}}

Theorem \ref{thm.UX.2} can be derived from Theorem \ref{thm.UX.1} by combining
some addends that have the same $p_{\operatorname*{type}\sigma}$ factor.
Depending on the respective $\left(  -1\right)  ^{\varphi\left(
\sigma\right)  }$ factors, these addends either cancel each other out or
combine to form a multiple of $p_{\operatorname*{type}\sigma}$.

\begin{proof}
[Proof of Theorem \ref{thm.UX.2}.]We have assumed that $D$ is a tournament.
Hence, for any two distinct vertices $u$ and $v$ of $D$, we have the logical
equivalences%
\[
\left(  \left(  u,v\right)  \text{ is an arc of }D\right)
\ \Longleftrightarrow\ \left(  \left(  v,u\right)  \text{ is an arc of
}\overline{D}\right)
\]
and%
\[
\left(  \left(  u,v\right)  \text{ is an arc of }\overline{D}\right)
\ \Longleftrightarrow\ \left(  \left(  v,u\right)  \text{ is an arc of
}D\right)  .
\]
Therefore, the reversal\footnote{See Definition \ref{def.reqc.features} for
the meanings of \textquotedblleft reversal\textquotedblright\ and
\textquotedblleft nontrivial\textquotedblright.} of a nontrivial $D$-cycle is
always a nontrivial $\overline{D}$-cycle, and vice versa.

We define a map $\Psi:\mathfrak{S}_{V}\left(  D,\overline{D}\right)
\rightarrow\mathfrak{S}_{V}\left(  D\right)  $ as follows: If $\sigma
\in\mathfrak{S}_{V}\left(  D\right)  $, then we let $\Psi\left(
\sigma\right)  $ be the permutation obtained from $\sigma$ by reversing each
cycle of $\sigma$ that is a nontrivial $\overline{D}$-cycle (i.e., replacing
this cycle of $\sigma$ by its reversal, i.e., replacing $\sigma$ by
$\sigma^{-1}$ on all entries of this cycle)\footnote{Here is what this means
in rigorous terms: We let $\Psi\left(  \sigma\right)  $ be the permutation of
$V$ defined by setting%
\begin{align*}
\left(  \Psi\left(  \sigma\right)  \right)  \left(  z\right)   &  =%
\begin{cases}
\sigma^{-1}\left(  z\right)  , & \text{if }z\text{ is an entry of a cycle of
}\sigma\text{ that is a nontrivial }\overline{D}\text{-cycle};\\
\sigma\left(  z\right)  , & \text{otherwise}%
\end{cases}
\\
&  \ \ \ \ \ \ \ \ \ \ \text{for each }z\in V.
\end{align*}
The cycles of this permutation $\Psi\left(  \sigma\right)  $ are precisely
\par
\begin{itemize}
\item the reversals of those cycles of $\sigma$ that are nontrivial
$\overline{D}$-cycles, and
\par
\item the remaining cycles of $\sigma$.
\end{itemize}
}. This map $\Psi$ is well-defined (i.e., we really have $\Psi\left(
\sigma\right)  \in\mathfrak{S}_{V}\left(  D\right)  $ for each $\sigma
\in\mathfrak{S}_{V}\left(  D,\overline{D}\right)  $), because as we just said,
the reversal of a nontrivial $\overline{D}$-cycle is always a nontrivial
$D$-cycle. Moreover, the map $\Psi$ preserves the cycle type of a permutation
-- i.e., we have%
\begin{equation}
\operatorname*{type}\left(  \Psi\left(  \sigma\right)  \right)
=\operatorname*{type}\sigma\label{pf.thm.UX.2.types}%
\end{equation}
for each $\sigma\in\mathfrak{S}_{V}\left(  D,\overline{D}\right)  $.

Now, Theorem \ref{thm.UX.1} yields%
\begin{align}
U_{D}  &  =\sum_{\sigma\in\mathfrak{S}_{V}\left(  D,\overline{D}\right)
}\left(  -1\right)  ^{\varphi\left(  \sigma\right)  }%
\underbrace{p_{\operatorname*{type}\sigma}}%
_{\substack{=p_{\operatorname*{type}\left(  \Psi\left(  \sigma\right)
\right)  }\\\text{(by (\ref{pf.thm.UX.2.types}))}}}=\sum_{\sigma
\in\mathfrak{S}_{V}\left(  D,\overline{D}\right)  }\left(  -1\right)
^{\varphi\left(  \sigma\right)  }p_{\operatorname*{type}\left(  \Psi\left(
\sigma\right)  \right)  }\nonumber\\
&  =\sum_{\tau\in\mathfrak{S}_{V}\left(  D\right)  }\ \ \sum_{\substack{\sigma
\in\mathfrak{S}_{V}\left(  D,\overline{D}\right)  ;\\\Psi\left(
\sigma\right)  =\tau}}\left(  -1\right)  ^{\varphi\left(  \sigma\right)
}p_{\operatorname*{type}\tau}\ \ \ \ \ \ \ \ \ \ \left(
\begin{array}
[c]{c}%
\text{here, we have split up the sum}\\
\text{according to the value of }\Psi\left(  \sigma\right)
\end{array}
\right) \nonumber\\
&  =\sum_{\tau\in\mathfrak{S}_{V}\left(  D\right)  }\left(  \sum
_{\substack{\sigma\in\mathfrak{S}_{V}\left(  D,\overline{D}\right)
;\\\Psi\left(  \sigma\right)  =\tau}}\left(  -1\right)  ^{\varphi\left(
\sigma\right)  }\right)  p_{\operatorname*{type}\tau}. \label{pf.thm.UX.2.2}%
\end{align}

Now, we claim that each $\tau\in\mathfrak{S}_{V}\left(  D\right)  $ satisfies%
\begin{equation}
\sum_{\substack{\sigma\in\mathfrak{S}_{V}\left(  D,\overline{D}\right)
;\\\Psi\left(  \sigma\right)  =\tau}}\left(  -1\right)  ^{\varphi\left(
\sigma\right)  }=%
\begin{cases}
2^{\psi\left(  \tau\right)  }, & \text{if all cycles of }\tau\text{ have odd
length};\\
0, & \text{otherwise.}%
\end{cases}
\label{pf.thm.UX.2.cancel}%
\end{equation}

[\textit{Proof of (\ref{pf.thm.UX.2.cancel}):} Let $\tau\in\mathfrak{S}%
_{V}\left(  D\right)  $. Then, $\tau$ has exactly $\psi\left(  \tau\right)  $
many nontrivial cycles (by the definition of $\psi\left(  \tau\right)  $), and
all of these nontrivial cycles are $D$-cycles (by the definition of
$\mathfrak{S}_{V}\left(  D\right)  $). The permutations $\sigma\in
\mathfrak{S}_{V}\left(  D,\overline{D}\right)  $ that satisfy $\Psi\left(
\sigma\right)  =\tau$ can be obtained by choosing some of these nontrivial
cycles and reversing them, which turns them into $\overline{D}$-cycles. This
can be done in $2^{\psi\left(  \tau\right)  }$ many ways, since each of the
$\psi\left(  \tau\right)  $ many nontrivial cycles can be either reversed or
not. If all cycles of $\tau$ have odd length, then all $2^{\psi\left(
\tau\right)  }$ permutations $\sigma$ obtained in this way will satisfy
$\left(  -1\right)  ^{\varphi\left(  \sigma\right)  }=1$ (because
$\varphi\left(  \sigma\right)  =\sum_{\substack{\gamma\in\operatorname*{Cycs}%
\sigma;\\\gamma\text{ is a }D\text{-cycle}}}\left(  \underbrace{\ell\left(
\gamma\right)  }_{\text{odd}}-1\right)  $ will always be even); therefore, the
sum $\sum_{\substack{\sigma\in\mathfrak{S}_{V}\left(  D,\overline{D}\right)
;\\\Psi\left(  \sigma\right)  =\tau}}\left(  -1\right)  ^{\varphi\left(
\sigma\right)  }$ will be a sum of $2^{\psi\left(  \tau\right)  }$ many $1$s
and therefore simplify to $2^{\psi\left(  \tau\right)  }$. On the other hand,
if not all cycles of $\tau$ have odd length, then there is at least one cycle
$\delta$ of $\tau$ that has even length, and of course this cycle $\delta$
will be nontrivial (since a trivial cycle has odd length); thus, among the
permutations $\sigma\in\mathfrak{S}_{V}\left(  D,\overline{D}\right)  $ that
satisfy $\Psi\left(  \sigma\right)  =\tau$, there will be as many that have
$\delta$ reversed as ones that have $\delta$ not reversed, and the parities of
$\varphi\left(  \sigma\right)  $ for the former will be opposite from the
parities of $\varphi\left(  \sigma\right)  $ for the latter; thus, the sum
$\sum_{\substack{\sigma\in\mathfrak{S}_{V}\left(  D,\overline{D}\right)
;\\\Psi\left(  \sigma\right)  =\tau}}\left(  -1\right)  ^{\varphi\left(
\sigma\right)  }$ will have equally many $1$s and $-1$s among its addends, and
therefore will simplify to $0$. In either case, we obtain
(\ref{pf.thm.UX.2.cancel}).] \medskip

Now, (\ref{pf.thm.UX.2.2}) becomes%
\begin{align*}
U_{D}  &  =\sum_{\tau\in\mathfrak{S}_{V}\left(  D\right)  }\underbrace{\left(
\sum_{\substack{\sigma\in\mathfrak{S}_{V}\left(  D,\overline{D}\right)
;\\\Psi\left(  \sigma\right)  =\tau}}\left(  -1\right)  ^{\varphi\left(
\sigma\right)  }\right)  }_{\substack{=%
\begin{cases}
2^{\psi\left(  \tau\right)  }, & \text{if all cycles of }\tau\text{ have odd
length};\\
0, & \text{otherwise}%
\end{cases}
\\\text{(by (\ref{pf.thm.UX.2.cancel}))}}}p_{\operatorname*{type}\tau}\\
&  =\sum_{\tau\in\mathfrak{S}_{V}\left(  D\right)  }%
\begin{cases}
2^{\psi\left(  \tau\right)  }, & \text{if all cycles of }\tau\text{ have odd
length};\\
0, & \text{otherwise}%
\end{cases}
\ \ p_{\operatorname*{type}\tau}\\
&  =\sum_{\substack{\tau\in\mathfrak{S}_{V}\left(  D\right)  ;\\\text{all
cycles of }\tau\text{ have odd length}}}2^{\psi\left(  \tau\right)
}p_{\operatorname*{type}\tau}=\sum_{\substack{\sigma\in\mathfrak{S}_{V}\left(
D\right)  ;\\\text{all cycles of }\sigma\text{ have odd length}}%
}2^{\psi\left(  \sigma\right)  }p_{\operatorname*{type}\sigma}.
\end{align*}
This proves Theorem \ref{thm.UX.2}.
\end{proof}

\section{\label{sec.pf.cors}Proving the corollaries}

Let us now quickly go through the proofs of the corollaries we stated after
Theorem \ref{thm.UX.1} and after Theorem \ref{thm.UX.2}:

\begin{proof}
[Proof of Corollary \ref{cor.UX.p-int}.]We let $\mathbb{N}\left[  p_{1}%
,p_{2},p_{3},\ldots\right]  $ denote the set of all polynomials in
$p_{1},p_{2},p_{3},\ldots$ with coefficients in $\mathbb{N}$.

For each integer partition $\lambda$, we have%
\begin{equation}
p_{\lambda}\in\mathbb{N}\left[  p_{1},p_{2},p_{3},\ldots\right]
\label{pf.cor.UX.p-int.1}%
\end{equation}
(by the definition of $p_{\lambda}$).

Theorem \ref{thm.UX.1} yields%
\begin{align*}
U_{D}  &  =\sum_{\sigma\in\mathfrak{S}_{V}\left(  D,\overline{D}\right)
}\left(  -1\right)  ^{\varphi\left(  \sigma\right)  }%
\underbrace{p_{\operatorname*{type}\sigma}}_{\substack{\in\mathbb{N}\left[
p_{1},p_{2},p_{3},\ldots\right]  \\\text{(by (\ref{pf.cor.UX.p-int.1}))}}}\\
&  \in\sum_{\sigma\in\mathfrak{S}_{V}\left(  D,\overline{D}\right)  }\left(
-1\right)  ^{\varphi\left(  \sigma\right)  }\mathbb{N}\left[  p_{1}%
,p_{2},p_{3},\ldots\right]  \subseteq\mathbb{Z}\left[  p_{1},p_{2}%
,p_{3},\ldots\right]  .
\end{align*}
This proves Corollary \ref{cor.UX.p-int}.
\end{proof}

\begin{proof}
[Proof of Corollary \ref{cor.UX.if-each-D-cyc-odd}.]Let $2\mathbb{Z}$ denote
the set of all even integers.

Let $\sigma\in\mathfrak{S}_{V}\left(  D,\overline{D}\right)  $. The definition
of $\varphi\left(  \sigma\right)  $ in Theorem \ref{thm.UX.1} yields%
\[
\varphi\left(  \sigma\right)  =\sum_{\substack{\gamma\in\operatorname*{Cycs}%
\sigma;\\\gamma\text{ is a }D\text{-cycle}}}\underbrace{\left(  \ell\left(
\gamma\right)  -1\right)  }_{\substack{\in2\mathbb{Z}\\\text{(since }%
\ell\left(  \gamma\right)  \text{ is odd}\\\text{(because every }%
D\text{-cycle}\\\text{has odd length))}}}\in2\mathbb{Z},
\]
so that%
\begin{equation}
\left(  -1\right)  ^{\varphi\left(  \sigma\right)  }=1.
\label{pf.cor.UX.if-each-D-cyc-odd.1}%
\end{equation}

Theorem \ref{thm.UX.1} now yields%
\begin{align*}
U_{D}  &  =\sum_{\sigma\in\mathfrak{S}_{V}\left(  D,\overline{D}\right)
}\underbrace{\left(  -1\right)  ^{\varphi\left(  \sigma\right)  }%
}_{\substack{=1\\\text{(by (\ref{pf.cor.UX.if-each-D-cyc-odd.1}))}%
}}p_{\operatorname*{type}\sigma}\\
&  =\sum_{\sigma\in\mathfrak{S}_{V}\left(  D,\overline{D}\right)
}\underbrace{p_{\operatorname*{type}\sigma}}_{\substack{\in\mathbb{N}\left[
p_{1},p_{2},p_{3},\ldots\right]  \\\text{(by (\ref{pf.cor.UX.p-int.1}))}}}\\
&  \in\sum_{\sigma\in\mathfrak{S}_{V}\left(  D,\overline{D}\right)
}\mathbb{N}\left[  p_{1},p_{2},p_{3},\ldots\right]  \subseteq\mathbb{N}\left[
p_{1},p_{2},p_{3},\ldots\right]  .
\end{align*}
This proves Corollary \ref{cor.UX.if-each-D-cyc-odd}.
\end{proof}

\begin{proof}
[Proof of Corollary \ref{cor.UX.tournament-N}.]For each $\sigma\in
\mathfrak{S}_{V}$, let $\psi\left(  \sigma\right)  $ denote the number of
nontrivial cycles of $\sigma$.

Let $\sigma\in\mathfrak{S}_{V}\left(  D\right)  $ be a permutation whose all
cycles have odd length. We shall show that $2^{\psi\left(  \sigma\right)
}p_{\operatorname*{type}\sigma}\in\mathbb{N}\left[  p_{1},2p_{3},2p_{5}%
,2p_{7},\ldots\right]  $.

Indeed, let $k_{1},k_{2},\ldots,k_{s}$ be the lengths of all cycles of
$\sigma$, listed in decreasing order. Then, the numbers $k_{1},k_{2}%
,\ldots,k_{s}$ are odd (since all cycles of $\sigma$ have odd length).
Moreover, the definition of $\operatorname*{type}\sigma$ yields
$\operatorname*{type}\sigma=\left(  k_{1},k_{2},\ldots,k_{s}\right)  $.
Furthermore,
\begin{align*}
\psi\left(  \sigma\right)   &  =\left(  \text{\# of nontrivial cycles of
}\sigma\right) \\
&  =\left(  \text{\# of cycles of }\sigma\text{ that have length }>1\right) \\
&  =\left(  \text{\# of }i\in\left[  s\right]  \text{ such that }%
k_{i}>1\right) \\
&  \ \ \ \ \ \ \ \ \ \ \ \ \ \ \ \ \ \ \ \ \left(  \text{since the lengths of
all cycles of }\sigma\text{ are }k_{1},k_{2},\ldots,k_{s}\right) \\
&  =\sum_{i=1}^{s}\left[  k_{i}>1\right]
\end{align*}
(here, we are using the Iverson bracket notation), so that%
\begin{equation}
2^{\psi\left(  \sigma\right)  }=2^{\sum_{i=1}^{s}\left[  k_{i}>1\right]
}=\prod_{i=1}^{s}2^{\left[  k_{i}>1\right]  }.
\label{pf.cor.UX.tournament-N.1}%
\end{equation}
Now, recall that $\operatorname*{type}\sigma=\left(  k_{1},k_{2},\ldots
,k_{s}\right)  $. Hence, the definition of $p_{\operatorname*{type}\sigma}$
yields
\begin{equation}
p_{\operatorname*{type}\sigma}=p_{k_{1}}p_{k_{2}}\cdots p_{k_{s}}=\prod
_{i=1}^{s}p_{k_{i}}. \label{pf.cor.UX.tournament-N.2}%
\end{equation}
Multiplying the equalities (\ref{pf.cor.UX.tournament-N.1}) and
(\ref{pf.cor.UX.tournament-N.2}), we obtain
\begin{align}
2^{\psi\left(  \sigma\right)  }p_{\operatorname*{type}\sigma}  &  =\left(
\prod_{i=1}^{s}2^{\left[  k_{i}>1\right]  }\right)  \left(  \prod_{i=1}%
^{s}p_{k_{i}}\right)  =\prod_{i=1}^{s}\underbrace{\left(  2^{\left[
k_{i}>1\right]  }p_{k_{i}}\right)  }_{\substack{\in\left\{  p_{1}%
,2p_{3},2p_{5},2p_{7},\ldots\right\}  \\\text{(since }k_{i}\text{ is
odd}\\\text{(because }k_{1},k_{2},\ldots,k_{s}\text{ are odd))}}}\nonumber\\
&  =\left(  \text{a product of }s\text{ elements of the set }\left\{
p_{1},2p_{3},2p_{5},2p_{7},\ldots\right\}  \right) \nonumber\\
&  \in\mathbb{N}\left[  p_{1},2p_{3},2p_{5},2p_{7},\ldots\right]  .
\label{pf.cor.UX.tournament-N.4}%
\end{align}

Forget that we fixed $\sigma$. We thus have proved
(\ref{pf.cor.UX.tournament-N.4}) for each permutation $\sigma\in
\mathfrak{S}_{V}\left(  D\right)  $ whose all cycles have odd length. Now,
Theorem \ref{thm.UX.2} yields%
\[
U_{D}=\sum_{\substack{\sigma\in\mathfrak{S}_{V}\left(  D\right)  ;\\\text{all
cycles of }\sigma\text{ have odd length}}}\underbrace{2^{\psi\left(
\sigma\right)  }p_{\operatorname*{type}\sigma}}_{\substack{\in\mathbb{N}%
\left[  p_{1},2p_{3},2p_{5},2p_{7},\ldots\right]  \\\text{(by
(\ref{pf.cor.UX.tournament-N.4}))}}}\in\mathbb{N}\left[  p_{1},2p_{3}%
,2p_{5},2p_{7},\ldots\right]  .
\]
This proves Corollary \ref{cor.UX.tournament-N}.
\end{proof}

\section{\label{sec.pf.thm.UX.3}Proof of Theorem \ref{thm.UX.3}}

The proof of Theorem \ref{thm.UX.3} is a slightly more complicated variant of
our above proof of Theorem \ref{thm.UX.2}.

\begin{proof}
[Proof of Theorem \ref{thm.UX.3}.]\textbf{(b)} First, we attempt to gain a
better understanding of risky cycles.

We start by noticing that the reversal of a risky rotation-equivalence class
is again risky.

We have assumed that there exist no two distinct vertices $u$ and $v$ of $D$
such that both pairs $\left(  u,v\right)  $ and $\left(  v,u\right)  $ belong
to $A$. In other words, if $\left(  u,v\right)  $ is an arc of $D$ with $u\neq
v$, then $\left(  v,u\right)  $ is not an arc of $D$, and thus $\left(
v,u\right)  $ must be an arc of $\overline{D}$.

Hence, if $v$ is any $D$-cycle of length $\geq2$, then the reversal of $v$
must be a $\overline{D}$-cycle, and thus cannot be a $D$-cycle. Therefore, in
particular, if $v$ is a risky rotation-equivalence class of tuples of elements
of $V$, then either $v$ or the reversal of $v$ is a $D$-cycle (by the
definition of \textquotedblleft risky\textquotedblright), but not both at the
same time.

Consequently, if $v$ is a risky rotation-equivalence class of tuples of
elements of $V$, then $v$ and the reversal of $v$ cannot be identical, i.e.,
we must have%
\begin{equation}
v\neq\operatorname*{rev}v. \label{pf.thm.UX.3.not-rev}%
\end{equation}

We define a subset $\mathfrak{S}_{V}^{\circ}\left(  D,\overline{D}\right)  $
of $\mathfrak{S}_{V}\left(  D,\overline{D}\right)  $ by%
\[
\mathfrak{S}_{V}^{\circ}\left(  D,\overline{D}\right)  :=\left\{  \sigma
\in\mathfrak{S}_{V}\left(  D,\overline{D}\right)  \ \mid\ \text{each risky
cycle of }\sigma\text{ is a }D\text{-cycle}\right\}  .
\]

We define a map $\Gamma:\mathfrak{S}_{V}\left(  D,\overline{D}\right)
\rightarrow\mathfrak{S}_{V}^{\circ}\left(  D,\overline{D}\right)  $ as
follows: If $\sigma\in\mathfrak{S}_{V}\left(  D,\overline{D}\right)  $, then
we let $\Gamma\left(  \sigma\right)  $ be the permutation obtained from
$\sigma$ by reversing each risky cycle of $\sigma$ that is not a $D$-cycle
(i.e., replacing this cycle of $\sigma$ by its reversal, i.e., replacing
$\sigma$ by $\sigma^{-1}$ on all entries of this cycle). This map $\Gamma$ is
well-defined (i.e., we really have $\Gamma\left(  \sigma\right)
\in\mathfrak{S}_{V}^{\circ}\left(  D,\overline{D}\right)  $ for each
$\sigma\in\mathfrak{S}_{V}\left(  D,\overline{D}\right)  $), because if a
risky tuple is not a $D$-cycle, then its reversal is a $D$-cycle (by the
definition of \textquotedblleft risky\textquotedblright). Moreover, the map
$\Gamma$ preserves the cycle type of a permutation -- i.e., we have%
\begin{equation}
\operatorname*{type}\left(  \Gamma\left(  \sigma\right)  \right)
=\operatorname*{type}\sigma\label{pf.cor.UX.p-pos.types}%
\end{equation}
for each $\sigma\in\mathfrak{S}_{V}\left(  D,\overline{D}\right)  $.

Now, Theorem \ref{thm.UX.1} yields%
\begin{align}
U_{D}  &  =\sum_{\sigma\in\mathfrak{S}_{V}\left(  D,\overline{D}\right)
}\left(  -1\right)  ^{\varphi\left(  \sigma\right)  }%
\underbrace{p_{\operatorname*{type}\sigma}}%
_{\substack{=p_{\operatorname*{type}\left(  \Gamma\left(  \sigma\right)
\right)  }\\\text{(by (\ref{pf.cor.UX.p-pos.types}))}}}=\sum_{\sigma
\in\mathfrak{S}_{V}\left(  D,\overline{D}\right)  }\left(  -1\right)
^{\varphi\left(  \sigma\right)  }p_{\operatorname*{type}\left(  \Gamma\left(
\sigma\right)  \right)  }\nonumber\\
&  =\sum_{\tau\in\mathfrak{S}_{V}^{\circ}\left(  D,\overline{D}\right)
}\ \ \sum_{\substack{\sigma\in\mathfrak{S}_{V}\left(  D,\overline{D}\right)
;\\\Gamma\left(  \sigma\right)  =\tau}}\left(  -1\right)  ^{\varphi\left(
\sigma\right)  }p_{\operatorname*{type}\tau}\ \ \ \ \ \ \ \ \ \ \left(
\begin{array}
[c]{c}%
\text{here, we have split up the sum}\\
\text{according to the value of }\Gamma\left(  \sigma\right)
\end{array}
\right) \nonumber\\
&  =\sum_{\tau\in\mathfrak{S}_{V}^{\circ}\left(  D,\overline{D}\right)
}\left(  \sum_{\substack{\sigma\in\mathfrak{S}_{V}\left(  D,\overline
{D}\right)  ;\\\Gamma\left(  \sigma\right)  =\tau}}\left(  -1\right)
^{\varphi\left(  \sigma\right)  }\right)  p_{\operatorname*{type}\tau}.
\label{pf.cor.UX.p-pos.2}%
\end{align}

Now, we claim that each $\tau\in\mathfrak{S}_{V}^{\circ}\left(  D,\overline
{D}\right)  $ satisfies%
\begin{equation}
\sum_{\substack{\sigma\in\mathfrak{S}_{V}\left(  D,\overline{D}\right)
;\\\Gamma\left(  \sigma\right)  =\tau}}\left(  -1\right)  ^{\varphi\left(
\sigma\right)  }=%
\begin{cases}
\left(  -1\right)  ^{\varphi\left(  \tau\right)  }, & \text{if no cycle of
}\tau\text{ is risky};\\
0, & \text{otherwise.}%
\end{cases}
\label{pf.cor.UX.p-pos.cancel}%
\end{equation}

[\textit{Proof of (\ref{pf.cor.UX.p-pos.cancel}):} Let $\tau\in\mathfrak{S}%
_{V}^{\circ}\left(  D,\overline{D}\right)  $. Let $c_{1},c_{2},\ldots,c_{k}$
be the risky cycles of $\tau$. All of these $k$ risky cycles $c_{1}%
,c_{2},\ldots,c_{k}$ are $D$-cycles (since $\tau\in\mathfrak{S}_{V}^{\circ
}\left(  D,\overline{D}\right)  $). The permutations $\sigma\in\mathfrak{S}%
_{V}\left(  D,\overline{D}\right)  $ that satisfy $\Gamma\left(
\sigma\right)  =\tau$ can be obtained by choosing some of these $k$ risky
cycles $c_{1},c_{2},\ldots,c_{k}$ of $\tau$ and reversing them, which turns
them into $\overline{D}$-cycles (because if $v$ is any $D$-cycle of length
$\geq2$, then the reversal of $v$ must be a $\overline{D}$-cycle). This can be
done in $2^{k}$ many ways, since each of the $k$ risky cycles $c_{1}%
,c_{2},\ldots,c_{k}$ can be either reversed or not\footnote{Fineprint: All of
these $k$ risky cycles are distinct from their reversals (by
(\ref{pf.thm.UX.3.not-rev})). Thus, each of the $2^{k}$ possible choices of
risky cycles to reverse leads to a different permutation $\sigma
\in\mathfrak{S}_{V}\left(  D,\overline{D}\right)  $.}. The sum $\sum
_{\substack{\sigma\in\mathfrak{S}_{V}\left(  D,\overline{D}\right)
;\\\Gamma\left(  \sigma\right)  =\tau}}\left(  -1\right)  ^{\varphi\left(
\sigma\right)  }$ thus has $2^{k}$ many addends, and each of these addends
corresponds to one way to decide which of the $k$ risky cycles $c_{1}%
,c_{2},\ldots,c_{k}$ to reverse and which not to reverse. If $k=0$, then this
sum therefore simplifies to $\left(  -1\right)  ^{\varphi\left(  \tau\right)
}$. If, on the other hand, we have $k\neq0$, then this sum equals
$0$\ \ \ \ \footnote{\textit{Proof.} Assume that $k\neq0$. Thus, $k\geq1$, so
that the risky cycle $c_{1}$ exists. If $\sigma\in\mathfrak{S}_{V}\left(
D,\overline{D}\right)  $ is such that $\Gamma\left(  \sigma\right)  =\tau$,
then either the cycle $c_{1}$ or its reversal (but not both) is a cycle of
$\sigma$. Thus,
\begin{align}
&  \sum_{\substack{\sigma\in\mathfrak{S}_{V}\left(  D,\overline{D}\right)
;\\\Gamma\left(  \sigma\right)  =\tau}}\left(  -1\right)  ^{\varphi\left(
\sigma\right)  }\nonumber\\
&  =\sum_{\substack{\sigma\in\mathfrak{S}_{V}\left(  D,\overline{D}\right)
;\\\Gamma\left(  \sigma\right)  =\tau;\\c_{1}\text{ is a cycle of }\sigma
}}\left(  -1\right)  ^{\varphi\left(  \sigma\right)  }+\sum_{\substack{\sigma
\in\mathfrak{S}_{V}\left(  D,\overline{D}\right)  ;\\\Gamma\left(
\sigma\right)  =\tau;\\c_{1}\text{ is not a cycle of }\sigma}}\left(
-1\right)  ^{\varphi\left(  \sigma\right)  }.
\label{pf.cor.UX.p-pos.cancel.pf.1}%
\end{align}
The two sums on the right hand side of this equality have the same number of
addends, and there is in fact a bijection between the addends of the former
and those of the latter (given by replacing the cycle $c_{1}$ by its reversal
or vice versa). Moreover, this bijection toggles the parity of the number
$\varphi\left(  \sigma\right)  $ (that is, it changes this number from odd to
even or vice versa), since $\varphi\left(  \sigma\right)  $ is defined to be
the sum $\sum_{\substack{\gamma\in\operatorname*{Cycs}\sigma;\\\gamma\text{ is
a }D\text{-cycle}}}\left(  \ell\left(  \gamma\right)  -1\right)  $ (which
contains the odd addend $\ell\left(  c_{1}\right)  -1$ when $c_{1}$ is a cycle
of $\sigma$, but does not contain this addend when $c_{1}$ is not a cycle of
$\sigma$). Hence, this bijection flips the sign $\left(  -1\right)
^{\varphi\left(  \sigma\right)  }$. Therefore, the addends in the first sum on
the right hand side of (\ref{pf.cor.UX.p-pos.cancel.pf.1}) cancel those in the
second. Therefore, the two sums add up to $0$. The equality
(\ref{pf.cor.UX.p-pos.cancel.pf.1}) thus simplifies to $\sum_{\substack{\sigma
\in\mathfrak{S}_{V}\left(  D,\overline{D}\right)  ;\\\Gamma\left(
\sigma\right)  =\tau}}\left(  -1\right)  ^{\varphi\left(  \sigma\right)  }=0$,
qed.}. Combining the results from both of these cases, we obtain
\begin{align*}
\sum_{\substack{\sigma\in\mathfrak{S}_{V}\left(  D,\overline{D}\right)
;\\\Gamma\left(  \sigma\right)  =\tau}}\left(  -1\right)  ^{\varphi\left(
\sigma\right)  }  &  =%
\begin{cases}
\left(  -1\right)  ^{\varphi\left(  \tau\right)  }, & \text{if }k=0;\\
0, & \text{otherwise}%
\end{cases}
\\
&  =%
\begin{cases}
\left(  -1\right)  ^{\varphi\left(  \tau\right)  }, & \text{if no cycle of
}\tau\text{ is risky};\\
0, & \text{otherwise.}%
\end{cases}
\end{align*}
(since $k$ is the number of risky cycles of $\tau$). This proves
(\ref{pf.cor.UX.p-pos.cancel}).] \medskip

Now, (\ref{pf.cor.UX.p-pos.2}) becomes%
\begin{align*}
U_{D}  &  =\sum_{\tau\in\mathfrak{S}_{V}^{\circ}\left(  D,\overline{D}\right)
}\underbrace{\left(  \sum_{\substack{\sigma\in\mathfrak{S}_{V}\left(
D,\overline{D}\right)  ;\\\Gamma\left(  \sigma\right)  =\tau}}\left(
-1\right)  ^{\varphi\left(  \sigma\right)  }\right)  }_{\substack{=%
\begin{cases}
\left(  -1\right)  ^{\varphi\left(  \tau\right)  }, & \text{if no cycle of
}\tau\text{ is risky};\\
0, & \text{otherwise}%
\end{cases}
\\\text{(by (\ref{pf.cor.UX.p-pos.cancel}))}}}p_{\operatorname*{type}\tau}\\
&  =\sum_{\tau\in\mathfrak{S}_{V}^{\circ}\left(  D,\overline{D}\right)  }%
\begin{cases}
\left(  -1\right)  ^{\varphi\left(  \tau\right)  }, & \text{if no cycle of
}\tau\text{ is risky};\\
0, & \text{otherwise}%
\end{cases}
\ \ p_{\operatorname*{type}\tau}\\
&  =\sum_{\substack{\tau\in\mathfrak{S}_{V}^{\circ}\left(  D,\overline
{D}\right)  ;\\\text{no cycle of }\tau\text{ is risky}}%
}\ \ \underbrace{\left(  -1\right)  ^{\varphi\left(  \tau\right)  }%
}_{\substack{=1\\\text{(since no cycle of }\tau\text{ is risky,}\\\text{and
thus it is easy to see}\\\text{that }\varphi\left(  \tau\right)  \text{ is
even)}}}p_{\operatorname*{type}\tau}\\
&  =\sum_{\substack{\tau\in\mathfrak{S}_{V}^{\circ}\left(  D,\overline
{D}\right)  ;\\\text{no cycle of }\tau\text{ is risky}}%
}p_{\operatorname*{type}\tau}=\sum_{\substack{\tau\in\mathfrak{S}_{V}\left(
D,\overline{D}\right)  ;\\\text{no cycle of }\tau\text{ is risky}%
}}p_{\operatorname*{type}\tau}%
\end{align*}
(since the permutations $\tau\in\mathfrak{S}_{V}^{\circ}\left(  D,\overline
{D}\right)  $ that have no risky cycles are precisely the permutations
$\tau\in\mathfrak{S}_{V}\left(  D,\overline{D}\right)  $ that have no risky
cycles\footnote{This follows trivially from the definition of $\mathfrak{S}%
_{V}^{\circ}\left(  D,\overline{D}\right)  $.}). Renaming the summation index
$\tau$ as $\sigma$ on the right hand side, we obtain%
\[
U_{D}=\sum_{\substack{\sigma\in\mathfrak{S}_{V}\left(  D,\overline{D}\right)
;\\\text{no cycle of }\sigma\text{ is risky}}}p_{\operatorname*{type}\sigma}.
\]
This proves Theorem \ref{thm.UX.3} \textbf{(b)}.

\textbf{(a)} This follows trivially from part \textbf{(b)}, since $p_{\lambda
}\in\mathbb{N}\left[  p_{1},p_{2},p_{3},\ldots\right]  $ for each partition
$\lambda$.
\end{proof}

\section{\label{sec.redei}Recovering Redei's and Berge's theorems}

We shall now derive two well-known theorems from Theorem \ref{thm.UX.1} and
Theorem \ref{thm.UX.2}.

We recall Convention \ref{conv.number} and Definition \ref{def.hamp}. The two
theorems we shall derive are the following:

\begin{theorem}
[R\'{e}dei's Theorem]\label{thm.tourn.redei}Let $D$ be a tournament. Then, the
\# of hamps of $D$ is odd. Here, we agree to consider the empty list $\left(
{}\right)  $ as a hamp of the empty tournament with $0$ vertices.
\end{theorem}

\begin{theorem}
[Berge's Theorem]\label{thm.hamp.Dbar}Let $D$ be a digraph. Then,%
\[
\left(  \text{\# of hamps of }\overline{D}\right)  \equiv\left(  \text{\# of
hamps of }D\right)  \operatorname{mod}2.
\]

\end{theorem}

Theorem \ref{thm.tourn.redei} originates in Laszlo R\'{e}dei's 1933 paper
\cite{Redei33} (see \cite[proof of Theorem 14]{Moon13} for an English
translation of his proof). Theorem \ref{thm.hamp.Dbar} was found by Claude
Berge (see \cite[\S 10.1, Theorem 1]{Berge76}, \cite[\S 10.1, Theorem
1]{Berge91}, \cite[solution to problem 7.8]{Tomesc85}, \cite[Exercise
5.19]{Lovasz07} or \cite[Theorem 1.3.6]{17s-lec7} for his proof, and
\cite[Corollaire 5.1]{Lass02} for another). Berge used Theorem
\ref{thm.hamp.Dbar} to give a new and simpler proof of Theorem
\ref{thm.tourn.redei} (see \cite[\S 10.2, Theorem 6]{Berge91} or
\cite[Exercise 5.20]{Lovasz07} or \cite[Theorem 1.6.1]{17s-lec7}).

We can now give new proofs for both theorems. This will rely on the symmetric
function $U_{D}$ and also on a few simple tools:

We define $\zeta:\operatorname*{QSym}\rightarrow\mathbb{Z}$ to be the
evaluation homomorphism that sends each quasisymmetric function $f\in
\operatorname*{QSym}$ to its evaluation $f\left(  1,0,0,0,\ldots\right)  $
(obtained by setting $x_{1}$ to be $1$ and setting all other variables
$x_{2},x_{3},x_{4},\ldots$ to be $0$). Note that $\zeta$ is a $\mathbb{Z}%
$-algebra homomorphism.\footnote{We don't really need $\operatorname*{QSym}$
here. We could just as well define $\zeta$ on the ring of bounded-degree power
series (that is, of all power series $f\in\mathbb{Z}\left[  \left[
x_{1},x_{2},x_{3},\ldots\right]  \right]  $ for which there exists an
$N\in\mathbb{N}$ such that no monomial of degree $>N$ appears in $f$).
However, we cannot define $\zeta$ on the whole ring $\mathbb{Z}\left[  \left[
x_{1},x_{2},x_{3},\ldots\right]  \right]  $, since $\zeta$ would have to send
$1+x_{1}+x_{1}^{2}+x_{1}^{3}+\cdots$ to $1+1+1^{2}+1^{3}+\cdots$.} We shall
show two simple lemmas:

\begin{lemma}
\label{lem.zeta.L}Let $n\in\mathbb{N}$. Let $I$ be a subset of $\left[
n-1\right]  $. Then, $\zeta\left(  L_{I,\ n}\right)  =\left[  I=\varnothing
\right]  $ (where we are using the Iverson bracket notation).
\end{lemma}

\begin{proof}
The definition of $L_{I,\ n}$ yields%
\[
L_{I,\ n}=\sum_{\substack{i_{1}\leq i_{2}\leq\cdots\leq i_{n};\\i_{p}%
<i_{p+1}\text{ for each }p\in I}}x_{i_{1}}x_{i_{2}}\cdots x_{i_{n}}.
\]

When we apply $\zeta$ to the sum on the right hand side (i.e., substitute $1$
for $x_{1}$ and substitute $0$ for $x_{2},x_{3},x_{4},\ldots$), any addend
that contains at least one of the variables $x_{2},x_{3},x_{4},\ldots$ becomes
$0$, whereas any addend that only contains copies of $x_{1}$ becomes $1$.
Hence, $\zeta\left(  L_{I,\ n}\right)  $ is the number of addends that only
contain copies of $x_{1}$. But this number is $1$ if $I=\varnothing$ (namely,
in this case, the addend for $\left(  i_{1},i_{2},\ldots,i_{n}\right)
=\left(  1,1,\ldots,1\right)  $ fits the bill), and is $0$ if $I\neq
\varnothing$ (because in this case, the condition \textquotedblleft%
$i_{p}<i_{p+1}$ for each $p\in I$\textquotedblright\ forces at least one of
the $n$ numbers $i_{1},i_{2},\ldots,i_{n}$ in each addend $x_{i_{1}}x_{i_{2}%
}\cdots x_{i_{n}}$ to be larger than $1$, and therefore each addend contains
at least one of $x_{2},x_{3},x_{4},\ldots$). Thus, altogether, this number is
$\left[  I=\varnothing\right]  $.\ This proves Lemma \ref{lem.zeta.L}.
\end{proof}

\begin{lemma}
\label{lem.zeta.p}Let $\lambda$ be any partition. Then,%
\[
\zeta\left(  p_{\lambda}\right)  =1.
\]

\end{lemma}

\begin{proof}
Write the partition $\lambda$ in the form $\lambda=\left(  \lambda_{1}%
,\lambda_{2},\ldots,\lambda_{k}\right)  $, where the $k$ entries $\lambda
_{1},\lambda_{2},\ldots,\lambda_{k}$ are positive. Then, the definition of
$p_{\lambda}$ yields $p_{\lambda}=p_{\lambda_{1}}p_{\lambda_{2}}\cdots
p_{\lambda_{k}}$. Hence,%
\begin{align}
\zeta\left(  p_{\lambda}\right)   &  =\zeta\left(  p_{\lambda_{1}}%
p_{\lambda_{2}}\cdots p_{\lambda_{k}}\right)  =\zeta\left(  p_{\lambda_{1}%
}\right)  \zeta\left(  p_{\lambda_{2}}\right)  \cdots\zeta\left(
p_{\lambda_{k}}\right) \nonumber\\
&  \ \ \ \ \ \ \ \ \ \ \ \ \ \ \ \ \ \ \ \ \left(  \text{since }\zeta\text{ is
a }\mathbb{Z}\text{-algebra homomorphism}\right) \nonumber\\
&  =\prod_{i=1}^{k}\zeta\left(  p_{\lambda_{i}}\right)  .
\label{pf.lem.zeta.p.1}%
\end{align}
However, for each positive integer $n$, we have $p_{n}=x_{1}^{n}+x_{2}%
^{n}+x_{3}^{n}+\cdots$ (by the definition of $p_{n}$) and%
\begin{align}
\zeta\left(  p_{n}\right)   &  =p_{n}\left(  1,0,0,0,\ldots\right)
\ \ \ \ \ \ \ \ \ \ \left(  \text{by the definition of }\zeta\right)
\nonumber\\
&  =\underbrace{1^{n}}_{=1}+\underbrace{0^{n}+0^{n}+0^{n}+\cdots
}_{\substack{=0\\\text{(since }n\text{ is positive)}}%
}\ \ \ \ \ \ \ \ \ \ \left(  \text{since }p_{n}=x_{1}^{n}+x_{2}^{n}+x_{3}%
^{n}+\cdots\right) \nonumber\\
&  =1. \label{pf.lem.zeta.p.2}%
\end{align}
Hence, (\ref{pf.lem.zeta.p.1}) becomes%
\[
\zeta\left(  p_{\lambda}\right)  =\prod_{i=1}^{k}\underbrace{\zeta\left(
p_{\lambda_{i}}\right)  }_{\substack{=1\\\text{(by (\ref{pf.lem.zeta.p.2}%
),}\\\text{since }\lambda_{i}\text{ is positive)}}}=\prod_{i=1}^{k}1=1.
\]
This proves Lemma \ref{lem.zeta.p}.
\end{proof}

\begin{lemma}
\label{lem.zeta.U}Let $D$ be a digraph. Then,%
\[
\zeta\left(  U_{D}\right)  =\left(  \text{\# of hamps of }\overline{D}\right)
.
\]

\end{lemma}

\begin{proof}
Write $D$ as $D=\left(  V,A\right)  $, and set $n:=\left\vert V\right\vert $.
Then, $\overline{D}=\left(  V,\ \left(  V\times V\right)  \setminus A\right)
$. Hence, a hamp of $\overline{D}$ is the same as a $V$-listing $w$ such that
each $i\in\left[  n-1\right]  $ satisfies $\left(  w_{i},w_{i+1}\right)
\in\left(  V\times V\right)  \setminus A$. In other words, a hamp of
$\overline{D}$ is the same as a $V$-listing $w$ such that no $i\in\left[
n-1\right]  $ satisfies $\left(  w_{i},w_{i+1}\right)  \in A$. In other words,
a hamp of $\overline{D}$ is the same as a $V$-listing $w$ that satisfies
$\operatorname*{Des}\left(  w,D\right)  =\varnothing$ (because
$\operatorname*{Des}\left(  w,D\right)  $ is defined to be the set of all
$i\in\left[  n-1\right]  $ satisfying $\left(  w_{i},w_{i+1}\right)  \in A$).
Therefore,%
\begin{align}
&  \left(  \text{\# of hamps of }\overline{D}\right) \nonumber\\
&  =\left(  \text{\# of }V\text{-listings }w\text{ that satisfy }%
\operatorname*{Des}\left(  w,D\right)  =\varnothing\right)  .
\label{pf.lem.zeta.U.1}%
\end{align}

The definition of $U_{D}$ yields%
\[
U_{D}=\sum_{w\text{ is a }V\text{-listing}}L_{\operatorname*{Des}\left(
w,D\right)  ,\ n}.
\]
Hence,%
\begin{align*}
\zeta\left(  U_{D}\right)   &  =\zeta\left(  \sum_{w\text{ is a }%
V\text{-listing}}L_{\operatorname*{Des}\left(  w,D\right)  ,\ n}\right) \\
&  =\sum_{w\text{ is a }V\text{-listing}}\underbrace{\zeta\left(
L_{\operatorname*{Des}\left(  w,D\right)  ,\ n}\right)  }_{\substack{=\left[
\operatorname*{Des}\left(  w,D\right)  =\varnothing\right]  \\\text{(by Lemma
\ref{lem.zeta.L})}}}\ \ \ \ \ \ \ \ \ \ \left(  \text{since the map }%
\zeta\text{ is }\mathbb{Z}\text{-linear}\right) \\
&  =\sum_{w\text{ is a }V\text{-listing}}\left[  \operatorname*{Des}\left(
w,D\right)  =\varnothing\right] \\
&  =\left(  \text{\# of }V\text{-listings }w\text{ that satisfy }%
\operatorname*{Des}\left(  w,D\right)  =\varnothing\right) \\
&  =\left(  \text{\# of hamps of }\overline{D}\right)
\ \ \ \ \ \ \ \ \ \ \left(  \text{by (\ref{pf.lem.zeta.U.1})}\right)  .
\end{align*}
This proves Lemma \ref{lem.zeta.U}.
\end{proof}

We can now state a formula for the \# of hamps of $\overline{D}$ for any
digraph $D$:

\begin{theorem}
\label{thm.hamps-formula-Dbar}Let $D=\left(  V,A\right)  $ be a digraph. Then:

\textbf{(a)} Set%
\[
\varphi\left(  \sigma\right)  :=\sum_{\substack{\gamma\in\operatorname*{Cycs}%
\sigma;\\\gamma\text{ is a }D\text{-cycle}}}\left(  \ell\left(  \gamma\right)
-1\right)  \ \ \ \ \ \ \ \ \ \ \text{for each }\sigma\in\mathfrak{S}_{V}.
\]
Then,
\[
\left(  \text{\# of hamps of }\overline{D}\right)  =\sum_{\sigma
\in\mathfrak{S}_{V}\left(  D,\overline{D}\right)  }\left(  -1\right)
^{\varphi\left(  \sigma\right)  }.
\]

\textbf{(b)} We have $\left(  \text{\# of hamps of }\overline{D}\right)
\equiv\left\vert \mathfrak{S}_{V}\left(  D,\overline{D}\right)  \right\vert
\operatorname{mod}2$.
\end{theorem}

\begin{proof}
\textbf{(a)} Theorem \ref{thm.UX.1} yields%
\[
U_{D}=\sum_{\sigma\in\mathfrak{S}_{V}\left(  D,\overline{D}\right)  }\left(
-1\right)  ^{\varphi\left(  \sigma\right)  }p_{\operatorname*{type}\sigma}.
\]
Hence,%
\begin{align*}
\zeta\left(  U_{D}\right)   &  =\zeta\left(  \sum_{\sigma\in\mathfrak{S}%
_{V}\left(  D,\overline{D}\right)  }\left(  -1\right)  ^{\varphi\left(
\sigma\right)  }p_{\operatorname*{type}\sigma}\right) \\
&  =\sum_{\sigma\in\mathfrak{S}_{V}\left(  D,\overline{D}\right)  }\left(
-1\right)  ^{\varphi\left(  \sigma\right)  }\underbrace{\zeta\left(
p_{\operatorname*{type}\sigma}\right)  }_{\substack{=1\\\text{(by Lemma
\ref{lem.zeta.p},}\\\text{applied to }\lambda=\operatorname*{type}%
\sigma\text{)}}}\ \ \ \ \ \ \ \ \ \ \left(  \text{since }\zeta\text{ is
}\mathbb{Z}\text{-linear}\right) \\
&  =\sum_{\sigma\in\mathfrak{S}_{V}\left(  D,\overline{D}\right)  }\left(
-1\right)  ^{\varphi\left(  \sigma\right)  }.
\end{align*}
However, Lemma \ref{lem.zeta.U} yields%
\[
\zeta\left(  U_{D}\right)  =\left(  \text{\# of hamps of }\overline{D}\right)
.
\]
Comparing these two equalities, we find%
\[
\left(  \text{\# of hamps of }\overline{D}\right)  =\sum_{\sigma
\in\mathfrak{S}_{V}\left(  D,\overline{D}\right)  }\left(  -1\right)
^{\varphi\left(  \sigma\right)  }.
\]
This proves Theorem \ref{thm.hamps-formula-Dbar} \textbf{(a)}.

\textbf{(b)} Theorem \ref{thm.hamps-formula-Dbar} \textbf{(a)} yields%
\[
\left(  \text{\# of hamps of }\overline{D}\right)  =\sum_{\sigma
\in\mathfrak{S}_{V}\left(  D,\overline{D}\right)  }\underbrace{\left(
-1\right)  ^{\varphi\left(  \sigma\right)  }}_{\substack{\equiv
1\operatorname{mod}2\\\text{(since }\left(  -1\right)  ^{k}\equiv
1\operatorname{mod}2\\\text{for any }k\in\mathbb{Z}\text{)}}}\equiv
\sum_{\sigma\in\mathfrak{S}_{V}\left(  D,\overline{D}\right)  }1=\left\vert
\mathfrak{S}_{V}\left(  D,\overline{D}\right)  \right\vert \operatorname{mod}%
2.
\]
This proves Theorem \ref{thm.hamps-formula-Dbar} \textbf{(b)}.
\end{proof}

We are now ready to prove R\'{e}dei's and Berge's theorems:

\begin{proof}
[Proof of Theorem \ref{thm.hamp.Dbar}.]We have $\mathfrak{S}_{V}\left(
\overline{D},D\right)  =\mathfrak{S}_{V}\left(  D,\overline{D}\right)  $
(since the digraphs $D$ and $\overline{D}$ play symmetric roles in the
definition of $\mathfrak{S}_{V}\left(  D,\overline{D}\right)  $). However, it
is also easy to see (using the definition of the complement of a digraph) that
$\overline{\overline{D}}=D$.

Theorem \ref{thm.hamps-formula-Dbar} \textbf{(b)} yields
\begin{equation}
\left(  \text{\# of hamps of }\overline{D}\right)  \equiv\left\vert
\mathfrak{S}_{V}\left(  D,\overline{D}\right)  \right\vert \operatorname{mod}%
2. \label{pf.thm.hamp.Dbar.1}%
\end{equation}
However, Theorem \ref{thm.hamps-formula-Dbar} \textbf{(b)} (applied to
$\overline{D}$ instead of $D$) yields%
\[
\left(  \text{\# of hamps of }\overline{\overline{D}}\right)  \equiv\left\vert
\mathfrak{S}_{V}\left(  \overline{D},\overline{\overline{D}}\right)
\right\vert \operatorname{mod}2.
\]
We can rewrite this as%
\[
\left(  \text{\# of hamps of }D\right)  \equiv\left\vert \mathfrak{S}%
_{V}\left(  \overline{D},D\right)  \right\vert \operatorname{mod}2
\]
(since $\overline{\overline{D}}=D$). Hence,%
\begin{align*}
\left(  \text{\# of hamps of }D\right)   &  \equiv\left\vert \mathfrak{S}%
_{V}\left(  \overline{D},D\right)  \right\vert =\left\vert \mathfrak{S}%
_{V}\left(  D,\overline{D}\right)  \right\vert \ \ \ \ \ \ \ \ \ \ \left(
\text{since }\mathfrak{S}_{V}\left(  \overline{D},D\right)  =\mathfrak{S}%
_{V}\left(  D,\overline{D}\right)  \right) \\
&  \equiv\left(  \text{\# of hamps of }\overline{D}\right)  \operatorname{mod}%
2\ \ \ \ \ \ \ \ \ \ \left(  \text{by (\ref{pf.thm.hamp.Dbar.1})}\right)  .
\end{align*}
This proves Theorem \ref{thm.hamp.Dbar}.
\end{proof}

\begin{proof}
[Proof of Theorem \ref{thm.tourn.redei}.]Write the tournament $D$ as
$D=\left(  V,A\right)  $. Set $n:=\left\vert V\right\vert $.

For each $\sigma\in\mathfrak{S}_{V}$, let $\psi\left(  \sigma\right)  $ denote
the number of nontrivial cycles of $\sigma$. Then, Theorem \ref{thm.UX.2}
yields%
\[
U_{D}=\sum_{\substack{\sigma\in\mathfrak{S}_{V}\left(  D\right)  ;\\\text{all
cycles of }\sigma\text{ have odd length}}}2^{\psi\left(  \sigma\right)
}p_{\operatorname*{type}\sigma}.
\]
Hence,%
\begin{align}
\zeta\left(  U_{D}\right)   &  =\zeta\left(  \sum_{\substack{\sigma
\in\mathfrak{S}_{V}\left(  D\right)  ;\\\text{all cycles of }\sigma\text{ have
odd length}}}2^{\psi\left(  \sigma\right)  }p_{\operatorname*{type}\sigma
}\right) \nonumber\\
&  =\sum_{\substack{\sigma\in\mathfrak{S}_{V}\left(  D\right)  ;\\\text{all
cycles of }\sigma\text{ have odd length}}}2^{\psi\left(  \sigma\right)
}\underbrace{\zeta\left(  p_{\operatorname*{type}\sigma}\right)
}_{\substack{=1\\\text{(by Lemma \ref{lem.zeta.p},}\\\text{applied to }%
\lambda=\operatorname*{type}\sigma\text{)}}}\ \ \ \ \ \ \ \ \ \ \left(
\text{since }\zeta\text{ is }\mathbb{Z}\text{-linear}\right) \nonumber\\
&  =\sum_{\substack{\sigma\in\mathfrak{S}_{V}\left(  D\right)  ;\\\text{all
cycles of }\sigma\text{ have odd length}}}2^{\psi\left(  \sigma\right)
}\label{pf.thm.tourn.redei.zUD=}\\
&  =\underbrace{2^{\psi\left(  \operatorname*{id}\nolimits_{V}\right)  }%
}_{\substack{=1\\\text{(since }\psi\left(  \operatorname*{id}\nolimits_{V}%
\right)  =0\text{)}}}+\sum_{\substack{\sigma\in\mathfrak{S}_{V}\left(
D\right)  ;\\\text{all cycles of }\sigma\text{ have odd length;}\\\sigma
\neq\operatorname*{id}\nolimits_{V}}}\underbrace{2^{\psi\left(  \sigma\right)
}}_{\substack{\equiv0\operatorname{mod}2\\\text{(since }\psi\left(
\sigma\right)  \geq1\\\text{(because }\sigma\neq\operatorname*{id}%
\nolimits_{V}\text{ shows}\\\text{that }\sigma\text{ has at least}\\\text{one
nontrivial cycle))}}}\nonumber\\
&  \ \ \ \ \ \ \ \ \ \ \ \ \ \ \ \ \ \ \ \ \left(
\begin{array}
[c]{c}%
\text{here, we have split off the addend for }\sigma=\operatorname*{id}%
\nolimits_{V}\\
\text{from the sum (since }\operatorname*{id}\nolimits_{V}\in\mathfrak{S}%
_{V}\left(  D\right)  \text{, and since}\\
\text{all cycles of }\operatorname*{id}\nolimits_{V}\text{ have odd length)}%
\end{array}
\right) \nonumber\\
&  \equiv1+\underbrace{\sum_{\substack{\sigma\in\mathfrak{S}_{V}\left(
D\right)  ;\\\text{all cycles of }\sigma\text{ have odd length;}\\\sigma
\neq\operatorname*{id}\nolimits_{V}}}0}_{=0}=1\operatorname{mod}2.\nonumber
\end{align}
In view of%
\begin{align*}
\zeta\left(  U_{D}\right)   &  =\left(  \text{\# of hamps of }\overline
{D}\right)  \ \ \ \ \ \ \ \ \ \ \left(  \text{by Lemma \ref{lem.zeta.U}%
}\right) \\
&  \equiv\left(  \text{\# of hamps of }D\right)  \operatorname{mod}%
2\ \ \ \ \ \ \ \ \ \ \left(  \text{by Theorem \ref{thm.hamp.Dbar}}\right)  ,
\end{align*}
we can rewrite this as%
\[
\left(  \text{\# of hamps of }D\right)  \equiv1\operatorname{mod}2.
\]
In other words, the \# of hamps of $D$ is odd. This proves Theorem
\ref{thm.tourn.redei}.
\end{proof}

\section{\label{sec.mod4}A modulo-$4$ improvement of Redei's theorem}

We can extend Redei's theorem (Theorem \ref{thm.tourn.redei}) to a somewhat
stronger result:

\begin{theorem}
\label{thm.tourn.mod4}Let $D$ be a tournament. Then,
\[
\left(  \text{\# of hamps of }D\right)  \equiv1+2\left(  \text{\# of
nontrivial odd }D\text{-cycles}\right)  \operatorname{mod}4.
\]
Here:

\begin{itemize}
\item We agree to consider the empty list $\left(  {}\right)  $ as a hamp of
the empty tournament with $0$ vertices (even though it is not a path).

\item We say that a $D$-cycle is \emph{odd} if its length is odd.

\item We say that a $D$-cycle is \emph{nontrivial} if its length is $>1$.
(This was already said in Definition \ref{def.reqc.features} \textbf{(e)}.)
\end{itemize}
\end{theorem}

To prove this, we shall need a simple lemma:

\begin{lemma}
\label{lem.psi=1}Let $D=\left(  V,A\right)  $ be a digraph. For each
$\sigma\in\mathfrak{S}_{V}$, let $\psi\left(  \sigma\right)  $ denote the
number of nontrivial cycles of $\sigma$. Let $\mathfrak{S}_{V}%
^{\operatorname*{odd}}\left(  D\right)  $ denote the set of all permutations
$\sigma\in\mathfrak{S}_{V}\left(  D\right)  $ such that all cycles of $\sigma$
have odd length. Then,%
\begin{align*}
&  \left(  \text{\# of permutations }\sigma\in\mathfrak{S}_{V}%
^{\operatorname*{odd}}\left(  D\right)  \text{ satisfying }\psi\left(
\sigma\right)  =1\right) \\
&  =\left(  \text{\# of nontrivial odd }D\text{-cycles}\right)  .
\end{align*}
(We are here using the same notations as in Theorem \ref{thm.tourn.mod4}.)
\end{lemma}

\begin{proof}
If $\gamma=\left(  a_{1},a_{2},\ldots,a_{k}\right)  _{\sim}$ is any $D$-cycle
(or, more generally, any cycle of the digraph $\left(  V,\ V\times V\right)
$), then $\operatorname*{perm}\gamma$ shall denote the permutation of $V$ that
sends the elements $a_{1},a_{2},\ldots,a_{k-1},a_{k}$ to $a_{2},a_{3}%
,\ldots,a_{k},a_{1}$ (respectively) while leaving all other elements of $V$
unchanged. (This permutation $\operatorname*{perm}\gamma$ is what is usually
called \textquotedblleft the cycle $\left(  a_{1},a_{2},\ldots,a_{k}\right)
$\textquotedblright\ in group theory.)

If $\gamma$ is any nontrivial $D$-cycle, then the permutation
$\operatorname*{perm}\gamma$ belongs to $\mathfrak{S}_{V}\left(  D\right)  $
(since its only nontrivial cycle is $\gamma$, which is a $D$-cycle) and
satisfies $\psi\left(  \operatorname*{perm}\gamma\right)  =1$ (by the
definition of $\psi\left(  \operatorname*{perm}\gamma\right)  $). Moreover, if
$\gamma$ is a nontrivial \textbf{odd} $D$-cycle, then this permutation
$\operatorname*{perm}\gamma$ furthermore has the property that all its cycles
have odd length (since its only nontrivial cycle $\gamma$ is odd, whereas its
trivial cycles have length $1$, which is also odd), i.e., belongs to
$\mathfrak{S}_{V}^{\operatorname*{odd}}\left(  D\right)  $ (since we know that
it belongs to $\mathfrak{S}_{V}\left(  D\right)  $). Thus, we obtain a map%
\begin{align*}
&  \text{from }\left\{  \text{nontrivial odd }D\text{-cycles}\right\} \\
&  \text{to }\left\{  \text{permutations }\sigma\in\mathfrak{S}_{V}%
^{\operatorname*{odd}}\left(  D\right)  \text{ satisfying }\psi\left(
\sigma\right)  =1\right\}
\end{align*}
which sends each nontrivial odd $D$-cycle $\gamma$ to the permutation
$\operatorname*{perm}\gamma$. This map is furthermore injective (because any
distinct nontrivial $D$-cycles $\gamma$ and $\delta$ will always give rise to
different permutations $\operatorname*{perm}\gamma$ and $\operatorname*{perm}%
\delta$) and surjective\footnote{\textit{Proof.} If $\sigma\in\mathfrak{S}%
_{V}\left(  D\right)  $ is a permutation satisfying $\psi\left(
\sigma\right)  =1$, then $\sigma=\operatorname*{perm}\gamma$ where $\gamma$ is
the unique nontrivial cycle of $\sigma$. Moreover, this cycle $\gamma$ is a
$D$-cycle (since $\sigma\in\mathfrak{S}_{V}\left(  D\right)  $). If we
furthermore assume that $\sigma\in\mathfrak{S}_{V}^{\operatorname*{odd}%
}\left(  D\right)  $, then this cycle $\gamma$ has odd length (since
$\sigma\in\mathfrak{S}_{V}^{\operatorname*{odd}}\left(  D\right)  $ entails
that all cycles of $\sigma$ have odd length), i.e., is odd.}. Thus, this map
is bijective. Hence, the bijection principle yields%
\begin{align*}
&  \left(  \text{\# of nontrivial odd }D\text{-cycles}\right) \\
&  =\left(  \text{\# of permutations }\sigma\in\mathfrak{S}_{V}%
^{\operatorname*{odd}}\left(  D\right)  \text{ satisfying }\psi\left(
\sigma\right)  =1\right)  .
\end{align*}
This proves Lemma \ref{lem.psi=1}.
\end{proof}

We can now prove Theorem \ref{thm.tourn.mod4}:

\begin{proof}
[Proof of Theorem \ref{thm.tourn.mod4}.]We use the same notations as in
Section \ref{sec.redei}. Write the tournament $D$ as $D=\left(  V,A\right)  $.

For each $\sigma\in\mathfrak{S}_{V}$, let $\psi\left(  \sigma\right)  $ denote
the number of nontrivial cycles of $\sigma$. Let $\mathfrak{S}_{V}%
^{\operatorname*{odd}}\left(  D\right)  $ denote the set of all permutations
$\sigma\in\mathfrak{S}_{V}\left(  D\right)  $ such that all cycles of $\sigma$
have odd length. Note that the identity permutation $\operatorname*{id}%
\nolimits_{V}$ belongs to $\mathfrak{S}_{V}^{\operatorname*{odd}}\left(
D\right)  $, since all its cycles are trivial.

Then, from (\ref{pf.thm.tourn.redei.zUD=}), we have%
\begin{align*}
\zeta\left(  U_{D}\right)   &  =\sum_{\substack{\sigma\in\mathfrak{S}%
_{V}\left(  D\right)  ;\\\text{all cycles of }\sigma\text{ have odd length}%
}}2^{\psi\left(  \sigma\right)  }=\sum_{\sigma\in\mathfrak{S}_{V}%
^{\operatorname*{odd}}\left(  D\right)  }2^{\psi\left(  \sigma\right)  }\\
&  \ \ \ \ \ \ \ \ \ \ \ \ \ \ \ \ \ \ \ \ \left(
\begin{array}
[c]{c}%
\text{since the permutations }\sigma\in\mathfrak{S}_{V}\left(  D\right)
\text{ such that all cycles}\\
\text{of }\sigma\text{ have odd length are precisely the elements of
}\mathfrak{S}_{V}^{\operatorname*{odd}}\left(  D\right)
\end{array}
\right) \\
&  \equiv\sum_{\substack{\sigma\in\mathfrak{S}_{V}^{\operatorname*{odd}%
}\left(  D\right)  ;\\\psi\left(  \sigma\right)  =0}}\underbrace{2^{\psi
\left(  \sigma\right)  }}_{\substack{=1\\\text{(since }\psi\left(
\sigma\right)  =0\text{)}}}+\sum_{\substack{\sigma\in\mathfrak{S}%
_{V}^{\operatorname*{odd}}\left(  D\right)  ;\\\psi\left(  \sigma\right)
=1}}\underbrace{2^{\psi\left(  \sigma\right)  }}_{\substack{=2\\\text{(since
}\psi\left(  \sigma\right)  =1\text{)}}}+\sum_{\substack{\sigma\in
\mathfrak{S}_{V}^{\operatorname*{odd}}\left(  D\right)  ;\\\psi\left(
\sigma\right)  \geq2}}\underbrace{2^{\psi\left(  \sigma\right)  }%
}_{\substack{\equiv0\operatorname{mod}4\\\text{(since }\psi\left(
\sigma\right)  \geq2\text{)}}}\\
&  \ \ \ \ \ \ \ \ \ \ \ \ \ \ \ \ \ \ \ \ \left(
\begin{array}
[c]{c}%
\text{here, we have split our sum according to}\\
\text{whether }\psi\left(  \sigma\right)  \text{ is }0\text{ or }1\text{ or
}\geq2
\end{array}
\right) \\
&  \equiv\sum_{\substack{\sigma\in\mathfrak{S}_{V}^{\operatorname*{odd}%
}\left(  D\right)  ;\\\psi\left(  \sigma\right)  =0}}1+\sum_{\substack{\sigma
\in\mathfrak{S}_{V}^{\operatorname*{odd}}\left(  D\right)  ;\\\psi\left(
\sigma\right)  =1}}2+\underbrace{\sum_{\substack{\sigma\in\mathfrak{S}%
_{V}^{\operatorname*{odd}}\left(  D\right)  ;\\\psi\left(  \sigma\right)
\geq2}}0}_{=0}\\
&  =\underbrace{\sum_{\substack{\sigma\in\mathfrak{S}_{V}^{\operatorname*{odd}%
}\left(  D\right)  ;\\\psi\left(  \sigma\right)  =0}}1}_{=\left(  \text{\# of
permutations }\sigma\in\mathfrak{S}_{V}^{\operatorname*{odd}}\left(  D\right)
\text{ satisfying }\psi\left(  \sigma\right)  =0\right)  \cdot1}\\
&  \ \ \ \ \ \ \ \ \ \ +\underbrace{\sum_{\substack{\sigma\in\mathfrak{S}%
_{V}^{\operatorname*{odd}}\left(  D\right)  ;\\\psi\left(  \sigma\right)
=1}}2}_{=\left(  \text{\# of permutations }\sigma\in\mathfrak{S}%
_{V}^{\operatorname*{odd}}\left(  D\right)  \text{ satisfying }\psi\left(
\sigma\right)  =1\right)  \cdot2}\\
&  =\underbrace{\left(  \text{\# of permutations }\sigma\in\mathfrak{S}%
_{V}^{\operatorname*{odd}}\left(  D\right)  \text{ satisfying }\psi\left(
\sigma\right)  =0\right)  }_{\substack{=1\\\text{(since the only permutation
}\sigma\in\mathfrak{S}_{V}^{\operatorname*{odd}}\left(  D\right)
\\\text{satisfying }\psi\left(  \sigma\right)  =0\text{ is the identity
permutation)}}}\cdot\,1\\
&  \ \ \ \ \ \ \ \ \ \ +\underbrace{\left(  \text{\# of permutations }%
\sigma\in\mathfrak{S}_{V}^{\operatorname*{odd}}\left(  D\right)  \text{
satisfying }\psi\left(  \sigma\right)  =1\right)  }_{\substack{=\left(
\text{\# of nontrivial odd }D\text{-cycles}\right)  \\\text{(by Lemma
\ref{lem.psi=1})}}}\cdot\,2\\
&  =1\cdot1+\left(  \text{\# of nontrivial odd }D\text{-cycles}\right)
\cdot2\\
&  =1+2\left(  \text{\# of nontrivial odd }D\text{-cycles}\right)
\operatorname{mod}4.
\end{align*}
Comparing this with%
\[
\zeta\left(  U_{D}\right)  =\left(  \text{\# of hamps of }\overline{D}\right)
\ \ \ \ \ \ \ \ \ \ \left(  \text{by Lemma \ref{lem.zeta.U}}\right)  ,
\]
we obtain%
\begin{align}
&  \left(  \text{\# of hamps of }\overline{D}\right) \nonumber\\
&  \equiv1+2\left(  \text{\# of nontrivial odd }D\text{-cycles}\right)
\operatorname{mod}4. \label{pf.thm.tourn.mod4.3}%
\end{align}

However, recall that $D$ is a tournament. Hence, the tournament axiom shows
that a pair $\left(  u,v\right)  $ of two distinct elements of $V$ is an arc
of $D$ if and only if the pair $\left(  v,u\right)  $ is not. In other words,
a pair $\left(  u,v\right)  $ of two distinct elements of $V$ is an arc of $D$
if and only if the pair $\left(  v,u\right)  $ is an arc of $\overline{D}$.
Thus, if $v=\left(  v_{1},v_{2},\ldots,v_{k}\right)  $ is a hamp of $D$, then
its reversal $\operatorname*{rev}v=\left(  v_{k},v_{k-1},\ldots,v_{1}\right)
$ is a hamp of $\overline{D}$. Hence, we obtain a map%
\begin{align*}
\left\{  \text{hamps of }D\right\}   &  \rightarrow\left\{  \text{hamps of
}\overline{D}\right\}  ,\\
v  &  \mapsto\operatorname*{rev}v.
\end{align*}
This map is furthermore easily seen to be injective and surjective. Hence, it
is bijective. Thus, the bijection principle yields%
\begin{align*}
\left(  \text{\# of hamps of }D\right)   &  =\left(  \text{\# of hamps of
}\overline{D}\right) \\
&  \equiv1+2\left(  \text{\# of nontrivial odd }D\text{-cycles}\right)
\operatorname{mod}4
\end{align*}
(by (\ref{pf.thm.tourn.mod4.3})). This proves Theorem \ref{thm.tourn.mod4}.
\end{proof}

\section{\label{sec.antipode}The antipode and the omega involution}

Next, we will discuss how the Redei-Berge symmetric functions $U_{D}$
interplay with two well-known involutions on the ring $\Lambda$: the omega
involution $\omega$ and the antipode map $S$.

We shall not recall the standard definitions of these involutions $\omega$ and
$S$ (see, e.g., \cite[\S 2.4]{GriRei}); however, we shall briefly state the
few properties that will be used in what follows. Both the \emph{omega
involution} $\omega$ and the \emph{antipode }$S$ of $\Lambda$ are
endomorphisms of the $\mathbb{Z}$-algebra $\Lambda$; they satisfy the
equalities%
\begin{equation}
S\left(  p_{n}\right)  =-p_{n} \label{eq.def.antipode.S}%
\end{equation}
and%
\begin{equation}
\omega\left(  p_{n}\right)  =\left(  -1\right)  ^{n-1}p_{n}
\label{eq.def.antipode.omega}%
\end{equation}
for every positive integer $n$ (see \cite[Proposition 2.4.1 (i)]{GriRei} and
\cite[Proposition 2.4.3 (c)]{GriRei}). Moreover, if $f\in\Lambda$ is a
homogeneous power series of degree $n$, then%
\begin{equation}
S\left(  f\right)  =\left(  -1\right)  ^{n}\omega\left(  f\right)
\label{eq.def.antipode.Somega}%
\end{equation}
(this is \cite[Proposition 2.4.3 (e)]{GriRei}). We now claim the following theorem:

\begin{theorem}
\label{thm.antipode.U}Let $D=\left(  V,A\right)  $ be a digraph. Then,%
\begin{equation}
\omega\left(  U_{D}\right)  =U_{\overline{D}}. \label{eq.thm.antipode.U.omega}%
\end{equation}
Furthermore, if $n:=\left\vert V\right\vert $, then%
\begin{equation}
S\left(  U_{D}\right)  =\left(  -1\right)  ^{n}U_{\overline{D}}.
\label{eq.thm.antipode.U.S}%
\end{equation}

\end{theorem}

\begin{proof}
The definition of $\overline{D}$ yields that $\overline{\overline{D}}=D$.
Hence, the definition of $\mathfrak{S}_{V}\left(  D,\overline{D}\right)  $
yields that $\mathfrak{S}_{V}\left(  \overline{D},D\right)  =\mathfrak{S}%
_{V}\left(  D,\overline{D}\right)  $.

For each $\sigma\in\mathfrak{S}_{V}$, we set%
\[
\varphi\left(  \sigma\right)  :=\sum_{\substack{\gamma\in\operatorname*{Cycs}%
\sigma;\\\gamma\text{ is a }D\text{-cycle}}}\left(  \ell\left(  \gamma\right)
-1\right)  \ \ \ \ \ \ \ \ \ \ \text{and}\ \ \ \ \ \ \ \ \ \ \overline
{\varphi}\left(  \sigma\right)  :=\sum_{\substack{\gamma\in
\operatorname*{Cycs}\sigma;\\\gamma\text{ is a }\overline{D}\text{-cycle}%
}}\left(  \ell\left(  \gamma\right)  -1\right)  .
\]

Now, it is easy to see that%
\begin{equation}
\omega\left(  \left(  -1\right)  ^{\varphi\left(  \sigma\right)
}p_{\operatorname*{type}\sigma}\right)  =\left(  -1\right)  ^{\overline
{\varphi}\left(  \sigma\right)  }p_{\operatorname*{type}\sigma}
\label{pf.thm.antipode.U.4}%
\end{equation}
for each $\sigma\in\mathfrak{S}_{V}\left(  D,\overline{D}\right)  $.

[\textit{Proof of (\ref{pf.thm.antipode.U.4}):} Let $\sigma\in\mathfrak{S}%
_{V}\left(  D,\overline{D}\right)  $. Let $k_{1},k_{2},\ldots,k_{s}$ be the
lengths of all cycles of $\sigma$, listed in decreasing order. Then, the
definition of $\operatorname*{type}\sigma$ yields $\operatorname*{type}%
\sigma=\left(  k_{1},k_{2},\ldots,k_{s}\right)  $. Hence,
\begin{equation}
p_{\operatorname*{type}\sigma}=p_{\left(  k_{1},k_{2},\ldots,k_{s}\right)
}=p_{k_{1}}p_{k_{2}}\cdots p_{k_{s}}=\prod_{\gamma\in\operatorname*{Cycs}%
\sigma}p_{\ell\left(  \gamma\right)  } \label{pf.thm.antipode.U.4.pf.0}%
\end{equation}
(since $k_{1},k_{2},\ldots,k_{s}$ are the lengths of all cycles of $\sigma$).
Hence,%
\begin{align}
\omega\left(  \left(  -1\right)  ^{\varphi\left(  \sigma\right)
}p_{\operatorname*{type}\sigma}\right)   &  =\omega\left(  \left(  -1\right)
^{\varphi\left(  \sigma\right)  }\prod_{\gamma\in\operatorname*{Cycs}\sigma
}p_{\ell\left(  \gamma\right)  }\right) \nonumber\\
&  =\left(  -1\right)  ^{\varphi\left(  \sigma\right)  }\prod_{\gamma
\in\operatorname*{Cycs}\sigma}\underbrace{\omega\left(  p_{\ell\left(
\gamma\right)  }\right)  }_{\substack{=\left(  -1\right)  ^{\ell\left(
\gamma\right)  -1}p_{\ell\left(  \gamma\right)  }\\\text{(by
(\ref{eq.def.antipode.omega}))}}}\ \ \ \ \ \ \ \ \ \ \left(
\begin{array}
[c]{c}%
\text{since }\omega\text{ is a }\mathbb{Z}\text{-algebra}\\
\text{homomorphism}%
\end{array}
\right) \nonumber\\
&  =\left(  -1\right)  ^{\varphi\left(  \sigma\right)  }\prod_{\gamma
\in\operatorname*{Cycs}\sigma}\left(  \left(  -1\right)  ^{\ell\left(
\gamma\right)  -1}p_{\ell\left(  \gamma\right)  }\right) \nonumber\\
&  =\left(  -1\right)  ^{\varphi\left(  \sigma\right)  }\underbrace{\left(
\prod_{\gamma\in\operatorname*{Cycs}\sigma}\left(  -1\right)  ^{\ell\left(
\gamma\right)  -1}\right)  }_{=\left(  -1\right)  ^{\sum_{\gamma
\in\operatorname*{Cycs}\sigma}\left(  \ell\left(  \gamma\right)  -1\right)  }%
}\underbrace{\prod_{\gamma\in\operatorname*{Cycs}\sigma}p_{\ell\left(
\gamma\right)  }}_{\substack{=p_{\operatorname*{type}\sigma}\\\text{(by
(\ref{pf.thm.antipode.U.4.pf.0}))}}}\nonumber\\
&  =\left(  -1\right)  ^{\varphi\left(  \sigma\right)  }\left(  -1\right)
^{\sum_{\gamma\in\operatorname*{Cycs}\sigma}\left(  \ell\left(  \gamma\right)
-1\right)  }p_{\operatorname*{type}\sigma}. \label{pf.thm.antipode.U.4.pf.1}%
\end{align}

However, each $\gamma\in\operatorname*{Cycs}\sigma$ is either a $D$-cycle or a
$\overline{D}$-cycle (since $\sigma\in\mathfrak{S}_{V}\left(  D,\overline
{D}\right)  $), but cannot be both at the same time (since $D$ and
$\overline{D}$ have no arcs in common). Thus,%
\begin{align*}
\sum_{\gamma\in\operatorname*{Cycs}\sigma}\left(  \ell\left(  \gamma\right)
-1\right)   &  =\underbrace{\sum_{\substack{\gamma\in\operatorname*{Cycs}%
\sigma;\\\gamma\text{ is a }D\text{-cycle}}}\left(  \ell\left(  \gamma\right)
-1\right)  }_{\substack{=\varphi\left(  \sigma\right)  \\\text{(by the
definition of }\varphi\left(  \sigma\right)  \text{)}}}+\underbrace{\sum
_{\substack{\gamma\in\operatorname*{Cycs}\sigma;\\\gamma\text{ is a }%
\overline{D}\text{-cycle}}}\left(  \ell\left(  \gamma\right)  -1\right)
}_{\substack{=\overline{\varphi}\left(  \sigma\right)  \\\text{(by the
definition of }\overline{\varphi}\left(  \sigma\right)  \text{)}}}\\
&  =\varphi\left(  \sigma\right)  +\overline{\varphi}\left(  \sigma\right)  .
\end{align*}
Thus, (\ref{pf.thm.antipode.U.4.pf.1}) rewrites as%
\[
\omega\left(  \left(  -1\right)  ^{\varphi\left(  \sigma\right)
}p_{\operatorname*{type}\sigma}\right)  =\underbrace{\left(  -1\right)
^{\varphi\left(  \sigma\right)  }\left(  -1\right)  ^{\varphi\left(
\sigma\right)  +\overline{\varphi}\left(  \sigma\right)  }}_{=\left(
-1\right)  ^{\overline{\varphi}\left(  \sigma\right)  }}%
p_{\operatorname*{type}\sigma}=\left(  -1\right)  ^{\overline{\varphi}\left(
\sigma\right)  }p_{\operatorname*{type}\sigma}.
\]
This proves (\ref{pf.thm.antipode.U.4}).] \medskip

Now, Theorem \ref{thm.UX.1} yields%
\begin{equation}
U_{D}=\sum_{\sigma\in\mathfrak{S}_{V}\left(  D,\overline{D}\right)  }\left(
-1\right)  ^{\varphi\left(  \sigma\right)  }p_{\operatorname*{type}\sigma}.
\label{pf.thm.antipode.U.8}%
\end{equation}
Also, Theorem \ref{thm.UX.1} (applied to $\overline{D}$, $\left(  V\times
V\right)  \setminus A$ and $\overline{\varphi}$ instead of $D$, $A$ and
$\varphi$) yields%
\begin{align*}
U_{\overline{D}}  &  =\sum_{\sigma\in\mathfrak{S}_{V}\left(  \overline
{D},D\right)  }\left(  -1\right)  ^{\overline{\varphi}\left(  \sigma\right)
}p_{\operatorname*{type}\sigma}\\
&  =\sum_{\sigma\in\mathfrak{S}_{V}\left(  D,\overline{D}\right)
}\underbrace{\left(  -1\right)  ^{\overline{\varphi}\left(  \sigma\right)
}p_{\operatorname*{type}\sigma}}_{\substack{=\omega\left(  \left(  -1\right)
^{\varphi\left(  \sigma\right)  }p_{\operatorname*{type}\sigma}\right)
\\\text{(by (\ref{pf.thm.antipode.U.4}))}}}\ \ \ \ \ \ \ \ \ \ \left(
\text{since }\mathfrak{S}_{V}\left(  \overline{D},D\right)  =\mathfrak{S}%
_{V}\left(  D,\overline{D}\right)  \right) \\
&  =\sum_{\sigma\in\mathfrak{S}_{V}\left(  D,\overline{D}\right)  }%
\omega\left(  \left(  -1\right)  ^{\varphi\left(  \sigma\right)
}p_{\operatorname*{type}\sigma}\right) \\
&  =\omega\left(  \underbrace{\sum_{\sigma\in\mathfrak{S}_{V}\left(
D,\overline{D}\right)  }\left(  -1\right)  ^{\varphi\left(  \sigma\right)
}p_{\operatorname*{type}\sigma}}_{\substack{=U_{D}\\\text{(by
(\ref{pf.thm.antipode.U.8}))}}}\right)  \ \ \ \ \ \ \ \ \ \ \left(
\text{since }\omega\text{ is }\mathbb{Z}\text{-linear}\right) \\
&  =\omega\left(  U_{D}\right)  .
\end{align*}
This proves (\ref{eq.thm.antipode.U.omega}).

Now, let $n:=\left\vert V\right\vert $. Then, the definition of $U_{D}$ easily
yields that $U_{D}$ is homogeneous of degree $n$. Hence,
(\ref{eq.def.antipode.Somega}) (applied to $f=U_{D}$) yields
\[
S\left(  U_{D}\right)  =\left(  -1\right)  ^{n}\omega\left(  U_{D}\right)
=\left(  -1\right)  ^{n}U_{\overline{D}}\ \ \ \ \ \ \ \ \ \ \left(  \text{by
(\ref{eq.thm.antipode.U.omega})}\right)  .
\]
Thus, (\ref{eq.thm.antipode.U.S}) is proved. This completes the proof of
Theorem \ref{thm.antipode.U}.
\end{proof}

Theorem \ref{thm.antipode.U} can also be proved directly from the definition
of $U_{D}$, using the formula for the antipode of a fundamental quasisymmetric
function (\cite[(5.2.7)]{GriRei}). Indeed, three different proofs of Theorem
\ref{thm.antipode.U} (specifically, of (\ref{eq.thm.antipode.U.omega})) are
found in \cite{Chow96} (where (\ref{eq.thm.antipode.U.omega}) appears as
\cite[Corollary 2]{Chow96}), one of which is doing just this. A fourth proof
can be found in \cite[(6)]{Wisema07}.

We can use Theorem \ref{thm.antipode.U} to give a new proof of Berge's theorem
(Theorem \ref{thm.hamp.Dbar}). For this purpose, we recall the $\mathbb{Z}%
$-algebra homomorphism $\zeta$ introduced in Section \ref{sec.redei}. We need
another simple property of this $\zeta$:

\begin{lemma}
\label{lem.zeta.parity}Let $f\in\mathbb{Z}\left[  p_{1},p_{2},p_{3}%
,\ldots\right]  $. Then, $\zeta\left(  \omega\left(  f\right)  \right)
\equiv\zeta\left(  f\right)  \operatorname{mod}2$.
\end{lemma}

\begin{proof}
Let $\pi:\mathbb{Z}\rightarrow\mathbb{Z}/2$ be the projection map that sends
each integer to its congruence class modulo $2$. This $\pi$ is a $\mathbb{Z}%
$-algebra homomorphism.

For each positive integer $n$, we have
\begin{align*}
\zeta\left(  \omega\left(  p_{n}\right)  \right)   &  =\zeta\left(  \left(
-1\right)  ^{n-1}p_{n}\right)  \ \ \ \ \ \ \ \ \ \ \left(  \text{by
(\ref{eq.def.antipode.omega})}\right) \\
&  =\underbrace{\left(  -1\right)  ^{n-1}}_{\equiv1\operatorname{mod}2}%
\zeta\left(  p_{n}\right)  \ \ \ \ \ \ \ \ \ \ \left(  \text{since }%
\zeta\text{ is }\mathbb{Z}\text{-linear}\right) \\
&  \equiv\zeta\left(  p_{n}\right)  \operatorname{mod}2
\end{align*}
and thus%
\[
\pi\left(  \zeta\left(  \omega\left(  p_{n}\right)  \right)  \right)
=\pi\left(  \zeta\left(  p_{n}\right)  \right)
\]
(since two integers $a$ and $b$ satisfy $a\equiv b\operatorname{mod}2$ if and
only if $\pi\left(  a\right)  =\pi\left(  b\right)  $). In other words, for
each positive integer $n$, we have%
\[
\left(  \pi\circ\zeta\circ\omega\right)  \left(  p_{n}\right)  =\left(
\pi\circ\zeta\right)  \left(  p_{n}\right)  .
\]
In other words, the two maps $\pi\circ\zeta\circ\omega$ and $\pi\circ\zeta$
agree on each of the generators $p_{1},p_{2},p_{3},\ldots$ of the $\mathbb{Z}%
$-algebra $\mathbb{Z}\left[  p_{1},p_{2},p_{3},\ldots\right]  $. Since these
two maps are $\mathbb{Z}$-algebra homomorphisms (because $\pi$, $\zeta$ and
$\omega$ are $\mathbb{Z}$-algebra homomorphisms), this shows that these two
maps agree on the entire $\mathbb{Z}$-algebra $\mathbb{Z}\left[  p_{1}%
,p_{2},p_{3},\ldots\right]  $. Hence, $\left(  \pi\circ\zeta\circ
\omega\right)  \left(  f\right)  =\left(  \pi\circ\zeta\right)  \left(
f\right)  $. In other words, $\pi\left(  \zeta\left(  \omega\left(  f\right)
\right)  \right)  =\pi\left(  \zeta\left(  f\right)  \right)  $. In other
words, $\zeta\left(  \omega\left(  f\right)  \right)  \equiv\zeta\left(
f\right)  \operatorname{mod}2$ (since two integers $a$ and $b$ satisfy
$a\equiv b\operatorname{mod}2$ if and only if $\zeta\left(  a\right)
=\zeta\left(  b\right)  $). This proves Lemma \ref{lem.zeta.parity}.
\end{proof}

\begin{proof}
[Second proof of Theorem \ref{thm.hamp.Dbar}.]From
(\ref{eq.thm.antipode.U.omega}), we obtain $\omega\left(  U_{D}\right)
=U_{\overline{D}}$.

Corollary \ref{cor.UX.p-int} yields $U_{D}\in\mathbb{Z}\left[  p_{1}%
,p_{2},p_{3},\ldots\right]  $. Hence, Lemma \ref{lem.zeta.parity} (applied to
$f=U_{D}$) yields that%
\[
\zeta\left(  \omega\left(  U_{D}\right)  \right)  \equiv\zeta\left(
U_{D}\right)  \operatorname{mod}2.
\]
In view of%
\[
\zeta\left(  U_{D}\right)  =\left(  \text{\# of hamps of }\overline{D}\right)
\ \ \ \ \ \ \ \ \ \ \left(  \text{by Lemma \ref{lem.zeta.U}}\right)
\]
and%
\begin{align*}
\zeta\left(  \underbrace{\omega\left(  U_{D}\right)  }_{=U_{\overline{D}}%
}\right)   &  =\zeta\left(  U_{\overline{D}}\right)  =\left(  \text{\# of
hamps of }\overline{\overline{D}}\right)  \ \ \ \ \ \ \ \ \ \ \left(
\begin{array}
[c]{c}%
\text{by Lemma \ref{lem.zeta.U},}\\
\text{applied to }\overline{D}\text{ instead of }D
\end{array}
\right) \\
&  =\left(  \text{\# of hamps of }D\right)  \ \ \ \ \ \ \ \ \ \ \left(
\text{since }\overline{\overline{D}}=D\right)  ,
\end{align*}
we can rewrite this as%
\[
\left(  \text{\# of hamps of }D\right)  \equiv\left(  \text{\# of hamps of
}\overline{D}\right)  \operatorname{mod}2.
\]
This proves Theorem \ref{thm.hamp.Dbar} again.
\end{proof}

\section{\label{sec.matrixgen}A multiparameter deformation}

Let us now briefly discuss a multiparameter deformation of the Redei-Berge
symmetric functions $U_{D}$, which replaces the digraph $D$ by an arbitrary matrix.

We fix a commutative ring $\mathbf{k}$, which we shall now be using instead of
$\mathbb{Z}$ as a base ring for our power series.

We fix an $n\in\mathbb{N}$, and a set $V$ with $n$ elements.

For any $a\in V\times V$, we fix an element $t_{a}\in\mathbf{k}$. (Thus, the
family $\left(  t_{\left(  i,j\right)  }\right)  _{i,j\in V}$ of these
elements can be viewed as a $V\times V$-matrix.)

For any $a\in V\times V$, we set $s_{a}:=t_{a}+1\in\mathbf{k}$.

The following definition is inspired by a comment from Mike Zabrocki:

\begin{definition}
We define the \emph{deformed Redei--Berge symmetric function} $\widetilde{U}%
_{t}$ to be the formal power series%
\begin{align*}
\widetilde{U}_{t}  &  =\sum_{\substack{w=\left(  w_{1},w_{2},\ldots
,w_{n}\right)  \\\text{is a }V\text{-listing}}}\ \ \sum_{i_{1}\leq i_{2}%
\leq\cdots\leq i_{n}}\left(  \prod_{\substack{k\in\left[  n-1\right]
;\\i_{k}=i_{k+1}}}s_{\left(  w_{k},w_{k+1}\right)  }\right)  x_{i_{1}}%
x_{i_{2}}\cdots x_{i_{n}}\\
&  \in\mathbf{k}\left[  \left[  x_{1},x_{2},x_{3},\ldots\right]  \right]  .
\end{align*}

\end{definition}

For example, if $n=2$ and $V=\left\{  1,2\right\}  $, then%
\begin{align*}
\widetilde{U}_{t}  &  =\sum_{i_{1}<i_{2}}x_{i_{1}}x_{i_{2}}+\sum_{i_{1}=i_{2}%
}t_{\left(  1,2\right)  }x_{i_{1}}x_{i_{2}}+\sum_{i_{1}<i_{2}}x_{i_{1}%
}x_{i_{2}}+\sum_{i_{1}=i_{2}}t_{\left(  2,1\right)  }x_{i_{1}}x_{i_{2}}\\
&  =\sum_{i<j}x_{i}x_{j}+\sum_{i}t_{\left(  1,2\right)  }x_{i}^{2}+\sum
_{i<j}x_{i}x_{j}+\sum_{i}t_{\left(  2,1\right)  }x_{i}^{2}\\
&  =p_{1}^{2}+\left(  s_{\left(  1,2\right)  }+s_{\left(  2,1\right)
}-1\right)  p_{2}\\
&  =p_{1}^{2}+\left(  t_{\left(  1,2\right)  }+t_{\left(  2,1\right)
}+1\right)  p_{2}%
\end{align*}

For a more complicated example, if $n=2$ and $V=\left\{  1,2,3\right\}  $,
then a longer computation shows that%
\begin{align*}
\widetilde{U}_{t}  &  =p_{1}^{3}+\left(  s_{\left(  1,2\right)  }+s_{\left(
2,1\right)  }+s_{\left(  1,3\right)  }+s_{\left(  3,1\right)  }+s_{\left(
2,3\right)  }+s_{\left(  3,2\right)  }-3\right)  p_{2}p_{1}\\
&  \ \ \ \ \ \ \ \ \ \ +\left(  s_{\left(  1,2\right)  }s_{\left(  2,3\right)
}+s_{\left(  2,3\right)  }s_{\left(  3,1\right)  }+s_{\left(  3,1\right)
}s_{\left(  1,2\right)  }\right. \\
&  \ \ \ \ \ \ \ \ \ \ \ \ \ \ \ \ \ \ \ \ \left.  +\ s_{\left(  1,3\right)
}s_{\left(  3,2\right)  }+s_{\left(  3,2\right)  }s_{\left(  2,1\right)
}+s_{\left(  2,1\right)  }t_{\left(  1,3\right)  }\right. \\
&  \ \ \ \ \ \ \ \ \ \ \ \ \ \ \ \ \ \ \ \ \left.  -\ s_{\left(  1,2\right)
}-s_{\left(  2,1\right)  }-s_{\left(  1,3\right)  }-s_{\left(  3,1\right)
}-s_{\left(  2,3\right)  }-s_{\left(  3,2\right)  }+2\right)  p_{3}\\
&  =p_{1}^{3}+\left(  t_{\left(  1,2\right)  }+t_{\left(  2,1\right)
}+t_{\left(  1,3\right)  }+t_{\left(  3,1\right)  }+t_{\left(  2,3\right)
}+t_{\left(  3,2\right)  }+3\right)  p_{2}p_{1}\\
&  \ \ \ \ \ \ \ \ \ \ +\left(  t_{\left(  1,2\right)  }t_{\left(  2,3\right)
}+t_{\left(  2,3\right)  }t_{\left(  3,1\right)  }+t_{\left(  3,1\right)
}t_{\left(  1,2\right)  }\right. \\
&  \ \ \ \ \ \ \ \ \ \ \ \ \ \ \ \ \ \ \ \ \left.  +\ t_{\left(  1,3\right)
}t_{\left(  3,2\right)  }+t_{\left(  3,2\right)  }t_{\left(  2,1\right)
}+t_{\left(  2,1\right)  }t_{\left(  1,3\right)  }\right. \\
&  \ \ \ \ \ \ \ \ \ \ \ \ \ \ \ \ \ \ \ \ \left.  +\ t_{\left(  1,2\right)
}+t_{\left(  2,1\right)  }+t_{\left(  1,3\right)  }+t_{\left(  3,1\right)
}+t_{\left(  2,3\right)  }+t_{\left(  3,2\right)  }+2\right)  p_{3}.
\end{align*}

Why are we calling $\widetilde{U}_{t}$ a deformation of $U_{D}$ ?

\begin{example}
Let $D=\left(  V,A\right)  $ be a digraph. Set $\mathbf{k}=\mathbb{Z}$, and
let%
\[
t_{a}:=%
\begin{cases}
-1, & \text{if }a\in A;\\
0, & \text{if }a\notin A
\end{cases}
\ \ \ \ \ \ \ \ \ \ \text{for each }a\in V\times V.
\]
Then, $\widetilde{U}_{t}=U_{D}$, as can be seen by comparing the definitions.
\end{example}

All the above results leading up to Theorem \ref{thm.UX.1} can be extended to
this deformation, culminating in the following deformation of Theorem
\ref{thm.UX.1}:

\begin{theorem}
\label{thm.UX.1deform}We have%
\[
\widetilde{U}_{t}=\sum_{\sigma\in\mathfrak{S}_{V}}\left(  \prod_{\gamma\text{
is a cycle of }\sigma}\left(  \prod_{a\in\operatorname*{CArcs}\gamma}%
s_{a}-\prod_{a\in\operatorname*{CArcs}\gamma}t_{a}\right)  \right)
p_{\operatorname*{type}\sigma}.
\]

\end{theorem}

Alternatively, Theorem \ref{thm.UX.1deform} can be deduced from Theorem
\ref{thm.UX.1} via the \textquotedblleft multilinearity
trick\textquotedblright: View each $t_{a}$ as an indeterminate, and argue that
both sides in Theorem \ref{thm.UX.1deform} are polynomials in degree $\leq1$
in these indeterminates (over the base ring $\mathbf{k}\left[  \left[
x_{1},x_{2},x_{3},\ldots\right]  \right]  $). Thus, in order to prove their
equality, it suffices to prove that they are equal when each $t_{a}$ is
specialized to either $0$ or $-1$. But this is precisely the claim of Theorem
\ref{thm.UX.1}. (Thus, Theorem \ref{thm.UX.1deform} is not essentially more
general than Theorem \ref{thm.UX.1}.)

Theorem \ref{thm.UX.1deform} shows that the $\widetilde{U}_{t}$ are
$p$-integral symmetric functions (taking the $t_{a}$ as \textquotedblleft
integers\textquotedblright). There do not seem to be any good opportunities
for generalizing any of Theorem \ref{thm.UX.2} and Theorem \ref{thm.UX.3}, however.

\end{document}